\pdfoutput=1
\RequirePackage{ifpdf}
\ifpdf 
\documentclass[pdftex]{sigma}
\else
\documentclass{sigma}
\fi

\usepackage{amscd}
\usepackage{mathrsfs}
\usepackage{mathdots}
\usepackage{tikz}
\usepackage{euscript}
\usepackage[all]{xy}

\numberwithin{equation}{section}

\newtheorem{Theorem}{Theorem}[section]
\newtheorem*{Theorem*}{Theorem}
\newtheorem{Corollary}[Theorem]{Corollary}
\newtheorem{Lemma}[Theorem]{Lemma}
\newtheorem{Proposition}[Theorem]{Proposition}
\newtheorem{Assumption}[Theorem]{Assumption}
\newtheorem{Claim}[Theorem]{Claim}

 { \theoremstyle{definition}
\newtheorem{Definition}[Theorem]{Definition}

\newtheorem{Example}[Theorem]{Example}
\newtheorem{Remark}[Theorem]{Remark}
}

\newcommand\rank{\mathop{\rm rank}\nolimits}
\newcommand\im{\mathop{\rm Im}\nolimits}
\newcommand\coker{\mathop{\rm coker}\nolimits}

\newcommand{\Tr}{\mathop{\rm Tr}\nolimits}
\newcommand\Hom{\mathop{\rm Hom}\nolimits}

\newcommand\End{\mathop{\rm End}\nolimits}

\newcommand\Spec{\mathop{\rm Spec}\nolimits}
\newcommand\Hilb{\mathop{\rm Hilb}\nolimits}
\newcommand{\length}{\mathop{\rm length}\nolimits}
\newcommand{\res}{\mathop{\sf res}\nolimits}

\newcommand{\Sym}{\mathop{\rm Sym}\nolimits}

\newcommand{\balpha}{\boldsymbol \alpha}

\begin{document}
\allowdisplaybreaks

\newcommand{\arXivNumber}{2108.09667}

\renewcommand{\PaperNumber}{013}

\FirstPageHeading

\ShortArticleName{Moduli Space of Factorized Ramified Connections}

\ArticleName{Moduli Space of Factorized Ramified Connections\\ and Generalized Isomonodromic Deformation}

\Author{Michi-aki INABA}

\AuthorNameForHeading{M.-a.~Inaba}

\Address{Department of Mathematics, Kyoto University, Kyoto 606-8502, Japan}
\Email{\href{mailto:email@address}{inaba@math.kyoto-u.ac.jp}}

\ArticleDates{Received January 21, 2022, in final form February 23, 2023; Published online March 22, 2023}

\Abstract{We introduce the notion of factorized ramified structure on a generic ramified irregular singular connection on a smooth projective curve. By using the deformation theory of connections with factorized ramified structure, we construct a canonical 2-form on the moduli space of ramified connections. Since the factorized ramified structure provides a~duality on the tangent space of the moduli space, the 2-form becomes nondegenerate. We prove that the 2-form on the moduli space of ramified connections is ${\rm d}$-closed via constructing an unfolding of the moduli space. Based on the Stokes data, we introduce the notion of local generalized isomonodromic deformation for generic unramified irregular singular connections on a unit disk. Applying the Jimbo--Miwa--Ueno theory to generic unramified connections, the local generalized isomonodromic deformation
is equivalent to the extendability of the family of connections to an integrable connection. We give the same statement for ramified connections. Based on this principle of Jimbo--Miwa--Ueno theory, we construct a global generalized isomonodromic deformation on the moduli space of generic ramified connections by constructing a horizontal lift of a universal family of connections. As a consequence of the global generalized isomonodromic deformation, we can lift the relative symplectic form on the moduli space to a total closed form, which is called a generalized isomonodromic 2-form.\looseness=-1}


\Classification{14D20; 53D30; 34M56; 34M40}

\vspace{-6mm}

{\small \setcounter{tocdepth}{1}
\tableofcontents}

\section{Introduction}\label{section1}

Let $C$ be a complex smooth projective curve and $D$ be an effective divisor on $C$.
Consider an algebraic vector bundle $E$ on $C$ of rank $r$ and a rational connection
$\nabla\colon E\longrightarrow E\otimes\Omega_C(D)$
admitting poles along $D$.
The connection $\nabla$ is said to be logarithmic at $x\in D$
if it has at most a simple pole at~$x$.
The notion of logarithmic connection is well formulated in~\cite{Simpson-0}
by adding parabolic structure on the underlying vector bundle.
In~\cite{Simpson-0}, C.T.~Simpson established the Riemann--Hilbert correspondence
as an isomorphism between parabolic logarithmic connections and filtered local systems.
The most important point of~\cite{Simpson-0} is the non-abelian Hodge theory,
which connects parabolic logarithmic connections with parabolic Higgs bundles
through a harmonic metric.
Its effect on the geometry of the corresponding two algebraic moduli spaces
seems mysterious to the author.

The connection $\nabla$ is said to be irregular singular at $x\in D$,
if it cannot be reduced to a~logarithmic connection via a meromorphic transform around~$x$.
So the order of pole of $\nabla$ at $x$ is at least two.
An irregular singular connection $\nabla$ is locally written
$\nabla|_U={\rm d}+A(z){\rm d}z/z^m$
for a~matrix $A(z)$ of holomorphic functions in $z$,
where $m$ is the order of pole of $\nabla$ at $x$ and $z$ is a~local holomorphic coordinate
on a neighborhood $U$ of~$x$.
We say that~$\nabla$ is generic unramified at~$x$ if
the leading term~$A(0)$ has $r$ distinct eigenvalues.
Among the irregular singular connections,
a generic unramified connection is of most generic type.
The next generic irregular singular connections are
generic ramified connections.
In this paper, we say that a connection $(E,\nabla)$ is generic $\nu$-ramified at $x$ if
the formal completion $\big(\widehat{E},\widehat{\nabla}\big)$ at $x$ is isomorphic to
$(\mathbb{C}[[w]],\nabla_{\nu})$,
where $w=z^{\frac{1}{r}}$,
$\nu(w)\in\sum_{l=0}^{mr-r} \mathbb{C}w^l{\rm d}w/w^{mr-r+1}$,
the formal connection $\nabla_{\nu}$ is defined by
\begin{equation}
\label{equation: naive generic ramified connection}
 \nabla_{\nu(w)}\colon \ \mathbb{C}[[w]] \ni f(w)
 \mapsto {\rm d}f(w)+f(w)\nu(w)
 \in \mathbb{C}[[w]]\otimes\frac{{\rm d}z}{z^m}
\end{equation}
and the $w{\rm d}w/w^{mr-r+1}$-coefficient of $\nu(w)$ does not vanish.

The moduli space of logarithmic connections is well formulated by adding the
parabolic structure and it is smooth and has a symplectic structure.
It is constructed in the work with K.~Iwasaki and M.-H.~Saito in~\cite{Inaba-1, IIS-1}.
For unramified irregular singular connections,
the moduli space is analytically constructed by O.~Biquard and P.~Boalch in~\cite{Biquard-Boalch} together with establishing the
non-abelian Hodge theory.
The algebraic construction of the moduli space of generic unramified irregular singular connections
was done in the work with Masa-Hiko Saito in~\cite{Inaba-Saito}
by using the same method as in the logarithmic case.
Compared with the unramified connections,
it is a~more difficult task to construct the
moduli space of ramified connections.
Over the trivial bundle on $\mathbb{P}^1$,
Bremer and Sage construct, in~\cite{Bremer-Sage-1},
the moduli space of ramified connections
via a careful consideration of the formal ramified structure
from a viewpoint of representation theory.
In a higher genus case, the moduli space of ramified connections of generic ramified type
is constructed by the author in \cite{Inaba-2}.
T.~Pantev and B.~To\"{e}n introduce in~\cite{Pantev-Toen} the derived geometric approach to
the moduli space of connections in a general abstract setting.

Both in logarithmic and unramified irregular singular cases,
the moduli space of connections has a natural symplectic structure.
Roughly speaking, the moduli space of parabolic logarithmic connections
is a torsor over the moduli space of parabolic bundles,
which is locally isomorphic to the cotangent bundle.
So the moduli space has a natural symplectic structure,
though we precisely need a more careful consideration to the locus of
non-simple underlying parabolic bundles.
The method of parabolic structure is also valid for the construction of symplectic form
on the moduli space of unramified irregular singular connections.
However, in the case of ramified irregular singular connections,
the method of parabolic structure does not go well with the construction of symplectic form.
In \cite[Theorem 4.1]{Inaba-2}, we proved the existence of a symplectic form
on the moduli space of ramified connections,
but the proof of nondegeneracy was not given directly.
It is reduced to the nondegeneracy
in the case of logarithmic or unramified irregular singular connections
by using an argument of codimension.
So, in \cite{Inaba-2}, we could not find a duality on the tangent space
like in logarithmic or unramified irregular singular case.
In this paper, we introduce the notion of factorized ramified structure,
which supplies the place of parabolic structure.
It induces a canonical duality on the tangent space of
the moduli space of ramified connections which was not done in \cite{Inaba-2}.
In order to see it easily, we adopt a simpler setting than~\cite{Inaba-2},
while we follow almost the same formulation of the moduli space
constructed in~\cite{Inaba-2}.\looseness=-1

Let us see a rough idea of factorized ramified structure.
Assume that a rank $r$ irregular singular connection $(E,\nabla)$ is formally isomorphic to
$(\mathbb{C}[[w]],\nabla_{\nu(w)})$ at $x$ for
$\nabla_{\nu(w)}$ defined in~(\ref{equation: naive generic ramified connection}).
Let $N$ be the endomorphism of $E|_{mx}$,
which corresponds to the action of $w$ on $\mathbb{C}[w]/(w^{mr})$.
Then we can consider the ${\mathcal O}_{mx}[T]$-module structure on
$E|_{mx}$ defined by $P(T)v:=P(N)v$ for a polynomial
$P(T)$ in ${\mathcal O}_{mx}[T]$.
By the elementary linear algebra, we can see that there is an isomorphism
${\mathcal O}_{mx}[T]/(T^r-z)\xrightarrow{\sim} E|_{mx}$
of ${\mathcal O}_{mx}[T]$-modules.
The dual $E|_{mx}^{\vee}$ also has the ${\mathcal O}_{mx}[T]$-module structure
via the map $^tN$ and we have an isomorphism
$E|_{mx}^{\vee} \xrightarrow{\sim} {\mathcal O}_{mx}[T]/(T^r-z)$
of ${\mathcal O}_{mx}[T]$-modules.
Composing these isomorphisms, we get an isomorphism
$\theta\colon E|_{mx}^{\vee} \stackrel{\sim}\longrightarrow E|_{mx}$
of ${\mathcal O}_{mx}[T]$-modules.
Set $\kappa:=\theta^{-1}\circ N$.
Then $\theta$ induces a perfect pairing
$\vartheta\colon E|_{mx}^{\vee}\times E|_{mx}^{\vee}\longrightarrow{\mathcal O}_{mx}$
which becomes symmetric
and $\kappa$ induces a pairing
$\varkappa\colon E|_{mx}\times E|_{mx} \longrightarrow {\mathcal O}_{mx}$
which is also symmetric.
Roughly speaking, a factorized ramified structure on $(E,\nabla)$ at $x$ is
given by $(\theta,\kappa)$ or~$(\vartheta,\varkappa)$.

The purpose of introducing factorized ramified structure
is to construct a duality on the tangent space of the moduli space.
So we require it to go well with the deformation theory.
In that context, all the conditions for the connection $(E,\nabla)$
should be given only by the restriction $(E,\nabla)|_{mx}$ to the divisor $mx$
and the rational one form $\nu(w)$ should be considered modulo holomorphic forms in~$w$.
Under such setting, the endomorphism $N$ on $E|_{mx}$ in fact has an ambiguity
in $z^{m-1}$-term, while the restriction $N|_{(m-1)x}$ is uniquely determined
from~$\nabla|_{mx}$ and~$\nu(w)$.
We take account of this ambiguity in the precise formulation of
factorized $\nu$-ramified structure in
Definition \ref{def-fac-connection}.

In Section~\ref{section: definition regular singular, unramified or ramified structure},
we introduce the notion of logarithmic $\lambda$-parabolic structure
and that of generic unramified $\mu$-parabolic structure, which locally
characterize the parabolic connections introduced in~\cite{Inaba-1}
and the unramified parabolic connections introduced in~\cite{Inaba-Saito}, respectively.
We also recall the notion of generic $\nu$-ramified structure given in~\cite{Inaba-2}.
In Section~\ref{section: factorized ramified structure},
we introduce the notion of factorized $\nu$-ramified structure and
prove that it is equivalent to the generic $\nu$-ramified structure given
in Section~\ref{section: definition regular singular, unramified or ramified structure}.
In Section~\ref{section: recovery of formal structure},
we see that a generic $\nu$-ramified structure enables us to
recover a formal isomorphism to the connection $\nabla_{\nu}$
in~(\ref{equation: naive generic ramified connection}).
In Section~\ref{section: construction of the moduli space},
we give a~construction of the moduli space of connections with $(\lambda,\mu,\nu)$-structure
(Theorem~\ref{theorem: existence of the moduli space})
using an embedding to the moduli space of parabolic triples constructed in~\cite{IIS-1}.
It is a~variant of the standard method of the GIT-construction of the moduli space
established by C.T.~Simpson in~\cite{Simpson-1,Simpson-2}.
The following is an important property of the moduli space
(see Theorem \ref{theorem: existence of symplectic form and d-closedness}
in a~precise setting).\looseness=-1

\begin{Theorem} \label{theorem: 1}
There exists a canonical symplectic form on the moduli space of connections
with $(\lambda,\mu,\nu)$-structure.
\end{Theorem}

For the construction of the canonical $2$-form in Theorem \ref{theorem: 1}
(or Theorem \ref{theorem: existence of symplectic form and d-closedness} precisely),
we describe the tangent space of the moduli space
using the hypercohomology of a complex defined in Section \ref{section: tangent space}.
In Section~\ref{section: smoothness of the moduli},
we see that this tangent space has a canonical duality
(Proposition~\ref{proposition: nondegenerate pairing on tangent space})
coming from the factorized ramified structure,
which gives a canonical nondegenerate $2$-form.
This duality is also of benefit to prove the smoothness of the moduli space.
We also need to prove that the canonical $2$-form is ${\rm d}$-closed.
For its proof, we construct an unfolding of the moduli space of connections
with $(\lambda,\mu,\nu)$-structure in Section~\ref{section: symplectic structure}.
An unfolding means a deformation of the moduli space to logarithmic moduli spaces.
A factorized ramified structure enables us to construct such an unfolding
in an easy way.
By reducing to the fact that the canonical $2$-form on the logarithmic
moduli space is ${\rm d}$-closed, we can complete the proof of
Theorem~\ref{theorem: 1}.

The main aim of considering the moduli space of connections with $(\lambda,\mu,\nu)$-structure
is to construct the generalized isomonodromic deformation
that fits in with our setting of the moduli space.
In the logarithmic case, the isomonodromic deformation naively means that
the monodromy representation corresponding to the connection is constant.
Over the trivial bundle on~$\mathbb{P}^1$,
the isomonodromic deformation is classically known as the Schlesinger equation.
The formulation of isomonodromic deformation in a higher genus case requires
an appropriate setting of the moduli space of connections, which is done
in the work with K.~Iwasaki and M.-H.~Saito in~\cite{IIS-1} and in~\cite{Inaba-1}.
A cohomological description of the isomonodromic deformation on the moduli space
is also established by I.~Biswas, V.~Heu, J.~Hurtubise and A.~Komyo
in \cite{Biswas-Heu-Hurtubise-1, Biswas-Heu-Hurtubise-2,Komyo-1}.
Conceptually, the isomonodromic deformation is obtained by
pulling back, via the Riemann--Hilbert morphism,
the local trivial foliation on the family of character varieties.

For irregular singular connections, we cannot recover a meromorphic connection
from the naive monodromy data and we need to consider the Stokes data.
By virtue of the theorem of Deligne, Malgrange and Sibuya
\cite[Theorems~4.5.1 and 4.7.3]{Babbitt-Varadarajan},
there is a bijective correspondence between the local meromorphic connections
and the Stokes data on a punctured disk.
The generalized isomonodromic deformation means a family of irregular singular connections,
whose corresponding monodromy representation equipped with the Stokes data
is locally constant.
In~\cite{Jimbo-Miwa-Ueno},
M.~Jimbo, T.~Miwa and K.~Ueno established
the formulation of generalized isomonodromic deformation of
generic unramified irregular singular connections over the trivial bundle on $\mathbb{P}^1$
and described its differential equation completely.
The generalized isomonodromic deformation was also introduced by B. Malgrange
in \cite{Malgrange}.
The purpose of this paper is to extend this theory to higher genus case
including generic ramified connections.
In order to realize the formulation of generalized isomonodromic deformation in such a general setting,
we need the moduli space of connections
with $(\lambda,\mu,\nu)$-structure constructed
in Section~\ref{section: construction of the moduli space}.\looseness=1

In \cite{Boalch-1}, P.~Boalch constructs the moduli space of unramified connections
over the trivial bundle on $\mathbb{P}^1$ and describes the generalized isomonodromic deformation
in~\cite{Jimbo-Miwa-Ueno} through the correspondence with the wild character variety
which is the moduli space of monodromy Stokes data.
P.~Boalch extends the framework of wild character variety to the higher genus case
in~\cite{Boalch-2}.
In~\cite{van-der-Put-Saito},
M.~van der Put and M.-H.~Saito gives
another construction of the moduli space of monodromy Stokes data,
which includes all possible singularities, and
provides the explicit descriptions of the moduli spaces in the case of Painlev\'e equations.
I.Krichever also extends the argument by Jimbo, Miwa and Ueno in~\cite{Jimbo-Miwa-Ueno}
to the higher genus case and describes the generalized isomonodromic $2$-form in~\cite{Krichever}.
Placing importance on the Simpson's framework of Betti and de Rham correspondence
in~\cite{Simpson-2},
the generalized isomonodromic deformation is formulated
via the full moduli space of generic unramified connections on curves of general genus
in the work with M.-H.~Saito in~\cite{Inaba-Saito} and in~\cite{Inaba-3}.
C.~Bremer and D.~Sage establish the generalized isomonodromic deformation
of ramified connections over the trivial bundle on $\mathbb{P}^1$ in~\cite{Bremer-Sage-2}
and they prove the integrability condition of the generalized isomonodromic deformation
via examining a property of the corresponding differential ideal.
Their work is based on the construction of the moduli space in~\cite{Bremer-Sage-1},
which partially uses the method by P.~Boalch in~\cite{Boalch-1}.

In Section \ref{section: local analytic theory}, we recall a brief sketch of
the local analytic theory of ramified irregular singular connections.
First we consider the pullback of a generic ramified connection to a local analytic
ramified cover.
After applying an elementary transform of vector bundle to
the pullback of the ramified connection,
we get an unramified irregular connection.
Such a~process is called a~shearing transformation method
\cite[Section 19.3]{Wasow}.
Its description is given by K.~Diarra, F.~Loray and A.~Komyo
in \cite{Diarra-Loray,Komyo-2}
for rank $2$ ramified connections on~$\mathbb{P}^1$.
On the other hand,
we give a brief idea of producing the Stokes data corresponding to the unramified
connection on the local analytic ramified cover.
Then we give a definition of local generalized isomonodromic deformation
of generic unramified irregular singular connections on a unit disk
in Definition~\ref{definition: local generalized isomonodromic deformation}.
Applying the Jimbo--Miwa--Ueno theory in~\cite{Jimbo-Miwa-Ueno}
to the local setting, we get the following theorem
(see Theorem~\ref{theorem: Jimbo-Miwa-Ueno equation} precisely).

\begin{Theorem}[Jimbo, Miwa and Ueno] \label{theorem: 2}
A family of generic unramified irregular singular connections on a unit disk
is a local generalized isomonodromic deformation
if and only if it can be extended to an integrable connection.
\end{Theorem}

Precisely, there are ambiguities in the asymptotic solutions in our setting
and our proof of Theorem \ref{theorem: 2}
(Theorem \ref{theorem: Jimbo-Miwa-Ueno equation} precisely)
follows from the asymptotic property of flat solutions,
which is essentially the result by T. Mochizuki in \cite[Chapter 20]{Mochizuki}.
Using Theorem \ref{theorem: 2} (precisely Theorem \ref{theorem: Jimbo-Miwa-Ueno equation}),
we get a similar statement for local ramified connections
in Corollary \ref{corollary: rigorous ramified local generalized isomonodromic deformation},
which is a main consequence of Section \ref{section: local analytic theory}.

Based on the viewpoint of Theorem \ref{theorem: 2}
(precisely Theorem \ref{theorem: Jimbo-Miwa-Ueno equation} and its consequence
Corollary~\ref{corollary: rigorous ramified local generalized isomonodromic deformation}),
we formulate the generalized isomonodromic deformation
on the moduli space of ramified connections
in Section~\ref{section: generalized isomonodromy equation}.
For the construction, we introduce in Section~\ref{section: horizontal lift}
the notion of horizontal lift (Definition~\ref{definition: horizontal lift in one vector case})
of the universal family of connections
on the moduli space.
The horizontal lift is locally a restriction of the family of integrable connections, given in
Theorem~\ref{theorem: 2}
(precisely Corollary~\ref{corollary: rigorous ramified local generalized isomonodromic deformation}),
to a first order infinitesimal neighborhood of the base parameter space.
Nevertheless, it is defined purely algebraically.
In the case of logarithmic or unramified irregular singular connections,
the notion of horizontal lift is introduced in~\cite{Inaba-1,Inaba-3,Inaba-Saito}.
We can prove the existence and the uniqueness of the horizontal lift
in Propositions~\ref{proposition: horizontal lift in one variable}
and~\ref{proposition: horizontal lift in two variable of deep order},
whose proof needs an isomorphism
$(E,\nabla)|_{qx}\cong(\mathbb{C}[[w]],\nabla_{\nu})|_{qx}$
in deep order (for $q=2m-1$ or $q=3m-1$),
that is proved in Proposition~\ref{proposition: normalization of formal type over finite scheme}.
The existence of horizontal lift in Proposition~\ref{proposition: horizontal lift in one variable}
produces a~tangent splitting
$\Psi\colon \pi_{\mathcal T}^* T_{\mathcal T}
\longrightarrow
T_{M^{\balpha}_{{\mathcal C},{\mathcal D}}(\lambda,\tilde{\mu},\tilde{\nu})}$
in Section~\ref{section: generalized isomonodromy equation},
equation~(\ref{equation: generalized isomonodromic splitting homomorphism}),
where
$M^{\balpha}_{{\mathcal C},{\mathcal D}}(\lambda,\tilde{\mu},\tilde{\nu})$
is a family of moduli spaces of $\balpha$-stable connections
with $(\lambda,\tilde{\mu},\tilde{\nu})$-structure
and ${\mathcal T}$ is the space of time variables parameterizing
local exponents and curves with divisors.
We call the subbundle
$\im\Psi\subset T_{M^{\balpha}_{{\mathcal C},{\mathcal D}}(\lambda,\tilde{\mu},\tilde{\nu})}$
the generalized isomonodromic subbundle
(Definition \ref{definition: generalized isomonodromic subbundle}).
The main purpose of this paper is the following theorem
(see Theorem \ref{theorem: integrability of generalized isomonodromy} precisely).

\begin{Theorem} \label{theorem: 3} The generalized isomonodromic subbundle $\im\Psi$ of
$T_{M^{\balpha}_{{\mathcal C},{\mathcal D}}(\lambda,\tilde{\mu},\tilde{\nu})}$
satisfies the integrability condition
$[\im\Psi,\im\Psi]\subset\im\Psi$.
\end{Theorem}

In the proof of the above theorem, we need the uniqueness of the horizontal lift
with respect to two deformation parameters $\epsilon_1$, $\epsilon_2$,
which is proved in Proposition~\ref{proposition: horizontal lift in two variable of deep order}.
We can prove the integrability condition of Theorem~\ref{theorem: 3}
by looking at the $\epsilon_1\epsilon_2$-term of the horizontal lift.

By Theorem~\ref{theorem: 3}
(or Theorem \ref{theorem: integrability of generalized isomonodromy}),
the generalized isomonodromic subbundle $\im\Psi$ determines a foliation on the moduli space
$M^{\balpha}_{{\mathcal C},{\mathcal D}}(\lambda,\tilde{\mu},\tilde{\nu})$,
which we call the generalized isomonodromic foliation
(Definition \ref{definition: generalized isomonodromic foliation}).
We regard the generalized isomonodromic subbundle or the induced foliation
as the generalized isomonodromic deformation.
However, our construction of generalized isomonodromic deformation is not complete,
because we do not establish the generalized Riemann--Hilbert correspondence
between the moduli space of connections and the wild character variety.
The construction of wild character variety in~\cite{Boalch-Yamakawa} will be
a key work in that framework.

The generalized isomonodromic deformation is known to be characterized by a canonical $2$-form,
which is introduced in~\cite{Jimbo-Miwa-Ueno} and extended to higher genus case
in~\cite{Krichever}.
The works \cite{Boalch-1,Bremer-Sage-2}
are also based on this principle.
By means of the generalized isomonodromic subbundle
$\im\Psi$ constructed in Theorem \ref{theorem: 3},
we can extend the relative symplectic form
given in Theorem~\ref{theorem: 1} to a total $2$-form
(Definition~\ref{definition: generalized isomonodromic 2-form}),
which we call the generalized isomonodromic $2$-form.
Using the generalized isomonodromic foliation produced by Theorem~\ref{theorem: 3},
we can prove in Corollary~\ref{corollary d-closedness of generalized isomonodromic 2-form}
that the generalized isomonodromic $2$-form is ${\rm d}$-closed.

\section[Logarithmic, unramified irregular singular or ramified irregular singular structure\\ on a connection]{Logarithmic, unramified irregular singular\\ or ramified irregular singular structure on a connection}\label{section: definition regular singular, unramified or ramified structure}

Let $C$ be a complex smooth projective curve of genus $g$.
We consider an effective divisor
$D=D_{\mathrm{log}}+D_{\mathrm{un}}+D_{\mathrm{ram}}$ on $C$,
where
$D_{\mathrm{log}}$, $D_{\mathrm{un}}$ and $D_{\mathrm{ram}}$ are mutually disjoint,
$D_{\mathrm{log}}$ is a reduced divisor,
$D_{\mathrm{un}}=\sum_{x\in D_{\mathrm{un}}} m_x x$
and
$D_{\mathrm{ram}}=\sum_{x\in D_{\mathrm{ram}}} m_x x$
are multiple divisors with $m_x\geq 2$ for
$x\in D_{\mathrm{un}} \cup D_{\mathrm{ram}}$.

For each point $x\in D_{\mathrm{log}}$,
we fix a tuple $(\lambda_0^x,\dots,\lambda_{r-1}^x)\in\mathbb{C}^r$
and put
$\lambda^x:=(\lambda^x_k)_{0\leq k\leq r-1}$
and $\lambda:=(\lambda^x)_{x\in D_{\mathrm{log}}}$.

For $x\in D_{\mathrm{un}}$, we take
$\mu^x_0,\dots,\mu^x_{r-1}\in \Omega^1_C(m_x x)\big|_{m_x x}$
whose leading terms are mutually distinct.
In other words, $\mu^x_k- \mu^x_{k'}$ is a generator of the ${\mathcal O}_{m_xx}$-module
$\Omega^1_C(m_x x)\big|_{m_xx}$ for $k\neq k'$.
We write
$\mu^x:=(\mu^x_k)_{0\leq k\leq r-1}$
and
$\mu:=(\mu^x)_{x\in D_{\mathrm{un}}}$.

Let $E$ be an algebraic vector bundle on $C$
of rank $r$ and let
$\nabla\colon E\longrightarrow E\otimes \Omega^1_C(D)$
be an algebraic connection admitting poles along $D$.

\begin{Definition}
We say that $l^x$ is a logarithmic $\lambda^x$-parabolic structure on $(E,\nabla)$
at $x\in D_{\mathrm{log}}$,
if it is a filtration $E|_x=l^x_0\supset\cdots\supset l^x_{r-1}\supset l^x_r=0$ satisfying
$(\res_x(\nabla)-\lambda^x_k\mathrm{id})(l^x_k)\subset l^x_{k+1}$
for $k=0,\dots,r-1$, where
$\res_x(\nabla)\colon E|_x \longrightarrow E|_x$
is the linear map determined by taking the residue at $x$.
\end{Definition}

\begin{Definition}We say that $\ell^x$ is a generic unramified $\mu^x$-parabolic structure
on $(E,\nabla)$ at $x\in D_{\mathrm{un}}$, if it is a filtration
$E|_{m_xx}=\ell^x_0\supset\cdots\supset \ell^x_{r-1}\supset \ell^x_r=0$
satisfying
$\ell^x_k/\ell^x_{k+1}\cong {\mathcal O}_{m_xx}$ and
$(\nabla|_{m_xx}-\mu^x_k\mathrm{id})(\ell^x_k)\subset \ell^x_{k+1}$
for $k=0,\dots,r-1$,
where $\nabla|_{m_xx} \colon E|_{m_xx} \longrightarrow E\otimes\Omega_X^1(D)|_{m_xx}$
is the ${\mathcal O}_{m_xx}$-homomorphism given by the restriction of
$\nabla$ to the finite subscheme $m_xx\subset X$.
\end{Definition}

For each $x\in D_{\mathrm{ram}}$, we take a generator
$z$ of the maximal ideal of the local ring
${\mathcal O}_{C,x}$.
Assume that
\[
 \nu^x_0(z)\in\Omega^1_C(D_{\mathrm{ram}})|_{m_xx},
 \qquad
 \nu^x_1(z),\dots,\nu^x_{r-1}(z)\in\Omega^1_C(D_{\mathrm{ram}})|_{(m_x-1)x}
\]
are given and that the leading term of $\nu_1^x(z)$ does not vanish.
In other words, $\nu_1^x(z)$ is a generator of
the ${\mathcal O}_{C,x}$-module $\Omega^1_C(D_{\mathrm{ram}})|_{(m_x-1)x}$.
We take a variable $w$ with $w^r=z$ and put
\[
 \nu^x(w)=\nu^x_0(z)+\nu^x_1(z)w+\cdots+\nu^x_{r-1}(z)w^{r-1}.
\]
We write
$\nu=(\nu^x(w))_{x\in D_{\mathrm{ram}}}$.
Furthermore, we assume the following

\begin{Assumption}\rm
We assume that
\[
 d \ := \
 -\sum_{x\in D_{\mathrm{log}}} \sum_{k=0}^{r-1} \lambda^x_k
 -\sum_{x\in D_{\mathrm{un}}} \sum_{k=0}^{r-1}\res_{x}(\mu^x_k)
 -\sum_{x\in D_{\mathrm{ram}}} \left(r\res_x(\nu^x_0)+\frac{r-1}{2}\right)
\]
is an integer.
\end{Assumption}

Next we recall the formulation of ramified connection given in \cite{Inaba-2}.
In this paper, we give a simplified version,
since the formulation in \cite{Inaba-2} is somewhat complicated.
Before stating the precise definition,
we will see the reason why we introduce a filtration on $E|_{m_xx}$.
What we want to consider is a connection $(E,\nabla)$ on $C$ with a formal isomorphism
$(E,\nabla)\otimes\widehat{\mathcal O}_{C,x}\cong(\mathbb{C}[[w]],\nabla_{\nu^x})$.
However, it is difficult to treat the formal isomorphism
in the construction of the moduli space
and also in the deformation theory.
It is rather convenient to formulate the ramified condition
only by the data of the restriction
$(E,\nabla)|_{m_xx}$.
With respect to the frame of $E|_{m_xx}$
corresponding to $1,w,\dots,w^{r-1}$,
the representation matrix of $\nabla|_{mx}$ is
\[
 \begin{pmatrix}
 \nu_0^x(z) & z\nu_{r-1}^x(z) & \cdots & z\nu_1^x(z) \\
 \nu_1^x(z) & \nu_0^x(z)+\frac{{\rm d}z}{rz} & \cdots & z\nu_2^x(z) \\
 \vdots & \vdots & \ddots & \vdots \\
 \nu_{r-1}^x(z) & \nu_{r-2}^x(z) & \cdots & \nu_0^x(z)+\frac{(r-1){\rm d}z}{r}
 \end{pmatrix}.
\]
However, the assumption on $\nabla|_{m_xx}$ by the above matrix is too strict
and that does not go well with the formulation of the moduli space.
It is rather better to allow ambiguities in $\nabla|_{m_xx}$
which is given by
\begin{equation}
\label{equation: representation matrix with ambiguity}
 \begin{pmatrix}
 \nu_0^x(z) & z\nu_{r-1}^x(z) & \cdots & z\nu_1^x(z) \\
 \nu_1^x(z)+a_{1,0}\frac{{\rm d}z}{z} & \nu_0^x(z)+\frac{{\rm d}z}{rz} & \cdots & z\nu_2^x(z) \\
 \vdots & \vdots & \ddots & \vdots \\
 \nu_{r-1}^x(z)+a_{r-1,0}\frac{{\rm d}z}{z} & \nu_{r-2}^x(z)+a_{r-1,1}\frac{{\rm d}z}{z}
 & \cdots & \nu_0^x(z)+\frac{(r-1){\rm d}z}{r}
 \end{pmatrix}
\end{equation}
where $a_{i,j}\in\mathbb{C}$ for $r-1\geq i>j\geq 0$.
Indeed, if $\nabla|_{m_xx}$ is given by the above matrix with ambiguities,
there is a formal isomorphism
$\big(\widehat{E},\widehat{\nabla}\big)\cong(\mathbb{C}[[w]],\nabla_{\nu^x})$
(which will follow from \cite[Proposition 1.3]{Inaba-2} or
Corollary \ref{corollary: formal structure} later).
In order to allow the ambiguities of $\nabla|_{m_xx}$ as above,
we introduce a filtration
$E|_{m_xx}=V^x_0\supset V^x_1\supset\cdots\supset V^x_{r-1}\supset z V^x_0$.
If we identify $E|_{m_xx}$ with $\mathbb{C}[[w]]/(w^{m_xr})$
via the formal isomorphism,
then we set $V^x_k:=(w^k)/(w^{m_xr})$
and $L_k^x:=(w^k)/\big(w^{m_xr-r+k+1}\big)$,
where $(w^k)$ the ideal of $\mathbb{C}[w]$ generated by
$w^k$.
Then all the conditions in the following definition will be obvious.

\begin{Definition}[{\cite[Definitions 1.2 and~2.1] {Inaba-2}}]
\label{definition of ramified structure}
Let $(E,\nabla)$ be a pair of an algebraic vector bundle $E$ of rank $r$ on $C$ and
an algebraic connection $\nabla$ on $E$.
We say that a tuple
${\mathcal V}^x=\big( \big( V^x_k,L^x_k,\pi^x_k\big)_{0\leq k\leq r-1} ,
\big(\phi^x_k\big)_{1\leq k\leq r}\big)$
is a generic $\nu$-ramified structure
on $(E,\nabla)$ at $x\in D_{\mathrm{ram}}$, if
\begin{itemize}\itemsep=0pt
\item[(i)]
$E|_{m_x x}=V^x_0\supset V^x_1\supset \cdots \supset V^x_{r-1}\supset
V^x_r=z V^x_0$
is a filtration by ${\mathcal O}_{m_x x}$-submodules
which satisfies
$\length(V^x_k/V^x_{k+1})=1$
and
$\nabla|_{m_x x}(V^x_k)\subset V^x_k\otimes\Omega^1_C(D)|_{m_x x}$
for $0\leq k\leq r-1$,
\item[(ii)]
$\pi^x_k \colon V^x_k \otimes\mathbb{C}[w]/(w^{m_x r-r+1})
\longrightarrow L^x_k$
is a quotient free $\mathbb{C}[w]/(w^{m_x r-r+1})$-module of rank one
for $0\leq k \leq r-1$ such that the restrictions
$\pi^x_k|_{V^x_k}\colon V^x_k\hookrightarrow
V^x_k\otimes\mathbb{C}[w]/(w^{m_x r-r+1})
\xrightarrow{\pi^x_k} L^x_k$
are surjective and that the diagrams
\[
 \begin{CD}
 V^x_k \otimes \mathbb{C}[w]/\big(w^{m_x r-r+1}\big)
 @> \pi^x_k >> L^x_k \\
 @V \nabla|_{m_x x} VV @VV \nu^x(w)+\frac{k}{r}\frac{{\rm d}z}{z} V \\
 V^x_k \otimes\Omega^1_C(D) \otimes \mathbb{C}[w]/\big(w^{m_x r-r+1}\big)
 @> \pi^x_k \otimes 1 >> L^x_k \otimes\Omega^1_C(D)
 \end{CD}
\]
are commutative for $0\leq k\leq r-1$,
\item[(iii)]
$\phi_k\colon L^x_k \longrightarrow w L^x_{k-1}$ for $1\leq k\leq r-1$
and
$\phi^x_r \colon (z)/\big(z^{m_x+1}\big)\otimes L^x_0
\longrightarrow w L^x_{r-1}$
are surjective $\mathbb{C}[w]$-homomorphisms
such that the diagrams
\[
 \begin{CD}
 V^x_k \otimes \mathbb{C}[w]/\big(w^{m_x r-r+1}\big)
 @> \pi^x_k >> L^x_k \\
 @VVV @VV \phi^x_k V \\
 V^x_{k-1} \otimes \mathbb{C}[w]/\big(w^{m_x r-r+1}\big)
 @> \pi^x_{k-1} >> L^x_{k-1}
 \end{CD}
\]
are commutative
for $1\leq k \leq r-1$ and that the diagram
\[
 \begin{CD}
 (z)/\big(z^{m_x+1}\big)\otimes V^x_0 \otimes\mathbb{C}[w]/\big(w^{m_x r-r+1}\big)
 @> 1\otimes \pi^x_0 >>
 (z)/\big(w^{m_xr+1}\big) \otimes L^x_0 \\
 @VVV @VV \phi^x_r V \\
 V^x_{r-1} \otimes\mathbb{C}[w]/\big(w^{m_x r-r+1}\big)
 @> \pi^x_{r-1} >> L^x_{r-1}
 \end{CD}
\]
is commutative,
\item[(iv)]
there are isomorphisms
$\psi^x_k\colon L^x_k\xrightarrow{\sim} (w)/\big(w^{m_xr-r+2}\big)\otimes L^x_{k-1}$
of $\mathbb{C}[w]$-modules for $1\leq k\leq r-1$
such that the composition
$L^x_k\xrightarrow{\psi^x_k} (w)/\big(w^{m_xr-r+2}\big)\otimes L^x_{k-1}
\longrightarrow w L^x_{k-1}$
coincides with~$\phi^x_k$
and that
the composition
\begin{align*}
& (z)/\big(w^{m_xr+1}\big)\otimes L^x_0
 \xrightarrow{\phi^x_r}
 L^x_{r-1}
 \xrightarrow[\sim]{\psi^x_{r-1}}
 (w)/\big(w^{m_xr-r+2}\big)\otimes L^x_{r-2}
 \\
 & \xrightarrow[\sim]{\psi^x_{r-2}}\cdots\xrightarrow[\sim]{\psi^x_{1}}
\big((w)/\big(w^{m_xr-r+2}\big)\big)^{\otimes r-1}\otimes L^x_0
 \xrightarrow{\sim}
 \big(w^{r-1}\big)/\big(w^{m_xr}\big)\otimes L^x_0
\end{align*}
coincides with the $\mathbb{C}[w]$-homomorphism
obtained by tensoring $L^x_0$ to
the canonical map
$(z)/\big(w^{m_xr+1}\big)
\longrightarrow \big(w^{r-1}\big)/\big(w^{m_xr}\big)$.
\end{itemize}
Two ramified structures
$\big(V^x_k,L^x_k,\pi^x_k,\phi^x_k\big)$
and $\big(V'^x_k,L'^x_k,\pi'^x_k,\phi'^x_k\big)$
on $(E,\nabla)$ at $x\in D_{\mathrm{ram}}$ are equivalent
if $V^x_k=V'^x_k$ for $0\leq k\leq r$,
there are isomorphisms
$\sigma_k\colon L^x_k\xrightarrow{\sim} L'^x_k$
of $\mathbb{C}[w]$-modules for $0\leq k\leq r-1$
such that the diagrams
\[
 \begin{CD}
 V^x_k @>\pi^x_k |_{V^x_k}>> L^x_k \\
 \parallel & & @V \cong V \sigma_k V \\
 V^x_k @>\pi'^x_k|_{V^x_k}>> L'^x_k
 \end{CD}
 \quad \quad (0\leq k\leq r-1)
 \qquad
 \begin{CD}
 L^x_k @>\phi^x_k >> L^x_{k-1} \\
 @V\sigma_k V \cong V @V \cong V \sigma_{k-1} V \\
 L'^x_k @>\phi'^x_k >> L'^x_{k-1}
 \end{CD}
 \quad\quad (1\leq k\leq r-1)
\]
and the diagram
\[
 \begin{CD}
 (z)/\big(w^{m_xr+1}\big)\otimes L^x_0 @>\phi^x_r>> L^x_{r-1} \\
 @V\mathrm{id}\otimes\sigma_0 V \cong V @V \cong V\sigma_{r-1}V \\
 (z)/\big(w^{m_xr+1}\big)\otimes L'^x_0 @>\phi'^x_r>> L'^x_{r-1}
 \end{CD}
\]
are commutative.
\end{Definition}

\begin{Remark} \rm
In the condition (iv) of Definition \ref{definition of ramified structure},
the composition
$\psi^x_1\circ\cdots\circ\psi^x_{r-1}\circ\phi^x_r$
is independent of the choices of the lifts
$\psi_k$ of $\phi_k$ taken for $1\leq k\leq r-1$.
In particular, the condition (iv) is independent of the choices of $\psi_k$.
\end{Remark}

\begin{Example}
\label{example: ramified structure} \rm
Let us consider the typical case
$(E,\nabla)\otimes\widehat{\mathcal O}_{C,x}
= (\mathbb{C}[[w]],\nabla_{\nu})$,
where $z\in{\mathcal O}_{C,x}$ is a generator of the maximal ideal,
$w=z^{\frac{1}{r}}$ and the connection $\nabla_{\nu}$ is given by
\[
 \nabla_{\nu} \colon \
 \mathbb{C}[[w]] \ \ni \ f(w) \ \mapsto \ df(w)+f(w)\nu \ \in \mathbb{C}[[w]]\otimes \frac{{\rm d}z}{z^m}.
\]
In this case, a generic $\nu$-ramified structure in Definition~\ref{definition of ramified structure}
is given in the following way.
We consider the filtration
$\mathbb{C}[[w]]/z^m\mathbb{C}[[w]]\supset(w)/\big(w^{mr}\big)\supset\cdots\supset
\big(w^{r-1}\big)/\big(w^{mr}\big)\supset z\mathbb{C}[[w]]/z^m\mathbb{C}[[w]]$
and put $V_k:=\big(w^k\big)/\big(w^{mr}\big)$ for $0\leq k\leq r-1$.
We put
$L_k:=\big(w^k\big)/\big(w^{mr-r+k+1}\big)$
and regard it as a~$\mathbb{C}[w]/\big(w^{mr-r+1}\big)$-module.
The canonical surjection
\[
V_k=\big(w^k\big)/\big(w^{mr}\big)\longrightarrow \big(w^k\big)/\big(w^{mr-r+k+1}\big)=L_k
\]
induces a surjective homomorphism
\[
 \pi_k\colon \ V_k\otimes_{\mathbb{C}[z]/(z^m)}\mathbb{C}[w]/\big(w^{mr-r+1}\big)
 \longrightarrow L_k
\]
of $\mathbb{C}[w]/\big(w^{mr-r+1}\big)$-modules.
Then the conditions~(i), (iii), (iv) of Definition~\ref{definition of ramified structure}
are obvious for such data.
Since the restriction
\[
 \nabla_{\nu}\colon \ w^k\mathbb{C}[[w]]\longrightarrow w^k\mathbb{C}[[w]]\otimes\frac{{\rm d}z}{z^m}
\]
satisfies the equality
\begin{align*}
 \nabla_{\nu}\big( w^kf(w) \big)& =kw^{k-1}{\rm d}w \, f(w)+w^k{\rm d}f(w)+w^kf(w)\nu\\
 & =w^kf(w)\frac{k}{r}\frac{{\rm d}z}{z}+w^k ( {\rm d}f(w)+f(w)\nu ),
\end{align*}
we can also see the commutativity of the diagrams
in Definition~\ref{definition of ramified structure}\,(ii).

We will see later in Corollary \ref{corollary: formal structure}
that any connection with generic $\nu$-ramified structure at $x$
is in fact isomorphic to the one given in this example.
\end{Example}

\begin{Definition}
\rm
We say that $(E,\nabla,l,\ell,{\mathcal V})$
is a connection with $(\lambda,\mu,\nu)$-structure, if
\begin{itemize}\itemsep=0pt
\item[(i)]
$E$ is an algebraic vector bundle of rank $r$ on $C$ of degree $d$,
\item[(ii)]
$\nabla\colon E \longrightarrow E\otimes \Omega^1_C(D)$
is an algebraic connection admitting poles along $D$,
\item[(iii)] $l=(l^x)_{x\in D_{\mathrm{log}}}$ is a tuple of
logarithmic $\lambda^x$-parabolic structures
$l^x$ on $(E,\nabla)$ at $x\in D_{\mathrm{log}}$,
\item[(iii)]
$\ell=(\ell^x)_{x\in D_{\mathrm{un}}}$ is a tuple of generic unramified
$\mu^x$-parabolic structures $\ell^x$ on $(E,\nabla)$ at $x\in D_{\mathrm{un}}$,
\item[(iv)]
${\mathcal V}=({\mathcal V}^x)_{x\in D_{\mathrm{ram}}}$
is a tuple of generic
$\nu^x$-ramified structures ${\mathcal V}^x$ on $(E,\nabla)$
at $x\in D_{\mathrm{ram}}$.
\end{itemize}
\end{Definition}

We take a tuple
$\balpha=\left( \alpha^x_k \right)^{x\in D}_{1\leq k\leq r} $
of positive rational numbers such that
$0<\alpha^x_1<\cdots<\alpha^x_r<1$ for any $x\in D$
and that
$\alpha^x_k\neq\alpha^{x'}_{k'}$ for $(x,k)\neq(x',k')$.

For a non-zero subbundle $F$ of $E$, we write
\begin{gather*}
 \mathrm{pardeg}^{\balpha}(F)=
 \deg F+\sum_{x\in D_{\mathrm{log}}}\sum_{k=1}^r
 \alpha^x_k \length((F|_{x}\cap l^x_{k-1})/(F|_{x}\cap l^x_k))
 \\
\hphantom{\mathrm{pardeg}^{\balpha}(F)=}{}
 + \sum_{x\in D_{\mathrm{un}}}\sum_{k=1}^r
 \alpha^x_k \length ( (F|_{n_x x}\cap \ell^x_{k-1}) / (F|_{n_xx}\cap \ell^x_k) )
 \\
\hphantom{\mathrm{pardeg}^{\balpha}(F)=}{}
 +\sum_{x\in D_{\mathrm{ram}}}\sum_{k=1}^r
 \alpha^x_k\length \left( (F|_{m_xx}\cap V^x_{k-1})/(F|_{m_xx}\cap V^x_k)\right).
\end{gather*}

\begin{Definition} \label{definition: stability}
We say that a connection
$(E,\nabla,l,\ell,{\mathcal V})$ with $(\lambda,\mu,\nu)$-structure
is $\balpha$-stable (resp. $\balpha$-semistable) if
the inequality
\begin{equation*}
 \frac { \mathrm{pardeg}^{\balpha}(F) }
 {\rank F}
 <
 \frac { \mathrm{pardeg}^{\balpha}(E) }
 {\rank E}
 \qquad
 \left( \text{resp. }
 \frac { \mathrm{pardeg}^{\balpha}(F) }
 {\rank F}
 \leq
 \frac { \mathrm{pardeg}^{\balpha}(E) }
 {\rank E} \right)
\end{equation*}
holds for any subbundle $0\neq F\subsetneq E$ satisfying
$\nabla(F)\subset F\otimes\Omega^1_C(D)$.
\end{Definition}

\begin{Remark}If $D_{\mathrm{ram}}\neq\varnothing$,
then we can see
$(E,\nabla)\otimes\widehat{\mathcal O}_{C,x}\cong(\mathbb{C}[[w]],\nabla_{\nu})$
by Corollary~\ref{corollary: formal structure},
which will be proved later.
Since $(\mathbb{C}[[w]],\nabla_{\nu})$ is irreducible,
$(E,\nabla)$ is also irreducible and
$(E,\nabla,l,\ell,{\mathcal V})$ is automatically $\balpha$-stable for
any parabolic weight $\balpha$ in this case.
\end{Remark}

\section {Factorized ramified structure}
\label{section: factorized ramified structure}

In this section, we introduce the notion of factorized ramified structure
which is a rephrasing of generic $\nu$-ramified structure
in Definition \ref{definition of ramified structure}.
This notion is useful for the description of symplectic form later.
In the Introduction, we saw a rough idea of factorized ramified structure.
Before giving the precise definition of factorized ramified structure,
we will see another aspect of the ambiguity
in (\ref{equation: representation matrix with ambiguity}),
which affects the definition of factorized ramified structure.

Let $(E,\nabla)$ be a connection on $C$ with a formal isomorphism
$(E,\nabla)\otimes\widehat{\mathcal O}_{C,x}
\cong (\mathbb{C}[[w]], \nabla_{\nu})$,
where $z$ is a generator of the maximal ideal of ${\mathcal O}_{C,x}$, $w^r=z$
and $\nabla_{\nu}$ is the connection defined in~(\ref{equation: naive generic ramified connection}).
Write
$\nu(w)=\sum_{k=0}^{r-1} c_k(z) w^k {\rm d}z/z^{m_xx}$
with $c_0(z),\dots,c_{r-1}(z)\in
{\mathcal O}_{m_xx}$
and ${c_1(z)\in{\mathcal O}_{m_xx}^{\times}}$.
There is in fact an ambiguity coming from the choice of $z$,
but it can be expressed by a modification of $\nu$ and
we do not pursue this point any more.
Recall that the endomorphism~$N$ on~$E|_{m_xx}$ corresponds to
the action of $w$ via the isomorphism
$E|_{m_xx}\xrightarrow{\sim} \mathbb{C}[w]/\big(w^{rm_x}\big)$.
Since~$c_1(z)$ is invertible,
there is a polynomial
$P(T)\in {\mathcal O}_{m_xx}[T]$ satisfying
the equality
\[
 w=P\big(c_0(z)+c_1(z)w+\cdots+c_{r-1}w^{r-1}\big)
\]
in the ring ${\mathcal O}_{m_xx}[w]/(w^r-z)$.
Since the equality $\nabla|_{(m_x-1)x}=\nu(N)|_{(m_x-1)x}$ holds,
$N|_{(m_x-1)x}$ is uniquely determined from $\nabla|_{m_xx}$ by substitution to $P(T)$.
However, $N$ always has an ambiguity in the $z^{m_x-1}$-coefficients.
This ambiguity causes the ambiguity in the matrix
(\ref{equation: representation matrix with ambiguity}) of $\nabla|_{m_xx}$.

In order to see more precisely,
consider the filtration
$E|_{m_x}=V_0\supset V_1\supset\cdots\supset V_{r-1}\supset V_r=zV_0$
given in Definition \ref{definition of ramified structure}\,(i).
Since this filtration is determined by $N(V_i)=V_{i+1}$,
it is uniquely determined from $\nabla|_{m_xx}$.
Then the restriction $N|_{V_i}$ induces an endomorphism on $V_i/z^{m_x-1}V_{i+1}$,
which is uniquely determined from $\nabla|_{m_x}$.
So the factorization $N=\theta\circ\kappa$ will be justified when we replace it
with the induced maps on $V_i/z^{m_x-1}V_{i+1}$ or on its dual.
Although we need a careful consideration for the expression
of these induced maps, all the conditions in the following definition will be natural.

Let $C$, $D_{\mathrm{log}}$, $D_{\mathrm{un}}$, $D_{\mathrm{ram}}$, $\nu$, $z$, $w$ be as in Section~\ref{section1}
and let $(E,\nabla)$ be a pair of an algebraic vector bundle $E$ of rank $r$ on $C$ and
an algebraic connection $\nabla$ on $E$ with poles along~$D$.

\begin{Definition} \label{def-fac-connection} \rm
We say that a tuple
$(V_k,\vartheta_k,\varkappa_k)_{0\leq k\leq r-1}$
is a factorized $\nu$-ramified structure on $(E,\nabla)$ at $x\in D_{\mathrm{ram}}$, if
\begin{itemize}\itemsep=0pt
\item[(i)] $E|_{m_xx}=V_0\supset V_1\supset\cdots\supset
V_{r-1}\supset V_r=zV_0$
is a filtration by ${\mathcal O}_{m_xx}$-submodules
satisfying $\nabla|_{m_xx}(V_k)\subset V_k\otimes\Omega^1_C(D)$
and $\length(V_k/V_{k+1})=1$
for $0\leq k\leq r-1$,
\item[(ii)]
for $\overline{V}_k:=V_k/z^{m_x-1}V_{k+1}$
and
$\overline{W}_k:=\big( \overline{V}_{r-k-1} \big)^{\vee}
=\Hom_{{\mathcal O}_{m_xx}}\big(\overline{V}_{r-k-1},{\mathcal O}_{m_xx}\big)$,
\[
 \vartheta_k\colon \ \overline{W}_k\times \overline{W}_{r-k-1}
 \longrightarrow {\mathcal O}_{m_xx}
\]
is an ${\mathcal O}_{m_xx}$-bilinear pairing
for $0\leq k\leq r-1$
such that
the equality $\vartheta_k(v,v')=\vartheta_{r-k-1}(v',v)$
holds for $v\in \overline{W}_k$ and $v'\in \overline{W}_{r-k-1}$
and that
the induced homomorphisms
\[
 \theta_k\colon \ \overline{W}_k \stackrel{\sim}\longrightarrow
 \Hom(\overline{W}_{r-k-1},{\mathcal O}_{m_xx})
 =\overline{V}_k
 \qquad
 (0\leq k\leq r-1)
\]
are isomorphisms,
which make
the diagrams
\[
 \begin{CD}
 \overline{W}_k @>>> \overline{W}_{k-1} \\
 @V \theta_k V \cong V @V \theta_{k-1} V \cong V \\
 \overline{V}_k @>>> \overline{V}_{k-1}
 \end{CD}
 \quad (1\leq k\leq r-1)
 \hspace{40pt}
 \begin{CD}
 (z)/\big(z^{m_x+1}\big)\otimes\overline{W}_{0} @>>> \overline{W}_{r-1} \\
 @V 1\otimes\theta_{0} V \cong V @V \theta_{r-1} V \cong V \\
 (z)/\big(z^{m_x+1}\big)\otimes \overline{V}_{0} @>>\sim> \overline{V}_{r-1}
 \end{CD}
\]
commutative,
where the horizontal arrow $\overline{W}_k \longrightarrow \overline{W}_{k-1}$
is the dual of
$\overline {V}_{r-k} \longrightarrow \overline {V}_{r-k-1}$
and the horizontal arrow
$(z)/\big(z^{m_x+1}\big)\otimes \overline {W}_0
\longrightarrow \overline {W}_{r-1}$
is induced by tensoring $(z)/\big(z^{m_x+1}\big)$
to
$\overline {W}_0=\Hom (\overline {V}_{r-1},{\mathcal O}_{m_xx})
\longrightarrow
\Hom((z)/\big(z^{m_x+1}\big)\otimes \overline {V}_0,{\mathcal O}_{m_xx})
=
\big((z)/\big(z^{m_x+1}\big)\big)^{\vee} \otimes \overline {W}_{r-1}$,
\item[(iii)] for $0\leq k\leq r-1$,
\[
 \varkappa_k\colon \ \overline{V}_k \times \overline{V}_{r-k-1}
 \longrightarrow {\mathcal O}_{m_xx}
\]
is an ${\mathcal O}_{m_xx}$-bilinear pairing
such that the equality
$\varkappa_k(v,v')=\varkappa_{r-k-1}(v',v)$
holds for ${v\in \overline{V}_k}$, $v'\in \overline{V}_{r-k-1}$
and that
the induced homomorphisms
\[
 \kappa_k\colon \ \overline{V}_k \longrightarrow
 \Hom_{{\mathcal O}_{m_xx}}\big(\overline{V}_{r-k-1},{\mathcal O}_{m_xx}\big)
 =\overline{W}_k
 \qquad
 (0\leq k\leq r-1)
\]
make the diagrams
\[
 \begin{CD}
 \overline{V}_k @>>> \overline{V}_{k-1} \\
 @V \kappa_k VV @V \kappa_{k-1} VV \\
 \overline{W}_k @>>> \overline{W}_{k-1}
 \end{CD}
 \quad (1\leq k\leq r-1)
 \hspace{20pt}
 \begin{CD}
 (z)/\big(z^{m_x+1}\big)\otimes\overline{V}_0 @>>> \overline{V}_{r-1} \\
 @V 1\otimes\kappa_0 VV @V \kappa_{r-1} VV \\
 (z)/\big(z^{m_x+1}\big)\otimes\overline{W}_0 @>>> \overline{W}_{r-1}
 \end{CD}
\]
commutative,
\item[(iv)]
the composition
$N_k := \langle \vartheta_k , \varkappa_k \rangle = \theta_k\circ\kappa_k
\colon \overline{V}_k\longrightarrow \overline{V}_k$
satisfies the equalities
$(N_k)^r = z \, \mathrm{id}_{\bar{V}_k}$
and $(N_k)^{m_xr-r+1}=0$,
from which
the injective ring homomorphism
\begin{equation} \label{ring homomorphism defining C[w]-module structure}
 \mathbb{C}[w]/\big(w^{m_xr-r+1}\big)
 \ni \overline{f(w)} \mapsto f(N_k)
 \in \End_{{\mathcal O}_{m_xx}}\big(\overline{V}_k\big)
\end{equation}
is induced and
the diagrams
\[
 \begin{CD}
 V_k @>\nabla|_{m_xx}>> V_k\otimes\Omega^1_C(D) \\
 @VVV @VVV \\
 \overline{V}_k
 @> \nu (N_k) +\frac{k}{r} \frac{{\rm d}z}{z} >>
 \overline{V}_k\otimes\Omega^1_C(D)
 \end{CD}
\]
are commutative
for $k=0,1,\dots,r-1$,
\item[(v)]
with respect to the $\mathbb{C}[w]$-module structure on $\overline{V}_k$
defined by the ring homomorphism~(\ref{ring homomorphism defining C[w]-module structure}),
there are $\mathbb{C}[w]$-isomorphisms
$\psi_k\colon \overline{V}_k\xrightarrow{\sim}
(w)/\big(w^{m_xr-r+2}\big)\otimes\overline{V}_{k-1}$
such that the composition
\[
 \overline{V}_k \xrightarrow[\sim]{\psi_k} (w)/\big(w^{m_xr-r+2}\big)\otimes\overline{V}_{k-1}
 \longrightarrow w \overline{V}_{k-1}\hookrightarrow\overline{V}_{k-1}
\]
coincides with the homomorphism
$\overline{V}_k\longrightarrow\overline{V}_{k-1}$
induced by the inclusion $V_k\hookrightarrow V_{k-1}$ and that the composition
\begin{gather*}
 (z)/\big(z^{m_x+1}\big)\otimes \overline{V}_0\rightarrow \overline{V}_{r-1}
 \xrightarrow[\sim]{\psi_{r-1}} (w)/\big(w^{m_xr-r+2}\big)\otimes \overline{V}_{r-2}\\
 \qquad{}
 \xrightarrow[\sim]{\psi_{r-2}}\cdots\xrightarrow[\sim]{\psi_{1}}
 \big(w^{r-1}\big)/\big(w^{m_xr}\big)\otimes \overline{V}_0
\end{gather*}
coincides with the homomorphism
$(z)/\big(z^{m_x+1}\big)\otimes \overline{V}_0
\longrightarrow \big(w^{r-1}\big)/\big(w^{m_xr}\big)\otimes \overline{V}_0$
obtained by tensoring $\overline{V}_0$ to the canonical homomorphism
$(z)/\big(z^{m_x+1}\big)\longrightarrow \big(w^{r-1}\big)/\big(w^{m_xr}\big)$.
\end{itemize}
Two factorized ramified structures
$(V_k,\vartheta_k,\varkappa_k)$
and
$(V'_k,\vartheta'_k,\varkappa'_k)$
are equivalent if
$V_k~=V'_k$ and
$\langle \vartheta_k,\varkappa_k \rangle=N_k=N'_k=\langle \vartheta'_k,\varkappa'_k \rangle$
for any $k$ and
there are isomorphisms
$\varsigma_k\colon \overline{W}_k\xrightarrow{\sim}
\overline{W}_k$ satisfying
$^tN_{r-k-1}\circ\varsigma_k=\varsigma_k\circ\,^tN_{r-k-1}$, \
$\theta'_k=\theta_k\circ\varsigma_k$, \
$\kappa'_k=\varsigma_k^{-1}\circ\kappa_k$
and the commutative diagrams
\[
 \begin{CD}
 (z)/\big(z^{m_x+1}\big)\otimes\overline{W}_0 @>>> \overline{W}_{r-1} \\
 @V 1\otimes\varsigma_0 V \cong V @V \varsigma_{r-1} V \cong V \\
 (z)/\big(z^{m_x+1}\big)\otimes\overline{W}_0 @>>> \overline{W}_{r-1}
 \end{CD}
 \hspace{40pt}
 \begin{CD}
 \overline{W}_k @>>> \overline{W}_{k-1} \\
 @V \varsigma_k V \cong V @V \varsigma_{k-1} V \cong V \\
 \overline{W}_k @>>> \overline{W}_{k-1}
 \end{CD}
 \quad (1\leq k\leq r-1).
\]
\end{Definition}

\begin{Remark}\rm
The condition $\vartheta_k(v,v')=\vartheta_{r-k-1}(v',v)$ for $v\in\overline{W}_k$,
$v'\in \overline{W}_{r-k-1}$
in Definition~\ref{def-fac-connection}\,(ii)
is equivalent to the condition
$^t(\theta_k)=\theta_{r-k-1}$
under the identifications
$\overline{W}_{r-k-1} = \big(\overline{V}_k\big)^{\vee}$
and
$\big(\overline{W}_k\big)^{\vee} = \overline{V}_{r-k-1}$
for $0\leq k\leq r-1$.
Similarly, the condition $\varkappa_k(v,w)=\varkappa_{r-k-1}(w,v)$ for $v\in\overline{V}_k$,
$w\in \overline{V}_{r-k-1}$
in Definition~\ref{def-fac-connection}\,(ii)
is equivalent to the condition
$^t\kappa_k=\kappa_{r-k-1}$
under the identifications
$\big(\overline{W}_k\big)^{\vee}=\overline{V}_{r-k-1}$,
$\big(\overline{V}_k\big)^{\vee}=\overline{W}_{r-k}$,
\end{Remark}

For a factorized $\nu$-ramified structure
$(V_k,\vartheta_k,\varkappa_k)$ on $(E,\nabla)$,
we can regard the ${\mathcal O}_{m_xx}$-module
$\overline{V}_k=V_k/z^{m_x-1}V_{k+1}$
as a~$\mathbb{C}[w]$-module by using the ring homomorphism in Definition~\ref{def-fac-connection}\,(iv),
(\ref{ring homomorphism defining C[w]-module structure})
and we have $\overline{V}_k\cong\mathbb{C}[w]/\big(w^{m_xr-r+1}\big)$.
The canonical surjection $V_k\longrightarrow\overline{V}_k$
induces a~surjection
$
 \pi_k\colon V_k\otimes_{\mathbb{C}[z]/(z^{m_x})}\mathbb{C}[w]/\big(w^{m_xr-r+1}\big)
 \longrightarrow \overline{V}_k
$
of $\mathbb{C}[w]/\big(w^{m_xr-r+1}\big)$-modules.
For $1\leq k\leq r-1$,
the canonical inclusion $\iota_k\colon V_k\hookrightarrow V_{k-1}$
induces a homomorphism
$\overline{\iota}_{k}\colon\overline{V}_k\longrightarrow\overline{V}_{k-1}$
and the canonical homomorphism
$(z)/\big(z^{m_x+1}\big)\otimes V_{0}\rightarrow
z V_0\hookrightarrow V_{r-1}$
induces a~homomorphism
$\overline{\iota}_r\colon
(z)/\big(z^{m_x+1}\big)\otimes\overline{V}_0\longrightarrow \overline{V}_{r-1}$.
Then $\big(V_k,\overline{V}_k,\pi_k,\overline{\iota}_k\big)$
becomes a generic $\nu$-ramified structure on~$(E,\nabla)$ at $x\in D_{\mathrm{ram}}$
in the sense of Definition~\ref{definition of ramified structure}.

\begin{Proposition} \label{prop:correspondence-factorized}
The correspondence
$(V_k,\vartheta_k,\varkappa_k)\mapsto
\big(V_k,\overline{V}_k,\pi_k,\overline{\iota}_k\big)$
gives a bijection between
the set of equivalence classes of factorized $\nu$-ramified structures
on $(E,\nabla)$ at $x\in D_{\mathrm{ram}}$
and the set of isomorphism classes of
generic $\nu$-ramified structures on $(E,\nabla)$ at $x\in D_{\mathrm{ram}}$.
\end{Proposition}

\begin{proof}
We will construct the inverse correspondence.
Let $(V_k,L_k,\pi_k,\phi_k)$
be a generic $\nu$-ramified structure on $(E,\nabla)$ at $x\in D_{\mathrm{ram}}$.
By Definition~\ref{definition of ramified structure}\,(ii), the restriction
$\pi_k|_{V_k}\colon V_k\longrightarrow L_k$
is a~surjection, which induces the isomorphism
$\overline{V}_k=V_k/z^{m_x-1}V_{k+1}
\xrightarrow{\sim} L_k$.
Take a~generator~$\bar{e}_0$ of $L_0$ as a $\mathbb{C}[w]$-module.
Let $\bar{e}_k$ be the element of $L_k$
which corresponds to $w^k\otimes \bar{e}_0$
via the isomorphism
\[
 L_k\xrightarrow[\sim]{\psi_k}(w)\otimes L_{k-1}
 \xrightarrow[\sim]{\psi_{k-1}}\cdots\xrightarrow[\sim]{\psi_1}\big(w^k\big)\otimes L_0.
\]
Since $\pi_k|_{V_k}$ is surjective, we can take an element $e_k\in V_k$ satisfying
$\pi_k(e_k)=\bar{e}_k$.
Then $e_0,e_1,\dots,e_{r-1}$ is a basis of the free ${\mathcal O}_{m_xx}$-module
$E|_{m_xx}$ and
we have
\begin{gather*}
 \pi_k(e_l)
 =
 (\phi_{k+1}\circ\cdots\circ\phi_l)(\pi_l(e_l))
 =w^{l-k}\pi_k(e_k) \qquad \text{if $k\leq l\leq r-1$},
 \\
 \pi_k(ze_l)
 =
 (\phi_{k+1} \circ\cdots\circ\phi_r)(z\otimes\pi_0(e_l))
 =w^{r-k+l}\pi_k(e_k)
 \qquad \text{if $0\leq l<k$}.
\end{gather*}
Furthermore, $V_k$ is generated by $e_k,e_{k+1},\dots,e_{r-1},ze_0,\dots,ze_{k-1}$.
If we define a homomorphism
$
 N\colon E|_{m_xx}\longrightarrow E|_{m_xx}
$
by
\[
 N(e_k)
 =
 \begin{cases}
 e_{k+1} & \text{if $0\leq k\leq r-2$}, \\
 ze_0 & \text{if $k=r-1$},
 \end{cases}
\]
then $N$ preserves $V_k$ and the diagram
\[
 \begin{CD}
 V_k @> \pi_k|_{V_k}>> L_k \\
 @V N|_{V_k} VV @VV w V \\
 V_k @> \pi_k|_{V_k} >> L_k
 \end{CD}
\]
is commutative.
By the definition, we have the equality $N^r=z\cdot\mathrm{id}_{E|_{m_xx}}$.
The induced ring homomorphism
\[
 {\mathcal O}_{m_xx}[w]/\big(w^r-z\big)
 \ni f(w) \mapsto f(N) \in \End_{{\mathcal O}_{m_xx}}(E|_{m_xx})
\]
endows $E|_{m_xx}$ with a structure of ${\mathcal O}_{m_xx}[w]$-module.
Since the minimal polynomial of
$N|_x$ is $w^r$ whose degree is $r$,
we can see $E|_x\cong\mathbb{C}[w]/(w^r)$ by elementary linear algebra.
By Nakayama's lemma, we can extend it to an isomorphism
\begin{equation} \label{equation: isomorphism by PID theory}
 E|_{m_xx}\cong{\mathcal O}_{m_xx}[w]/\big(w^r-z\big)
\end{equation}
of ${\mathcal O}_{m_xx}[w]$-modules.
Similarly, the endomorphism
$^tN$ on $E|_{m_xx}^{\vee}$
induces a structure of ${\mathcal O}_{m_xx}[w]$-module and we have an isomorphism
\begin{equation} \label{equation: isomorphism by PID theory (dual)}
 E|_{m_xx}^{\vee}\cong{\mathcal O}_{m_xx}[w]/\big(w^r-z\big).
\end{equation}
Combining (\ref{equation: isomorphism by PID theory}) and
(\ref{equation: isomorphism by PID theory (dual)}),
we get an isomorphism
\[
 \theta\colon \ E|_{m_xx}^{\vee} \stackrel{\sim}\longrightarrow E|_{m_xx}
\]
of ${\mathcal O}_{m_xx}[w]$-modules.
Let
\begin{equation} \label{equation: total symmetric pairing}
 \vartheta\colon \
 E|_{m_xx}^{\vee} \times E|_{m_xx}^{\vee} \longrightarrow {\mathcal O}_{m_xx}
\end{equation}
be the corresponding bilinear pairing defined by
$\vartheta(v^*,w^*)=w^*(\theta(v^*))$
for $v^*,w^*\in E|_{m_xx}^{\vee}$.
Take a generator $e^*$ of $E|_{m_xx}^{\vee}$
as an ${\mathcal O}_{m_xx}[w]$-module.
Then any element $v^*,w^*\in E|_{m_xx}^{\vee}$
can be written $v^*=P\big({}^tN\big)e^*$, $w^*=Q\big({}^tN\big)e^*$
for polynomials $P(w),Q(w)\in{\mathcal O}_{m_xx}[w]$ in $w$.
So we have
\begin{align}
 \vartheta(v^*,w^*)=w^*(\theta(v^*))
 & =
 \big(Q\big({}^tN\big)e^*\big) \big(\theta(P\big({}^tN\big)e^*)\big)\nonumber \\
 &=
 (e^*\circ Q(N)) (P(N)(\theta(e^*)))\nonumber\\
 &=
 (e^*\circ Q(N)\circ P(N)\circ \theta) (e^*)\nonumber \\
 &= (e^*\circ P(N)\circ Q(N)\circ \theta) (e^*)
 =\vartheta(w^*,v^*). \label{equation: symmetric property of theta}
\end{align}
In other words, the pairing $\vartheta$ defined in (\ref{equation: total symmetric pairing})
is symmetric,
which is also equivalent to $^t\theta=\theta$.
If we put
\[
 \kappa:=\theta^{-1}\circ N \colon \
 E|_{m_xx} \longrightarrow E|_{m_xx}^{\vee},
\]
then we have $\theta\circ\kappa=N$.
By the similar calculation to (\ref{equation: symmetric property of theta}),
we can see that the bilinear pairing
\[
 \varkappa \colon \ E|_{m_xx} \times E|_{m_xx}
 \longrightarrow {\mathcal O}_{m_xx},
\]
determined by $\varkappa(v,w)=\kappa(v)(w)$ is also symmetric,
which is equivalent to $^t\kappa=\kappa$.

Now we put
\[
 W_k:=\big\{
 v^*\in E|_{m_xx}^{\vee} \mid
 v^*\big(z^{m_x-1}V_{r-k}\big)=0 \big\}
 =\ker \big(z^{m_x-1}\big({}^tN\big)^{r-k}\big)
\]
for $0\leq k\leq r$.
Then we get the exact commutative diagram
\[
 \begin{CD}
 & & z^{m_x-1}W_{k+1} & = & z^{m_x-1}W_{k+1} & \\
 & & @VVV @VVV & \\
 0 @>>> W_k @>>> E|_{m_xx}^{\vee} @>>>
 & (z^{m_x-1}V_{r-k})^{\vee} & \longrightarrow 0 \\
 & & @VVV @VVV & \parallel & & \\
 0 @>>> \overline{V}_{r-k-1}^{\vee} @>>> V_{r-k-1}^{\vee} @>>>
 & (z^{m_x-1}V_{r-k})^{\vee} & \longrightarrow 0 \\
 & & @VVV @VVV \\
 & & 0 & & 0.
 \end{CD}
\]
So we have an isomorphism
\[
 W_k/z^{m_x-1}W_{k+1} \stackrel{\sim}\longrightarrow
 \overline{V}_{r-k-1}^{\vee}=\overline{W}_k.
\]
Using $W_k=\ker\big(z^{m_x-1}\big({}^tN\big)^{r-k}\big)$,
we can see
\[
\theta(W_k)
=\theta \big(\ker\big(z^{m_x-1}\big({}^tN\big)^{r-k}\big)\big)
=\ker\big(z^{m_x-1}N^{r-k}\big)=V_k.
\]
So $\theta|_{W_k}$ induces an isomorphism
$\theta_k\colon \overline{W}_k\xrightarrow{\sim} \overline{V}_k$
which makes the diagram
\[
 \begin{CD}
 W_k @>\theta|_{W_k}>\sim> V_k \\
 @VVV @VVV \\
 \overline{W}_k @>\theta_k>\sim> \overline{V}_k
 \end{CD}
\]
commutative.
By the equality
$\kappa=\theta^{-1}N$, we have
$\kappa(V_k)\subset W_k$ for $0\leq k\leq r$ and
get the commutative diagram
\[
 \begin{CD}
 V_k @>\kappa|_{V_k}>\sim> W_k \\
 @VVV @VVV \\
 \overline{V}_k @>\kappa_k>\sim> \overline{W}_k.
 \end{CD}
\]

We can associate $(\vartheta_k,\varkappa_k)$ to $(\theta_k,\kappa_k)$
and the conditions (ii) and (iii) of Definition \ref{def-fac-connection}
follow from the properties of $\theta$, $\kappa$.
The other conditions (i), (iv) and (v) of Definition~\ref{def-fac-connection}
are satisfied by that of $(V_k,L_k,\pi_k,\phi_k)$.
So we get a factorized $\nu(w)$-ramified structure
$(V_k,\vartheta_k,\varkappa_k)$.

Assume that there is another factorized ramified structure
$(V_k,\vartheta'_k,\varkappa'_k)$
which gives the same generic $\nu$-ramified structure
$(V_k,L_k,\pi_k,\phi_k)$.
Recall that $\overline{V}_k\xrightarrow{\sim}L_k$.
So we have $\theta'_k\circ\kappa'_k=N_k=\theta_k\circ\kappa_k$,
because both sides correspond to the multiplication by $w$ on $L_k$.
Since the diagram
\[
 \begin{CD}
 \overline{W}_k \ @>\theta'_k>> \ \overline{V}_k \\
 @V ^tN_k= V \kappa'_k\circ\theta'_k V @V N_k V =\theta'_k\circ\kappa'_k V \\
 \overline{W}_k \ @>\theta'_k>> \ \overline{V}_k
 \end{CD}
\]
is commutative,
$\theta'_k\colon \overline{W}_k\xrightarrow{\sim}\overline{V}_k$
is an isomorphism of
free $\mathbb{C}[w]/\big(w^{m_xr-r+1}\big)$-modules of rank one.
So there is an element
$\beta_k(w)\in\mathbb{C}[w]/\big(w^{m_xr-r+1}\big)^{\times}$ such that
$\theta_k=\theta'_k\circ \beta_k(\,^tN_k)$.
Then we also have
$\kappa_k=\beta_k\big({}^tN_k\big)^{-1}\circ\kappa'_k$.
Taking account of the compatibility of
$(\theta'_k,\kappa'_k)$ with $(\theta'_{k-1},\kappa'_{k-1})$,
we can see
$\beta_k(w)\equiv \beta_{k-1}(w) \pmod{w^{mr-r}}$
for $k=1,\dots,r-1$.
Thus we have
$(V_k,\vartheta'_k,\varkappa'_k)\sim(V_k,\vartheta_k,\varkappa_k)$.
In other words, the equivalence class of
factorized $\nu$-ramified structure $(V_k,\theta_k,\kappa_k)$
is uniquely determined by the generic $\nu$-ramified structure
$(V_k,L_k,\pi_k,\phi_k)$.
So we can define a~correspondence
\[
 (V_k,L_k,\pi_k,\phi_k)\mapsto
 (V_k,\theta_k,\kappa_k)
\]
and it is the inverse to the correspondence stated in the proposition.
\end{proof}

\begin{Example}\label{example: formal factorized ramified structure}
We will see what the factorized ramified structure is
in the typical case explained in Example~\ref{example: ramified structure}.
We have
$(E,\nabla)\otimes\widehat{\mathcal O}_{C,x}=
(\mathbb{C}[[w]],\nabla_{\nu})$
in that case
and the filtration in Definition~\ref{def-fac-connection}\,(i) is given by
$V_k=\big(w^k\big)/\big(w^{mr}\big)$ for $0\leq k\leq r$.
Consider the trace map
\begin{equation*}
 \Tr \colon \ \mathbb{C}[[w]] \longrightarrow \mathbb{C}[[z]].
\end{equation*}
For
$f(w)\in\mathbb{C}[[w]]$,
$\Tr(f(w))$ is defined as the trace of the endomorphism
$\mathbb{C}[[w]]\xrightarrow{f(w)}\mathbb{C}[[w]]$
on the free $\mathbb{C}[[z]]$-module $\mathbb{C}[[w]]$
of rank $r$.
By construction, we have
$\Tr\big(z^l\big)=rz^l$ and $\Tr\big(w^kz^l\big)=0$ for $1\leq k\leq r-1$.
So the above map induces a homomorphism
$\Tr\colon\mathbb{C}[w]/\big(w^{mr-r+1}\big)
\rightarrow\mathbb{C}[z]/(z^{m})$,
which also induces
\begin{gather*}
 \mathbf{Tr}\colon \
 \big(w^{r-1}\big)/
 \big(w^{mr}\big)\otimes\Omega^1_{\mathbb{C}[[w]]/\mathbb{C}}\!
 =
 \mathbb{C}[w]/\big(w^{mr-r+1}\big)\otimes
 \Omega^1_{\mathbb{C}[[z]]/\mathbb{C}}
 \!\xrightarrow {\Tr\otimes \mathrm{id}}\!
 \mathbb{C}[z]/\big(z^{m}\big)\otimes\Omega^1_{\mathbb{C}[[z]]/\mathbb{C}}.
\end{gather*}
Then we can define a pairing
\[
 \Theta_k\colon \
 \big(w^k\big)/\big(w^{mr-r+k+1}\big) \times \big(w^{r-k-1}\big)/\big(w^{mr-k}\big)
 \longrightarrow \mathbb{C}[z]/\big(z^m\big)
\]
by setting
\[
 \Theta_k \big( f(w),g(w) \big){\rm d}z=\mathbf{Tr} \big( f(w)g(w){\rm d}w \big)
\]
for $f(w)\in \big(w^k\big)/\big(w^{mr-r+k+1}\big)$ and
$g(w)\in \big(w^{r-k-1}\big)/\big(w^{mr-k}\big)$.
By the construction, the induced $\mathbb{C}[z]/(z^m)$-homomorphism
$\big(w^k\big)/\big(w^{mr-r+k+1}\big) \rightarrow
\big(\big(w^{r-k-1}\big)/\big(w^{mr-k}\big)\big)^{\vee}$
is an isomorphism.
If we denote the inverse of this homomorphism by
\[
\theta_k\colon \ \big(\big(w^{r-k-1}\big)/\big(w^{mr-k}\big)\big)^{\vee}
\xrightarrow{\sim}
\big(w^k\big)/\big(w^{mr-r+k+1}\big),
\]
then $\theta_k$ induces a pairing
\[
 \vartheta_k\colon \
 \big(\big(w^{r-k-1}\big)/\big(w^{mr-k}\big)\big)^{\vee} \times \big(\big(w^k\big)/\big(w^{mr-r+k+1}\big)\big)^{\vee}
 \longrightarrow
 \mathbb{C}[z]/\big(z^m\big)
\]
satisfying
$\vartheta_k(v,v')=\vartheta_{r-k-1}(v',v)$
for $v\!\in\! \big(\big(w^{r-k-1}\big)/\big(w^{mr-k}\big)\big)^{\vee}$
and $v'\!\in\! \big(\big(w^k\big)/\big(w^{mr-r+k+1}\big)\big)^{\vee}\!$.
We can also define a pairing
\[
 \varkappa_k\colon \
 \big(w^k\big)/\big(w^{mr-r+k+1}\big) \times \big(w^{r-k-1}\big)/\big(w^{mr-k}\big)
 \longrightarrow \mathbb{C}[z]/\big(z^m\big)
\]
by setting
\[
 \varkappa_k(f(w),g(w))=\Theta_k(wf(w)g(w))
\]
for $f(w)\in \big(w^k\big)/\big(w^{mr-r+k+1}\big)$ and
$g(w)\in \big(w^{r-k-1}\big)/\big(w^{mr-k}\big)$.
We can see that the filtration
$\mathbb{C}[[w]]/z^m\mathbb{C}[[w]]
\supset (w)/\big(w^{mr}\big)\supset \big(w^2\big)/\big(w^{mr}\big)
\supset\cdots\supset \big(w^{r-1}\big)/\big(w^{mr}\big)
\supset z\mathbb{C}[[w]]/z^m\mathbb{C}[[w]]$
together with $(\vartheta_k,\varkappa_k)_{0\leq k\leq r-1}$
gives a factorized $\nu$-ramified structure
on $(E,\nabla)$ at~$x$.
\end{Example}

\begin{Remark}We can extend the notion of generic $\nu$-ramified structure
or that of factorized $\nu$-ramified structure in a relative setting.
So, if $S$ is a noetherian scheme (or a noetherian ring) and
if $(E,\nabla)$ is a pair of a vector bundle $E$ on $C\times S$
and a connection $\nabla$ on $E$,
we can mention about a generic $\nu$-ramified structure
on $(E,\nabla)$.
\end{Remark}

\section[Recovery of formal structure from a generic ramified structure]{Recovery of formal structure from a generic ramified\\ structure}
\label{section: recovery of formal structure}

In this section, we will see in Corollary \ref{corollary: 3 conditions are equivalent}
that the generic ramified condition
given in the Introduction
is equivalent to the generic ramified structure
(Definition \ref{definition of ramified structure})
or the factorized $\nu$-ramified structure
(Definition \ref{def-fac-connection}).
The most essential point is to recover a formal isomorphism
from a generic ramified structure or a factorized $\nu$-ramified structure
(in Corollary \ref{corollary: formal structure}).
In fact, we proved it in \cite[Proposition 1.3]{Inaba-2}
by using the Hukuhara--Levert--Turrittin theorem
(see \cite[Proposition 1.4.1]{Babbitt-Varadarajan} or \cite[Theorem 6.8.1]{Sibuya} for example).
In this paper, we will examine it by a direct computation only by using regular formal transforms
rather than formal Laurent transforms in the Hukuhara--Levert--Turrittin theorem.
It has the advantage of applying to (\ref{equation: family of formal isomorphisms})
or (\ref{equation: formal transform in w}) later.
For such applications, we actually require a formal isomorphism in a relative setting
in Corollary \ref{corollary: formal structure}.

Let $A$ be a noetherian ring over $\mathbb{C}$.
Take a flat family $U\longrightarrow \Spec A$ of
smooth affine curves over $\Spec A$
and let $\tilde{x}$ be a section of $U$ over $\Spec A$.
We can take a local defining equation
$z\in{\mathcal O}_U$ of $\tilde{x}$.
Let $w$ be a variable satisfying $w^r=z$.
We take an integer $m$ with $m\geq 2$.
Choose
\begin{equation} \label{equation: coefficient data of ramified exponent}
 \big(a^{(0)}_0,a^{(0)}_1,\dots,a^{(0)}_{m-1}\big)\in A^m,
 \qquad
 \big(a^{(k)}_0,a^{(k)}_1,\dots,a^{(k)}_{m-2}\big)\in A^{m-1}
 \qquad (k=1,\dots,r-1)
\end{equation}
with the condition $a^{(1)}_0\in A^{\times}$.
Using the data~(\ref{equation: coefficient data of ramified exponent}),
we put
\begin{equation} \label{equation: ramified exponent given by coefficient data}
 \nu_0(z)
 =
 \sum_{l=0}^{m-1} a^{(0)}_l z^l \frac{{\rm d}z}{z^m} ,
 \qquad
 \nu_k(z)
 =
 \sum_{l=0}^{m-2} a^{(k)}_l z^l \frac{{\rm d}z}{z^m}
 \qquad
 (k=1,\dots,r-1)
\end{equation}
and set
\begin{equation} \label{equation: total ramified exponent given by coefficient data}
 \nu(w):=
 \nu_0(z)+\nu_1(z)w+\cdots+\nu_{r-1}(z)w^{r-1}.
\end{equation}
For an integer $q$ with $q\geq m$,
we can regard $A[w]/(w^{qr})$ as a free $A[z]/(z^q)$-module of rank~$r$.
Define the $A$-linear homomorphism
\[
 \nabla_{\nu}|_{q\tilde{x}} \colon \
 A[w]/(w^{qr}) \longrightarrow
 A[w]/(w^{qr})\otimes \Omega^1_{U/A}(m\tilde{t})|_{q\tilde{x}}
\]
by setting
$\nabla_{\nu}|_{q\tilde{x}}(f(w))=df(w)+f(w)\nu(w)$
for $f(w)\in A[w]/(w^{qr})$.

We need the following proposition in the construction of
generalized isomonodromic deformation later in
Sections~\ref{section: horizontal lift} and~\ref{section: generalized isomonodromy equation}.

\begin{Proposition}\label{proposition: normalization of formal type over finite scheme}
Let the notations be as
in \eqref{equation: coefficient data of ramified exponent},
\eqref{equation: ramified exponent given by coefficient data}
and \eqref{equation: total ramified exponent given by coefficient data}
with the assumption that the leading coefficient
$a^{(1)}_0$ of $\nu_1(z)$ is invertible in $A$.
Take a vector bundle $E$ on $U$ of rank $r$ and a connection
$\nabla\colon E\longrightarrow E\otimes \Omega^1_{U/A}(m\tilde{x})$
with a generic $\nu$-ramified structure
\[
\big( (V_k,\pi_k,,L_k)_{0\leq k\leq r-1} ,(\phi_k)_{1\leq k\leq r}\big)
\]
at $\tilde{x}$.
Then, for any integer $q$ with $q\geq m$, there is an isomorphism
\[
 \sigma \colon E|_{q\tilde{x}}
 \xrightarrow{\sim}
 (A[z]/(z^{q}))[w]/(w^r-z)
 \cong A[w]/(w^{qr})
\]
which makes the diagram
\[
 \begin{CD}
 E|_{q\tilde{x}} @>\sigma>\sim> A[w]/(w^{qr}) \\
 @V \nabla|_{q\tilde{x}} VV
 @V \nabla_{\nu}|_{q\tilde{x}} VV \\
 E|_{q \tilde{x}} \otimes \Omega^1_{U/A}(m\tilde{x})|_{q\tilde{x}}
 @>\sigma\otimes 1>\sim>
 A[w]/(w^{qr}) \otimes \frac {{\rm d}z} {z^m}
 \end{CD}
\]
commutative.
\end{Proposition}

\begin{proof}Let $\tilde{V}_k$ be the pullback of $V_k$ via the canonical surjection
$E|_{q\tilde{x}} \longrightarrow E|_{m\tilde{x}}$
for $0\leq k \leq r-1$.
We take a generator $e'_0\in L_0$ as an
$A[w]/\big(w^{mr-r+1}\big)$-module.
By the condition~(iv) of Definition~\ref{definition of ramified structure},
there is a composition of isomorphisms
\[
 L_k\xrightarrow[\sim]{\psi_k} (w)\otimes L_{k-1}\xrightarrow[\sim]{\psi_{k-1}}
 \cdots\xrightarrow[\sim]{\psi_1}\big(w^k\big)\otimes L_0.
\]
Let $e'_k\in L_k$ be the element corresponding to $w^k\otimes e'_0$
via this isomorphism.
Since
$ \pi_k|_{V_k} \colon V_k \xrightarrow {\pi_k|_{V_k}} L_k $
is surjective, we can take
$\bar{e}_k\in V_k$
satisfying
$\pi_k(\bar{e}_k)=e'_k$.
Then we have
\begin{gather*}
 \pi_k(\bar{e}_{k+l})
 =
 w^l\pi_k(\bar{e}_k)
 \qquad
 \text{for $0\leq l\leq r-k-1$},
 \\
 \pi_k(z\bar{e}_l)
 =
 w^{r-k+l}\pi_k(\bar{e}_k)
 \qquad
 \text{for $0\leq l\leq k-1$}.
\end{gather*}
We take lifts $e_0,e_1,\dots,e_{r-1}\in E|_{q\tilde{x}}$
of $\bar{e}_0,\bar{e}_1,\dots,\bar{e}_{r-1}\in E|_{m\tilde{x}}$.
The commutativity of the diagram in
Definition~\ref{definition of ramified structure}\,(ii)
yields the equality
\begin{gather*}
 \nabla|_{q\tilde{x}}(e_k)
 \equiv
 \left( \nu_0(z) + \frac {k\,{\rm d}z} {r z} \right)\! e_k
 +
 \sum_{l=k+1}^{r-1}\! \nu_{l-k}(z) e_l
 +\sum_{l=0}^{k-1} z \nu_{r+l-k}(z) e_l
 \quad
 \left(\! \bmod \ z^{m-1}\tilde{V}_{k+1}\frac{{\rm d}z}{z^m} \right)
\end{gather*}
for $k=0,1,\dots,r-1$.
Applying the following lemma to the cases
\[
 (q',s)=(m,1),(m,2),\dots,(m,r),(m+1,1),(m+1,2),\dots,(q,1),\dots,(q,r-1)
\]
successively, we get the proposition.
\end{proof}

\begin{Lemma}Let $q'$, $s$ be integers with $m\leq q'\leq q$ and $1\leq s\leq r$.
Assume that the equalities
\begin{gather} \label{equation: assumption of formal equivalence modulo first case}
 \nabla|_{q\tilde{x}}(e_k)
 \equiv
 \left( \nu_0+\frac{k\,{\rm d}z}{r z} \right) e_k
 +\sum_{l=k+1}^{r-1} \nu_{l-k} e_l
 +\sum_{l=0}^{k-1} \nu_{r+l-k} z e_l
 \qquad \left(\! \bmod \ z^{q'-1}\tilde{V}_{k+s}\frac{{\rm d}z}{z^m} \right)\!\!\!
\end{gather}
hold for $0\leq k < r-s$
and
the equalities
\begin{equation} \label{equation: assumption of formal equivalence modulo second case}
 \nabla|_{q\tilde{x}}(e_k)
 \equiv
 \left( \nu_0+\frac{k\,{\rm d}z}{r\,z} \right) e_k
 +\sum_{l=k+1}^{r-1} \nu_{l-k} e_l
 +\sum_{l=0}^{k-1} \nu_{r+l-k} z e_l
 \qquad \left(\! \bmod \ z^{q'}\tilde{V}_{k+s-r}\frac{{\rm d}z}{z^m} \right)
\end{equation}
hold for $r-s\leq k\leq r-1$.
Then there exist $c,b_1,\dots,b_{r-1}\in A$
such that the replacement
\begin{gather}
 \tilde{e}_0
 =
\begin{cases}
 e_0+c z^{q'-m} e_s & \text{if $1\leq s\leq r-1$}, \\
 e_0+c z^{q'-m+1}e_0 & \text{if $s=r$},
 \end{cases}\nonumber
 \\
 \tilde{e}_k
 =
 \begin{cases}
 e_k+c z^{q'-m}e_{k+s}+b_kz^{q'-1}e_{k+s-1}
 & \text{if $k+s <r$ and $1\leq k\leq r-1$},
 \\
 e_k+c z^{q'-m+1} e_{k+s-r}+b_k z^{q'-1}e_{k+s-1}
 & \text{if $k+s= r$ and $1\leq k\leq r-1$},
 \\
 e_k+c z^{q'-m+1} e_{k+s-r}+b_kz^{q'}e_{k+s-1-r}
 & \text{if $k+s>r$ and $1\leq k\leq r-1$}
 \end{cases}\label{equation: replacement in general case}
 \end{gather}
leads to the equalities
\begin{gather} \label{equation: equivalence in the next step case 1}
 \nabla|_{q\tilde{x}}(\tilde{e}_k)
 \equiv
 \left( \nu_0+\frac{k\,{\rm d}z}{r z} \right) \tilde{e}_k
 +\sum_{l=k+1}^{r-1} \nu_{l-k} \tilde{e}_l
 +\sum_{l=0}^{k-1} \nu_{r+l-k} z \tilde{e}_l
 \quad \left(\! \bmod \ z^{q'-1}\tilde{V}_{k+s+1}\frac{{\rm d}z}{z^m} \right)
\end{gather}
for $0\leq k<r-s-1$
and the equalities
\begin{gather} \label{equation: equivalence in the next step case 2}
 \nabla|_{q\tilde{x}}(\tilde{e}_k)
 \equiv
 \left( \nu_0+\frac{k\,{\rm d}z}{r z} \right) \tilde{e}_k
 +\sum_{l=k+1}^{r-1}\! \nu_{l-k} \tilde{e}_l
 +\sum_{l=0}^{k-1} \nu_{r+l-k} z \tilde{e}_l
 \quad \left(\! \bmod \ z^{q'}\tilde{V}_{k+s+1-r}\frac{{\rm d}z}{z^m} \right)\!\!\!
\end{gather}
for $r-s-1\leq k\leq r-1$.
\end{Lemma}

\begin{proof}
By the assumption (\ref{equation: assumption of formal equivalence modulo first case}),
we can find
$\eta_0,\dots,\eta_{r-s-1}\in z^{q'-1}\Omega^1_{U/A}(D)|_{q\tilde{x}}$
satisfying the equalities
\begin{gather*}
 \nabla|_{q\tilde{x}} (e_k)
 \equiv
 \left( \nu_0+\frac{k\,{\rm d}z}{r z} \right) e_k
 +\sum_{l=k+1}^{r-1}\nu_{l-k} e_l\\
 \hphantom{\nabla|_{q\tilde{x}} (e_k) \equiv}{}
 +\sum_{l=0}^{k-1} \nu_{r+l-k} z e_l
 +\eta_k e_{k+s}
 \qquad
 \left(\! \bmod \ z^{q'-1}\tilde{V}_{k+s+1}\frac{{\rm d}z}{z^m} \right)
\end{gather*}
for $0\leq k<r-s-1$ and the equality
\begin{gather*}
 \nabla|_{q\tilde{x}} (e_k)
 \equiv
 \left( \nu_0+\frac{k\,{\rm d}z}{r z} \right) e_k
 +\sum_{l=k+1}^{r-1}\nu_{l-k} e_l
 +\sum_{l=0}^{k-1} \nu_{r+l-k} z e_l
 +\eta_k e_{r-1}
 \quad
 \left(\! \bmod \ z^{q'}\tilde{V}_0\frac{dz}{z^m} \right)
\end{gather*}
for $k+s=r-1$.
By the assumption (\ref{equation: assumption of formal equivalence modulo second case}),
we can find $\eta_{r-s},\dots,\eta_{r-1}\in z^{q'-1}\Omega^1_{U/A}(D)|_{q\tilde{x}}$
satisfying the equalities
\begin{gather*}
 \nabla|_{q\tilde{x}} (e_k)
 \equiv
 \left( \nu_0+\frac{k\,{\rm d}z}{r z} \right) e_k
 +\sum_{l=k+1}^{r-1}\nu_{l-k} e_l\\
 \hphantom{\nabla|_{q\tilde{x}} (e_k) \equiv}{}
 +\sum_{l=0}^{k-1} \nu_{r+l-k} z e_l
 +\eta_k z e_{k+s-r}
 \qquad
 \left(\! \bmod \ z^{q'}\tilde{V}_{k+s+1-r}\frac{{\rm d}z}{z^m} \right)
\end{gather*}
for $r-s\leq k \leq r-1$.
We will determine $c,b_1,\dots,b_{r-1}\in A$
so that the substitution of~(\ref{equation: replacement in general case})
enables the equalities~(\ref{equation: equivalence in the next step case 1})
and~(\ref{equation: equivalence in the next step case 2}) to hold.

Consider the substitution of $\tilde{e}_k$
for $0\leq k<r-s$.
In that case, we have
\begin{gather*}
 \nabla|_{q\tilde{x}}(\tilde{e}_k) =
 \nabla|_{q\tilde{x}}(e_k)+(q'-m)c z^{q'-m-1}{\rm d}z e_{k+s}
 +cz^{q'-m}\nabla(e_{k+s})
 \\
\hphantom{\nabla|_{q\tilde{x}}(\tilde{e}_k) =}{}
 +(q'-1)b_kz^{q'-2}{\rm d}z e_{k+s-1}
 +b_kz^{q'-1}\nabla(e_{k+s-1}).
\end{gather*}
If we put $b_0:=0$ and $b_r:=0$,
then we can calculate the above substitution
in the following,
while using
$b_kz^{q'-1}\nu_{l-k-s+1} e_l \equiv 0 \pmod{z^{q'-1}\tilde{V}_{k+s+1}{\rm d}z/z^m}$
for $l\geq k+s+1$
in the second equality;
\begin{gather*}
 \nabla|_{q\tilde{x}} (\tilde{e}_k) \equiv
 (q'-m)c z^{q'-m-1} {\rm d}z e_{k+s}
 +(q'-1)b_k z^{q'-2} dz e_{k+s-1}
 +\left( \nu_0 + \frac {k \,{\rm d}z} {r z} \right) e_k
 +\nu_1 e_{k+1}
 \\
 \hphantom{\nabla|_{q\tilde{x}} (\tilde{e}_k) \equiv}{}
 +\sum_{l=k+2}^{r-1}\nu_{l-k} e_l
 +\sum_{l=0}^{k-1} \nu_{r+l-k} z e_l
 + c z^{q'-m} \Big( \nu_0 +\frac{(k+s)}{r}\frac{{\rm d}z}{z} \Big) e_{k+s}
 \\
 \hphantom{\nabla|_{q\tilde{x}} (\tilde{e}_k) \equiv}{}
 +c z^{q'-m} \nu_1 e_{k+s+1} + \sum_{l=k+s+2}^{r-1} c z^{q'-m} \nu_{l-k-s} e_l
 +\sum_{l=0}^{k+s-1} c z^{q'-m+1} \nu_{r+l-k-s} e_l
 \\
 \hphantom{\nabla|_{q\tilde{x}} (\tilde{e}_k) \equiv}{}
 +b_k z^{q'-1}\nu_0 e_{k+s-1} + b_k z^{q'-1}\nu_1 e_{k+s}
 +\sum_{l=k+s+1}^{r-1} b_k z^{q'-1}\nu_{l-k-s+1} e_l\\
 \hphantom{\nabla|_{q\tilde{x}} (\tilde{e}_k) \equiv}{}
 +\sum_{l=0}^{k+s-2} b_k z^{q'} \nu_{r+l-k-s+1} e_l
 +\eta_k e_{k+s}
 \\
\hphantom{\nabla|_{q\tilde{x}} (\tilde{e}_k)}{}
 \equiv
 \left( \!\nu_0\!+\frac {k\,{\rm d}z} {r z} \!\right)\! \tilde{e}_k
 +\left(\! \nu_s +\eta_k+\frac { ( (q'\!-m)r+s )z^{q'-1} c } {r z^m} {\rm d}z
 +(b_k\!-b_{k+1})z^{q'-1}\nu_1 \!\right)\! \tilde{e}_{k+s}
 \\
 \hphantom{\nabla|_{q\tilde{x}} (\tilde{e}_k) \equiv}{}
 +\nu_1 \tilde{e}_{k+1}
 +\sum_{k+2\leq l\leq r-1, l \neq k+s} \nu_{l-k} \tilde{e}_l
 \ +\sum_{l=0}^{k-1} \nu_{r+l-k} z \tilde{e}_l
 \quad
 \left( \!\bmod \ z^{q'-1}\tilde{V}_{k+s+1} \frac {{\rm d}z} {z^m} \right).
\end{gather*}
We can similarly calculate the substitution of $\tilde{e}_k$ for
$r-s\leq k \leq r-1$ and we have
\begin{gather*}
 \nabla|_{q\tilde{x}} ( \tilde{e}_k ) \equiv
 \left( \nu_0 + \frac {k\,{\rm d}z} {rz} \right) \tilde{e}_k\\
 \hphantom{\nabla|_{q\tilde{x}} ( \tilde{e}_k ) \equiv}{}
 +\left( \nu_s +\eta_k + \frac { ( (q'\!-m)r+s )z^{q'-1} c } {r z^m} {\rm d}z
 +(b_k-b_{k+1})z^{q'-1}\nu_1 \right) z \tilde{e}_{k+s-r}
 \\
 \hphantom{\nabla|_{q\tilde{x}} ( \tilde{e}_k ) \equiv}{}
 +\sum_{k+1\leq l \leq r-1, l \neq k+s} \nu_{l-k} \tilde{e}_l
 \ +\sum_{l=0}^{k-1} \nu_{r+l-k} z \tilde{e}_l
 \quad
 \left( \!\bmod \ z^{q'}\tilde{V}_{k+s+1-r} \frac {{\rm d}z} {z^m} \right).
\end{gather*}
So it is sufficient to solve the equation
\begin{equation*}
 \begin{pmatrix}
 \frac{((q'-m)r+s)z^{q'-1}}{r z^m} {\rm d}z & -z^{q'-1}\nu_1 & 0 & \cdots & 0 \vspace{1mm}\\
 \frac { ((q'-m)r+s)z^{N'-1} } {r z^m} {\rm d}z & z^{q'-1}\nu_1 & -z^{q'-1}\nu_1 & \ddots & 0 \vspace{1mm}\\
 \frac { ((q'-m)r+s)z^{q'-1} } {r z^m} {\rm d}z & 0 & z^{q'-1}\nu_1 & \ddots & 0 \vspace{1mm}\\
 \vdots & \vdots & \ddots & \ddots & -z^{q'-1}\nu_1 \\
 \frac{((q'-m)r+s)z^{q'-1}} {r z^m} {\rm d}z
 & 0 & \cdots & 0 & z^{q'-1}\nu_1
 \end{pmatrix}
 \begin{pmatrix} c \\ b_1 \\ b_2 \\ \vdots \\ b_{r-2} \\ b_{r-1} \end{pmatrix}
 =
 \begin{pmatrix} -\eta_0 \\ -\eta_1 \\ -\eta_2 \\
 \vdots \\ -\eta_{r-2} \\ -\eta_{r-1} \end{pmatrix},
\end{equation*}
which is possible because the $r\times r$ matrix of the left hand side
is invertible.
\end{proof}

Under the setting
(\ref{equation: coefficient data of ramified exponent}),
(\ref{equation: ramified exponent given by coefficient data})
and (\ref{equation: total ramified exponent given by coefficient data}),
let
$\nabla_{\nu}\colon A[[w]]\longrightarrow A[[w]]\otimes \Omega^1_{U/A}\big(m\tilde{t}\big)$
be the relative formal connection
defined by $\nabla_{\nu}(f(w))={\rm d}f(w)+f(w)\nu$
for $f(w)\in A[[w]]$.
If we take the inverse limit of the isomorphisms
$ (E,\nabla)\otimes A[z]/(z^{q})\stackrel{\sim}\longrightarrow
 (A[[w]]/(w^{qr}),\nabla_{\nu}|_{q\tilde{x}})$
constructed in Proposition \ref{proposition: normalization of formal type over finite scheme},
we get the following corollary.

\begin{Corollary} \label{corollary: formal structure}
Under the same assumption as
Proposition~{\rm \ref{proposition: normalization of formal type over finite scheme}},
there is an isomorphism
\[
 (E,\nabla)\otimes A[[z]]
 \cong (A[[w]],\nabla_{\nu}).
\]
\end{Corollary}

If a connection $(E,\nabla)$ has a formal isomorphism
$(E,\nabla)\otimes\widehat{\mathcal O}_{C,x}\cong
(\mathbb{C}[[w]],\nabla_{\nu})$ at $x$,
then it induces a generic $\nu$-ramified structure
as in Example~\ref{example: ramified structure}.
Conversely, the above corollary enables us to
recover a formal isomorphism
from a $\nu$-ramified structure in
Definition \ref{definition of ramified structure}
or a~factorized $\nu$-ramified structure in
Definition \ref{def-fac-connection}.
So we have the following corollary.

\begin{Corollary}\label{corollary: 3 conditions are equivalent}
Let $(E,\nabla)$ be a pair of a vector bundle $E$ of rank $r$ on a curve $C$
and a~connection $\nabla\colon E \longrightarrow E\otimes\Omega^1_C(D)$
with poles along the divisor $D$ whose multiplicity at $x$ is~$m$.
Take a~generator $z$ of the maximal ideal of ${\mathcal O}_{X,x}$
and a~variable $w$ with $w^r=z$.
Consider a rational one form
$\nu(w)=\nu_0(z)+\nu_1(z)w+\cdots+\nu_{r-1}(z)w^{r-1}$
such that
$\nu_0(z)\in\sum_{i=0}^{m-1}\mathbb{C}z^{i-m}{\rm d}z$,
$\nu_k(z)\in \sum_{i=0}^{m-2}\mathbb{C}z^{i-m}{\rm d}z$
for $1\leq k\leq r-1$
and that the leading term of $\nu_1(z)$ does not vanish.
Then the following conditions are equivalent.
\begin{itemize}\itemsep=0pt
\item[$(1)$]
$(E,\nabla)$ is generic $\nu$-ramified at $x$, that is,
$(\widehat{E},\widehat{\nabla})\cong (\mathbb{C}[[w]],\nabla_{\nu})$.
\item[$(2)$]
There is a generic $\nu$-ramified structure on $(E,\nabla)$ at $x$
in the sense of Definition~{\rm \ref{definition of ramified structure}}.
\item[$(3)$]
There is a factorized $\nu$-ramified structure on $(E,\nabla)$ at $x$
in the sense of Definition~{\rm \ref{def-fac-connection}}.
\end{itemize}
\end{Corollary}

\section{Construction of the moduli space of connections}\label{section: construction of the moduli space}

The moduli space of ramified connections is constructed in~\cite{Inaba-2}.
Since some notations in this paper are different from those in~\cite{Inaba-2},
we recall the construction of the moduli space in our setting.

Let $n_{\mathrm{log}}$, $n_{\mathrm{un}}$, $n_{\mathrm{ram}}$ be non-negative integers
and put $n=n_{\mathrm{log}}+n_{\mathrm{un}}+n_{\mathrm{ram}}$.
Consider the moduli stack ${\mathcal M}_{g,n}$
of $n$-pointed curves
$\big( C,x^{(\mathrm{log})}_1,\dots,x^{(\mathrm{log})}_{n_{\mathrm{log}}},
x^{(\mathrm{un})}_1,\dots,x^{(\mathrm{un})}_{n_{\mathrm{un}}},
x^{(\mathrm{ram})}_1,\dots,x^{(\mathrm{ram})}_{n_{\mathrm{ram}}}\big)$
of genus $g$ over $\Spec \mathbb{C}$.
We can take a smooth algebraic scheme ${\mathcal H}$ over $\Spec\mathbb{C}$
with a smooth surjective morphism
$H\longrightarrow{\mathcal M}_{g,n}$.
Indeed, we can take a subscheme $H'$ of $\Hilb_{\mathbb{P}^L}$
parameterizing the $l$-th canonical embeddings
$C\hookrightarrow \mathbb{P}(H^0(\omega_C^l))$
of smooth projective curves $C$ of genus $g$ for a~fixed large $l$
if $g\geq 2$.
If $g=1$, we take $H'$ as the open subset of
$\mathbb{P}_*(H^0({\mathcal O}_{\mathbb{P}^2}(3)))$
parameterizing the smooth cubic curves in~$\mathbb{P}^2$.
If $g=0$, we take~$H'$ as a point.
In any case, there is a universal family ${\mathcal Z}\subset\mathbb{P}^L\times H'$
of curves over $H'$.
Then the open subscheme ${\mathcal H}$ of the fiber product of~$n$ copies of~${\mathcal Z}$
over $H'$ parameterizing the distinct $n$ points on the curves satisfies our request.
We can take a universal family
$\big(
{\mathcal C}\times{\mathcal H},
\big( \tilde{x}^{\mathrm{log}}_i \big)_{1\leq i\leq n_{\mathrm{log}}} ,
( \tilde{x}^{\mathrm{un}}_i )_{1\leq i\leq n_{\mathrm{un}}} ,
( \tilde{x}^{\mathrm{ram}}_i )_{1\leq i\leq n_{\mathrm{ram}}} \big)$
consisting of flat family of curves of genus $g$ over ${\mathcal H}$
and sections
$\tilde{x}^{\mathrm{log}}_i$ ($1\leq i\leq n_{\mathrm{log}}$),
$\tilde{x}^{\mathrm{un}}_i$ ($1\leq i\leq n_{\mathrm{un}}$),
$\tilde{x}^{\mathrm{ram}}_i$ ($1\leq i\leq n_{\mathrm{ram}}$)
of ${\mathcal C}$ over ${\mathcal H}$.
We denote the ideal sheaf of $\tilde{x}^{\mathrm{un}}_i$
(resp. $\tilde{x}^{\mathrm{ram}}_j$)
by $I_{\tilde{x}^{\mathrm{un}}_i}$ (resp. $I_{\tilde{x}^{\mathrm{ram}}_j}$).

Assume that integers $m^{\mathrm{un}}_i\geq 2$ are given for $1\leq i\leq n_{\mathrm{un}}$
and integers $m^{\mathrm{ram}}_i\geq 2$ are given for $1\leq i\leq n_{\mathrm{ram}}$.
We put
\begin{align*}
 &{\mathcal D}_{\mathrm{log}}
 :=
 \sum_{i=1}^{n_{\mathrm{log}}} \tilde{x}^{\mathrm{log}}_i,
 \qquad
 {\mathcal D}_{\mathrm{un}}
 :=
 \sum_{i=1}^{n_{\mathrm{un}}} m^{\mathrm{un}}_i \tilde{x}^{\mathrm{un}}_i,
 \qquad
 {\mathcal D}_{\mathrm{ram}}
 :=
 \sum_{i=1}^{n_{\mathrm{ram}}} m^{\mathrm{ram}}_i \tilde{x}^{\mathrm{ram}}_i,
 \\
 &{\mathcal D}
 :=
 {\mathcal D}_{\mathrm{log}}+{\mathcal D}_{\mathrm{un}}+{\mathcal D}_{\mathrm{ram}}.
\end{align*}
Let ${\mathcal X}$ be the maximal open subset of
\[
 \Spec \mathrm{Sym}_{{\mathcal O}_{\mathcal H}}\bigg(
 {\mathcal H}{\rm om}_{{\mathcal O}_{\mathcal H}}
 \bigg( \bigoplus_{i=1}^{n_{\mathrm{un}}}
 I_{\tilde{x}^{\mathrm{un}}_i}/(I_{\tilde{x}^{\mathrm{un}}_i})^{m^{\mathrm{un}}_i+1}
 \oplus\bigoplus_{j=1}^{n_{\mathrm{ram}}}
 I_{\tilde{x}^{\mathrm{ram}}_j}/(I_{\tilde{x}^{\mathrm{ram}}_j})^{m^{\mathrm{ram}}_j+1},
 \, {\mathcal O}_{\mathcal H} \bigg)\bigg)
\]
such that the restriction $\bar{z}$ of the universal section to ${\mathcal X}$
gives a generator of
$\big(I_{\tilde{x}^{\mathrm{un}}_i}/I_{\tilde{x}^{\mathrm{un}}_i}^{m^{\mathrm{un}}_i+1}\big)
\otimes_{{\mathcal O}_{\mathcal H}} {\mathcal O}_{\mathcal X}$
at each
$\tilde{x}^{\mathrm{un}}_i$
(resp.\ a generator of $\big(I_{\tilde{x}^{\mathrm{ram}}_j}/I_{\tilde{x}^{\mathrm{ram}}_j}^{m^{\mathrm{ram}}_j+1}\big)
\otimes_{{\mathcal O}_{\mathcal H}} {\mathcal O}_{\mathcal X}$
at each $\tilde{x}^{\mathrm{ram}}_j$).

Fix complex numbers
\begin{gather*}
 \lambda
 =
 \big( \lambda^{(i)}_k \big)
 _{1\leq i\leq n_{\mathrm{log}} , \, 0\leq k\leq r-1} \
 \in\mathbb{C}^{rn_{\mathrm{log}}},
 \\
 c^{\mathrm{un}}
 =
 \big( c^{\mathrm{un}}_{i,k} \big)
 _{1\leq i\leq n_{\mathrm{un}} , \, 0\leq k\leq r-1} \,
 \in\mathbb{C}^{r n_{\mathrm{un}}},
 \\
 c^{\mathrm{ram}}
 =
\big( c^{\mathrm{ram}}_i \big)
 _{1\leq i\leq n_{\mathrm{ram}}}\in\mathbb{C}^{n_{\mathrm{ram}}},
\end{gather*}
which satisfy the equality
\begin{equation*}
 d+\sum_{i=1}^{n_{\mathrm{log}}}\sum_{k=0}^{r-1}\lambda^{(i)}_k
 +\sum_{i=1}^{n_{\mathrm{un}}}\sum_{k=0}^{r-1} c^{\mathrm{un}}_{i,k}
 +\sum_{i=1}^{n_{\mathrm{ram}}} \left(rc^{\mathrm{ram}}_i+\frac{r-1}{2} \right)
 =0
\end{equation*}
for an integer $d$.
We set
\[
 \mathbf{V}
 =\Spec \left( \mathrm{Sym}_{{\mathcal O}_{\mathcal X}} \bigg(\bigoplus_{i=1}^{n_{\mathrm{un}}}
 {\mathcal O}_{\mathcal X}^{\oplus (m^{\mathrm{un}}_i-1)r}
 \oplus
 \bigoplus_{j=1}^{n_{\mathrm{ram}}} {\mathcal O}_{\mathcal X}
 ^{\oplus (m^{\mathrm{ram}}_j-1)r} \bigg) \right)
\]
and
take universal sections
\begin{gather*}
 \big( \tilde{a}^{\mathrm{un}}_{i,k,j}\big)
 _{1\leq i\leq n_{\mathrm{un}}, 0\leq k \leq r-1, 0\leq j\leq m^{\mathrm{un}}_i-2}
 \in\bigoplus_{i=1}^{n_{\mathrm{un}}}\bigoplus_{k=0}^{r-1}\bigoplus_{j=0}^{m^{\mathrm{un}}_i-2}
 {\mathcal O}_{\mathbf{V}},
 \\
\big(\tilde{a}^{\mathrm{ram}}_{i,k,j}\big)
 _{1\leq i\leq n_{\mathrm{ram}},0\leq k\leq r-1,0\leq j\leq m^{\mathrm{ram}}_i-2}
 \in\bigoplus_{i=1}^{n_{\mathrm{ram}}}\bigoplus_{k=0}^{r-1}\bigoplus_{j=0}^{m^{\mathrm{ram}}_i-2}
 {\mathcal O}_{\mathbf{V}}.
\end{gather*}
Let ${\mathcal T}$ be the Zariski open subset of
$\mathbf{V}$ defined by
\[
 {\mathcal T}=
 \left\{ t\in \mathbf{V} \, \middle| \,
 \begin{array}{@{}l@{}}
 \text{for each $1\leq i\leq n_{\mathrm{un}}$,
 $\tilde{a}^{\mathrm{un}}_{i,k,0}(t)\neq \tilde{a}^{\mathrm{un}}_{i,k',0}(t)$ for $k\neq k'$,}
 \\
 \text{and $\tilde{a}^{\mathrm{ram}}_{i,1,0}(t)\neq 0$ for any $1\leq i\leq n_{\mathrm{ram}}$}
 \end{array}
 \right\}.
\]
We take a lift $z$ of $\overline{z}$ as a local algebraic function in a neighborhood of
${\mathcal D}$ and rephrase the above universal sections by setting
\begin{gather*}
 \tilde{\mu}_k(\bar{z})
 =
 \sum_{i=1}^{n_{\mathrm{un}}} \big(
 \tilde{a}^{\mathrm{un}}_{i,k,0}
 +\cdots
 +\tilde{a}^{\mathrm{un}}_{i,k,m^{\mathrm{un}}_i-2} \, \bar{z}^{m^{\mathrm{un}}_i-2}
 +c^{\mathrm{un}}_{i,k} \, \bar{z}^{m^{\mathrm{un}}_i-1}
\big) \frac{{\rm d}z} {z^{m^{\mathrm{un}}_i}} \bigg|_{m^{\mathrm{un}}_i(\tilde{x}_i)_{\mathcal T}}
 \ (0\leq k\leq r-1),
 \\
 \tilde{\nu}_0(\bar{z})
 =
 \sum_{i=1}^{n_{\mathrm{ram}}} \big(
 \tilde{a}^{\mathrm{ram}}_{i,0,0}+\tilde{a}^{\mathrm{ram}}_{i,0,1}\,\bar{z}+\cdots
 +\tilde{a}^{\mathrm{ram}}_{i,0,m^{\mathrm{ram}}_i-2} \, \bar{z}^{m^{\mathrm{ram}}_i-2}
 +c^{\mathrm{ram}}_i \, \bar{z}^{m^{\mathrm{ram}}_i-1} \big)
 \frac{{\rm d}z} {z^{m^{\mathrm{ram}}_i}} \bigg|_{m^{\mathrm{ram}}_i(\tilde{x}_i)_{\mathcal T}},
 \\
 \tilde{\nu}_k(\bar{z})
 =
 \sum_{i=1}^{n_{\mathrm{ram}}} \big(
 \tilde{a}^{\mathrm{ram}}_{i,k,0}+\tilde{a}^{\mathrm{ram}}_{i,k,1}\,\bar{z}+\cdots
 +\tilde{a}^{\mathrm{ram}}_{i,k,m^{\mathrm{ram}}_i-2} \, \bar{z}^{m^{\mathrm{ram}}_i-2}
\big) \frac{{\rm d}z} {z^{m^{\mathrm{ram}}_i}}
 \bigg|_{m^{\mathrm{ram}}_i(\tilde{x}_i)_{\mathcal T}}
 \quad (1\leq k\leq r-1),
 \\
 \tilde{\nu}(w)
=
 \tilde{\nu}_0(\bar{z})+\tilde{\nu}_1(\bar{z})w+\cdots+\tilde{\nu}_{r-1}(\bar{z})w^{r-1}
\end{gather*}
and we write
$\tilde{\mu}:=(\tilde{\mu}_k)_{0\leq k\leq r-1}$
and $\tilde{\nu}:=\tilde{\nu}(w)$.
Note that the restriction of the differential forms
$\frac{{\rm d}z} {z^{m^{\mathrm{un}}_i}} \big|_{m^{\mathrm{un}}_i(\tilde{x}_i)_{\mathcal T}}$,
$\frac{{\rm d}z} {z^{m^{\mathrm{ram}}_i}} \big|_{m^{\mathrm{ram}}_i(\tilde{x}_i)_{\mathcal T}}$
are independent of the choice of the representative $z$ of
$\overline{z}$ and are uniquely determined by $\overline{z}$.

We fix a parabolic weight
$\balpha=
\big( \big(\alpha^{\mathrm{log}}_k\big)^{1\leq i\leq n_{\mathrm{log}}}_{1\leq k\leq r} ,
\big(\alpha^{\mathrm{un}}_{i,k}\big)^{1\leq i\leq n_{\mathrm{un}}}_{1\leq k\leq r} ,
\big(\alpha^{\mathrm{ram}}_{i,k}\big)^{1\leq i\leq n_{\mathrm{ram}}}_{1\leq k\leq r} \big)$
as in Definition~\ref{definition: stability}.

\begin{Theorem}[{\cite[Theorem 2.1]{Inaba-2}}]\label{theorem: existence of the moduli space}
There exists a relative coarse moduli space
$M^{\balpha}_{{\mathcal C},{\mathcal D}}(\lambda,\tilde{\mu},\tilde{\nu})
\longrightarrow {\mathcal T}$
of $\balpha$-stable connections with $(\lambda,\tilde{\mu},\tilde{\nu})$-structure
on $({\mathcal C},{\mathcal D})$.
Furthermore,
$M^{\balpha}_{{\mathcal C},{\mathcal D}}(\lambda,\tilde{\mu},\tilde{\nu})
\longrightarrow {\mathcal T}$
is a quasi-projective morphism.
\end{Theorem}

\begin{proof}
We use the same argument as in the proof of
\cite[Theorem 2.1]{Inaba-Saito} and
\cite[Theorem 2.1]{Inaba-2}.
Consider the moduli functor
${\mathcal M}$ of
tuples $(E,\nabla,l,\ell,(V_k))$ consisting of rank $r$ vector bundles~$E$,
connections $\nabla$ admitting poles along~${\mathcal D}$
and parabolic structure $\ell,l,(V_k)$ along ${\mathcal D}$
satisfying $\balpha$-stability.
Then we can embed ${\mathcal M}$ to a locally closed subfunctor of
the moduli functor of stable parabolic triples $(E_1,E_2,\phi,\nabla,F_*(E_1))$,
whose existence is proved in
\cite[Theorem~5.1]{IIS-1}.
So we can get a moduli space $M$ which represents the \'etale sheafification of
${\mathcal M}$ and $M$ is quasi-projective over~$\mathcal T$.
We can construct a quasi-projective scheme $M_{\lambda,\tilde{\mu}}$ over
$M$ which parameterizes $(\lambda,\tilde{\mu})$-structure on $(E,\nabla)$
compatible with~$l$, $\ell$
as in the proof of
\cite[Theorem~2.1]{Inaba-1}
and \cite[Theorem~2.1]{Inaba-Saito}.

We only have to construct a parameter space of $\tilde{\nu}$-ramified structure
over $M_{\lambda,\tilde{\mu}}$
such that the filtration in Definition~\ref{definition of ramified structure}\,(i)
coincides with the given filtration~$(V_k)$.
There is an \'etale surjective morphism
$M'\longrightarrow M_{\lambda,\tilde{\mu}}$
with a universal family
$\big(\tilde{E},\tilde{\nabla},\tilde{l},\tilde{\ell},\big(\tilde{V}_k\big)\big)$ on
${\mathcal C}_{M'}$ over $M'$.
We set
\[
 A_w:=\prod_{i=1}^{n_{\mathrm{ram}}}{\mathcal O}_{M'}[w]/\big(w^{m^{\mathrm{ram}}_ir-r+1}\big).
\]
Since $A_w$ is a finite module over $M'$,
we can construct a locally closed subscheme $Q$ of a product of Quot-schemes over $M'$
such that the set of $S$-valued points of $Q$ is
\[
 Q(S)=
 \left\{ \big( \tilde{V}_k\otimes (A_w)_S \xrightarrow{\pi_k} L_k \big)_{0\leq k\leq r-1}
 \middle|
 \begin{array}{@{}l@{}}
 \text{$L_k$ is a quotient $A_w$ module of
 $\tilde{V}_k\otimes (A_w)_S$ and}
 \\
 \text{$L_k$ is a locally free $(A_w)_S$-module of rank one}
 \end{array}
 \right\}.
\]
Let $\pi_k\colon\tilde{V}_k\otimes (A_w)_Q\longrightarrow \tilde{L}_k$
be the universal quotient sheaf.
There exists a maximal locally closed subscheme
$\Sigma \subset Q$ such that the restrictions
$(\pi_k)_{\Sigma}\big|_{(\tilde{V}_k)_{\Sigma}}
\colon \big(\tilde{V}_k\big)_{\Sigma} \longrightarrow \big(\tilde{L}_k\big)_{\Sigma}$ are surjective,
the diagrams
\[
 \begin{CD}
 \big(\tilde{V}_k\big)_{\Sigma} @>\tilde{\nabla}|_{({\mathcal D}_{\mathrm{ram}})_{\Sigma}}>>
 (\tilde{V}_k)_{\Sigma}\otimes
 \Omega^1_{{\mathcal C}_{\Sigma}/{\Sigma}}(({\mathcal D}_{\mathrm{ram}})_{\Sigma}) \\
 @V \pi_k VV @VV \pi_k\otimes\mathrm{id} V \\
 \big(\tilde{L}_k\big)_{\Sigma} @>\nu(w)+\frac{kdz}{rz}>>
 \big(\tilde{L}_k\big)_{\Sigma}\otimes \Omega^1_{{\mathcal C}_{\Sigma}/{\Sigma}}
 (({\mathcal D}_{\mathrm{ram}})_{\Sigma})
 \end{CD}
\]
are commutative for $0\leq k\leq r-1$,
each composition $\big(\tilde{V}_k\big)_{\Sigma}\longrightarrow \big(\tilde{V}_{k-1}\big)_{\Sigma}
\xrightarrow{\pi_{k-1}} \big(\tilde{L}_{k-1}\big)_{\Sigma}$
factors through an $(A_w)_{\Sigma}$-homomorphism
$\tilde{\phi}_k\colon \big(\tilde{L}_k\big)_{\Sigma}\longrightarrow \big(\tilde{L}_{k-1}\big)_{\Sigma}$
whose image is $w\big(\tilde{L}_{k-1}\big)_{\Sigma}$
for $1\leq k\leq r-1$
and the composition
$(z)\otimes (V_0)_{\Sigma}\longrightarrow (V_{r-1})_{\Sigma}
\xrightarrow{\pi_{r-1}} \big(\tilde{L}_{r-1}\big)_{\Sigma}$
factors thorough an $(A_w)_{\Sigma}$-homomorphism
$\tilde{\phi}_r\colon (z)\otimes \big(\tilde{L}_0\big)_{\Sigma}\longrightarrow (\tilde{L}_{r-1})_{\Sigma}$
whose image is $w\big(\tilde{L}_{r-1}\big)_{\Sigma}$.
We denote the free $(A_w)_{\Sigma}$-module
$\bigoplus_{i=1}^{n_{\mathrm{ram}}}\big(w^k\big)/\big(w^{k+m_i^{\mathrm{ram}}r-r+1}\big)$ simply by $\big(w^k\big)$.
Consider the affine space bundle
\[
 \mathbf{V}_k=
 \Spec \mathrm{Sym}_{{\mathcal O}_{\Sigma}}
\big({\mathcal H}{\rm om}_{{\mathcal O}_{\Sigma}}
\big({\mathcal H}{\rm om}_{(A_w)_{\Sigma}} \big(
 \big(\tilde{L}_k\big)_{\Sigma} , ((w)\otimes A_w \otimes \tilde{L}_{k-1})_{\Sigma} \big) ,
 {\mathcal O}_{\Sigma}
\big) \big)
 \longrightarrow {\Sigma}
\]
for $k=1,\dots,r-1$ and take a universal section
\[
 \psi_k\colon \
 \big(\tilde{L}_k\big)_{\mathbf{V}_k}\longrightarrow \big((w)\otimes A_w\otimes \tilde{L}_{k-1}\big)_{\mathbf{V}_k}.
\]
There is a morphism
\[
 c_k\colon \
 \mathbf{V}_k \longrightarrow
 \Spec \mathrm{Sym}_{{\mathcal O}_{\Sigma}}
\big({\mathcal H}{\rm om}_{{\mathcal O}_{\Sigma}}
\big({\mathcal H}{\rm om}_{(A_w)_{\Sigma}} \big( \big(\tilde{L}_k\big)_{\Sigma}, \big(w \tilde{L}_{k-1}\big)_{\Sigma}\big) ,
 {\mathcal O}_{\Sigma}
\big) \big)
\]
over $\Sigma$ defined by the composition
\[
 \big(\tilde{L}_k\big)_{\mathbf{V}_k} \xrightarrow{\psi_k}
 \big((w)\otimes A_w\otimes \tilde{L}_{k-1}\big)_{\mathbf{V}_k}
 \longrightarrow
 \big(w\tilde{L}_{k-1}\big)_{\mathbf{V}_k}.
\]
Over the fiber
$c_k^{-1}\big(\tilde{\phi}_k\big)\subset \mathbf{V}_k$,
the composition
\[
 \big(\tilde{L}_k\big)_{c_k^{-1}(\tilde{\phi}_k)} \xrightarrow{\psi_k}
 \big((w)\otimes A_w\otimes \tilde{L}_{k-1}\big)_{c_k^{-1}(\tilde{\phi}_k)}
 \longrightarrow
 \big(w\tilde{L}_{k-1}\big)_{c_k^{-1}(\tilde{\phi}_k)}
\]
coincides with
$\big(\tilde{\phi}_k\big)_{c_k^{-1}(\tilde{\phi}_k)}\colon
\big(\tilde{L}_k\big)_{c_k^{-1}(\tilde{\phi}_k)}
\longrightarrow
\big(w\tilde{L}_{k-1}\big)_{c_k^{-1}(\tilde{\phi}_k)}$,
which is surjective.
So, we can see by the Nakayama's lemma, that
$(\psi_k)_{c_k^{-1}(\tilde{\phi}_k)} \colon
\big(\tilde{L}_k\big)_{c_k^{-1}(\tilde{\phi}_k)} \longrightarrow
(w) \otimes \big(\tilde{L}_{k-1}\big)_{c_k^{-1}(\tilde{\phi}_k)}$
is surjective and then $(\psi_k)_{c_k^{-1}(\tilde{\phi}_k)} $ is isomorphic
because it is a surjection between locally free $(A_w)_{c_k^{-1}(\tilde{\phi}_k)}$-modules
of rank one.
Consider the group scheme $G$ over $\Sigma$
whose set of $S$-valued points is
\[
 G(S)=
 \prod_{i=1}^{n_{\mathrm{ram}}}\big(1+H^0({\mathcal O}_S)z^{m^{\mathrm{ram}}_i-1}\big),
\]
where each component
$\big(1+H^0({\mathcal O}_S)z^{m^{\mathrm{ram}}_i-1}\big)$
is regarded as a subgroup of the group of invertible elements of
$H^0((A_w)_S)$.
Then there is a canonical action of $G$ on the product
$Y:=\prod_{k=1}^{r-1}c_k^{-1}\big(\tilde{\phi}_k\big)$
and
\[
 Y=\prod_{k=1}^{r-1}c_k^{-1}\big(\tilde{\phi}_k\big) \longrightarrow \Sigma
\]
is a $G$-torsor.
Consider the composition
\begin{gather*}
\begin{split}
 & \psi_1\circ\cdots\circ\psi_{r-1}\circ\phi_r
 \colon \
 (z)\otimes \big(\tilde{L}_0\big)_Y\xrightarrow{\tilde{\phi}_r}
 \big(\tilde{L}_{r-1}\big)_Y\xrightarrow[\sim]{\psi_{r-1}}
 (w)\otimes \big(\tilde{L}_{r-2}\big)_Y
\\
& \hphantom{\psi_1\circ\cdots\circ\psi_{r-1}\circ\phi_r
 \colon \
 (z)\otimes \big(\tilde{L}_0\big)_Y\xrightarrow{\tilde{\phi}_r}
 \big(\tilde{L}_{r-1}\big)_Y}{}
 \xrightarrow[\sim]{\psi_{r-2}}\cdots \xrightarrow[\sim]{\psi_1}
 \big(w^{r-1}\big)\otimes \big(\tilde{L}_0\big)_Y.
 \end{split}
\end{gather*}
Then there exists a maximal closed subscheme
$Z\subset Y$ such that
the composition
$\big(\psi_1\circ\cdots\circ\psi_{r-1}\circ\tilde{\phi}_r\big)_Y$
coincides with the canonical homomorphism
$(z)\otimes\big(\tilde{L}_0\big)_Y\longrightarrow \big(w^{r-1}\big)\otimes \big(\tilde{L}_0\big)_Y$
induced by the inclusion $(z)\hookrightarrow \big(w^{r-1}\big)$.
By the construction, $Z$ is invariant under the action of~$G$.
So~$Z$ descends to a closed subscheme
$\Sigma_{\tilde{\nu}}\subset\Sigma$.
We can see that the quasi-projective scheme $\Sigma_{\tilde{\nu}}$
over~$M'$ descends to a quasi-projective scheme
$M^{\balpha}_{{\mathcal C},{\mathcal D}}(\lambda,\tilde{\mu},\tilde{\nu})$
over $M_{\lambda,\tilde{\mu}}$
which is the desired moduli space.
\end{proof}

\section[Tangent space of the moduli space using factorized ramified structure]{Tangent space of the moduli space\\ using factorized ramified structure}
\label{section: tangent space}

The aim of introducing the factorized ramified structure is to
construct a duality on the tangent space of the moduli space,
which was not achieved in~\cite{Inaba-2}.
We will first describe the tangent space of the moduli space
by the infinitesimal deformation of factorized ramified structure.

Let the notation be as in Section~\ref{section: construction of the moduli space}.
Take a point $t\in{\mathcal T}$.
We will describe the tangent space of the fiber
$M^{\balpha}_{{\mathcal C},{\mathcal D}}(\lambda,\tilde{\mu},\tilde{\nu})_t$
of the moduli space over $t$.
We write
$C:={\mathcal C}_t$, $D:={\mathcal D}_t$,
$D_{\mathrm{log}}=({\mathcal D}_{\mathrm{log}})_t$,
$D_{\mathrm{un}}=({\mathcal D}_{\mathrm{un}})_t$,
$D_{\mathrm{ram}}=({\mathcal D}_{\mathrm{ram}})_t$
and
$(\mu,\nu):=(\tilde{\mu},\tilde{\nu})_t$.
We put
$m_x:=m^{\mathrm{un}}_i$ for $x=\tilde{x}^{\mathrm{un}}_i|_t$
and
$m_x:=m^{\mathrm{ram}}_i$ for $x=\tilde{x}^{\mathrm{ram}}_i|_t$.

Let
$(E,\nabla,l,\ell,{\mathcal V})$
be a connection on $(C,D)$ with
$(\lambda,\mu,\nu)$-structure.
If we put
\[
 l_k:=
 \bigoplus_{x\in D_{\mathrm{log}}} l^x_k,
 \qquad
 \ell_k
 :=
 \bigoplus_{x\in D_{\mathrm{un}}} \ell^x_k,
\]
then we get filtrations
$E|_{D_{\mathrm{log}}}=l_0\supset l_1\supset\cdots\supset l_{r-1}\supset l_r=0$,
$E|_{D_{\mathrm{un}}}=\ell_0\supset\ell_1\supset\cdots\supset\ell_{r-1}\supset l_r=0$
such that
$l_k/l_{k+1}\cong{\mathcal O}_{D_{\mathrm{log}}}$
and $\ell_k/\ell_{k+1}\cong{\mathcal O}_{D_{\mathrm{un}}}$
for $0\leq k\leq r-1$.
If we put
\[
 V_k:=
 \bigoplus_{x\in D_{\mathrm{ram}}} V^x_k,
 \qquad
 \overline{V}_k:=
 \bigoplus_{x\in D_{\mathrm{ram}}} \overline{V}^x_k,
 \qquad
 \overline{W}_k:=
 \bigoplus_{x\in D_{\mathrm{ram}}} \overline{W}^x_k,
\]
then we get a filtration
$E|_{D_{\mathrm{ram}}}=V_0\supset V_1\supset\cdots\supset V_{r-1}\supset V_r=zV_0$
with surjections $V_k\longrightarrow \overline{V}_k$
and isomorphisms
$\overline{W}_k\cong\Hom_{{\mathcal O}_{D_{\mathrm{ram}}}}
(\overline{V}_{r-k-1}, {\mathcal O}_{D_{\mathrm{ram}}} )$
for $0\leq k\leq r-1$.

Define a complex ${\mathcal G}^{\bullet}$ of sheaves on $C$ by
setting
\begin{gather} \label{equation: definition of G^0,G^1}
 {\mathcal G}^0
=
 \left\{ u\in {\mathcal E}{\rm nd}(E) \, \middle|
 \begin{array}{@{}l@{}}
 \text{$u|_{D_{\mathrm{log}}}(l_k)\subset l_k$,
 $u|_{D_{\mathrm{un}}}(\ell_k)\subset \ell_k$ and $u|_{D_{\mathrm{ram}}}(V_k)\subset V_k$} \\
 \text{for $0\leq k \leq r-1$}
 \end{array}
 \right\}, \\
 {\mathcal G}^1
 =
 \left\{v\in{\mathcal E}{\rm nd}(E)\otimes\Omega^1_C(D) \,\middle|
 \begin{array}{@{}l@{}}
 \text{$v|_{D_{\mathrm{log}}}(l_k)\subset l_{k+1}\otimes\Omega^1_C(D)$,
 $v|_{D_{\mathrm{un}}}(\ell_k)\subset \ell_{k+1}\otimes\Omega^1_C(D)$} \\
 \text{and $v|_{D_{\mathrm{ram}}}(V_k)\subset V_k\otimes\Omega^1_C(D)$ for $0\leq k\leq r-1$}
 \end{array}
 \right\},\nonumber
\end{gather}
and by defining the homomorphism
\begin{equation} \label{equation: homomorphism of complex G}
 d^0_{\mathcal G^{\bullet}}\colon \
 {\mathcal G}^0 \ni u \mapsto \nabla \circ u-(u\otimes 1)\circ \nabla
 \in {\mathcal G}^1.
\end{equation}
The meaning of the hypercohomology
$\mathbb{H}^1({\mathcal G}^{\bullet})$ is the tangent space of
the moduli space of connections
$(E,\nabla)$ on $C$ equipped with logarithmic $\lambda$-parabolic structure along $D_{\rm log}$,
generic unramified $\mu$-parabolic structure along $D_{\rm un}$
and a filtration $E|_{D_{\rm ram}}=V_0\supset V_1\supset\cdots\supset V_{r-1}\supset V_r=zV_0$
preserved by $\nabla$.
For the description of the tangent space of the moduli space
$M^{\balpha}_{{\mathcal C},{\mathcal D}}(\lambda,\tilde{\mu},\tilde{\nu})_t$,
we will construct the data of deformation of factorized ramified structure
in addition to $\mathbb{H}^1({\mathcal G}^{\bullet})$.

For
$(v_k) \in \bigoplus_{k=0}^{r-1}
\Hom\big(\overline{V}_k,\overline{V}_k\otimes\Omega^1_C(D)\big)$,
consider the diagrams
\begin{gather}
 \begin{CD}
 z{\mathcal O}_{D_{\mathrm{ram}}}\otimes \overline{V}_0 @> v_k >>
 z{\mathcal O}_{D_{\mathrm{ram}}}\otimes\overline{V}_0\otimes\Omega^1_C(D) \\
 @VVV @VVV \\
 \overline{V}_{r-1} @> v_{k-1} >> \overline{V}_{r-1}\otimes\Omega^1_C(D),
 \end{CD}
 \nonumber\vspace{1mm}\\
 \begin{CD}
 \overline{V}_k @> v_k >> \overline{V}_k\otimes\Omega^1_C(D) \\
 @VVV @VVV \\
 \overline{V}_{k-1} @> v_{k-1} >> \overline{V}_{k-1}\otimes\Omega^1_C(D)
 \end{CD}
 \qquad (1\leq k\leq r-1).\label{equation: commutative diagram defining G^1}
\end{gather}
If we put
\begin{equation*}
 G^1
 =
 \left\{
 (v_k) \in \bigoplus_{k=0}^{r-1}
 \Hom(\overline{V}_k,\overline{V}_k\otimes\Omega^1_C(D))
 \, \middle|
 \begin{array}{@{}l@{}}
 \text{all the diagrams in (\ref{equation: commutative diagram defining G^1})} \\
 \text{are commutative}
 \end{array}\right\},
\end{equation*}
then there is a canonical homomorphism
\[
 \varpi_G\colon \
 {\mathcal G}^1\longrightarrow G^1
\]
defined by
$\varpi_G(v)=\big( \overline{v|_{D_{\mathrm{ram}}}} \big)_k$,
where
$\overline{v|_{D_{\mathrm{ram}}}} \colon \overline{V}_k
\longrightarrow \overline{V}_k\otimes\Omega^1_C(D)$
is the homomorphism induced by $v|_{D_{\mathrm{ram}}}$.
We can see the surjectivity of $\varpi_G$ by the following lemma,
which is often used later.

\begin{Lemma} \label{lemma: extension of endomorphisms}
For any tuple $(h_k)\in \prod_{k=0}^{r-1} \End_{{\mathcal O}_{D_{\mathrm{ram}}}}\big(\overline{V}_k\big)$
of endomorphisms satisfying the commutative diagrams
\[
 \begin{CD}
 z{\mathcal O}_{D_{\mathrm{ram}}}\otimes \overline{V}_0 @> \mathrm{id}\otimes h_0 >>
 z{\mathcal O}_{D_{\mathrm{ram}}}\otimes\overline{V}_0 \\
 @VVV @VVV \\
 \overline{V}_{r-1} @> h_{r-1} >> \overline{V}_{r-1},
 \end{CD}
 \qquad
 \begin{CD}
 \overline{V}_k @> h_k >> \overline{V}_k \\
 @VVV @VVV \\
 \overline{V}_{k-1} @> h_{k-1} >> \overline{V}_{k-1}
 \end{CD}
 \qquad (1\leq k\leq r-1),
\]
there exists an endomorphism
$h\in\End_{{\mathcal O}_{D_{\mathrm{ram}}}}(E|_{D_{\mathrm{ram}}})$ satisfying
$h(V_k)\subset V_k$ and the commutative diagrams
\[
 \begin{CD}
 V_k @> \pi_k >> \overline{V}_k \\
 @V h|_{V_k} VV @VV h_k V \\
 V_k @>\pi_k >> \overline{V}_k
 \end{CD}
\]
for $0\leq k\leq r-1$.
Moreover, $\Tr(h)\in{\mathcal O}_{D_{\mathrm{ram}}}$ is uniquely determined by
$(h_k)$ and independent of the choice of $h$.
\end{Lemma}

\begin{proof}
Let $e_0,\dots,e_{r-1}$ be the basis of $E|_{D_{\mathrm{ram}}}$ taken in the proof of
Proposition \ref{prop:correspondence-factorized}.
Then we can write
\[
 h_k(e_k)=a_{k,k}e_k+\overline{a}_{k+1,k}e_{k+1}+\cdots+\overline{a}_{r-1,k}e_{r-1}
 +za_{0,k}e_1+\cdots+za_{k-1,k}e_{k-1}
\]
for $a_{k,k}\in{\mathcal O}_{D_{\mathrm{ram}}}$ and
$\overline{a}_{l,k}\in {\mathcal O}_{D'_{\mathrm{ram}}}$ for $l\neq k$,
where we put $D'_{\mathrm{ram}}:=\sum_{x\in D_{\mathrm{ram}}} (m_x-1)x$
and $za_{l,k}$ is the image of $z\otimes\overline{a}_{l,k}$ under the isomorphism
$(z)\otimes {\mathcal O}_{D'_{\mathrm{ram}}}\xrightarrow{\sim} z{\mathcal O}_{D_{\mathrm{ram}}}$
for $l<k$
We can see that a lift $h\in\End_{{\mathcal O}_{D_{\mathrm{ram}}}}(E|_{D_{\mathrm{ram}}})$
of $(h_k)$ desired in the lemma is given by the matrix
\[
 \begin{pmatrix}
 a_{0,0} & za_{0,1} & \cdots & za_{0,r-1} \\
 a_{1,0} & a_{1,1} & \cdots & z_{1,r-1} \\
 \vdots & \vdots & \ddots & \vdots \\
 a_{r-1,0} & a_{r-1,1} & \cdots & a_{r-1,r-1}
 \end{pmatrix}
\]
with respect to the basis $e_0,\dots,e_{r-1}$,
where $a_{l,k}\in{\mathcal O}_{D_{\mathrm{ram}}}$
are lifts of $\overline{a}_{k,l}$ for $l> k$.
In particular, we obtain the existence of $h$.
The ambiguities of $h$ are the lower triangular entries
$a_{i,j}$ with $i>j$.
So its trace
$\Tr(h)=a_{0,0}+\cdots+a_{r-1,r-1}$ is independent of the choice of $h$.
\end{proof}

The trace pairing
$\Tr\colon \ker(\varpi_G) \otimes {\mathcal G}^0
\ni v\otimes u \mapsto \Tr(v\circ u)\in \Omega^1_C$
induces an isomorphism
\begin{equation*}
 \ker \varpi_G \stackrel{\sim}\longrightarrow \big({\mathcal G}^0\big)^{\vee}\otimes \Omega^1_C.
\end{equation*}

For $(\tau_k)\in \bigoplus_{k=0}^{r-1} \Hom\big(\overline{W}_k,\overline{V}_k\big)$,
consider the diagrams
\begin{equation} \label{equation: commutativity condition of symmetric tensor 1}
\begin{CD}
 \overline{W}_k @>>> \overline{W}_{k-1} \\
 @V \tau_k VV @V \tau_{k-1} VV \\
 \overline{V}_k @>>> \overline{V}_{k-1}
 \end{CD}
 \quad (1\leq k\leq r-1),
 \qquad
 \begin{CD}
 z {\mathcal O}_{D_{\mathrm{ram}}} \otimes \overline{W}_0 @>>> \overline{W}_{r-1} \\
 @V \mathrm{id}\otimes\tau_0 V V @V \tau_{r-1} VV \\
 z {\mathcal O}_{D_{\mathrm{ram}}} \otimes \overline{V}_0 @>>> \overline{V}_{r-1}
 \end{CD}
\end{equation}
and for
$(\xi_k)\in \bigoplus_{k=0}^{r-1} \Hom\big(\overline{V}_k,\overline{W}_k\big)$,
consider the diagrams
\begin{equation} \label{equation: commutativity condition of symmetric tensor 2}
 \begin{CD}
 \overline{V}_k @>>> \overline{V}_{k-1} \\
 @V \xi_k VV @V \xi_{k-1} VV \\
 \overline{W}_k @>>> \overline{W}_{k-1}
 \end{CD}
 \quad (1\leq k\leq r-1),
 \qquad
 \begin{CD}
 z{\mathcal O}_{D_{\mathrm{ram}}} \otimes \overline{V}_0 @>>> \overline{V}_{r-1} \\
 @V \mathrm{id}\otimes\xi_0 VV @V \xi_{r-1} VV \\
 z{\mathcal O}_{D_{\mathrm{ram}}} \otimes \overline{W}_0 @>>> \overline{W}_{r-1}.
 \end{CD}
\end{equation}
Then we put
\begin{gather*}
 \Sym^2\big(\overline{W}\big)
 =
 \left\{ (\tau_k)\in \bigoplus_{k=0}^{r-1} \Hom\big(\overline{W}_k,\overline{V}_k\big)
 \, \middle| \,
 \begin{array}{@{}l@{}}
 \text{the diagrams (\ref{equation: commutativity condition of symmetric tensor 1})
 are commutative} \\
 \text{and $^t\tau_{r-k-1}=\tau_k$ for $0\leq k\leq r-1$}
 \end{array}
 \right\},
 \\
 \Sym^2\big(\overline{V}\big)
 =
 \left\{
 (\xi_k)\in \bigoplus_{k=0}^{r-1} \Hom\big(\overline{V}_k,\overline{W}_k\big)
 \, \middle| \,
 \begin{array}{@{}l@{}}
 \text{the diagrams (\ref{equation: commutativity condition of symmetric tensor 2})
 are commutative} \\
 \text{and $^t\xi_{r-k-1}=\xi_k$ for $0\leq k\leq r-1$}
 \end{array}
 \right\}
\end{gather*}
and put
\begin{gather*}
 A^0 =
 \left\{ (a_k(w))\in \bigoplus_{x\in D_{\mathrm{ram}}}\prod_{k=0}^{r-1} \mathbb{C}[w]/\big(w^{m_xr-r+1}\big)
 \, \middle| \,
 \begin{array}{@{}l@{}}
 \text{$w \, (a_k(w)-a_{k+1}(w))=0$} \\
 \text{for $0\leq k\leq r-2$}
 \end{array}
 \right\},
 \\
 A^1 =
 \Hom_{{\mathcal O}_{D_{\mathrm{ram}}}} \big(A^0,{\mathcal O}_{D_{\mathrm{ram}}}\big).
\end{gather*}

We need the following lemma which is similar to Lemma~\ref{lemma: extension of endomorphisms}.

\begin{Lemma} \label{lemma: extension of symmetric pairings}
Assume that $(\tau_k)\in \Sym^2\big(\overline{W}\big)$
and $(\xi_k)\in \Sym^2\big(\overline{V}\big)$ are given.
Then there are homomorphisms
$\tau\colon E|_{D_{\mathrm{ram}}}^{\vee}\longrightarrow E|_{D_{\mathrm{ram}}}$,
$\xi\colon E|_{D_{\mathrm{ram}}} \longrightarrow E|_{D_{\mathrm{ram}}}^{\vee}$
satisfying $^t\tau=\tau$, $^t\xi=\xi$,
$\tau(W_k)\subset V_k$, $\xi(V_k)\subset W_k$ and the commutative diagrams
\[
 \begin{CD}
 W_k @> \tau|_{W_k} >> V_k \\
 @VVV @VVV \\
 \overline{W}_k @>\tau_k>> \overline{V}_k,
 \end{CD}
 \qquad
 \begin{CD}
 V_k @> \xi|_{V_k} >> W_k \\
 @VVV @VVV \\
 \overline{V}_k @>\xi_k>> \overline{W}_k
 \end{CD}
\]
for $k=0,1\dots,r-1$,
where
$W_k=\bigoplus_{x\in D_{\mathrm{ram}}} \ker(z^{m_x-1}(\,^tN)^{r-k})
\subset E|_{D_{\mathrm{ram}}}^{\vee}$.
\end{Lemma}

\begin{proof}
Choose the basis $e_0,\dots,e_{r-1}$ of $E|_{D_{\mathrm{ram}}}$ taken in the proof of
Proposition~\ref{prop:correspondence-factorized}
and its dual basis $e^*_0,\dots,e^*_{r-1}$.
Since $\tau_k\big(\overline{W}_k\big)\subset \overline{V}_k$, we can write
\begin{gather*}
 \tau_k(e^*_{r-k-1}) =
 zb_{0,r-k-1}e_0+\cdots +zb_{k-1,r-k-1}e_{k-1}+b_{k,r-k-1}e_k\\
 \hphantom{\tau_k(e^*_{r-k-1}) =}{}
 +
 \overline{b}_{k+1,r-k-1}e_{k+1}+\cdots+\overline{b}_{r-1,r-k-1}e_{r-1},
\end{gather*}
where $\overline{b}_{l,r-k-1}\in{\mathcal O}_{D'_{\mathrm{ram}}}$ for $l\geq k+1$
and $b_{l,r-k-1}\in {\mathcal O}_{D_{\mathrm{ram}}}$for $l\leq k$.
Take a lift $b_{l,r-k-1}\in{\mathcal O}_{D_{\mathrm{ram}}}$ of $\overline{b}_{l,r-k-1}$
for $l\geq k+1$.
Then we have
\begin{gather*}
 zb_{l,r-k-1}=\tau_k(e^*_{r-k-1})(e^*_l) =
 \tau_{r-k-1}(e^*_l)(e^*_{r-k-1})
 \\
 \hphantom{zb_{l,r-k-1}}{}
 =\tau_{r-l-1}(e^*_l)(e^*_{r-k-1})
 =zb_{r-k-1,l}
 \qquad (\text{for $l\leq k-1$}),
 \\
 b_{k,r-k-1}=\tau_k(e^*_{r-k-1})(e^*_k)
 =
 \tau_{r-k-1}(e^*_k)(e^*_{r-k-1})
 =b_{r-k-1,k},
 \\
 zb_{l,r-k-1}
 =\tau_k(e^*_{r-k-1})(ze^*_l)
=\tau_{r-k-1}(ze^*_l)(e^*_{r-k-1})
 \\
\hphantom{zb_{l,r-k-1}}{}
=\tau_{r-l-1}(ze^*_l)(e^*_{r-k-1})=zb_{r-k-1,l}
 \qquad (\text{for $l\geq k+1$}).
\end{gather*}
After replacing $b_{r-k-1,l}$ for $l\geq k+1$,
we may assume $b_{l,r-k-1}=b_{r-k-1,l}$
for $l\geq k+1$.
Let
$\tau\colon E|_{D_{\mathrm{ram}}}^{\vee} \longrightarrow E|_{D_{\mathrm{ram}}}$
be the homomorphism given by the matrix
\[
 \begin{pmatrix}
 zb_{00}(z) & & \cdots & & b_{0,r-1}(z) \\
 \vdots & & & \rotatebox{80}{$\ddots$} \ & \vdots \\
 zb_{0,k}(z) & \cdots & b_{k,r-k-1}(z) & \cdots & b_{k,r-1}(z) \\
 \vdots & \rotatebox{80}{$\ddots$} & & & \vdots \\
 b_{0,r-1}(z) & & \cdots & & b_{r-1,r-1}(z)
 \end{pmatrix}
\]
with respect to the bases $(e^*_0,\dots,e^*_{r-1})$
and
$(e_0,\dots,e_{r-1})$.
Then we have $^t\tau=\tau$ and
$\tau$ also satisfies the other required conditions of the lemma.
The same statement holds for $(\xi_k)$.
\end{proof}

We define a complex
${\mathcal S}^{\bullet}_{\mathrm{ram}}$ by setting
\begin{gather} \label{equation: definition of S^0,S^1,S^2}
 {\mathcal S}^0_{\mathrm{ram}}
 =
 A^0,
 \qquad
 {\mathcal S}^1_{\mathrm{ram}}
 =
 \Sym^2\big(\overline{W}\big) \oplus \Sym^2\big(\overline{V}\big),
\qquad
 {\mathcal S}^2_{\mathrm{ram}}
 =
 G^1\oplus A^1
\end{gather}
and by setting the homomorphisms
\begin{gather}
 d_{{\mathcal S}^{\bullet}}^0\colon \
 {\mathcal S}^0_{\mathrm{ram}} \ni (a_k(w)) \mapsto
\big( \big( \theta_k\circ a_k(\kappa_k\circ\theta_k) \big) ,
 \big( {-}a_k(\kappa_k\circ\theta_k)\circ\kappa_k \big) \big)
 \in {\mathcal S}^1_{\mathrm{ram}},\nonumber \\
 d_{{\mathcal S}^{\bullet}}^1\colon \
 {\mathcal S}^1_{\mathrm{ram}}
 \ni ( (\tau_k), (\xi_k))
 \mapsto
\big( {-}(\delta_{(\tau_k,\xi_k)}) , \Theta_{(\tau_k,\xi_k)} \big)
 \in {\mathcal S}^2_{\mathrm{ram}},\label{equation: definition of complex S^dot}
\end{gather}
where $\delta_{(\tau_k,\xi_k)}\in G^1$
and $\Theta_{(\tau_k,\xi_k)}\in A^1$
are defined by
\begin{gather*}
 \delta_{(\tau_k,\xi_k)}
 =
 \bigg( \sum_{p=1}^{r-1}\sum_{l=1}^p \nu_p(z)
 N_k^{p-l} \circ\big(\theta_k\circ\xi_k+\tau_k\circ\kappa_k \big) \circ N_k^{l-1} \bigg),
 \\
 \Theta_{(\tau_k,\xi_k)} \big( \big(\overline{f_k(w)} \big) \big)
 =
 \Tr \big( f\circ (\theta\circ\xi+\tau\circ\kappa) \big),
\end{gather*}
where $\theta,\kappa$ are lifts of $(\theta_k),(\kappa_k)$
chosen as in the proof of Proposition \ref{prop:correspondence-factorized},
$\tau$, $\xi$ are lifts of $(\tau_k)$, $(\xi_k)$
given by Lemma \ref{lemma: extension of symmetric pairings}
and $f\in \End(E|_{D_{\mathrm{ram}}})$ is a lift of $(f_k(\theta_k\circ\kappa_k))$
given by Lemma \ref{lemma: extension of endomorphisms}.
By virtue of Lemma \ref{lemma: extension of endomorphisms},
we can see that $\Theta_{(\tau_k,\xi_k)}$ is independent of the choices of
$\theta$, $\kappa$, $\tau$, $\xi$ and $f$.
We can also check $d^1_{{\mathcal S}^{\bullet}}\circ d^0_{{\mathcal S}^{\bullet}}=0$.
The meaning of the cohomology
$H^1({\mathcal S}^{\bullet}_{\rm ram})$
is the first order deformation of factorized ramified structure.

We define a homomorphism of complexes
$\gamma^{\bullet} \colon {\mathcal G}^{\bullet}
\longrightarrow {\mathcal S}^{\bullet}_{\mathrm{ram}}[1]$
by
\begin{gather}
 \gamma^0 \colon \ {\mathcal G}^0 \ni u
 \mapsto
\big(
 \big( \overline { u|_{D_{\mathrm{ram}}} }\circ\theta_k
 + \theta_k\circ {}^t \overline { u|_{D_{\mathrm{ram}}} } \big) ,
 \big( {-}\kappa_k\circ \overline { u|_{D_{\mathrm{ram}}} }
 - {} ^t \overline { u|_{D_{\mathrm{ram}}} } \circ \kappa_k \big)
 \big)
 \in
 {\mathcal S}^1_{\mathrm{ram}},\nonumber
 \\
 \gamma^1 \colon \ {\mathcal G}^1 \ni v
 \mapsto
 (-\varpi_G(v),0)
 \in G^1\oplus A^1={\mathcal S}^2_{\mathrm{ram}}, \label{equation: chain map from G to S}
\end{gather}
where
$\overline{u|_{D_{\mathrm{ram}}}} \colon \overline{V}_k
\longrightarrow \overline{V}_k$
is the homomorphism induced by $u|_{D_{\mathrm{ram}}}$.

For $u\in {\mathcal G}^0$,
we have
\begin{gather*}
 \delta_{\gamma^0(u)}
 =
 \bigg( \sum_{p=1}^{r-1} \sum_{l=1}^p \nu_p(z)
 N_k^{p-l}\big(\overline{u|_{D_{\mathrm{ram}}}} N_k-N_k \overline{u|_{D_{\mathrm{ram}}}} \big) N_k^{l-1} \bigg)\\
 \hphantom{\delta_{\gamma^0(u)}}{}
 =
 \big( \overline{u|_{D_{\mathrm{ram}}}} \nu(N_k)-\nu(N_k) \overline{u|_{D_{\mathrm{ram}}}} \big).
\end{gather*}
On the other hand, the restriction
$\nabla|_{D_{\mathrm{ram}}}$ induces the homomorphism
$\nu(N_k)+\frac{k\,{\rm d}z}{rz}\mathrm{id}$ on $\overline{V}_k$.
So we have
$\delta_{\gamma^0(u)}=-\varpi_G(\nabla u-u\nabla)$.
Thus we have
$d_{{\mathcal S}^{\bullet}[1]}^0\gamma^0=\gamma^1 d_{{\mathcal G}^{\bullet}}^0$,
where $d_{{\mathcal S}^{\bullet}[1]}^0=-d^1_{{\mathcal S}^{\bullet}}$.
Set
\begin{equation} \label{equation: definition of tangent complex}
 {\mathcal F}^{\bullet}
 :=
 \mathrm{Cone} \big( {\mathcal G}^{\bullet} \xrightarrow {\gamma^{\bullet}}
 {\mathcal S}^{\bullet}_{\mathrm{ram}}[1] \big) [-1].
\end{equation}
So we have
\begin{gather*}
 {\mathcal F}^0
=
 {\mathcal G}^0\oplus A^0,
 \qquad
 {\mathcal F}^1
 =
 {\mathcal G}^1\oplus \Sym^2\big(\overline{V}\big)\oplus \Sym^2\big(\overline{W}\big),
 \qquad
 {\mathcal F}^2
=
 G^1\oplus A^1
\end{gather*}
and $d^0_{{\mathcal F}^{\bullet}}\colon {\mathcal F}^0 \longrightarrow {\mathcal F}^1$,
$d^1_{{\mathcal F}^{\bullet}}\colon {\mathcal F}^1 \longrightarrow {\mathcal F}^2$
are defined by
\begin{gather*}
 d^0_{{\mathcal F}^{\bullet}} ( u, (a_k(w)) )
 =
 \big(\nabla\circ u-(u\otimes\mathrm{id})\circ \nabla ,
 -\gamma^0(u)+d^0_{{\mathcal S}^{\bullet}}((a_k(w))) \big)
 \\
 d^1_{{\mathcal F}^{\bullet}} ( v, ( (\tau_k),(\xi_k) ) )
 =
 \big(\varpi_G(v)-(\delta_{(\tau_k,\xi_k)}) , \Theta_{(\tau_k,\xi_k)} \big).
\end{gather*}

Consider the complexes
${\mathcal F}_0^{\bullet}=
\big[ {\mathcal G}^0 \oplus {\mathcal S}^0_{\mathrm{ram}}
 \longrightarrow \Sym^2\big(\overline{W}\big) \big]$,
${\mathcal F}_1^{\bullet}=
 \big[
 {\mathcal G}^1 \oplus \Sym^2\big(\overline{V}\big)
 \longrightarrow
 {\mathcal S}^2_{\mathrm{ram}} \big]$
defined by
\begin{gather*}
 d_{{\mathcal F}_0^{\bullet}}^0
 \colon \
 {\mathcal G}^0 \oplus A^0
 \ni (u,(\overline{a_k(w)})) \mapsto
 \big({-}\overline { u|_{D_{\mathrm{ram}}} } \!\circ\theta_k\!
 - \theta_k\!\circ \!\! {}^t \overline { u|_{D_{\mathrm{ram}}} }\!
 + \theta_k\circ a_k(\kappa_k\circ\theta_k) \big)
 \!\in\! \Sym^2\big(\overline{W}\big),
 \\
 d^0_{{\mathcal F}_1^{\bullet}}
 \colon \
 {\mathcal G}^1 \oplus \Sym^2\big(\overline{V}\big)
 \ni (v,(\xi_k)) \mapsto
 (\varpi_G(v)-(\delta_{(0,\xi_k)}), (\Theta_{(0,\xi_k)}) )
 \in G^1\oplus A^1.
\end{gather*}
Then there is an exact sequence of complexes
\begin{equation*}
 0 \longrightarrow {\mathcal F}^{\bullet}_1[-1]
 \longrightarrow {\mathcal F}^{\bullet}
 \longrightarrow {\mathcal F}^{\bullet}_0
 \longrightarrow 0
\end{equation*}
which is expressed by the diagram
\[
 \begin{CD}
 0 @>>> {\mathcal G}^0\oplus A^0 @>>> {\mathcal G}^0\oplus A^0 \\
 @VVV @V d^0_{{\mathcal F}^{\bullet}} VV @V d^0_{{\mathcal F}_0^{\bullet}} V V \\
 {\mathcal G}^1\oplus \Sym^2\big(\overline{V}\big) @>>>
 {\mathcal G}^1\oplus \Sym^2\big(\overline{W}\big)\oplus \Sym^2\big(\overline{V}\big)
 @>>> \Sym^2(\overline{W}) \\
 @V d^0_{{\mathcal F}_1^{\bullet}} VV @V d^1_{{\mathcal F}^{\bullet}} VV @VVV \\
 G^1\oplus A^1 @>>> G^1\oplus A^1 @>>> 0.
 \end{CD}
\]
So we get the following exact sequence of hyper cohomologies:
\begin{gather}
 0\rightarrow
 \mathbf{H}^0({\mathcal F}^{\bullet})
 \rightarrow \mathbf{H}^0({\mathcal F}^{\bullet}_0)
 \rightarrow \mathbf{H}^0({\mathcal F}^{\bullet}_1)
 \rightarrow \mathbf{H}^1({\mathcal F}^{\bullet})
 \rightarrow \mathbf{H}^1({\mathcal F}^{\bullet}_0)
 \rightarrow \mathbf{H}^1({\mathcal F}^{\bullet}_1)\nonumber\\
 \hphantom{0}{}
 \rightarrow \mathbf{H}^2({\mathcal F}^{\bullet})
 \rightarrow 0.\label{equation: hyper cohomology exact sequence}
\end{gather}

\begin{Proposition} \label{prop: tangent space of the moduli space}
The relative tangent space
of $M^{\balpha}_{{\mathcal C},{\mathcal D}}(\lambda,\tilde{\mu},\tilde{\nu})$
over ${\mathcal T}$ at
$(E,\nabla,\{ l,\ell,{\mathcal V}\})$
is isomorphic to $\mathbf{H}^1({\mathcal F}^{\bullet})$.
\end{Proposition}

\begin{proof}
Take a point $t\in{\mathcal T}$ and
a point $y\in M^{\balpha}_{{\mathcal C},{\mathcal D}}(\lambda,\tilde{\mu},\tilde{\nu})$
over $t$
corresponding to a connection
$(E,\nabla,\{ l,\ell,{\mathcal V}\})$
with $(\lambda,\mu,\nu)$-structure.
Giving a tangent vector $v$ of the fiber
$M^{\balpha}_{{\mathcal C},{\mathcal D}}(\lambda,\tilde{\mu},\tilde{\nu})_t$
of the moduli space at $y$
is equivalent to giving
a flat family
$(\tilde{E},\tilde{\nabla},\{\tilde{l},\tilde{\ell},\tilde{\mathcal V}\})$
of connections with $(\lambda,\mu,\nu)$-structure
on $C\times\Spec\mathbb{C}[\epsilon]$
satisfying
$\big(\tilde{E},\tilde{\nabla},\big\{\tilde{l},\tilde{\ell},\tilde{\mathcal V}\big\}\big)
\otimes\mathbb{C}[\epsilon]/(\epsilon)
\cong (E,\nabla,\{ l,\ell,{\mathcal V}\})$,
where $\mathbb{C}[\epsilon]=\mathbb{C}[\epsilon]/\big(\epsilon^2\big)$.
Take an affine open covering
$\{U_{\alpha}\}$ of $C$
such that $E|_{U_{\alpha}}\cong{\mathcal O}_{U_{\alpha}}^{\oplus r}$ for any~$\alpha$.
Put $U_{\alpha}[\epsilon]:=U_{\alpha}\times\Spec\mathbb{C}[\epsilon]$.
We may assume that for each $x\in D$,
there exists exactly one index $\alpha$ satisfying
$x\in U_{\alpha}$ and that
each $U_{\alpha}$ contains at most one point in $D$.
We can take a~lift
$\varphi_{\alpha}\colon
E\otimes\mathbb{C}[\epsilon]|_{U_{\alpha}[\epsilon]} \stackrel{\sim}\longrightarrow
\tilde{E}|_{U_{\alpha}[\epsilon]}$
of the given isomorphism
$E|_{U_{\alpha}}\stackrel{\sim}\longrightarrow
\tilde{E}\otimes\mathbb{C}[\epsilon]/(\epsilon)|_{U_{\alpha}}$.
We may assume that $\varphi_{\alpha}$ preserves $l$
if ${\mathcal D}_{\mathrm{log}}\cap U_{\alpha}\neq\varnothing$
and preserves $\ell$
if ${\mathcal D}_{\mathrm{un}}\cap U_{\alpha}\neq\varnothing$.
If ${\mathcal D}_{\mathrm{ram}}\cap U_{\alpha}\neq \varnothing$,
then we may assume that
$\varphi_{\alpha}$ sends the filtration $\big\{V_k\otimes\mathbb{C}[\epsilon]\big\}$
to the filtration $\{\tilde{V}_k\}$.
Set
\begin{gather*}
 \epsilon u_{\alpha\beta}= \varphi_{\alpha}^{-1}\circ\varphi_{\beta}-\mathrm{id}, \\
 \epsilon v_{\alpha}=
 (\varphi_{\alpha}\otimes\mathrm{id})^{-1}
 \circ\tilde{\nabla}\circ\varphi_{\alpha}-\nabla\otimes\mathbb{C}[\epsilon], \\
 \epsilon \eta_{\alpha}
 =
 \big(
 \varphi_{\alpha}|_{D_{\mathrm{ram}}}^{-1}\circ\tilde{\theta}_k
 \circ \,^t\varphi_{\alpha}|_{D_{\mathrm{ram}}}^{-1}
 -\theta_k ,
 {}^t\varphi_{\alpha}|_{D_{\mathrm{ram}}}\circ\tilde{\kappa}_k
 \circ\varphi_{\alpha}|_{D_{\mathrm{ram}}}
 -\kappa_k \big).
\end{gather*}
Then we get a cohomology class
$[ \{u_{\alpha\beta}\}, \{ v_{\alpha} , (\eta_{\alpha}) \} ]
\in \mathbf{H}^1({\mathcal F}^{\bullet})$,
which can be checked to be independent of the choice of
$\{U_{\alpha},\varphi_{\alpha}\}$.

Conversely, assume that a cohomology class
$[ \{u_{\alpha\beta}\}, \{ v_{\alpha} \} , \{ \eta_{\alpha} \} ]
\in \mathbf{H}^1({\mathcal F}^{\bullet})$
is given.
We define
\begin{gather}
 \sigma_{\beta\alpha}
=
 \mathrm{id}+\epsilon u_{\alpha\beta}
 \colon \
 {\mathcal O}_{U_{\alpha\beta}[\epsilon]}^{\oplus r}
 \stackrel{\sim}\longrightarrow
 {\mathcal O}_{U_{\alpha\beta}[\epsilon]}^{\oplus r},\nonumber \\
 \nabla_{\alpha}
 =\nabla+\epsilon v_{\alpha}
 \colon \
 {\mathcal O}_{U_{\alpha}[\epsilon]}^{\oplus r}
 \longrightarrow
 {\mathcal O}_{U_{\alpha}[\epsilon]}^{\oplus r}
 \otimes\Omega^1_C(D).\label{equation: local patching of bundle isomorphism}
\end{gather}
If $U_{\alpha}\cap D_{\mathrm{log}}\neq\varnothing$
we put
$l_{\alpha}:=l|_{U_{\alpha}}\otimes\mathbb{C}[\epsilon]$
and we put
$\ell_{\alpha}:=\ell|_{U_{\alpha}\otimes\mathbb{C}[\epsilon]}$
if $U_{\alpha}\cap D_{\mathrm{un}}\neq\varnothing$.
Then we can see that
$\nabla_{\alpha}$ preserves $l_{\alpha}$
if $U_{\alpha}\cap D_{\mathrm{log}}\neq\varnothing$
and preserves $\ell_{\alpha}$
if $U_{\alpha}\cap D_{\mathrm{un}}\neq\varnothing$,
because $v_{\alpha}$ preserves $l|_{U_{\alpha}}$ and $\ell|_{U_{\alpha}}$ by the definition of
${\mathcal G}^1$.

Consider the case
 $U_{\alpha}\cap D_{\mathrm{ram}}=\{x\}$.
We can write
$\eta_{\alpha}=(\tau_k,\xi_k)_{0\leq k\leq r-1}$.
By the choice of $\eta_{\alpha}$, we have
$\delta_{(\tau_k,\xi_k)}=\overline{v_{\alpha}|_{D_{\mathrm{ram}}}}$ and
$\Theta_{(\tau_k,\xi_k)}=0$,
which yield the equalities
\begin{gather}
 \Tr \left( (\theta\circ\xi+\tau\circ\kappa)\circ N^j \right)=0,
 \qquad 0\leq j\leq r-1, \nonumber
 \\
 \sum_{p=1}^{r-1}\sum_{j=1}^p \nu_p(z) N_k^{p-j}(\theta_k\xi_k+\tau_k\kappa_k)N_k^{j-1}
 =\overline { v_{\alpha}|_{D_{\mathrm{ram}}} }, \qquad 0\leq k\leq r-1, \label{condition in G^1}
\end{gather}
where $N$, $\theta$ and $\kappa$ are lifts of $(N_k)$, $(\theta_k)$ and $(\kappa_k)$
chosen as in the proof of Proposition~\ref{prop:correspondence-factorized}
and~$\tau$,~$\xi$ are lifts of $(\tau_k)$, $(\xi_k)$
given by Lemma~\ref{lemma: extension of symmetric pairings}.

Since the minimal polynomial $w^r$ of $N|_x$ is of degree~$r$, we can see
from \cite[Lemma~1.4]{Inaba-3}. that
\begin{equation*}
 \im ( \mathrm{ad}(N) )
 =
\big\{ f\in\End_{{\mathcal O}_{m_xx}}(E|_{m_xx}) \, \big| \,
 \text{$\Tr\big(f\circ N^l\big)=0$ for any $l\geq 0$} \big\}.
\end{equation*}
So we can find an endomorphism
$f\in\End(E|_{m_xx})$ satisfying
$\theta\circ\xi+\tau\circ\kappa=f\circ N-N\circ f$.

Now we will construct a factorized ramified structure
on $\big({\mathcal O}_{U_{\alpha}[\epsilon]}^{\oplus r} , \nabla_{\alpha}\big)$.
We take $(V_k\otimes\mathbb{C}[\epsilon])_{0\leq k\leq r-1}$
as the relative version of the filtration in Definition~\ref{def-fac-connection}\,(i).
The homomorphisms
\begin{gather*}
 \theta_{k,\epsilon}
 :=
 \theta_k+\epsilon\tau_k
 \colon \
 \overline{W}_k\otimes\mathbb{C}[\epsilon]
 \longrightarrow \overline{V}_k\otimes\mathbb{C}[\epsilon],
 \\
 \kappa_{k,\epsilon}
:=
 \kappa_k+\epsilon\xi_k
 \colon \
 \overline{V}_k\otimes\mathbb{C}[\epsilon]
 \longrightarrow \overline{W}_k\otimes\mathbb{C}[\epsilon]
\end{gather*}
become lifts of $\theta_k$ and $\kappa_k$, respectively.
They determine bilinear pairings
\begin{gather*}
\begin{split}
& \vartheta_{k,\epsilon}
 \colon \
 \big(\overline{W}_k\otimes\mathbb{C}[\epsilon]\big)
 \times \big(\overline{W}_{r-k-1}\otimes\mathbb{C}[\epsilon] \big)
 \longrightarrow {\mathcal O}_{m_xx}\otimes\mathbb{C}[\epsilon],
 \\
& \varkappa_{k,\epsilon}
 \colon \
 \big( \overline{V}_k\otimes\mathbb{C}[\epsilon] \big)
 \times \big( \overline{V}_{r-k-1}\otimes\mathbb{C}[\epsilon] \big)
 \longrightarrow {\mathcal O}_{m_xx}\otimes\mathbb{C}[\epsilon],
 \end{split}
\end{gather*}
which satisfy the commutative diagrams in (ii), (iii)
of Definition \ref{def-fac-connection}.
Since $N^r=z\cdot\mathrm{id}_{E|_{m_xx}}$, the equality
\begin{align*}
 (N+\epsilon (\theta\circ\xi+\tau\circ\kappa))^r
 &=
 N^r+\epsilon\sum_{j=0}^{r-1} N^j\circ (\theta\circ\xi+\tau\circ\kappa)\circ N^{r-j-1}
 \\
 &=
 N^r+\epsilon\sum_{j=0}^{r-1} N^j\circ (f\circ N-N\circ f) \circ N^{r-j-1}
 \\
 &=
 N^r+f\circ N^r-N^r\circ f
 =z\, \mathrm{id}_{{\mathcal O}_{U_{\alpha}[\epsilon]}^{\oplus r}}
\end{align*}
holds.
By the equality (\ref{condition in G^1}),
\begin{gather*}
 \nu ( N_k+\epsilon (\theta_k\circ\xi_k+\tau_k\circ\kappa_k) )
 +k\,{\rm d}z/rz\,\mathrm{id}
 \\
\qquad{} =
 \nu(N_k)+k\,{\rm d}z/rz\,\mathrm{id}
 +\epsilon \sum_{p=1}^{r-1}\sum_{j=1}^p \nu_p(z) N_k^{p-j}
 (\theta_k\circ\xi_k+\tau_k\circ\kappa_k) N_k^{j-1}
 \\
\qquad{} =
 \nu(N_k)+k\,{\rm d}z/rz\,\mathrm{id}
 +\epsilon \overline { v_{\alpha}|_{D_{\rm ram} } }
\end{gather*}
coincides with the map induced by
$\nabla_{\alpha}$.
So the relative version of the condition (iv) of Definition~\ref{def-fac-connection}
is satisfied.
The endomorphism $N_k+\epsilon (\theta_k\circ\xi_k+\tau_k\circ\kappa_k)$ defines a
$\mathbb{C}[w]\otimes_{\mathbb{C}}\mathbb{C}[\epsilon]$-module structure
on $\overline{V}_k\otimes\mathbb{C}[\epsilon]$.
Define an isomorphism
\[
 \psi_{k,\epsilon} \colon \ \overline{V}_k\otimes\mathbb{C}[\epsilon]
 \stackrel{\sim} \longrightarrow
 (w)/\big(w^{m_xr-r+2}\big)\otimes \overline{V}_{k-1}\otimes\mathbb{C}[\epsilon]
\]
of $\mathbb{C}[w]\otimes_{\mathbb{C}}\mathbb{C}[\epsilon]$-modules
by setting
\[
 \psi_{k,\epsilon} \big( \pi_k\big( (N+\epsilon (\theta\circ\xi+\tau\circ\kappa))^ke_0 \big)\big)
 =
 w\otimes \pi_{k-1} \big( (N+\epsilon (\theta\circ\xi+\tau\circ\kappa))^{k-1}e_0 \big) ,
\]
where $\pi_k$ means $\pi_k\otimes\mathbb{C}[\epsilon]$.
Then the image of
$z\otimes \pi_0(e_0)$ via the composition
\begin{gather}
 (z)\otimes \overline{V}_0\otimes\mathbb{C}[\epsilon]
 \rightarrow \overline{V}_{r-1}\otimes\mathbb{C}[\epsilon]
 \xrightarrow[\sim]{\psi_{r-1,\epsilon}}
 (w)\otimes\overline{V}_{r-2}\otimes\mathbb{C}[\epsilon]
\nonumber\\
\hphantom{(z)\otimes \overline{V}_0\otimes\mathbb{C}[\epsilon]
 \rightarrow \overline{V}_{r-1}\otimes\mathbb{C}[\epsilon]}{}
 \xrightarrow[\sim]{\psi_{r-2,\epsilon}}\cdots \xrightarrow[\sim]{\psi_{1,\epsilon}}
 (w^{r-1})\otimes\overline{V}_0\otimes\mathbb{C}[\epsilon]\label{equation: deformation of composition}
\end{gather}
coincides with
$(\psi_{1,\epsilon}\circ\cdots\circ\psi_{r-1,\epsilon})
\big(\pi_{r-1}\big( (N+\epsilon (\theta\circ\xi+\tau\circ\kappa))^re_0\big)\big)
=w^r\otimes \pi_0(e_0)$.
Thus the composition~(\ref{equation: deformation of composition})
coincides with the homomorphism
$(z)\otimes\overline{V}_0\longrightarrow \big(w^{r-1}\big)\otimes L_0$
obtained by tensoring $\overline{V}_0\otimes\mathbb{C}[\epsilon]$ to
the canonical homomorphism $(z)\longrightarrow \big(w^{r-1}\big)$.

If we put
${\mathcal V}_{\alpha}:=
\big(\overline{V}_k\otimes\mathbb{C}[\epsilon],\vartheta_{k,\epsilon},\kappa_{k,\epsilon}\big)_{0\leq k\leq r-1}$,
then we can see from the above arguments that
$\big({\mathcal O}_{U_{\alpha}[\epsilon]}^{\oplus r} ,
\nabla_{\alpha} , {\mathcal V}_{\alpha}\big)$
is a flat family of local connections with $\nu(w)_x$-ramified structure which is a lift of
$(E,\nabla,{\mathcal V})|_{U_{\alpha}}$.
We can patch all the local connections
$\big({\mathcal O}_{U_{\alpha}[\epsilon]}^{\oplus r} ,l_{\alpha},\ell_{\alpha}
\nabla_{\alpha} , {\mathcal V}_{\alpha}\big)$
with $(\lambda,\mu,\nu)$-structure
via the isomorphisms $\sigma_{\beta\alpha}$
defined in~(\ref{equation: local patching of bundle isomorphism}).
Then we obtain a global flat family of connections
$\big(\tilde{E},\tilde{\nabla},\tilde{l},\tilde{\ell},\tilde{\mathcal V}\big)$
with $(\lambda,\mu,\nu)$-structure
which gives a tangent vector
$v\in T_{M^{\balpha}_{{\mathcal C},{\mathcal D}}(\lambda,\tilde{\mu},\tilde{\nu})/{\mathcal T}}(y)$
at $y$.
We can see from its construction that
the map
$[\{u_{\alpha\beta}\},\{v_{\alpha},\eta_{\alpha}\}]\mapsto v$
gives the desired inverse.
\end{proof}

\section{Smoothness of the moduli space}\label{section: smoothness of the moduli}

In this section, we assume the same notations as in
Sections~\ref{section: construction of the moduli space}
and~\ref{section: tangent space}.
Take a connection
$(E,\nabla,\{ l,\ell,{\mathcal V}\})
\in M^{\balpha}_{{\mathcal C},{\mathcal D}}(\lambda,\tilde{\mu},\tilde{\nu})_t$
with $(\lambda,\mu,\nu)$-structure.
We define a pairing
\[
 \Xi_{\mathrm{ram}} \colon \
 {\mathcal S}^1_{\mathrm{ram}} \times {\mathcal S}^1_{\mathrm{ram}}
 \longrightarrow \Omega^1_C(D)|_{D_{\mathrm{ram}}}
\]
by setting
\begin{gather}
 \Xi_{\mathrm{ram}} ( (\tau_k,\xi_k), (\tau'_k,\xi'_k) )\nonumber\\
\qquad{} :=
 \sum_{p=1}^{r-1}\sum_{j=1}^p
 \frac{\nu_p(z)}{2}
 \Tr\big( \tau' \circ \,^t\!N^{p-j} \circ \xi \circ N^{j-1}
 -N^{p-j} \circ \tau \circ \,^t\!N^{j-1} \circ \xi'
\big)\label{equation: definition of Xi}
\end{gather}
for $(\tau_k),(\tau'_k)\in \Sym^2\big(\overline{W}\big)$ and
$(\xi_k),(\xi'_k)\in \Sym^2\big(\overline{V}\big)$,
where $\tau,\tau'\in \Hom(E|_{D_{\mathrm{ram}}}^{\vee},E|_{D_{\mathrm{ram}}})$
are lifts of $(\tau_k),(\tau'_k)$
and $\xi,\xi'\in \Hom\big(E|_{D_{\mathrm{ram}}},E|_{D_{\mathrm{ram}}}^{\vee}\big)$
are lifts of $(\xi_k),(\xi'_k)$
given by Lemma \ref{lemma: extension of symmetric pairings},
respectively.

Take an affine open covering
$C=\bigcup_{\alpha} U_{\alpha}$
for the calculation of the hypercohomologies in \v{C}ech cohomology.
We define a bilinear pairing
\begin{gather} \label{equation: definition of symplectic form}
 \omega_{ (E,\nabla,\{ l,\ell,{\mathcal V}\}) } \colon \
 \mathbf{H}^1({\mathcal F}^{\bullet}) \times \mathbf{H}^1({\mathcal F}^{\bullet})
 \longrightarrow
 \mathbf{H}^2\big({\mathcal O}_C \rightarrow \Omega^1_C(D_{\mathrm{ram}})
 \rightarrow \Omega^1_C(D_{\mathrm{ram}})|_{D_{\mathrm{ram}}} \big)
 = \mathbb{C}
\end{gather}
by setting
\begin{gather}
 \omega_{ (E,\nabla,\{ l,\ell,{\mathcal V}\}) }
\big( \big[ \{ u_{\alpha\beta}\},
 \{ v_{\alpha}, \eta_{\alpha} \} \big] ,
 \big[ \{ u'_{\alpha\beta}\},
 \{ v'_{\alpha}, \eta'_{\alpha} \} \big]
\big)
 \nonumber\\
\qquad{}=
\big[
 \{ \Tr(u_{\alpha\beta}\circ u'_{\beta\gamma})\},
 \{ {-}\Tr(u_{\alpha\beta}\circ v'_{\beta}-v_{\alpha}\circ u'_{\alpha\beta})\},
 \{\Xi_{\mathrm{ram}}( \eta_{\alpha} , \eta'_{\alpha} )
 \}
\big] \label{expression of symplectic form}
\end{gather}
for
$u_{\alpha\beta},u'_{\alpha\beta}\in {\mathcal G}^0|_{U_{\alpha\beta}}$,
$v_{\alpha},v'_{\alpha}\in {\mathcal G}^1|_{U_{\alpha}}$,
$\eta_{\alpha},\eta'_{\alpha}
\in {\mathcal S}^1_{\mathrm{ram}}|_{U_{\alpha}}$
satisfying the cocycle conditions
\begin{alignat*}{3}
& \nabla u_{\alpha\beta}-u_{\alpha\beta}\nabla
=
 v_{\beta}-v_{\alpha} ,
 \qquad&&
 \gamma^1(v_{\alpha}) = d^1_{{\mathcal S}^{\bullet}}(\eta_{\alpha}) ,&
 \\
 &\nabla u'_{\alpha\beta}-u'_{\alpha\beta}\nabla
=
 v'_{\beta}-v'_{\alpha} ,
 \qquad&&
 \gamma^1(v'_{\alpha})
 =d^1_{{\mathcal S}^{\bullet}}(\eta'_{\alpha}),&
\end{alignat*}
where $d^1_{{\mathcal S}^{\bullet}}$ and $\gamma^1$ are defined in
(\ref{equation: definition of complex S^dot}) and
(\ref{equation: chain map from G to S}).
From the following lemma,
we can see that the pairing
$\omega_{ (E,\nabla,\{ l,\ell,{\mathcal V}\}) } \big( \big[ \{ (u_{\alpha\beta}\},
 \{ v_{\alpha}, \eta_{\alpha} \} \big] ,
 \big[ \{ (u'_{\alpha\beta}\},
 \{ v'_{\alpha}, \eta'_{\alpha} \} \big]
 \big)$
in (\ref{expression of symplectic form})
depends only on the cohomology classes
$ \big[ \{ (u_{\alpha\beta}\},
 \{ v_{\alpha}, \eta_{\alpha} \} \big] ,
 \big[ \{ (u'_{\alpha\beta}\},
 \{ v'_{\alpha}, \eta'_{\alpha} \} \big]
\in \mathbf{H}^1({\mathcal F}^{\bullet})$.

\begin{Lemma}
The equality
\[
 \omega_{ (E,\nabla,\{ l,\ell,{\mathcal V}\}) } \big(
 \big[ \{ u_{\alpha\beta}\},
 \{ v_{\alpha} , \eta_{\alpha} \} \big] ,
 \big[ \{ u'_{\alpha\beta}\},
 \{ v'_{\alpha} , \eta'_{\alpha} \} \big]
\big)=0
\]
holds
if there exists
$\{u_{\alpha},(a_{k,\alpha}(w))\}\in C^0\big(\{U_{\alpha}\},{\mathcal F}^0\big)$
which satisfies the equalities
\begin{gather*}
 u_{\alpha\beta} = u_{\beta}-u_{\alpha},
 \\
 v_{\alpha}= \nabla\circ u_{\alpha}-(u_{\alpha}\otimes\mathrm{id})\circ\nabla,
 \\
 \eta_{\alpha}=
 -\gamma^0(u_{\alpha})+d^0_{{\mathcal S}^{\bullet}}((a_{k,\alpha}(w))),
\end{gather*}
where $\gamma^0\colon {\mathcal G}^0\longrightarrow {\mathcal S}^1_{\mathrm{ram}}$
is defined in \eqref{equation: chain map from G to S} and
$d^0_{{\mathcal S}^{\bullet}} \colon {\mathcal S}^0_{\mathrm{ram}}
\longrightarrow {\mathcal S}^1_{\mathrm{ram}}$
is defined in~\eqref{equation: definition of complex S^dot}.
\end{Lemma}

\begin{proof}We put $c_{\alpha\beta}:=\Tr(u_{\alpha}\circ u'_{\alpha\beta})$
and $b_{\alpha}:=\Tr(u_{\alpha}\circ v'_{\alpha})$.
It is sufficient to prove the equality
\[
 d (\{c_{\alpha\beta}\},\{b_{\alpha}\})
 =
 \big(
 \{ \Tr(u_{\alpha\beta}\circ u'_{\beta\gamma})\},
 \{ {-}\Tr(u_{\alpha\beta}\circ v'_{\beta}-v_{\alpha}\circ u'_{\alpha\beta})\},
 \{\Xi_{\mathrm{ram}}( \eta_{\alpha} , \eta'_{\alpha} )
 \}\big).
\]
We need a certain amount of calculations for checking the above equality,
but we can do it in the same way as that of
\cite[pp.~37--39]{Inaba-3}.
\end{proof}

\begin{Proposition}\label{proposition: nondegenerate pairing on tangent space}
The bilinear pairing
$\omega_{ (E,\nabla,\{ l,\ell,{\mathcal V}\}) }\colon \mathbf{H}^1({\mathcal F}^{\bullet})
\times \mathbf{H}^1({\mathcal F}^{\bullet})
\longrightarrow \mathbb{C}$,
defined by the equality \eqref{expression of symplectic form}
in \eqref{equation: definition of symplectic form},
is a nondegenerate pairing.
\end{Proposition}

\begin{proof}The bilinear pairing
$\omega_{ (E,\nabla,\{ l,\ell,{\mathcal V}\}) }$ corresponds to a homomorphism
$\sigma\colon \mathbf{H}^1({\mathcal F}^{\bullet})
\longrightarrow\mathbf{H}^1({\mathcal F}^{\bullet})^{\vee}$
which induces the exact commutative diagram
\[
 \begin{CD}
 \mathbf{H}^0({\mathcal F}_0^{\bullet}) @>>> \mathbf{H}^0({\mathcal F}_1^{\bullet})
 @>>> \mathbf{H}^1({\mathcal F}^{\bullet}) @>>> \mathbf{H}^1({\mathcal F}_0^{\bullet})
 @>>> \mathbf{H}^1({\mathcal F}_1^{\bullet}) \\
 @V\sigma_1 VV @V\sigma_2 VV @V\sigma VV @V\sigma_3 VV @V\sigma_4 VV \\
 \mathbf{H}^1({\mathcal F}_1^{\bullet})^{\vee} @>>> \mathbf{H}^1({\mathcal F}_0^{\bullet})^{\vee}
 @>>> \mathbf{H}^1({\mathcal F}^{\bullet})^{\vee} @>>> \mathbf{H}^0({\mathcal F}_1^{\bullet})^{\vee}
 @>>> \mathbf{H}^0({\mathcal F}_0^{\bullet})^{\vee}.
 \end{CD}
\]
The homomorphism
$\sigma_2 \colon \mathbf{H}^0({\mathcal F}_1^{\bullet})
\longrightarrow
\mathbf{H}^1({\mathcal F}_0^{\bullet})^{\vee}$
is given by the pairing
\begin{gather*}
 \mathbf{H}^0({\mathcal F}_1^{\bullet})
 \times \mathbf{H}^1({\mathcal F}_0^{\bullet})
 \longrightarrow
 \mathbf{H}^2({\mathcal O}_C\rightarrow
 \Omega^1_C(D_{\mathrm{ram}})\rightarrow\Omega^1_C(D_{\mathrm{ram}})|_{D_{\mathrm{ram}}})
 \cong\mathbb{C}, \\
\big( [\{(v_{\alpha},(\xi_{k,\alpha}))\} ],
 [\{(u'_{\alpha\beta},(\tau'_{k,\alpha}))\} ] \big)
 \mapsto
 \big[\{ \Tr(v_{\alpha}\circ u'_{\alpha\beta})\},
 \{
 \Xi_{\mathrm{ram}} ( (0,\xi_{k,\alpha}) , (\tau'_{k,\alpha},0) )
 \}\big]
\end{gather*}
and $\sigma_3$ is defined similarly.
There is an exact commutative diagram
\begin{gather*}
 \begin{CD}
 0\longrightarrow & H^0\big(\!\ker\!\big({\mathcal G}^1\!\rightarrow G^1\big)\big) & \longrightarrow &
 \mathbf{H}^0({\mathcal F}_1^{\bullet}) & \longrightarrow &
 \ker\big(\!\Sym^2\!\big(\overline{V}\big)\!\rightarrow A^1\big) & \longrightarrow &
 H^1\big(\!\ker\!\big({\mathcal G}^1\!\rightarrow G^1\big)\big) \\
 & @V \eta_1 VV @V \sigma_2 VV
 @V \eta_2 VV @V \eta_3 VV \\
 0 \longrightarrow & H^1\big({\mathcal G}^0\big)^{\vee} & \longrightarrow &
 \mathbf{H}^1({\mathcal F}_0^{\bullet})^{\vee} & \longrightarrow &
 \ \coker\big(A^0\!\rightarrow \Sym^2\big(\overline{W}\big)\big)^{\vee}
 & \longrightarrow & H^0\big({\mathcal G}^0\big)^{\vee}
 \end{CD}
\end{gather*}
whose horizontal sequences are induced by
the exact sequences
\begin{gather*}
 0\longrightarrow \big[{\mathcal G}^1\rightarrow G^1\big]
 \longrightarrow {\mathcal F}^{\bullet}_1 \longrightarrow
 \big[\Sym^2\big(\overline{V}\big)\rightarrow A^1\big] \longrightarrow 0,
 \\
 0\longrightarrow \big[A^0\rightarrow \Sym^2\big(\overline{W}\big)\big] \longrightarrow
{\mathcal F}^{\bullet}_0 \longrightarrow {\mathcal G}^0\longrightarrow 0.
\end{gather*}
Since the trace pairing induces an isomorphism
$\ker\big({\mathcal G}^1\rightarrow G^1\big)\xrightarrow{\sim}
\big({\mathcal G}^0\big)^{\vee}\otimes\Omega^1_C$,
we can see by the Serre duality that $\eta_1$ and $\eta_3$ are isomorphisms.

The map
$\eta_2$ is induced by the trace pairing
\begin{align}
 \ker\big(\Sym^2\big(\overline{V}\big)\rightarrow A^1\big)\times \coker \big(A^0\rightarrow \Sym^2\big(\overline{W}\big)\big)
 & \longrightarrow \Omega^1_C(D_{\mathrm{ram}})|_{D_{\mathrm{ram}}},\nonumber
 \\
 ( (\xi_{k}) , (\tau_{k}) ) & \mapsto
 \Xi_{\mathrm{ram}} \left((0,\xi_k) , (\tau_k,0) \right)\label{equation: trace pairing on the symmetric vectors}
 \end{align}
composed with
$\Omega^1_C(D_{\mathrm{ram}})|_{D_{\mathrm{ram}}}
\longrightarrow
\mathbf{H}^2\big({\mathcal O}_C\rightarrow\Omega^1_C(D_{\mathrm{ram}})\to
\Omega^1_C(D_{\mathrm{ram}})|_{D_{\mathrm{ram}}}\big)$.

Assume that $(\xi_k)\in \ker\big(\Sym^2(\overline{V})\rightarrow A^1\big)$ satisfies
$\Xi_{\mathrm{ram}}((0,\xi_k),(\tau_k,0))=0$
for any $(\tau_k)\in \Sym^2\big(\overline{W}\big)$.
We can take a lift $\xi$ of $(\xi_k)$ given by Lemma \ref{lemma: extension of symmetric pairings}.
For any endomorphism $h\in \End(E|_{D_{\mathrm{ram}}})$,
$\psi:=z\big(h\circ\theta+\theta\circ {}^th\big)\colon E|_{D_{\mathrm{ram}}}^{\vee}\longrightarrow
E|_{D_{\mathrm{ram}}}$
is a homomorphism satisfying
$^t\psi=\psi$ and $\psi(W_k)\subset V_k$.
So $\psi$ induces $(\psi_k)\in \Sym^2\big(\overline{W}\big)$ and the equality
\[
 0=2\Xi_{\mathrm{ram}}((0,\xi_k),(\psi_k,0))=
 \sum_{p=1}^{r-1}\sum_{j=1}^p
 \nu_p(z)\Tr\big(
 z\big(h\circ\theta+\theta\circ{}^th\big)\circ{}^tN^{p-j}\circ\xi\circ N^{j-1}\big).
\]
holds by the assumption.
Since
\begin{align*}
 \sum_{j=1}^p \Tr\big(z \theta\circ{}^th\circ{}^tN^{p-j}\circ\xi\circ N^{j-1}\big)
 &=
 \sum_{j=1}^p \Tr\big(z {}^tN^{j-1}\circ\xi\circ N^{p-j}\circ h\circ\theta\big)
 \\
 &=\sum_{j=1}^p \Tr (z h\circ\theta\circ{}^tN^{p-j}\circ\xi\circ N^{j-1}\big),
\end{align*}
we can deduce
$\Tr\big( h\circ z\sum_{p=1}^{r-1}\sum_{j=1}^p
\nu_p(z) \theta\circ{}^tN^{p-j}\circ\xi\circ N^{j-1}\big)=0$.
Since the usual trace pairing is nondegenerate, we have
$z\sum_{p=1}^{r-1}\sum_{j=1}^p
\nu_p(z) {}^tN^{p-j}\circ\xi\circ N^{j-1}=0$.
Let
\[
 U=\begin{pmatrix}
 z^{m-1}a_{0,0} & \cdots & z^{m-1}a_{0,r-1} \\
 \vdots & & \vdots \\
 z^{m-1}a_{0,r-1,0} & \cdots & z^{m-1}a_{r-1,r-1}
 \end{pmatrix}
\]
be the symmetric matrix representing
$\sum_{p=1}^{r-1}\sum_{j=1}^p
\nu_p(z) \,^tN^{p-j}\circ\xi\circ N^j$
with respect to the bases $(e_0,\dots,e_{r-1})$ and $(e^*_0,\dots,e^*_{r-1})$.
Consider the trace pairing $\Tr(U(E_{ij}+E_{ji}))$
for $i+j>r-1$,
where $E_{ij}$ is the matrix whose $(i,j)$ entry is $1$ and the other entries are zero.
Then $E_{ij}+E_{ji}$ becomes a lift of an element of $\Sym^2\big(\overline{W}\big)$.
So we have
$\Tr(U(E_{ij}+E_{ji}))=z^{m-1}(a_{ij}+a_{ji})=0$.
Since $U$ is symmetric, we have $z^{m-1}a_{ij}=0$ for $i+j\geq r$.
So we have
\[
\sum_{p=1}^{r-1}\sum_{j=1}^p \nu_p(z) N_k^{p-j}\circ\theta_k\circ\xi_k\circ N_k^{j-1}=0
\]
for each $k$.
By the way, $(\xi_k)\in\ker\big(\Sym^2\big(\overline{V}\big)\rightarrow A^1\big)$ implies
$\Tr\big(\theta\circ\xi\circ N^l\big)=0$ for any $0\leq l\leq r-1$.
So there is an endomorphism $f\in \End(E|_{D_{\mathrm{ram}}})$
satisfying $\theta\circ\xi=Nf-fN$.
Moreover, we have $f(V_k)\subset V_k$ for any $k$.
Thus we have
\[
 0=\sum_{p=1}^{r-1}\sum_{j=1}^p \nu_p(z) N_k^{p-j}\circ
 (N_k\circ f_k-f_k\circ N_k)\circ N_k^{j-1}
 =\nu(N_k)f_k-f_k\nu(N_k)
\]
for each $0\leq k\leq r-1$, where $f_k$ is the endomorphism of $\overline{V}_k$
induced by~$f$.
Since the $w\frac{{\rm d}z}{z^m}$-coefficient of $\nu(w)$ does not vanish, we can deduce
$N_k\circ f_k-f_k\circ N_k=0$ from the above equality.
Thus we have $(\xi_k)=0$.
Hence the pairing
(\ref{equation: trace pairing on the symmetric vectors}) is a perfect pairing
of ${\mathcal O}_{D_{\mathrm{ram}}}$-modules, since
$\length \big( \ker\big(\Sym^2\big(\overline{V}\to A^1\big) \big)\big)
=\length \big( \coker \big(A^0\to\Sym^2\big(\overline{W}\big)\big)\big)$.
Note that the map
\[
 \Omega^1_C(D_{\mathrm{ram}})|_{D_{\mathrm{ram}}}
 \longrightarrow
 \mathbf{H}^2({\mathcal O}_C\rightarrow\Omega^1_C(D_{\mathrm{ram}})\to
 \Omega^1_C(D_{\mathrm{ram}})|_{D_{\mathrm{ram}}})\cong\mathbb{C}
\]
is identified with the residue map.
So we can see that the pairing
\[
 \ker\big(\Sym^2\big(\overline{V}\big)\rightarrow A^1\big)\times \coker \big(A^0\rightarrow \Sym^2\big(\overline{W}\big)\big)
 \longrightarrow \mathbb{C}
\]
induced by (\ref{equation: trace pairing on the symmetric vectors}) is a perfect pairing
of vector spaces, which means that
$\eta_2$ is an isomorphism.

Since $\eta_1$, $\eta_3$ and $\eta_2$ are isomorphic,
$\sigma_2\colon \mathbf{H}^0({\mathcal F}_1^{\bullet})
\xrightarrow{\sim} \mathbf{H}^1({\mathcal F}_0^{\bullet})^{\vee}$
is an isomorphism by the five lemma.
Then
$\sigma_3\colon
\mathbf{H}^1({\mathcal F}_0^{\bullet})\xrightarrow{\sim} \mathbf{H}^0({\mathcal F}_1^{\bullet})^{\vee}$
is also isomorphic because it is the dual of $\sigma_2$.

On the other hand,
$\sigma_1\colon \mathbf{H}^0({\mathcal F}_0^{\bullet})
\longrightarrow \mathbf{H}^1({\mathcal F}_1^{\bullet})^{\vee}$
is given by the pairing
\begin{gather*}
 \mathbf{H}^0({\mathcal F}_0^{\bullet}) \times
 \mathbf{H}^1({\mathcal F}_1^{\bullet})
 \longrightarrow
 \mathbf{H}^2\big({\mathcal O}_C\rightarrow
 \Omega^1_C(D_{\mathrm{ram}})\rightarrow\Omega^1_C(D_{\mathrm{ram}})|_{D_{\mathrm{ram}}}\big),
 \\
 \big( [\{u_{\alpha}, (a_{k,\alpha}\!(w)\!)\}] ,
 [\{v_{\alpha\beta}\},\{ (\overline{v}_{k,\alpha},b_{\alpha})\}] \big)\\
 \qquad{}
 \mapsto
 \big[ \{ \Tr(v_{\alpha\beta}\circ u_{\beta}) \} ,
 \big\{ \Tr(\bar{v}_{\alpha}\circ u_{\alpha})
+\tfrac{b_{\alpha}}{2}( \nu'(w) a_{k,\alpha}(w) ) \big\} \big],
\end{gather*}
where $\nu'(w):=\sum_{k=0}^{r-1}k\nu_k(z)w^{k-1}$
and $\overline{v}_{\alpha}\in\End(E|_{D_{\mathrm{ram}}})$
is a lift of $(\overline{v}_{k,\alpha})$ given by
Lemma \ref{lemma: extension of endomorphisms}.
We have the exact commutative diagram
\[
 \begin{CD}
 0=\ker\!\big(A^0\!\rightarrow \Sym^2\!\big(\big(\overline{W}\big)\big)\big) &\rightarrow
 \mathbf{H}^0({\mathcal F}_0^{\bullet}) & \rightarrow & H^0\!\big({\mathcal G}^0\big)
 & \rightarrow & \coker\!\big(A^0\!\rightarrow \Sym^2\!\big(\big(\overline{W}\big)\big)\big) \\
 & @V \sigma_1 VV @V ^t\eta_3 VV
 @V ^t\eta_2 VV \\
 0=\coker\!\big(\!\Sym^2\!\big(\big(\overline{V}\big)\big)\!\rightarrow A^1\big)^{\vee}\! &\rightarrow
 \mathbf{H}^1({\mathcal F}_1^{\bullet})^{\vee} & \rightarrow &
 H^1\!\big(\!\ker\big({\mathcal G}^1\!\rightarrow G^1\big)\big)^{\vee}\! & \rightarrow &
 \ker\!\big(\!\Sym^2\!\big(\big(\overline{V}\big)\big)\!\rightarrow A^1\big)^{\vee}
 \end{CD}
\]
and the five lemma implies that
$\sigma_1\colon \mathbf{H}^0({\mathcal F}^{\bullet}_0)\longrightarrow
\mathbf{H}^1({\mathcal F}^{\bullet}_1)^{\vee}$ is isomorphic because
$^t\eta_3$ and $^t\eta_2$ are isomorphic.

We can see that
$\sigma_4\colon
\mathbf{H}^1({\mathcal F}_1^{\bullet})\longrightarrow\mathbf{H}^0({\mathcal F}_0^{\bullet})^{\vee}$
is also isomorphic since it is the dual of $\sigma_1$.
Since $\sigma_1$, $\sigma_2$, $\sigma_3$, $\sigma_4$ are all isomorphic,
$\sigma\colon\mathbf{H}^1({\mathcal F}^{\bullet})
\longrightarrow\mathbf{H}^1({\mathcal F}^{\bullet})^{\vee}$
is isomorphic by the five lemma.
\end{proof}

We define a complex $\tilde{\Omega}^{\bullet}$ by setting
$\tilde{\Omega}^0={\mathcal O}_C$,
$\tilde{\Omega}^1=\Omega^1_C(D_{\mathrm{ram}})\oplus A^1$,
$\tilde{\Omega}^2=\Omega^1_C(D_{\mathrm{ram}})|_{D_{\mathrm{ram}}}\oplus A^1$
and
\begin{gather*}
 d_{\tilde{\Omega}^{\bullet}}^0
 \colon \
 {\mathcal O}_C \ni f \mapsto (df,0) \in \Omega^1_C(D_{\mathrm{ram}})\oplus A^1,
 \\
 d_{\tilde{\Omega}^{\bullet}}^1
 \colon \
 \Omega^1_C(D_{\mathrm{ram}}) \oplus A^1
 \ni (\eta, b) \mapsto ( (\eta|_{D_{\mathrm{ram}}}-b(\nu'(w))) , b)
 \in \Omega^1_C(D_{\mathrm{ram}})|_{D_{\mathrm{ram}}} \oplus A^1,
\end{gather*}
where the $k$-th component of
$(\nu'(w))\in A^0\otimes\Omega^1_C(D_{\mathrm{ram}})|_{D_{\mathrm{ram}}}$
is given by
$\nu'(w)=\sum_{j=0}^{r-1} j \nu_j(z)w^{j-1}$.
Then we can define a homomorphism of complexes
$\Tr^{\bullet}\colon
{\mathcal F}^{\bullet} \longrightarrow \tilde{\Omega}^{\bullet}$
by
\begin{gather*}
 \Tr^0
 \colon \
 {\mathcal F}^0={\mathcal G}^0 \oplus A^0
 \ni (u,(f_k(w))) \mapsto \Tr(u)\in {\mathcal O}_C,
 \\
 \Tr^1
 \colon \ {\mathcal F}^1={\mathcal G}^1\oplus \Sym^2\big(\overline{W}\big)\oplus \Sym^2\big(\overline{V}\big)
 \ni (v,(\tau_k),(\xi_k))\\
 \hphantom{\Tr^1 \colon \
 {\mathcal F}^1=}{} \mapsto (\Tr(v), (\Theta_{(\tau_k,\xi_k\!)})) \in
 \Omega^1_C(D_{\mathrm{ram}}) \oplus A^1,
 \\
 \Tr^2
 \colon \
 {\mathcal F}^2
 =G^1\oplus A^1
 \ni ((\overline{v}_k), b)) \mapsto ( \Tr(\overline{v}) , b )
 \in \Omega^1_C(D_{\mathrm{ram}})|_{D_{\mathrm{ram}}} \oplus A^1,
\end{gather*}
where $\overline{v}\in \Hom\big(E|_{D_{\mathrm{ram}}},E\otimes\Omega^1_D(D)|_{D_{\mathrm{ram}}}\big)$
is a lift of $(\overline{v}_k)$ given by
Lemma~\ref{lemma: extension of endomorphisms}.

\begin{Lemma}\label{lemma: trace map is isomorphic}
Assume that the endomorphism ring of $E$, preserving
$l$, $\ell$, ${\mathcal V}$ and commuting with~$\nabla$,
consists of the scalar multiplications~$\mathbb{C}\mathrm{id}_E$.
Then the map
\[
 \mathbf{H}^2(\Tr)\colon \
 \mathbf{H}^2({\mathcal F}^{\bullet})
 \longrightarrow \mathbf{H}^2\big(\tilde{\Omega}^{\bullet}\big)
 \cong \mathbf{H}^2(\Omega^{\bullet}_C)\cong\mathbb{C}
\]
is an isomorphism.
\end{Lemma}

\begin{proof}
First note that $\mathbf{H}^0({\mathcal F}^{\bullet})=\mathbb{C}$
because there are only scalar endomorphisms of $E$
commuting with $\nabla$ and preserving the
$(\lambda,\mu,\nu)$-structure.
Under the identification
$\mathbf{H}^0({\mathcal F}^{\bullet})\cong\mathbb{C}
\cong \mathbf{H}^2\big(\tilde{\Omega}^{\bullet}\big)^{\vee}$,
there is an exact commutative diagram
\[
 \begin{CD}
 \mathbf{H}^1({\mathcal F}^{\bullet}_0) @>>>
 \mathbf{H}^1({\mathcal F}^{\bullet}_1) @>>>
 \mathbf{H}^2({\mathcal F}^{\bullet}) & \longrightarrow 0
 \\
 @V \sigma_3 VV @V\sigma_4 VV @V \mathbf{H}^2(\Tr^{\bullet}) VV
 \\
 \mathbf{H}^0({\mathcal F}_1^{\bullet})^{\vee} @>>>
 \mathbf{H}^0({\mathcal F}_0^{\bullet})^{\vee}
 @>>> \mathbf{H}^0({\mathcal F}^{\bullet})^{\vee} & \longrightarrow 0.
 \end{CD}
\]
Since $\sigma_3$ and $\sigma_4$ are isomorphisms,
$\mathbf{H}^2(\Tr^{\bullet})$ is also an isomorphism.
\end{proof}

\begin{Remark} \rm
If $(E,\nabla,l,\ell,{\mathcal F})$
is $\balpha$-stable,
then the assumption of Lemma \ref{lemma: trace map is isomorphic} holds.
\end{Remark}

\begin{Theorem} \label{thm: smoothness of the moduli space}
The moduli space
$M^{\balpha}_{{\mathcal C},{\mathcal D}}(\lambda,\tilde{\mu},\tilde{\nu})$
of connections with $(\lambda,\tilde{\mu},\tilde{\nu})$-structure
is smooth over ${\mathcal T}$.
The dimension of the fiber
$M^{\balpha}_{{\mathcal C},{\mathcal D}}(\lambda,\tilde{\mu},\tilde{\nu})_t$
over $t\in {\mathcal T}$ is
$2r^2(g({\mathcal C}_t)-1)+2+r(r-1)\deg {\mathcal D}_t$
if it is non-empty.
\end{Theorem}

\begin{proof}
For the proof of the smoothness, take an Artinian local ring
$A$ over ${\mathcal T}$ with the maximal ideal $\mathfrak{m}$
and an ideal $I$ of $A$ satisfying $\mathfrak{m}I=0$.
Assume that a flat family
$(E,\nabla,l,\ell,{\mathcal V})$ of connections
on ${\mathcal C}\otimes A/I$ is given.
Consider the complex
${\mathcal F}^{\bullet}$ determined from
$(E,\nabla,l,\ell,{\mathcal V})\otimes A/\mathfrak{m}$
by (\ref{equation: definition of tangent complex}).
We take an affine open covering
$\{U_{\alpha}\}$ of ${\mathcal C}\otimes A$
as in the proof of Proposition \ref{prop: tangent space of the moduli space}.
If $U_{\alpha}\cap ({\mathcal D}_{\mathrm{ram}})_A=\varnothing$,
we can easily take a lift
$(E_{\alpha},\nabla_{\alpha},\{l_{\alpha},\ell_{\alpha},{\mathcal V}_{\alpha}\})$
of $(E,\nabla,\{l,\ell,{\mathcal V}\})|_{U_{\alpha}\otimes A/I}$.
If $U_{\alpha}\cap ({\mathcal D}_{\mathrm{ram}})_A\neq\varnothing$,
then we may assume that
${\mathcal V}\cap U_{\alpha}$ is given by a factorized $\tilde{\nu}$-ramified structure
$(V_k,\vartheta_k,\varkappa_k)$.
As in the proof of Proposition~\ref{prop:correspondence-factorized},
we can choose an endomorphism
$N$ on $E|_{({\mathcal D}_{\mathrm{ram}})_{A/I}}$
inducing $\theta_k\circ\kappa_k$ on $\overline{V}_k$
for $0\leq k\leq r-1$.
The representation matrix of $N$ is given~by
\[
 \begin{pmatrix}
 0 & 0 & \cdots & 0 & z \\
 1 & 0 & \cdots & 0 & 0 \\
 \vdots & \ddots & \ddots & \vdots & \vdots \\
 0 & \cdots & 1 & 0 & 0 \\
 0 & \cdots & 0 & 1 & 0
 \end{pmatrix}.
\]
with respect to the basis $e_0,\dots,e_{r-1}$
chosen as in the proof of Proposition \ref{prop:correspondence-factorized}.
Then we can give a factorization
$N=\theta\circ\kappa$
by the matrix factorization
\[
\begin{pmatrix}
 0 & 0 & \cdots & 0 & z \\
 1 & 0 & \cdots & 0 & 0 \\
 \vdots & \ddots & \ddots & \vdots & \vdots \\
 0 & \cdots & 1 & 0 & 0 \\
 0 & \cdots & 0 & 1 & 0
 \end{pmatrix}
 =
 \begin{pmatrix}
 0 & 0 & \cdots & 0 & 1 \\
 0 & 0 & \cdots & 1 & 0 \\
 \vdots & \vdots & \iddots & \iddots & \vdots \\
 0 & 1 & 0 & \cdots & 0 \\
 1 & 0 & 0 & \cdots & 0
 \end{pmatrix}
 \begin{pmatrix}
 0 & 0 & \cdots & 0 & 1 & 0 \\
 0 & 0 & \cdots & 1 & 0 & 0 \\
 \vdots \vspace{2pt}& \vdots & \iddots & \iddots & \vdots & \vdots \\
 0 & 1 & 0 & \cdots & 0 & 0 \\
 1 & 0 & 0 & \cdots & 0 & 0 \\
 0 & 0 & 0 & \cdots & 0 & z
 \end{pmatrix}
\]
with respect to the basis
$e_0,\dots,e_{r-1}$ of $E|_{({\mathcal D}_{\mathrm{ram}})_{A/I}}$
and its dual basis
$e_0^*,\dots,e_{r-1}^*$.
Let $E_{\alpha}$ be a free ${\mathcal O}_{U_{\alpha}}$-module with
$E_{\alpha}\otimes A/I=E|_{U_{\alpha}\otimes A/I}$.
Define
$\tilde{N}\colon E_{\alpha}|_{({\mathcal D}_{\mathrm{ram}})_A\cap U_{\alpha}}
\longrightarrow E_{\alpha}|_{({\mathcal D}_{\mathrm{ram}})_A\cap U_{\alpha}}$,
$\tilde{\theta}\colon E_{\alpha}|_{({\mathcal D}_{\mathrm{ram}})_A\cap U_{\alpha}}^{\vee}
\longrightarrow E_{\alpha}|_{({\mathcal D}_{\mathrm{ram}})_A\cap U_{\alpha}}$
and
$\tilde{\kappa}\colon E_{\alpha}|_{({\mathcal D}_{\mathrm{ram}})_A\cap U_{\alpha}}
\longrightarrow E_{\alpha}|_{({\mathcal D}_{\mathrm{ram}})_A\cap U_{\alpha}}^{\vee}$
by the same representation matrices as $N$, $\theta$ and $\kappa$ respectively.
Then $\tilde{N}$, $\tilde{\theta}$ and $\tilde{\kappa}$ are lifts of $N$, $\theta$ and $\kappa$
and they induce a lift
${\mathcal V}_{\alpha}=(\tilde{V}_k,\tilde{\vartheta}_k.\tilde{\varkappa}_k)$
of $(V_k,\vartheta_k,\varkappa_k)$ over $A$.
We can easily take a relative connection
$\nabla_{\alpha}$ on $E_{\alpha}$ which is a lift of
$\nabla|_{U_{\alpha}}$ and which is compatible with
${\mathcal V}_{\alpha}$.
So we obtain a lift
$(E_{\alpha},\nabla_{\alpha},\{l_{\alpha},\ell_{\alpha},{\mathcal V}_{\alpha}\})$
of $(E,\nabla,\{l,\ell,{\mathcal V}\})|_{U_{\alpha}\otimes A/I}$
when $U_{\alpha}\cap ({\mathcal D}_{\mathrm{ram}})_A\neq\varnothing$.

Take an isomorphism
$\theta_{\beta\alpha}\colon E_{\alpha}|_{U_{\alpha\beta}}
\stackrel{\sim}\longrightarrow E_{\beta}|_{U_{\alpha\beta}}$,
where $U_{\alpha\beta}=U_{\alpha}\cap U_{\beta}$.
If we put
\begin{gather*}
 u_{\alpha\beta\gamma}
 =
 \theta_{\gamma\alpha}^{-1}\circ\theta_{\gamma\beta}\circ\theta_{\beta\alpha}-\mathrm{id},
 \\
 v_{\alpha\beta}
 =
 (\theta_{\beta\alpha}\otimes\mathrm{id})^{-1}\circ\nabla_{\beta}\circ\theta_{\beta\alpha}
 -\nabla_{\alpha},
\end{gather*}
then the class
$[\{u_{\alpha\beta\gamma}\},\{v_{\alpha\beta}\}]
\in \mathbf{H}^2({\mathcal F}^{\bullet})\otimes I$
is nothing but the obstruction for the lifting of
$(E,\nabla,\{l,\ell,{\mathcal V}\})$
to a flat family of connections on ${\mathcal C}\otimes A$ over $A$.
We can see that the image
$\mathbf{H}^2(\Tr^{\bullet})([\{u_{\alpha\beta\gamma}\},\{v_{\alpha\beta}\}])$
under the isomorphism
$\mathbf{H}^2(\Tr^{\bullet})\colon
\mathbf{H}^2({\mathcal F}^{\bullet})\stackrel{\sim}\longrightarrow
\mathbf{H}^2(\Omega^{\bullet}_{{\mathcal C}\otimes A/\mathfrak{m}})$
is nothing but the obstruction for the lifting of
the determinant line bundle $\det (E,\nabla)$
with the induced connection.
Consider the moduli space
$M(\sum\lambda_k,\sum\mu_k,(r-1)dz/2+r\nu_0)$
of pairs $(L,\nabla_L)$
of a line bundle $L$ on the fibers of ${\mathcal C}$ over ${\mathcal T}$ and a connection
$\nabla_L$ on $L$ admitting poles along ${\mathcal D}$
whose residue along ${\mathcal D}_{\mathrm{log}}$ is
$\sum_{1\leq k \leq r}\lambda_k$,
whose restriction to ${\mathcal D}_{\mathrm{un}}$ is
$\sum_{1\leq k\leq r}\mu_k$
and whose restriction to ${\mathcal D}_{\mathrm{ram}}$ is
$\sum((r-1)dz/2+ r\nu_0)$.
Then $M(\sum\lambda_k,\sum\mu_k,(r-1)dz/2+r\nu_0)$ is smooth over ${\mathcal T}$,
since it is an affine space bundle over the relative Jacobian of
${\mathcal C}$ over ${\mathcal T}$.
In particular, we have
$\mathbf{H}^2(\Tr^{\bullet})([\{u_{\alpha\beta\gamma}\},\{v_{\alpha\beta}\}])=0$
which is equivalent to
$[\{u_{\alpha\beta\gamma}\},\{v_{\alpha\beta}\}]=0$.
Thus $M^{\balpha}_{{\mathcal C},{\mathcal D}}(\lambda,\tilde{\mu},\tilde{\nu})$
is smooth over ${\mathcal T}$.

By Proposition \ref{prop: tangent space of the moduli space},
the dimension of the moduli space at
$(E,\nabla,l,\ell,{\mathcal V})\otimes A/\mathfrak{m}$
is given by
$\dim \mathbf{H}^1({\mathcal F}^{\bullet})$.
We write ${\mathcal D}\otimes A/\mathfrak{m}=D$,
${\mathcal D}_{\mathrm{log}}\otimes A/\mathfrak{m}=D_{\mathrm{log}}$ and so on.
Using the exact sequence~(\ref{equation: hyper cohomology exact sequence})
and the equality
$\dim\mathbf{H}^0({\mathcal F}^{\bullet})=
\dim\mathbf{H}^2({\mathcal F}^{\bullet})=1$
by Lemma \ref{lemma: trace map is isomorphic}, we have
\begin{gather}
 \dim \mathbf{H}^1({\mathcal F}^{\bullet}) =
 \chi({\mathcal F}_1^{\bullet})-\chi({\mathcal F}_0^{\bullet})+2\nonumber \\
 \hphantom{\dim \mathbf{H}^1({\mathcal F}^{\bullet})}{}
 =\chi\big({\mathcal G}^1\big)-\dim_{\mathbb{C}} G^1
 +\dim_{\mathbb{C}} \Sym^2\big(\overline{V}\big) -\dim_{\mathbb{C}} A^1\nonumber \\
\hphantom{\dim \mathbf{H}^1({\mathcal F}^{\bullet})=}{}
 -\chi\big({\mathcal G}^0\big)-\dim_{\mathbb{C}} A^0
 +\dim_{\mathbb{C}} \Sym^2\big(\overline{W}\big) +2.\label{equation: formula of H^1-cohomology}
\end{gather}
Since
$\ker\big({\mathcal G}^1\rightarrow G^1\big)\cong \big({\mathcal G}^0\big)^{\vee}\otimes\Omega^1_C$,
we have
\begin{gather}
 \chi\big({\mathcal G}^1\big)-\dim_{\mathbb{C}}G^1 =
 -\chi\big({\mathcal G}^0\big)\nonumber \\
\hphantom{\chi\big({\mathcal G}^1\big)-\dim_{\mathbb{C}}G^1}{}=
 r^2(g-1)+(\deg D_{\mathrm{log}}+\deg D_{\mathrm{un}})r(r-1)/2
 +\sum_{x\in D_{\mathrm{ram}}}r(r-1)/2.\label{equation: chi of G^0, G^1}
\end{gather}
By the same method as in the proof of Lemma \ref{lemma: extension of symmetric pairings},
we can see that the elements of $\Sym^2\big(\overline{V}\big)$ are given by the data
\begin{gather*}
\begin{split}
& (a_{r-k-1,k}(z))_{0\leq k\leq r-1}
 \in \big(\mathbb{C}[z]/(z^{m_x})\big)^r \quad
 \text{such that $za_{r-k-1,k}=za_{k,r-k-1}$},
 \\
 & (\bar{a}_{ij}(z))_{0\leq i,j\leq r-1, i+j\neq r-1}
 \in \big(\mathbb{C}[z]/\big(z^{m_x-1}\big) \big)^{r^2-r}
 \quad \text{such that $\bar{a}_{ji}=\bar{a}_{ij}$} \quad (x\in D_{\mathrm{ram}})
\end{split}
\end{gather*}
and each $\xi_k\in\Hom\big(\overline{V}_k,\overline{W}_k\big)|_{m_xx}$ is given by the matrix
\[
 \begin{pmatrix}
 \overline{a}_{00}(z) & & \cdots & & za_{0,r-1}(z) \\
 \vdots & & & \rotatebox{80}{$\ddots$} \ & \vdots \\
 \overline{a}_{r-k-1,0}(z) & \cdots & a_{r-k-1,k}(z) & \cdots & za_{k,r-1}(z) \\
 \vdots & \rotatebox{80}{$\ddots$} & & & \vdots \\
 \overline{a}_{r-1,0}(z) & & \cdots & & za_{r-1,r-1}(z)
 \end{pmatrix},
\]
where $za_{i,j}$ is the image of $z\otimes \overline{a}_{i,j}$ by
$(z)\otimes{\mathcal O}_{(m_x-1)x}\xrightarrow{\sim}
z{\mathcal O}_{m_xx}$.
So we can see that
\begin{equation} \label{equation: tangent of factorized data}
 \dim_{\mathbb{C}} \Sym^2\big(\overline{V}\big)
 =\dim_{\mathbb{C}} \Sym^2\big(\overline{W}\big)
 =\sum_{x\in D_{\mathrm{ram}}} \left( r+\frac{1}{2}(m_x-1)r(r+1) \right).
\end{equation}
Finally note that
\begin{equation} \label{equation: tangent of automorphism of factorized structure}
 \dim_{\mathbb{C}}A^0=\dim_{\mathbb{C}}A^1=\sum_{x\in D_{\mathrm{ram}}} m_x r.
\end{equation}
Substituting (\ref{equation: chi of G^0, G^1}), (\ref{equation: tangent of factorized data})
and (\ref{equation: tangent of automorphism of factorized structure})
to (\ref{equation: formula of H^1-cohomology}),
we get the desired equality
$\dim \mathbf{H}^1({\mathcal F}^{\bullet})
=2r^2(g-1)+2+r(r-1)\deg D$.
\end{proof}

\section {Symplectic structure on the moduli space}
\label{section: symplectic structure}

In this section, we assume again the same notations as in
Section \ref{section: construction of the moduli space},
Section \ref{section: tangent space}
and
Section \ref{section: smoothness of the moduli}.

There is an \'etale surjective morphism
$M'\longrightarrow M^{\balpha}_{{\mathcal C},{\mathcal D}}\big(\lambda,\tilde{\mu},\tilde{\nu}\big)$,
such that there is a universal family of connections
$\big(\tilde{E},\tilde{\nabla},\{\tilde{l},\tilde{\ell},\tilde{\mathcal V}\}\big)$
on ${\mathcal C}_{M'}$ over $M'$.
We can define
a complex
${\mathcal G}^{\bullet}_{M'}$ on ${\mathcal C}_{M'}$
from
$(\tilde{E},\tilde{\nabla},\{\tilde{l},\tilde{\ell},\tilde{\mathcal V}\})$
in the same way as ${\mathcal G}^{\bullet}$
given by (\ref{equation: definition of G^0,G^1}), (\ref{equation: homomorphism of complex G}).
We can also define a complex
${\mathcal S}^{\bullet}_{\mathrm{ram},M'}$
on ${\mathcal C}_{M'}$
in the same way as
${\mathcal S}^{\bullet}_{\mathrm{ram}}$
given by (\ref{equation: definition of S^0,S^1,S^2}),
(\ref{equation: definition of complex S^dot}).
Then we can define a complex
\[
 \tilde{\mathcal F}^{\bullet}_{M'}
 :=
 \mathrm{Cone}\big( {\mathcal G}^{\bullet}_{M'} \rightarrow
 {\mathcal S}^{\bullet}_{\mathrm{ram},M'}[1] \big) [-1]
\]
in the same way as
${\mathcal F}^{\bullet}$ defined in (\ref{equation: definition of tangent complex}).

Let $p_{M'}\colon {\mathcal C}_{M'}\longrightarrow M'$ be the projection.
Then we can see by Proposition~\ref{prop: tangent space of the moduli space}
that the relative tangent bundle
$T_{M'/{\mathcal T}}$ of $M'$ over ${\mathcal T}$ is isomorphic to
$\mathbf{R}^1p_{M' *}\big(\tilde{\mathcal F}^{\bullet}_{M'}\big)$.
We can define a pairing
$\Xi_{\mathrm{ram}}\colon
{\mathcal S}^1_{\mathrm{ram},M'}\times {\mathcal S}^1_{\mathrm{ram},M'}
\longrightarrow
\Omega^1_{{\mathcal C}_{M'}/M'}({\mathcal D}_{M'})|_{({\mathcal D}_{\mathrm{ram}})_{M'}}$
in the same way as
(\ref{equation: definition of Xi}).
Consider the pairing
\begin{gather}
 \omega_{M'} \colon \
 \mathbf{R}^1 p_{M' *}\big(\tilde{\mathcal F}^{\bullet}_{M'}\big) \times
 \mathbf{R}^1p_{M' *}\big(\tilde{\mathcal F}^{\bullet}_{M'}\big)\longrightarrow \nonumber\\
 \hphantom{\omega_{M'} \colon} \
 \mathbf{R}^2 p_{M' *}\big[ {\mathcal O}_{\mathcal C}\to
 \Omega^1_{{\mathcal C}/{\mathcal T}} ({\mathcal D}_{\mathrm{ram}})
 \to \Omega^1_{{\mathcal C}/{\mathcal T}}({\mathcal D}_{\mathrm{ram}})|
 _{ {\mathcal D}_{\mathrm{ram}} } \big]_{M'}  \cong
 \mathbf{R}^2 p_{M' *} \Omega^{\bullet}_{{\mathcal C}_{M'}/M'}
 \cong {\mathcal O}_{M'}\label{equation: definition of symplectic form over etale cover}
\end{gather}
defined by
\begin{gather*}
 \omega_{M'}\big( \big[ \{ u_{\alpha\beta}\},
 \{ v_{\alpha}, \eta_{\alpha} \} \big] ,
 \big[ \{ u'_{\alpha\beta}\},
 \{ v'_{\alpha}, \eta'_{\alpha} \} \big]
\big) \\
 \qquad{}=
\big[
 \{ \Tr(u_{\alpha\beta}\circ u'_{\beta\gamma})\},
 \{ {-}\Tr(u_{\alpha\beta}\circ v'_{\beta}-v_{\alpha}\circ u'_{\alpha\beta})\},
 \{\Xi_{\mathrm{ram}}( \eta_{\alpha} , \eta'_{\alpha} )
 \}
\big]
\end{gather*}
in the same way as (\ref{expression of symplectic form}).
We can check $\omega_{M'}(v,v)=0$ for
$v\in \mathbf{R}^1p_{M' *}(\tilde{\mathcal F}^{\bullet}_{M'})$
and $\omega_{M'}$ descends to a ${\mathcal T}$-relative $2$-form
$\omega_{M^{\balpha}_{{\mathcal C},{\mathcal D}}(\lambda,\tilde{\mu},\tilde{\nu})}$ on
$M^{\balpha}_{{\mathcal C},{\mathcal D}}(\lambda,\tilde{\mu},\tilde{\nu})$.

\begin{Theorem}
\label{theorem: existence of symplectic form and d-closedness}
The $2$-form
$\omega_{M^{\balpha}_{{\mathcal C},{\mathcal D}}(\lambda,\tilde{\mu},\tilde{\nu})}$
defined by
\eqref{equation: definition of symplectic form over etale cover}
is a ${\mathcal T}$-relative symplectic form
on the moduli space
$M^{\balpha}_{{\mathcal C},{\mathcal D}}(\lambda,\tilde{\mu},\tilde{\nu})$
of $\balpha$-stable connections
on $({\mathcal C},{\mathcal D})$ with $(\lambda,\tilde{\mu},\tilde{\nu})$-structure.
\end{Theorem}

The restriction
$\omega_{M^{\balpha}_{{\mathcal C},{\mathcal D}}(\lambda,\tilde{\mu},\tilde{\nu})}|_p$
at each point
$p\in M^{\balpha}_{{\mathcal C},{\mathcal D}}(\lambda,\tilde{\mu},\tilde{\nu})$
is nondegenerate by
Proposition~\ref{proposition: nondegenerate pairing on tangent space}.
It remains to prove that
$d\omega_{M^{\balpha}_{{\mathcal C},{\mathcal D}}(\lambda,\tilde{\mu},\tilde{\nu})}=0$.
Since
$M^{\balpha}_{{\mathcal C},{\mathcal D}}(\lambda,\tilde{\mu},\tilde{\nu})$
is smooth over ${\mathcal T}$,
we only have to show the vanishing
$d\omega_{M^{\balpha}_{{\mathcal C},{\mathcal D}}(\lambda,\tilde{\mu},\tilde{\nu})_t}=0$
of the restriction to the fiber
$M^{\balpha}_{{\mathcal C},{\mathcal D}}(\lambda,\tilde{\mu},\tilde{\nu})_t$
over $t\in{\mathcal T}$.
For its proof we use a construction of an unfolding of
the moduli space.

Put ${\mathcal C}_t=C$,
${\mathcal D}_t=D$, $({\mathcal D}_{\mathrm{un}})_t=D_{\mathrm{un}}$,
$({\mathcal D}_{\mathrm{ram}})_t=D_{\mathrm{ram}}$ and
$(\lambda,\mu,\nu)=(\lambda,\tilde{\mu},\tilde{\nu})_t$.
For each ${x\in D}$, choose a~defining equation $z$ of
$D_{\mathrm{red}}$ on an affine open neighborhood of
$x$, which is a lift of $\overline{z}$.
Take distinct complex numbers
$s^x_1,\dots,s^x_{m_x-1},s^x_{m_x}\in\mathbb{C}$.
Let $D^x_{\mathrm{un},h}$
be the divisor on $C\times \Spec\mathbb{C}[h]$
defined by the equation
$(z-hs^x_1)\cdots(z-hs^x_{m_x})=0$
and put $D_{\mathrm{un},h}=\sum_{x\in D_{\mathrm{un}}} D^x_{\mathrm{un},h}$.
For each $x\in D_{\mathrm{ram}}$, take distinct complex numbers
$q^x_1,\dots,q^x_{m_x-1},q^x_{m_x}\in\mathbb{C}$
with $q^x_{m_x}=1$.
Let $D^x_{\mathrm{ram},h}$ be the divisor on $C\times\Spec\mathbb{C}[h]$
defined by the equation
$(z-h^rq^x_1)\cdots(z-h^rq^x_{m_x-1})(z-h^r)=0$
and put
$D_{\mathrm{ram},h}:=\sum_{x\in D_{\mathrm{ram}}} D^x_{\mathrm{ram},h}$.
We set
\[
 D_h:=D_{\mathrm{log}}+D_{\mathrm{un},h}+D_{\mathrm{ram},h}.
\]
Note that $D_h$ is a reduced divisor for generic $h$ and
it coincides with $D$ if $h=0$.
So we can take a Zariski open subset
$H^{\circ}$ of $\Spec\mathbb{C}[h]$ containing $0$ such that
$D_h$ is a reduced divisor for any
$h\in H^{\circ}\setminus\{0\}$.

For $x\in D_{\mathrm{un}}$, we can write
\[
 \mu_k|_{m_xx}
 =
 \big(b_{k,0}+b_{k,1}z+\cdots+b_{k,m_x-1}z^{m_x-1}\big)\frac{{\rm d}z}{z^{m_x}},
 \qquad k=0,\dots,r-1.
\]
We define
$\mu_{k,h}\in\Omega^1_{C\times\Spec\mathbb{C}[h]/\Spec\mathbb{C}[h]}
(D_{\mathrm{un},h})|_{D_{\mathrm{un},h}}$ by
\[
 \mu_{k,h}|_{D_{\mathrm{un},h}}
 =
 \frac{b_{k,0}+b_{k,1}z+\cdots+b_{k,m_x-1}z^{m_x-1}} {(z-hs^x_1)\cdots(z-hs^x_{m_x})} {\rm d}z,
 \qquad k=0,\dots,r-1.
\]
We can write
\begin{gather*}
 \nu^x_0(z)
 =
 \big(a^x_{0,0}+a^x_{0,1}z+\cdots+a^x_{0,m_x-2}z^{m_x-2}+a^x_{0,m_x-1}z^{m_x-1}\big)\frac{{\rm d}z}{z^{m_x}},
 \\
 \nu^x_k(z)
 =
 \big(a^x_{k,0}+a^x_{k,1}z+\cdots+a^x_{k,m_x-2}z^{m_x-2}\big)\frac{{\rm d}z}{z^{m_x}},
 \qquad k=1,\dots,r-1 .
\end{gather*}
Then we define
$\nu_{k,h}(z)\in\Omega^1_{C\times\Spec\mathbb{C}[h]/\Spec\mathbb{C}[h]}
(D_{\mathrm{ram},h}) |_{D_{\mathrm{ram},h}}$
for $0\leq k\leq r-1$ by
\begin{gather*}
 \nu_{0,h}(z) |_{D_{\mathrm{ram},h}}
 =
 \frac{a^x_{0,0}+a^x_{0,1}z+\cdots+a^x_{0,m_x-2}z^{m_x-2}+a^x_{0,m_x-1}z^{m_x-1}}
 {(z-h^rq^x_1)\cdots(z-h^rq^x_{m_x-1})(z-h^r)} {\rm d}z,
 \\
 \nu_{k,h}(z) |_{D_{\mathrm{ram},h}}
 =
 \frac{a^x_{k,0}+a^x_{k,1}z+\cdots+a^x_{k,m_x-2}z^{m_x-2}}
 {(z-h^rq^x_1)\cdots(z-h^rq^x_{m_x-1})(z-h^r)} {\rm d}z,
 \qquad k=1,\dots,r-1,
\end{gather*}
and we set
\[
 \nu_h(w)
 :=
 \nu_{0,h}(z)+\nu_{1,h}(z)w+\cdots+\nu_{r-1,h}(z)w^{r-1}.
\]
Consider the moduli space
\[
 {\mathcal M}_{H^{\circ}}
 =\left\{
 (E,\nabla, l,(\ell_k)_{0\leq k\leq r-1},(V_k,\vartheta_k,\varkappa_k)_{0\leq k\leq r-1})
 \right\}
 \longrightarrow H^{\circ},
\]
where
\begin{itemize}\itemsep=0pt
\item[(i)] $E$ is an algebraic vector bundle on $C$ of rank $r$ and degree $d$,
\item[(ii)] $\nabla\colon E\longrightarrow E\otimes\Omega^1_C(D_h)$
is a connection admitting poles along $D_h$,
\item[(iii)] $l$ is a logarithmic $\lambda$-parabolic structure on $(E,\nabla)$
along $D_{\mathrm{log}}$,
\item[(iv)] $E|_{D_{\mathrm{un},h}}=\ell_0\supset\cdots\supset \ell_{r-1}\supset\ell_r=0$
is a filtration
such that
$\ell_k/\ell_{k+1}\cong{\mathcal O}_{D_{\mathrm{un},h}}$ for any $k$
and that
$(\nabla|_{D_{\mathrm{un},h}}-\mu_{k,h}\mathrm{id})(\ell_k)\subset
\ell_{k+1}\otimes\Omega^1_C(D_{\mathrm{un},h})$
for any $k$,
\item[(v)]
$E|_{D_{\mathrm{ram},h}}=V_0\supset V_1\supset\cdots\supset V_{r-1}\supset
V_r= (z-h^r)V_0$
is a filtration by ${\mathcal O}_{D_{\mathrm{ram},h}}$-submodules
such that $V_j/V_{j+1}\cong{\mathcal O}_{D_{\mathrm{ram},h}}/(z-h^r)$ and
$\nabla|_{D^x_{\mathrm{ram},h}}(V_k)\subset
V_k\otimes\Omega^1_C(D_{\mathrm{ram},h})$
for $0\leq k\leq r-1$,
\item[(vi)]
for
$\overline{V}^x_k:V_k|_{D^x_{\mathrm{ram},h}}
\big/\prod_{j=1}^{m_x-1}(z-h^rq^x_j)V_{k+1}|_{D^x_{\mathrm{ram},h}}$
and
$\overline{W}^x_k=\Hom_{{\mathcal O}_{D^x_{\mathrm{ram},h}}}
(\overline{V}^x_{r-k-1},{\mathcal O}_{D^x_{\mathrm{ram},h}})$,
\[
 \vartheta^x_k\colon \ \overline{W}^x_k\times\overline{W}^x_{r-k-1}
 \longrightarrow {\mathcal O}_{D^x_{\mathrm{ram}}},
 \qquad 0\leq k\leq r-1,
\]
are ${\mathcal O}_{D^x_{\mathrm{ram},h}}$-bilinear pairings
such that the homomorphisms
$\theta^x_k\colon \overline{W}^x_k\longrightarrow
\big(\overline{W}^x_{r-k-1}\big)^{\vee}=\overline{V}^x_k$
induced by $\vartheta^x_k$ are isomorphisms,
the equalities $\vartheta^x_k(v,v')=\vartheta^x_{r-k-1}(v',v)$
hold for $v\in \overline{W}^x_k$, $v'\in\overline{W}^x_{r-k-1}$
and that
the equalities
$\vartheta^x_{k-1}\big( v_1|_{\overline{V}^x_{r-k}},v_2 \big)
=\vartheta_k \big( v_1,v_2|_{\overline{V}^x_k} \big)$
hold for
$v_1\in \overline{W}^x_k=\Hom\big(\overline{V}^x_{r-k-1},{\mathcal O}_{D^x_{\mathrm{ram},h}}\big)$,
$v_2\in \overline{W}^x_{r-k}=\Hom\big(\overline{V}^x_{k-1},{\mathcal O}_{D^x_{\mathrm{ram},h}}\big)$
when $1\leq k\leq r-1$
and the equality
$\vartheta_{r-1}((z-h^r)v_1,v_2)=\vartheta_0(v_1,(z-h^r)v_2)$
holds for $v_1,v_2\in \overline{W}^x_0$,
\item[(vii)]
$
 \varkappa^x_k\colon \overline{V}^x_k\times\overline{V}^x_{r-k-1}
 \longrightarrow {\mathcal O}_{D^x_{\mathrm{ram},h}}
$
are ${\mathcal O}_{D^x_{\mathrm{ram},h}}$-bilinear pairings
for $0\leq k\leq r-1$ such that
the equalities $\varkappa^x_k(v,v')=\varkappa^x_{r-k-1}(v',v)$
hold for $v\in\overline{V}^x_k$, $v'\in\overline{V}^x_{r-k-1}$,
the equalities $\varkappa^x_{k-1}(\overline{v_1},v_2)=\varkappa^x_k(v_1,\overline{v_2})$
hold for $v_1\in \overline{V}_k$, $v_2\in \overline{V}_{r-k}$
and for the image
$\overline{v_1}$ (resp.\ $\overline{v_2}$)
of $v_1$ (resp.\ $v_2$)
via the canonical map
$\overline{V}_k^x \rightarrow \overline{V}_{k-1}^x$
(resp.\ $\overline{V}_{r-k}^x \rightarrow \overline{V}_{r-k-1}^x$),
the equality
$\varkappa_{r-1}((z-h^r)v_1,v_2)=\varkappa_0(v_1,(z-h^r)v_2)$
holds for $v_1,v_2\in \overline{V}^x_0$
and that the equalities
$(\theta^x_k\circ\kappa^x_k)^r=(z-h^r)\cdot\mathrm{id}_{\overline{V}^x_k}$
hold for the homomorphisms
$\kappa^x_k\colon \overline{V}^x_k\longrightarrow
\big(\overline{V}_{r-k-1}^x\big)^{\vee}=\overline{W}^x_k$
induced by $\varkappa^x_k$,
\item[(viii)]
the homomorphism
\begin{align*}
 {\mathcal O}_{D^x_{\mathrm{ram},h}}[w]\big/\big(w^r-z+h^r, (z-h^rq^x_1)\cdots(z-h^rq^x_{m_x-1})w\big)
 & \longrightarrow
 \End_{{\mathcal O}_{D^x_{\mathrm{ram},h}}}\big(\overline{V}^x_k\big),
 \\
 \overline{f(w)}
 & \mapsto f\big(\theta^x_k\circ\kappa^x_k\big)
\end{align*}
is injective
and the diagrams
\[
 \begin{CD}
 V_k|_{D^x_{\mathrm{ram},h}} @> \nabla|_{D^x_{\mathrm{ram},h}} >>
 V_k|_{D^x_{\mathrm{ram},h}}\otimes\Omega^1_C(D_{\mathrm{ram},h}) \\
 @VVV @VVV \\
 \overline{V}^x_k @> \nu_h(\theta^x_k\circ\kappa^x_k)+\frac{k}{r}\frac{dz}{z-h^r} >>
 \overline{V}^x_k\otimes\Omega^1_C(D_{\mathrm{ram},h})
 \end{CD}
\]
are commutative for $k=0,1,\dots,r-1$,
\item[(ix)]
there is an isomorphism
$\psi_k\colon \overline{V}^x_k\xrightarrow {\sim}
(w)\big/\big(w^2(z-h^rq^x_1)\cdots(z-h^rq^x_{m_x-1})\big) \otimes \overline{V}^x_{k-1}$
which is a lift of $\overline{V}^x_k\longrightarrow \overline{V}^x_{k-1}$
such that the composition
\begin{gather*}
 (z-h^r)\big/\big(w(z-h^rq_1^x)\cdots(z-h^rq^x_{m_x-1})(z-h^r)\big) \otimes \overline{V}^x_0
 \longrightarrow \overline{V}^x_{r-1},
 \\
 \quad
 \xrightarrow[\sim]{\psi_{r-1}}\cdots\cdots\xrightarrow[\sim]{\psi_1}
 \big(w^{r-1}\big)\big/\big((z-h^rq^x_1)\cdots(z-h^rq^x_{m_x-1})(z-h^r)\big)\otimes \overline{V}^x_0
\end{gather*}
coincides with the homomorphism obtained by tensoring
$\overline{V}^x_0$ to
\begin{gather*}
 (w^r)\big/\big(w(z-h^rq_1^x)\cdots(z-h^rq^x_{m_x-1})(z-h^r)\big)\\
 \qquad{} \rightarrow
 (w^{r-1})\big/\big((z-h^rq^x_1)\cdots(z-h^rq^x_{m_x-1})(z-h^r)\big)
\end{gather*}
for $1\leq k\leq r-1$ and
\item[(x)]
the ring of endomorphisms of $E$ preserving $l$, $(\ell_k)$, $(V_k,\vartheta_k,\varkappa_k)$
and commuting with $\nabla$ consists of scalar endomorphisms
$\mathbb{C}\mathrm{id}_E$.
\end{itemize}

We can prove that the moduli space ${\mathcal M}_{H^{\circ}}$
exists as an algebraic space,
by modifying the proof of Theorem \ref{theorem: existence of the moduli space}.
The proof is rather easier
because we do not need a GIT construction.
So we omit the proof of the following proposition.

\begin{Proposition}
There exists a relative moduli space
${\mathcal M}_{H^{\circ}}\longrightarrow H^{\circ}$
as an algebraic space.
\end{Proposition}

Note that the fiber
${\mathcal M}_{H^{\circ},0}$
of the moduli space ${\mathcal M}_{H^{\circ}}$ over $h=0$
is the moduli space of simple connections
on $(C,D)$ with $(\lambda,\mu,\nu)$-structure.

There is a scheme $\tilde{\mathcal M}_{H^{\circ}}$
of finite type over $H^{\circ}$
with an \'etale surjective morphism
$\tilde{\mathcal M}_{H^{\circ}}\longrightarrow {\mathcal M}_{H^{\circ}}$
such that a universal family
$\big(\tilde{E}_{\tilde{\mathcal M}_{H^{\circ}}},\tilde{\nabla}_{\tilde{\mathcal M}_{H^{\circ}}},
\tilde{l}_{\tilde{\mathcal M}_{H^{\circ}}},\tilde{\ell}_{\tilde{\mathcal M}_{H^{\circ}}},
\tilde{\mathcal V}_{\tilde{\mathcal M}_{H^{\circ}}}\big)$
exists over $\tilde{\mathcal M}_{H^{\circ}}$.
We can define a complex
\begin{gather*}
 {\mathcal F}_{\tilde{\mathcal M}_{H^{\circ}}}^{\bullet}
 =
\big[
 {\mathcal G}^0_{\tilde{\mathcal M}_{H^{\circ}}}\oplus A^0_{\tilde{\mathcal M}_{H^{\circ}}}
 \rightarrow
 {\mathcal G}^1_{\tilde{\mathcal M}_{H^{\circ}}}\oplus
 \Sym^2\big(\big(\overline{W}_{\tilde{\mathcal M}_{H^{\circ}}}\big)\big)
 \oplus \Sym^2\big(\big(\overline{V}_{\tilde{\mathcal M}_{H^{\circ}}}\big)\big)\\
 \hphantom{{\mathcal F}_{\tilde{\mathcal M}_{H^{\circ}}}^{\bullet} =\big[}{}
 \rightarrow G^1_{\tilde{\mathcal M}_{H^{\circ}}} \oplus A^1_{\tilde{\mathcal M}_{H^{\circ}}}
\big]
\end{gather*}
from
$\big(\tilde{E}_{\tilde{\mathcal M}_{H^{\circ}}},\tilde{\nabla}_{\tilde{\mathcal M}_{H^{\circ}}},
\tilde{l}_{\tilde{\mathcal M}_{H^{\circ}}},\tilde{\ell}_{\tilde{\mathcal M}_{H^{\circ}}},
\tilde{\mathcal V}_{\tilde{\mathcal M}_{H^{\circ}}}\big)$
in a similar way to (\ref{equation: definition of tangent complex}).
We can see by the same argument as
Proposition \ref{prop: tangent space of the moduli space}
and
Theorem \ref{thm: smoothness of the moduli space}
that ${\mathcal M}_{H^{\circ}}$ is smooth over $H^{\circ}$
and
$\mathbf{R}^1 \big(p_{\tilde{\mathcal M }_{H^{\circ}}}\big)_*
\big({\mathcal F}^{\bullet}_{\tilde{\mathcal M}_{H^{\circ}}}\big)$
is the $H^{\circ}$-relative tangent bundle of $\tilde{\mathcal M}_{H^{\circ}}$.
We can define a pairing
\begin{gather}
 \omega_{\tilde{\mathcal M}_{H^{\circ}}}\colon \
 \mathbf{R}^1 \big(p_{\tilde{\mathcal M}_{H^{\circ}}}\big)_*
 \big({\mathcal F}^{\bullet}_{\tilde{\mathcal M}_{H^{\circ}}}\big)
 \times \mathbf{R}^1 \big(p_{\tilde{\mathcal M}_{H^{\circ}}}\big)_ *
 \big({\mathcal F}^{\bullet}_{\tilde{\mathcal M}_{H^{\circ}}}\big)
 \nonumber\\
\hphantom{\omega_{\tilde{\mathcal M}_{H^{\circ}}}\colon}{} \ \longrightarrow
 \mathbf{R}^2 \big(p_{\tilde{\mathcal M}_{H^{\circ}}}\big)_ *
 \big[ {\mathcal O}_{C\times H^{\circ}}\!
 \to
 \Omega^1_{C\times H^{\circ}/H^{\circ}} (D_{\mathrm{ram},h})\!
 \to \Omega^1_{C\times H^{\circ}/H^{\circ}} (D_{\mathrm{ram},h})
 \big| _{ D_{\mathrm{ram},h} } \big]_{\tilde{\mathcal M}_{H^{\circ}}}\nonumber
 \\
\hphantom{\omega_{\tilde{\mathcal M}_{H^{\circ}}}\colon}{} \ \cong
 \mathbf{R}^2 (p_{\tilde{\mathcal M}_{H^{\circ}}})_ *
 \Omega^{\bullet}_{C\times \tilde{\mathcal M}_{H^{\circ}}/\tilde{\mathcal M}_{H^{\circ}}}\!
 \cong {\mathcal O}_{\tilde{\mathcal M}_{H^{\circ}}}\label{equation: definition of unfolded relative 2-form}
\end{gather}
by the same formula as
(\ref{equation: definition of symplectic form over etale cover}).
We can see that it defines a relative
$2$-form $\omega_{{\mathcal M}_{H^{\circ}}}$ on ${\mathcal M}_{H^{\circ}}$ over~$H^{\circ}$.
The moduli space
$M^{\balpha}_{{\mathcal C},{\mathcal D}}(\lambda,\tilde{\mu},\tilde{\nu})_t$
is a Zariski open subset of the fiber
$({\mathcal M}_{H^{\circ}})_{0}$ over $h=0$
and the restriction
$\omega_{{\mathcal M}_{H^{\circ}}}\big|
_{M^{\balpha}_{{\mathcal C},{\mathcal D}}(\lambda,\tilde{\mu},\tilde{\nu})_t}$
is nothing but the $2$-form
$\omega_{M^{\balpha}_{{\mathcal C},{\mathcal D}}(\lambda,\tilde{\mu},\tilde{\nu})_t}$
on $M^{\balpha}_{{\mathcal C},{\mathcal D}}(\lambda,\tilde{\mu},\tilde{\nu})_t$
defined by~(\ref{equation: definition of symplectic form over etale cover}).
So Theorem~\ref{theorem: existence of symplectic form and d-closedness}
follows from the following proposition.

\begin{Proposition}
The relative $2$-form
$\omega_{{\mathcal M}_{H^{\circ}}}$ on ${\mathcal M}_{H^{\circ}}$
defined by \eqref{equation: definition of unfolded relative 2-form}
is $d$-closed:
$d\omega_{{\mathcal M}_{H^{\circ}}}=0$.
\end{Proposition}

\begin{proof}
Let
${\mathcal M}_{H^{\circ},h}$
be the fiber of the moduli space ${\mathcal M}_{H^{\circ}}$ over
generic $h\in H^{\circ}\setminus\{0\}$.

Consider the point $z=hs^x_j$ in $D^x_{\mathrm{un},h}$ for generic $h\in H^{\circ}$.
Then $\nabla$ is logarithmic at $z=hs^x_j$ and the filtration
$\ell|_{z=hs^x_j}$ is a logarithmic
$(\res_{z=hs^x_j}(\mu^x_{k,h})_{0\leq k\leq r-1})$-parabolic structure
at the point $z=hs^x_j$.

Consider the point $z=h^rq^x_j$ in $D^x_{\mathrm{ram},h}$
for generic $h\in H^{\circ}$.
Then the restriction
of $\theta^x_k\circ\kappa^x_k$ to
$\overline{V}^x_k|_{z=h^rq^x_j}=E|_{z=h^rq^x_j}$ satisfies
the equalities
$(\theta^x_k\circ\kappa^x_k|_{z=h^rq^x_j})^r-h^r(q^x_j-1)=0$
for $1\leq j \leq m_x-1$.
So it has $r$ distinct eigenvalues
$\zeta_r^sh\sqrt[r]{q^x_j-1}$ ($s=0,1,\dots,r-1$),
where $\zeta_r$ is a primitive $r$-th root of unity.
Then
\begin{gather*}
 \res_{z=h^rq^x_j}(\nabla)=
 \res_{z=h^rq^x_j}(\nu_0(z))
 + \res_{z=h^rq^x_j}(\nu_1(z))(\theta^x_k\circ\kappa^x_k)|_{z=h^rq^x_j}
 +\cdots
 \\
\hphantom{\res_{z=h^rq^x_j}(\nabla)=}{}
 + \res_{z=h^rq^x_j}(\nu_{r-1}(z))((\theta^x_k\circ\kappa^x_k)|_{z=h^rq^x_j})^{r-1}
\end{gather*}
also has $r$ distinct eigenvalues if $h$ is sufficiently generic.
The data of filtration $\{V_k\}$ given in (v) is equivalent to
the filtration
$E|_{z=h^r}=V^x_0|_{z=h^r}\supset\cdots\supset V^x_{r-1}|_{z=h^r}
\supset V^x_r|_{z=h^r}=0$
satisfying
$\big(\res_{z=h^r}(\nabla)-\big(\res_{z=h^r}(\nu_0)+\frac{k}{r}\big)\mathrm{id}\big)
\big(V^x_k\big|_{z=h^r}\big)\subset V^x_{k+1}\big|_{z=h^r}$
for $0\leq k\leq r-1$ at each $x$.
So the restriction
$(V^x_k|_{z=h^r})_{0\leq k\leq r-1}$
is a logarithmic parabolic structure on $(E,\nabla)$.

For generic $h$,
we define a complex ${\mathcal F}_{\tilde{\mathcal M}_{H^{\circ},h}}^{\mathrm{diag} \bullet}$
on the fiber $\tilde{\mathcal M}_{H^{\circ},h}$
by setting
\begin{gather*}
 {\mathcal F}^{\mathrm{diag},0}_{\tilde{\mathcal M}_{H^{\circ},h}}
 =
 \ker \big( {\mathcal G}^0_{\tilde{\mathcal M}_{H^{\circ},h}}
 \longrightarrow
 \coker \big( A^0_{\tilde{\mathcal M}_{H^{\circ},h}} \rightarrow
 \Sym^2\big(\big( \overline{W}_{\tilde{\mathcal M}_{H^{\circ},h}} \big)\big)
 \oplus \Sym^2\big( \big(\overline{V}_{\tilde{\mathcal M}_{H^{\circ},h},h}\big)\big)
\big),
 \\
 {\mathcal F}^{\mathrm{diag},1}_{\tilde{\mathcal M}_{H^{\circ},h}}
 =
 \ker \big( {\mathcal G}^1_{\tilde{\mathcal M}_{H^{\circ},h}}
 \rightarrow G^1_{\tilde{\mathcal M}_{H^{\circ},h}} \big),
 \\
 d^0_{{\mathcal F}^{\mathrm{diag},\bullet}_{\tilde{\mathcal M}_{H^{\circ},h}}}
 =
 d^0_{{\mathcal F}^{\bullet}_{\tilde{\mathcal M}_{H^{\circ},h}}}
 \big|_{{\mathcal F}^{\mathrm{diag},0}_{\tilde{\mathcal M}_{H^{\circ},h}}}
 \colon \
 {\mathcal F}^{\mathrm{diag},0}_{\tilde{\mathcal M}_{H^{\circ},h}} \longrightarrow
 {\mathcal F}^{\mathrm{diag},1}_{\tilde{\mathcal M}_{H^{\circ},h}}.
\end{gather*}
Note that ${\mathcal F}^{\mathrm{diag},0}_{\tilde{\mathcal M}_{H^{\circ},h}}$
is the sheaf of endomorphisms of $E$ preserving the eigen decomposition
of $\res_{z=h^r q^x_j}(\nabla)$
at $z=h^r q^x_j$ in $D^x_{\mathrm{ram},h}$ for $1\leq j\leq m_x-1$,
preserving the parabolic structure $l^x$ at each
$x\in D_{\mathrm{log}}$,
preserving the parabolic structure
$\big(\ell_k^x|_{z=hs^x_j}\big)_{0\leq k\leq r-1}$
at $z=hs^x_j$ in $D^x_{\mathrm{un},h}$ for $1\leq j\leq hs^x_{n_x}$
and preserving the parabolic structure
$\big(\overline{V}^x_k|_{z=h^r}\big)_{0\leq k\leq r-1}$
at $z=h^r$ in $D^x_{\mathrm{ram},h}$.
We can see that the canonical map
\[
 {\mathcal F}^{\mathrm{diag},\bullet}_{\tilde{\mathcal M}_{H^{\circ},h}}
 \longrightarrow {\mathcal F}^{\bullet}_{\tilde{\mathcal M}_{H^{\circ},h}}
\]
is a quasi-isomorphism.
On the other hand, we can define a complex
${\mathcal F}_{\mathrm{par}}^{\bullet}$ on $C\times \tilde{\mathcal M}_{H^{\circ},h}$
in the same way as
in the proof of \cite[Proposition 7.2]{Inaba-1}
by associating the parabolic structure induced by the eigen decomposition
at each point defined by $z=h^rq^x_j$ in $D^x_{\mathrm{ram},h}$
for $1\leq j\leq m^x-1$.
Then the canonical map
\[
 {\mathcal F}^{\mathrm{diag},\bullet}_{\tilde{\mathcal M}_{H^{\circ},h}}
 \longrightarrow {\mathcal F}^{\bullet}_{\mathrm{par}}
\]
is a quasi-isomorphism.
We can see
that the restriction $\omega_{\tilde{\mathcal M}_{H^{\circ},h}}$
to a generic fiber $\tilde{\mathcal M}_{H^{\circ},h}$
of the $2$-form $\omega_{\tilde{\mathcal M}_{H^{\circ}}}$
coincides with the $2$-form
constructed in \cite[Proposition~7.2]{Inaba-1},
because it is expressed by the same formula as~(\ref{equation: definition of symplectic form over etale cover}).
Since the $2$-form in \cite[Proposition~7.2]{Inaba-1} is ${\rm d}$-closed
by \cite[Proposition~7.3]{Inaba-1},
we have ${\rm d}\omega_{{\mathcal M}_{H^{\circ},h}}=0$ for generic $h$.
Thus we can deduce ${\rm d}\omega_{{\mathcal M}_{H^{\circ}}}=0$,
because ${\mathcal M}_{H^{\circ}}$ is smooth over $H^{\circ}$.
\end{proof}

\section[Local generalized isomonodromic deformation on a ramified covering]{Local generalized isomonodromic deformation\\ on a ramified covering}
\label{section: local analytic theory}

In this section, we will consider the pullback of a generic ramified connection
via a local analytic ramified covering map.
Furthermore, we will give a brief sketch of the Stokes data of the pullback
and its generalized isomonodromic deformation
established by Jimbo, Miwa and Ueno in \cite{Jimbo-Miwa-Ueno}.

Let $\Delta_z$ and $\Delta_w$ be unit disks
equipped with the variables $z$ and $w$, respectively.
Consider the ramified covering map
\begin{equation} \label{ramified covering map}
 p \colon \ \Delta_w \ni w \mapsto w^r=z \in \Delta_z.
\end{equation}
There is a canonical action of the Galois group
$\mathrm{Gal}(\Delta_w/\Delta_z)=\{ \sigma^k \mid| 0\leq k\leq r-1 \}$
which is generated by the automorphism
$\sigma\colon \Delta_w \ni w\mapsto \zeta_r w\in \Delta_w$,
where $\zeta_r=\exp\big(2\pi\sqrt{-1}/r\big)$ is a~primitive root of unity.

Take
$\nu_0(z)\in(\mathbb{C}+\mathbb{C}z+\cdots+\mathbb{C}z^{mr-r}){\rm d}z/z^m$,
$\nu_1(z)\in (\mathbb{C}^{\times}+\mathbb{C}z+\cdots+\mathbb{C}z^{mr-r-1}){\rm d}z/z^m$
and
$\nu_2(z),\dots,\nu_{r-1}(z)\in(\mathbb{C}+\mathbb{C}z+\cdots+\mathbb{C}z^{mr-r-1}){\rm d}z/z^m$.
Then we put
\[
 \nu(w):=\nu_0(z)+\nu_1(z)w+\cdots+\nu_{r-1}(z)w^{r-1},
\]
which is said to be a ramified exponent.
We define a formal connection
$\nabla_{\nu}$ on $\mathbb{C}[[w]]$ by
\[
 \nabla_{\nu} \colon \ \mathbb{C}[[w]] \ni f(w)
 \mapsto {\rm d}f(w)+f(w)\nu(w) \in \mathbb{C}[[w]]\otimes\frac{{\rm d}z}{z^m}.
\]

Let $(E,\nabla)$ be a meromorphic connection on $\Delta_z$
with a formal isomorphism
\begin{equation} \label{equation: formal isomorphism in single version}
 \big(\widehat{E},\widehat{\nabla}\big)
 :=
 (E,\nabla)\otimes\widehat{\mathcal O}_{\Delta_z,0}
 \stackrel{\sim}\longrightarrow (\mathbb{C}[[w]], \nabla_{\nu}).
\end{equation}
Consider the pullback
$(p^*E,p^*\nabla)$ of the meromorphic connection
$(E,\nabla)$ by the ramified cover~$p$ given in (\ref{ramified covering map}).
The formal isomorphism (\ref{equation: formal isomorphism in single version})
induces a canonical surjection
\begin{equation*}
 \pi \colon \ p^*E\otimes\widehat{\mathcal O}_{\Delta_w,0}
 =\widehat{E}\otimes_{\mathbb{C}[[z]]}\mathbb{C}[[w]] \longrightarrow \mathbb{C}[[w]]
\end{equation*}
which makes the diagram
\[
 \begin{CD}
 \widehat{E}\otimes\mathbb{C}[[w]] @>\pi>> \mathbb{C}[[w]] \\
 @V\widehat{\nabla}\otimes\mathrm{id}VV @VV \nabla_{\nu} V \\
 \widehat{E}\otimes\mathbb{C}[[w]]\otimes\frac{{\rm d}z}{z^m}
 @>\pi\otimes\mathrm{id}>> \mathbb{C}[[w]]\otimes \frac{{\rm d}z}{z^m}
 \end{CD}
\]
commutative.
The Galois transform of $\pi$ by the element
$\sigma^k$ of $\mathrm{Gal}(\Delta_w/\Delta_z)$ is given by
\[
 \sigma^k\circ\pi\circ\sigma^{-k} \colon \
 \widehat{E}\otimes\mathbb{C}[[w]] \xrightarrow[\sim]{\mathrm{id}_{\widehat{E}}\otimes\sigma^{-k}}
 \widehat{E}\otimes\mathbb{C}[[w]] \stackrel {\pi} \longrightarrow
 \mathbb{C}[[w]]\xrightarrow[\sim]{\sigma^k}\mathbb{C}[[w]],
\]
which makes the diagram
\[
 \begin{CD}
 \widehat{E}\otimes\mathbb{C}[[w]] @>\sigma^k\circ\pi\circ\sigma^{-k}>> \mathbb{C}[[w]] \\
 @V\widehat{\nabla}\otimes\mathrm{id}VV @VV \nabla_{\sigma^k\nu} V \\
 \widehat{E}\otimes\mathbb{C}[[w]]\otimes\frac{{\rm d}z}{z^m}
 @>(\sigma^k\circ\pi\circ\sigma^{-k})\otimes\mathrm{id}>> \mathbb{C}[[w]]\otimes \frac{{\rm d}z}{z^m}
 \end{CD}
\]
commutative,
where we put $\sigma^k\nu(w):=\nu\big(\zeta_r^kw\big)$.
So we get a morphism
\begin{equation} \label{comparison morphism over ramified cover}
 \varpi \colon \
 \big(p^*\widehat{E},p^*\widehat{\nabla}\big)
 \xrightarrow{ \bigoplus_{k=0}^{r-1} \sigma^k\circ\pi\circ\sigma^{-k} }
 \ \bigoplus_{k=0}^{r-1} (\mathbb{C}[[w]],\nabla_{\sigma^k\nu(w)}),
\end{equation}
whose underlying homomorphism on vector bundles over $\mathbb{C}[[w]]$
is generically isomorphic.
Choose a generator $e_0$ of the underlying bundle $\mathbb{C}[[w]]$
of $(\mathbb{C}[[w]],\nabla_{\nu})$ (we may choose $e_0=1$).
We denote the same element of the underlying bundle of
$(\mathbb{C}[[w]],\nabla_{\sigma^k\nu})$ by $\sigma^k(e_0)$.
Then we can define an action of
$\mathrm{Gal}(\Delta_w/\Delta_z)$ on the right-hand side of
(\ref{comparison morphism over ramified cover})
by setting
\[
 \sigma^l\cdot\sum_{k=0}^{r-1}f_k(w)\sigma^k(e_0)
 :=\sum_{k=0}^{r-1}f_k\big(\zeta_r^lw\big) \sigma^{k+l}(e_0).
\]
The connection
$\bigoplus_{k=0}^{r-1}\nabla_{\sigma^k\nu}$
on the right-hand side of~(\ref{comparison morphism over ramified cover})
commutes with the Galois action.
The morphism $\varpi$ in~(\ref{comparison morphism over ramified cover})
is a $\mathbb{C}[[w]]$-homomorphism, which
commutes with the connections
and with the Galois actions on the both sides.

We can see that the image
$\im\varpi$ of the homomorphism (\ref{comparison morphism over ramified cover})
is generated by
\[
 \left\{ \sum_{l=0}^{r-1} \zeta_r^{kl}w^k \sigma^l(e_0) \, \bigg| \, k=0,1,\dots,r-1 \right\}
\]
as a $\mathbb{C}[[w]]$-module.
Then we can check the inclusion
$w^{r-1}\cdot\bigoplus_{k=0}^{r-1}\mathbb{C}[[w]]\sigma^k(e_0)\subset\im\varpi$.
Consider the restriction
\begin{gather*}
 \varpi|_{w^{r-1}=0}
 \colon \
 \widehat{E}|_{w^{r-1}=0}\otimes\mathbb{C}[w]/\big(w^{r-1}\big)
 \xrightarrow {\varpi|_{w^{r-1}=0}}
 \im(\varpi|_{w^{r-1}=0})
 \subset \bigoplus_{k=0}^{r-1} \mathbb{C}[w]/\big(w^{r-1}\big) \cdot \sigma^k(e_0)
\end{gather*}
of the morphism $\varpi$ in (\ref{comparison morphism over ramified cover})
to the divisor on $\Delta_w$ defined by $w^{r-1}=0$.
Then the composition
\begin{equation*}
 \varphi \colon \
 p^*(E) \longrightarrow p^*(E)|_{w^{r-1}=0}
 =\widehat{E}\otimes\mathbb{C}[w]/\big(w^{r-1}\big)
 \xrightarrow{\varpi|_{w^{r-1}=0}} \im(\varpi|_{w^{r-1}=0})
\end{equation*}
commutes with $p^*(\nabla)$ and $\bigoplus_{k=0}^{r-1}\nabla_{\sigma^k\nu}|_{w^{r-1}=0}$.
So we have
\[
 (p^*\nabla)(\ker\varphi)\subset \ker\varphi\otimes \frac{{\rm d}w}{w^{mr-r+1}}.
\]
Consider the line bundle
${\mathcal O}_{\Delta_w}\big((r-1)\cdot\{0\}\big)$
on $\Delta_w$ with the connection
\begin{gather*}
 \nabla_{-\nu_0(z)}\colon \ {\mathcal O}_{\Delta_w}\big((r-1)\cdot\{0\}\big)
 \ni f(w) \mapsto {\rm d}f(w)-f(w) \nu_0(z) \\
 \hphantom{ \nabla_{-\nu_0(z)}\colon \ {\mathcal O}_{\Delta_w}\big((r-1)\cdot\{0\}\big)
 \ni f(w) \mapsto}{} \in
 {\mathcal O}_{\Delta_w}\big((r-1)\cdot\{0\}\big)\otimes\frac{{\rm d}w}{w^{mr-r+1}}.
\end{gather*}
If we modify $(\ker\varphi,p^*\nabla|_{\ker\varphi})$
by setting
\begin{equation} \label{equation: unramified connection by elementary transform}
 (E',\nabla')
 :=
 \big( \ker\varphi,(p^*\nabla)|_{\ker\varphi} \big)
 \otimes \big( {\mathcal O}_{\Delta_w}\big((r-1)\cdot\{0\}\big), \nabla_{-\nu_0(z)} \big),
\end{equation}
then the order of pole of $\nabla'$ at $w=0$ is $mr-r$.
Indeed, the morphism $\varpi$
in (\ref{comparison morphism over ramified cover})
induces a~formal isomorphism
\begin{equation*}
 \big(\widehat{E'},\widehat{\nabla'}\big)
 \stackrel{\sim}\longrightarrow
 \bigoplus_{k=0}^{r-1} \big(\mathbb{C}[[w]],\nabla_{\nu(\zeta_r^kw)-\nu_0(z)}\big)
\end{equation*}
and the matrix of the connection $\nabla_{\nu(\zeta_r^kw)-\nu_0(z)}$
of the right-hand side is
\[
 \begin{pmatrix}
 \displaystyle \sum_{k=1}^{r-1}\nu_k(z)w^k & 0 & \cdots & 0 \\
 0 & \displaystyle \sum_{k=1}^{r-1}\nu_k(z)\zeta_r^kw^k & & 0 \\
 \vdots & & \ddots & \vdots \\
 0 & 0 & \cdots & \displaystyle \sum_{k=1}^{r-1}\nu_k(z)\zeta_r^{k(r-1)}w^k
 \end{pmatrix}.
\]
Since the leading terms of the diagonal entries of the above matrix
are distinct,
$(E',\nabla')$ is a generic unramified connection.
Furthermore,
there is a canonical action of
$\mathrm{Gal}(\Delta_w/\Delta_z)$
on $(E',\nabla')$,
since $\varphi$ and
$\otimes \big( {\mathcal O}_{\Delta_w}\big((r-1)\cdot\{0\}\big), \nabla_{-\nu_0(z)} \big)$
preserve the Galois action.

\begin{Proposition}
\label{proposition: local correspondence between ramified and unramified connection}
The correspondence
$(E,\nabla)\mapsto
(E',\nabla')$
given by the formula
\eqref{equation: unramified connection by elementary transform}
is a bijection between the meromorphic $\nu$-ramified connections
$(E,\nabla)$ on $\Delta_z$ equipped with
a formal isomorphism
$\big(\widehat{E},\widehat{\nabla}\big)\xrightarrow{\sim}(\mathbb{C}[[w]],\nabla_{\nu})$
and the $\mathrm{Gal}(\Delta_w/\Delta_z)$-equivariant
$\big(\nu\big(\zeta_r^kw\big)-\nu_0(z)\big)_{0\leq k\leq r-1}$-unramified
meromorphic connections
$(E',\nabla')$ on $\Delta_w$ equipped with a Galois equivariant formal isomorphism
$\big(\widehat{E'},\widehat{\nabla'}\big)
\xrightarrow{\sim}
\bigoplus (\mathbb{C}[[w]],\nabla_{\sigma^k\nu})$.
\end{Proposition}

\begin{proof}
We have to give the inverse correspondence.
If $(E',\nabla')$ is a $\big(\nu\big(\zeta_r^kw\big)-\nu_0(z)\big)_{0\leq k\leq r-1}$-unramified
meromorphic connection on $\Delta_w$
compatible with an action of
$\mathrm{Gal}(\Delta_w/\Delta_z)$,
we put
\[
 \tilde{E}':=
 \ker\big( E' \longrightarrow
 \coker \big(
 \big( E'|_{w^{mr-r}=0} \big)^{ \mathrm{Gal}(\Delta_w/\Delta_z) }
 \otimes \mathbb{C}[w]/(w^{mr-r})
 \rightarrow E'|_{w^{mr-r}=0} \big) \big),
\]
where $\big( E'|_{w^{mr-r}=0} \big)^{ \mathrm{Gal}(\Delta_w/\Delta_z) }$
is the submodule of $E'|_{w^{mr-r}=0}$
consisting of the $\mathrm{Gal}(\Delta_w/\Delta_z)$-invariant sections.
Let $\big( \tilde{E}' \big)^{\mathrm{Gal}(\Delta_w/\Delta_z)}$
be the subsheaf of
$p_*\big(\tilde{E}'\big)$
consisting of $\mathrm{Gal}(\Delta_w/\Delta_z)$-invariant sections.
Then $\big( \tilde{E}' \big)^{\mathrm{Gal}(\Delta_w/\Delta_z)}$
becomes a locally free sheaf on $\Delta_z$ of rank $r$
and the connection
$\nabla'\big|_{\tilde{E}'}\otimes\nabla_{\nu_0(z)}$
on $\tilde{E}'$ descends to a connection
$\big( \nabla' \big|_{ \tilde{E}' } \otimes\nabla_{\nu_0(z)} \big) ^{ \mathrm{Gal}(\Delta_w/\Delta_z) }$
on $\big(\tilde{E}'\big)^{ \mathrm{Gal}(\Delta_w/\Delta_z) }$.
We can check that
$
 \big( \big(\tilde{E}'\big)^{ \mathrm{Gal}(\Delta_w/\Delta_z) } ,
 \big( \nabla' \big|_{ \tilde{E}' } \otimes\nabla_{\nu_0(z)} \big) ^{ \mathrm{Gal}(\Delta_w/\Delta_z) }
 \big)
$
is a meromorphic $\nu$-ramified connection on $\Delta_z$.
From the construction,
\[
 (E',\nabla')
 \mapsto
\big( \big(\tilde{E}'\big)^{ \mathrm{Gal}(\Delta_w/\Delta_z) } ,
 \big( \nabla' \big|_{ \tilde{E}' } \otimes\nabla_{\nu_0(z)}\big)^{ \mathrm{Gal}(\Delta_w/\Delta_z)}
\big)
\]
gives the inverse to
$(E,\nabla)\mapsto (E',\nabla')$.
\end{proof}

\begin{Remark}\rm
The process of getting the vector bundle $\ker\varphi$ or $E'$
from $p^*E$ is called an elementary transform
or a Hecke modification.
The construction of $(E',\nabla')$
from $(E,\nabla)$
is known \cite[Section 19.3]{Wasow}
as a shearing transformation method.
\end{Remark}

We will apply Proposition~\ref{proposition: local correspondence between ramified and unramified connection}
to a family of connections.
From now on,
let the notations
${\mathcal T}$, ${\mathcal C}$, $\lambda$, $\tilde{\mu}$, $\tilde{\nu}$
and
$M^{\balpha}_{{\mathcal C},{\mathcal D}}(\lambda,\tilde{\mu},\tilde{\nu})$
be as in Section~\ref{section: construction of the moduli space}.

We take a point $x=(\tilde{x}_i)_t\in ({\mathcal D}_{\mathrm{ram}})_t$
in the fiber over $t\in {\mathcal T}$.
We can take an analytic open neighborhood
${\mathcal T}^{\circ}$ of $t$ such that
$\bar{z}_{\mathcal T^{\circ}}$ can be extended to a local holomorphic function
$z\in{\mathcal O}_{{\mathcal C}_{\mathcal T^{\circ}}}^{\mathrm{hol}}$
whose zero set coincides with the section $\tilde{x}=(\tilde{x}_i)_{\mathcal T^{\circ}}$.
Precisely, there is an analytic open immersion
\[
 \Delta_z\times{\mathcal T^{\circ}}\hookrightarrow{\mathcal C}_{\mathcal T^{\circ}}
\]
for a unit disk $\Delta_z$, such that the coordinate of $\Delta_z$ corresponds to $z$.
We can assume the existence of a universal family
$\big(\!\tilde{E},\tilde{\nabla},\tilde{l},\tilde{\ell},\tilde{\mathcal V}\big)$
on some analytic open neighborhood
${M^{\circ}\!\subset\! M^{\balpha}_{{\mathcal C},{\mathcal D}}(\lambda,\tilde{\mu},\tilde{\nu})
\!\times_{\mathcal T} \!{\mathcal T^{\circ}}}\!.\!$
By Corollary~\ref{corollary: formal structure},
we may further assume that
there is an isomorphism
\begin{equation} \label{equation: family of formal isomorphisms}
 \big(\tilde{E},\tilde{\nabla}\big)\otimes\widehat{\mathcal O}_{\tilde{C}_{M^{\circ}},\tilde{x}}
 \stackrel{\sim}\longrightarrow
 \big({\mathcal O}_{M^{\circ}}^{\mathrm{hol}}[[w]],\nabla_{\tilde{\nu}}\big),
\end{equation}
where
$\displaystyle\widehat{\mathcal O}_{{\mathcal C}_{M^{\circ}},\tilde{x}}=
\lim_{\longleftarrow}{\mathcal O}^{\mathrm{hol}}_{{\mathcal C}_{M^{\circ}}}/I_{\tilde{x}}^j
\cong{\mathcal O}_{M^{\circ}}^{\mathrm{hol}}[[w]]$.
Consider a family of ramified covering maps (\ref{ramified covering map})
\[
 p_{M^{\circ}} \colon \ \Delta_w\times M^{\circ} \ni (w,y) \mapsto (w^r,y)
 \in \Delta_z\times M^{\circ}.
\]
We write $m:=m^{\mathrm{ram}}_i$ for simplicity.
As in the former argument,
the isomorphism (\ref{equation: family of formal isomorphisms})
induces a canonical surjection
\begin{equation*}
 \pi_{M^{\circ}} \colon \
 p_{M^{\circ}}^*\tilde{E}\otimes
 \widehat{\mathcal O}_{{\mathcal C}_{M^{\circ}},\tilde{x}}
 \longrightarrow
 {\mathcal O}_{M^{\circ}}^{\mathrm{hol}}[[w]],
\end{equation*}
which also induces a morphism
\begin{equation} \label{equation:comparison with diagonal on the ramified cover}
 \varpi_{M^{\circ}} \colon \
 \big(p_{M^{\circ}}^*\tilde{E},p_{M^{\circ}}^*\tilde{\nabla}\big)\otimes
 \widehat{\mathcal O}_{{\mathcal C}_{M^{\circ}},\tilde{x}}
 \xrightarrow{\bigoplus_{k=0}^{r-1}\sigma^k\circ\pi_{M^{\circ}}\circ\sigma^{-k}}
 \bigoplus_{k=0}^{r-1}
 \big({\mathcal O}_{M^{\circ}}^{\mathrm{hol}}[[w]],\nabla_{\sigma^k\tilde{\nu}}\big)
\end{equation}
between rank~$r$ connections over
${\mathcal O}^{\mathrm{hol}}_{M^{\circ}}[[w]]$.
Let $\tilde{x}'$ be the divisor on $\Delta_w\times M^{\circ}$
defined by the equation $w=0$.
The composition
\[
 \varphi_{M^{\circ}} \colon \
 p_{M^{\circ}}^*\big(\tilde{E}|_{\Delta_z\times M^{\circ}}\big)
 \longrightarrow
 p_{M^{\circ}}^*\big(\tilde{E}|_{\Delta_z\times M^{\circ}}\big)|_{(r-1)\tilde{x}'}
 \longrightarrow
 \im \big( \varpi_{M^{\circ}}|_{(r-1)\tilde{x}'} \big)
\]
is a surjective homomorphism
and we have
$\big(p_{M^{\circ}}^*\tilde{\nabla}\big) (\ker\varphi)
\subset \ker\varphi\otimes \Omega^1_{\Delta_w\times M^o/M^o}((mr-r+1)\tilde{x}')$.
Setting
\begin{equation} \label{equation: local connection on ramified cover}
 \big(\tilde{E}',\tilde{\nabla}'\big):=\big(\ker\varphi,p_{M^{\circ}}^*\tilde{\nabla}|_{\ker\varphi}\big)
 \otimes \big( {\mathcal O}_{\Delta_w\times M^{\circ}}^{\mathrm{hol}}((r-1)\tilde{x}') ,
 \nabla_{-\tilde{\nu}_0}\big),
\end{equation}
we get a connection
\[
 \tilde{\nabla}' \colon \
 \tilde{E}' \longrightarrow
 \tilde{E}' \otimes
\Omega^1_{\Delta_w\times M^{\circ}/M^{\circ}}\big((mr-r)\tilde{x}'\big).
\]

The morphism $\varpi_{M^{\circ}}$ in (\ref{equation:comparison with diagonal on the ramified cover})
induces an isomorphism
\begin{equation} \label{equation: relative ramified unramified correspondence}
 \big(\tilde{E}',\tilde{\nabla}'\big)\otimes\widehat{\mathcal O}_{\tilde{C}_{M^{\circ}},\tilde{x}}
 \stackrel{\sim}\longrightarrow
 \bigoplus_{k=0}^{r-1} \big( {\mathcal O}_{M^{\circ}}^{\mathrm{hol}}[[w]],
 \nabla_{\tilde{\nu}(\zeta_r^kw)-\tilde{\nu}_0(z)} \big).
\end{equation}
The connection $\nabla_{\tilde{\nu}(\zeta_r^kw)-\tilde{\nu}_0(z)} $
of the right-hand side is given by
$d+\Lambda(w,t)$ with
\begin{equation} \label{equation: diagonal matrix of normal form}
 \Lambda(w,t):=
 \begin{pmatrix}
 \displaystyle \sum_{k=1}^{r-1}\tilde{\nu}_k(z,t)w^k & 0 & \cdots & 0 \\
 0 & \displaystyle\sum_{k=1}^{r-1}\tilde{\nu}_k(z,t)\zeta_r^kw^k & & 0 \\
 \vdots & & \ddots & \vdots \\
 0 & 0 & \cdots & \displaystyle \sum_{k=1}^{r-1}\tilde{\nu}_k(z,t)\zeta_r^{k(r-1)}w^k
 \end{pmatrix}.
\end{equation}

Now we will see the corresponding Stokes data.
We set
$E'_0 :=
\big( {\mathcal O}_{\Delta_w\times {\mathcal T^{\circ}}}^{\mathrm{hol}}\big)^{\oplus r}$
and fix a connection
$\nabla'_0 \colon
E'_0 \longrightarrow E'_0 \otimes
\Omega^1_{\Delta_w\times {\mathcal T^{\circ}}/{\mathcal T^{\circ}}} ((mr-r)\tilde{x}')$
defined by
\[
 \begin{pmatrix} f_1 \\ \vdots \\ f_r \end{pmatrix}
 \mapsto
 \begin{pmatrix} {\rm d} f_1 \\ \vdots \\ {\rm d} f_r \end{pmatrix}
 +
 \Lambda(w,t)
 \begin{pmatrix} f_1 \\ \vdots \\ f_r \end{pmatrix}.
\]
We call $(E'_0,\nabla'_0)$ a normal form.

It is a general fact \cite[Proposition 2.2]{Jimbo-Miwa-Ueno} that
there is a matrix
$P(w,t)$ of formal power series in $w$ with coefficients
in ${\mathcal O}_{M^{\circ}}^{\mathrm{hol}}$,
which gives a formal isomorphism
\begin{equation} \label{equation: formal transform in w}
 (E'_0,\nabla'_0)\otimes\widehat{\mathcal O}_{\tilde{C}_{M^{\circ}},\tilde{x}}
 \xrightarrow[\sim]{P(w,t)}
 \big(\tilde{E}',\tilde{\nabla}'\big)\otimes\widehat{\mathcal O}_{\tilde{C}_{M^{\circ}},\tilde{x}}.
\end{equation}
If $\tilde{\nabla}'$ is given by
${\rm d}+A'(w,t){\rm d}w/w^{mr-r}$ for a matrix $A'(w,t)$ of holomorphic functions in~$w$,~$t$,
then we have
\[
 P(w,t)^{-1} {\rm d}P(w,t)
 +P(w,t)^{-1} A'(w,t)\frac{{\rm d}w}{w^{mr-r}} P(w,t)
 =
 \Lambda(w,t).
\]
In fact, we can give the formal transform $P(w,t)$
as the inverse of (\ref{equation: relative ramified unramified correspondence}),
which is induced by the formal transform
(\ref{equation: family of formal isomorphisms})
over ${\mathcal O}_{M^{\circ}}^{\mathrm{hol}}[[z]]$.
Indeed, if we denote the inverse formal transform of
(\ref{equation: family of formal isomorphisms})
by
\begin{equation} \label{equation: setting of Q(z)}
 Q(z,t) \colon \
 \big({\mathcal O}_{M^{\circ}}^{\mathrm{hol}}[[w]],\nabla_{\tilde{\nu}}\big)
 \stackrel{\sim}\longrightarrow
 \big(\tilde{E},\tilde{\nabla}\big)\otimes\widehat{\mathcal O}_{\tilde{C}_{M^{\circ}},\tilde{x}}
\end{equation}
and if we denote the rational gauge transform
$p_{M'}^*\big(\tilde{E}|_{\Delta_z\times M'}\big)
\hookrightarrow \tilde{E}'$ by $S(w)$,
then we can give~$P(w)$ by
\begin{equation} \label{equation: formal transform using z parameter}
 P(w,t)
 =
 S(w,t) Q(z,t)
 \begin{pmatrix}
 1 & w & \cdots & w^{r-1} \\
 1 & \zeta_rw & \cdots & \zeta_r^{r-1}w^{r-1} \\
 \vdots & \vdots & \ddots & \vdots \\
 1 & \zeta_r^{r-1}w & \cdots & \zeta_r^{(r-1)^2}w^{r-1}
 \end{pmatrix}^{-1}.
\end{equation}

\begin{Remark}
The above procedure
is explained in~\cite[Proposition~10]{Diarra-Loray} for the explicit case of rank $2$
connections on~$\mathbb{P}^1$.
\end{Remark}

Take any point
$u\in (\Delta_w\setminus\{0\})\times M^{\circ}$.
By the fundamental existence theorem \cite[Theorem~12.1]{Wasow}
of asymptotic solution,
there are a sector $\Gamma_u=\{w\in\Delta_w \mid a<\mathrm{arg}(w)<b\}$
in $\Delta_w\setminus\{0\}$ for some $a,b\in\mathbb{R}$
and an open subset $M_u\subset M^{\circ}$
satisfying $u \in \Gamma_u\times M_u$
such that there exists
a~fundamental solution
$Y_{\Sigma}(w,t)=(y_1(w,t),\dots,y_r(w,t))$
of $\tilde{\nabla}'$ on $\Sigma=\Gamma_u\times M_u$
satisfying the asymptotic property
\begin{equation} \label{equation: asymptotic property of fundamental solution}
 Y_{\Sigma}(w,t)
 \exp \left(\int\Lambda(w)\right)
 \sim
 P(w,t)
 \qquad
 \text{as $w\to 0$ on $\Sigma=\Gamma_u\times M_u$},
\end{equation}
where the path integral of $\Lambda(w)$, which is defined in~(\ref{equation: diagonal matrix of normal form}),
is with respect to the $w$-variable.
If we put $P(w,t)=\sum_{j=0}^{\infty} P_j(t) w^j$,
the asymptotic relation~(\ref{equation: asymptotic property of fundamental solution}) means
\begin{equation} \label{equation: uniform asymptotic}
 \lim_{w\to 0, w\in\Gamma_u}
 \frac{ \big\| Y_{\Sigma}(w,t)\exp \big( \int\Lambda(w) \big)
 -\sum_{j=0}^N P_j(t) w^j \big\| } { |w|^N }
 =0
\end{equation}
for any positive integer $N$
and the convergence in (\ref{equation: uniform asymptotic})
is uniform in $t\in M_u$.

Fix a point $t'\in M^{\circ}$.
Taking a finite subcover of $\{\Sigma=\Gamma_u\times M_u\}$,
we can choose an open neighborhood $U_{t'}$ of $t'$ in $M^{\circ}$
and a covering
$\{\Sigma\}$ of $(\Delta_w\setminus\{0\})\times U_{t'}$
such that each $\Sigma$ is of the form
$\Sigma=\Gamma_u\times U_{t'}$
for a sector $\Gamma_u$ in $\Delta_w\setminus\{0\}$.

If we take another $\Sigma'=\Gamma_{u'}\times U_{t'}$
in the above covering,
and if we choose a fundamental solution $Y_{\Sigma'}(w,t)$
on $\Sigma'$
with the same asymptotic property as~(\ref{equation: asymptotic property of fundamental solution})
on $\Sigma'$,
we can write
\begin{equation} \label{equation: comparison of fundamental solutions with asymptotic}
 Y_{\Sigma'}(w,t)=Y_{\Sigma}(w,t) C_{\Sigma,\Sigma'}(t)
\end{equation}
for a matrix $C_{\Sigma,\Sigma'}(t)$
constant in $w$.
We call $C_{\Sigma,\Sigma'}(t)$ a Stokes matrix.

\begin{Definition} \label{definition: local generalized isomonodromic deformation}
We say that a family of connections
$(\tilde{E}',\tilde{\nabla}')|_{\Delta_w\times {\mathcal L}}$
over a submanifold ${\mathcal L}\subset M^{\circ}$
is a local generalized isomonodromic deformation,
if for each $t'\in{\mathcal L}$,
we can take an open neighborhood ${\mathcal L}_{t'}$ of $t'$ in ${\mathcal L}$,
a replacement of the formal transform $P(w,t)$ in (\ref{equation: formal transform in w})
and a~covering $\{\Sigma=\Gamma_u\times {\mathcal L}_{t'}\}$
of $(\Delta_w\setminus\{0\})\times {\mathcal L}_{t'}$ for sectors
$\Gamma_u$ in $\Delta_w\setminus\{0\}$
such that
\begin{itemize}\itemsep=0pt
\item[(i)]
there is a fundamental solution
$Y_{\Sigma}(w,t)$ of $\tilde{\nabla}'|_{\Sigma}$
with the asymptotic property~(\ref{equation: asymptotic property of fundamental solution})
and
\item[(ii)]
all the Stokes matrices $C_{\Sigma,\Sigma'}(t)$
defined by~(\ref{equation: comparison of fundamental solutions with asymptotic})
are constant in $t\in {\mathcal L}_{t'}$.
\end{itemize}
\end{Definition}

\begin{Remark}\quad
\begin{itemize}\itemsep=0pt
\item[(1)]
 The ambiguity of the path integral $\int\Lambda(w)$
 in (\ref{equation: asymptotic property of fundamental solution})
 is included in the replacement of the formal transform $P(w,t)$ in
 Definition \ref{definition: local generalized isomonodromic deformation}.
\item[(2)]
In our definition of Stokes matrices $C_{\Sigma,\Sigma'}(t)$,
there is an ambiguity in the choice of the fundamental solution $Y_{\Sigma}(w,t)$.
On the other hand, \cite[Proposition 2.4]{Jimbo-Miwa-Ueno}
requires $\Sigma$ to be taken sufficiently large so that
there is no ambiguity in $Y_{\Sigma}(w,t)$.
Due to this difference, we will need an additional argument later
in Proposition \ref{proposition: extendability to integrable connection in w variable}.
\end{itemize}
\end{Remark}

Let us recall the argument in the proof of \cite[Theorem 3.1]{Jimbo-Miwa-Ueno}.
Assume that
${\mathcal L}\subset M^{\circ}$
is a~submanifold,
$\{ \Sigma\}$ is a covering of $(\Delta_w\setminus\{0\})\times{\mathcal L}$
as in Definition \ref{definition: local generalized isomonodromic deformation}
and that $Y_{\Sigma}(w,t)$ is a~fundamental solution of
$\tilde{\nabla}'|_{\Delta_w\times{\mathcal L}}$ on each $\Sigma$
such that all the matrices $C_{\Sigma,\Sigma'}(t)$ are constant in $t\in {\mathcal L}$.
We choose a local coordinate system
$(t_1,\dots,t_n)$ of ${\mathcal L}$
around $t'\in{\mathcal L}$.
Rewriting~(\ref{equation: comparison of fundamental solutions with asymptotic}),
we have
$ Y_{\Sigma}(w,t)^{-1} \, Y_{\Sigma'}(w,t)
=C_{\Sigma,\Sigma'}$,
which is constant in~$t$.
Differentiate it in
$t_1,\dots,t_n$, we have
\[
 -Y_{\Sigma}(w,t)^{-1}\, \frac{\partial Y_{\Sigma}(w,t)}{\partial t_j} \: Y_{\Sigma}(w,t)^{-1}\,Y_{\Sigma'}(w,t)
 +Y_{\Sigma}(w,t)^{-1} \, \frac{\partial Y_{\Sigma'}(w,t)}{\partial t_j} =0,
\]
which is equivalent to the equality
\begin{equation} \label{equation: patching condition of dt coefficient}
 -\frac{\partial Y_{\Sigma}(w,t)}{\partial t_j} Y_{\Sigma}(w,t)^{-1}
 =
 -\frac{\partial Y_{\Sigma'}(w,t)}{\partial t_j} Y_{\Sigma'}(w,t)^{-1}
\end{equation}
in
$\End({\mathcal O}_{\Sigma\cap\Sigma'}^{\oplus r})
\otimes\Omega^1_{{\Sigma}\cap{\Sigma'}}$.
So we get a matrix~$B_j(w,t)$ of single valued functions on $(\Delta_w\setminus\{0\})\times{\mathcal L}$
by patching the matrices~(\ref{equation: patching condition of dt coefficient}).

On the other hand, since the convergence in~(\ref{equation: uniform asymptotic})
is uniform in $t\in{\mathcal L}$,
the differentiation of~(\ref{equation: asymptotic property of fundamental solution})
in $t_j$ provides the asymptotic relation
\[
 \frac{\partial Y_{\Sigma}}{\partial t_j} \exp\left( \int\Lambda(w) \right)
 +Y_{\Sigma} \exp\left( \int\Lambda(w) \right) \int\frac{\partial \Lambda}{\partial t_j}
 \sim
 \frac{\partial P}{\partial t_j}
 \qquad \text{as $w\to 0$ on $\Sigma$}.
\]
Multiplying $w^{mr-r-1}P^{-1}\sim w^{mr-r-1}\exp\big({-}\int\Lambda(w)\big) Y_{\Sigma}^{-1}$
from the right to the above, we get
\begin{equation} \label{equation: asymptotic property of B}
 -w^{mr-r-1}B_j
 =
 w^{mr-r-1}\frac{\partial Y_{\Sigma}}{\partial t_j} Y_{\Sigma}^{-1}
 \sim
 w^{mr-r-1}
 \left( \frac{\partial P}{\partial t_j} P^{-1}
 - P \left( \int\frac{\partial \Lambda}{\partial t_j} \right) P^{-1}
\right)
\end{equation}
on $\Sigma$.
Note that the right-hand side of the above is a matrix of formal power series in $w$
without pole.
So the left-hand side of~(\ref{equation: asymptotic property of B})
is bounded on any~$\Sigma$.
Since $-w^{mr-r-1}B_j$ is also a matrix of single valued functions on
$(\Delta_w\setminus\{0\})\times{\mathcal L}$,
it is holomorphic on $\Delta_w\times{\mathcal L}$.
In other words,
$B_j(w,t)$ is a matrix of meromorphic functions on $\Delta_w\times{\mathcal L}$,
whose pole is of order at most~$mr-r-1$.

Recall that the matrix of $\tilde{\nabla}'$ is given by
\[
 -\frac {\partial Y_{\Sigma}(w,t)} {\partial w} Y_{\Sigma}(w,t)^{-1} {\rm d}w
 =
 A'(w,t)\frac{{\rm d}w}{w^{mr-r}}
\]
since $Y_{\Sigma}$ is a fundamental solution of $\tilde{\nabla}'$.
So we obtain a matrix of differential forms
\[
 A'(w,t)\frac{{\rm d}w}{w^{mr-r}}
 +\sum_{j=1}^N B_j(w,t){\rm d}t_j
\]
which determines a meromorphic connection
\begin{equation*}
 \big(\tilde{\nabla}'\big)^{\mathrm{flat}}
 \colon \
 \tilde{E}'|_{\Delta_w\times {\mathcal L}} \longrightarrow
 \tilde{E}'|_{\Delta_w\times{\mathcal L}} \otimes
 \Omega^1_{\Delta_w\times{\mathcal L}}
 ({\mathcal D}_{\mathcal L}\cap(\Delta_w\times{\mathcal L})).
\end{equation*}
By the definition, $(\tilde{\nabla}')^{\mathrm{flat}}$
is an extension of the relative connection
$\tilde{\nabla}'|_{\Delta_w\times{\mathcal L}}$.

The curvature form
of $\big(\tilde{\nabla}'\big)^{\mathrm{flat}}$ is
\begin{gather*}
 {\rm d} \left({-}\frac{\partial Y_{\Sigma}} {\partial w} Y_{\Sigma}^{-1} {\rm d}w
 -\sum_{j=1}^N \frac{\partial Y_{\Sigma}} {\partial t_j} Y_{\Sigma}^{-1} {\rm d}t_j \right)
 \\
\qquad\quad{} +
\left({-}\frac{\partial Y_{\Sigma}} {\partial w} Y_{\Sigma}^{-1} {\rm d}w
 -\sum_{j=1}^N \frac{\partial Y_{\Sigma}} {\partial t_j} Y_{\Sigma}^{-1} {\rm d}t_j \right)
 \wedge
\left( {-}\frac{\partial Y_{\Sigma}} {\partial w} Y_{\Sigma}^{-1} {\rm d}w
 -\sum_{j=1}^N \frac{\partial Y_{\Sigma}} {\partial t_j} Y_{\Sigma}^{-1} {\rm d}t_j \right)
 \\
\qquad{}=
 -\sum_{j=1}^N \left( \frac{\partial^2 Y_{\Sigma}} {\partial t_j\partial w}
 -\frac{\partial Y_{\Sigma}} {\partial w} Y_{\Sigma}^{-1}
 \frac{\partial Y_{\Sigma}} {\partial t_j} \right) Y_{\Sigma}^{-1} {\rm d}t_j \wedge {\rm d}w\\
 \qquad\quad{}
 -\sum_{j=1}^N \left( \frac{\partial^2 Y_{\Sigma}} {\partial w\partial t_j}
 -\frac{\partial Y_{\Sigma}}{\partial t_j}Y_{\Sigma}^{-1}\frac{\partial Y_{\Sigma}}{\partial w}
 \right) Y_{\Sigma}^{-1} {\rm d}w \wedge {\rm d}t_j
 \\
\qquad\quad{}
 -\sum_{j=1}^N
 \sum_{j'=1}^N \left( \frac{\partial^2Y_{\Sigma}}{\partial t_{j'}\partial t_j}{\rm d}t_{j'}
 -\frac{\partial Y_{\Sigma}} {\partial t_j}Y_{\Sigma}^{-1}
 \frac{\partial Y_{\Sigma}}{\partial t_{j'}} {\rm d}t_{j'}
 \right) \wedge Y_{\Sigma}^{-1} {\rm d}t_j\\
 \qquad\quad{}
 +\frac{\partial Y_{\Sigma}} {\partial w} Y_{\Sigma}^{-1} {\rm d}w \wedge
 \sum_{j=1}^N \frac{\partial Y_{\Sigma}} {\partial t_j} Y_{\Sigma}^{-1} {\rm d}t_j
 \\
\qquad\quad{}
 +\sum_{j=1}^N\frac{\partial Y_{\Sigma}} {\partial t_j} Y_{\Sigma}^{-1} {\rm d}t_j\wedge
 \frac{\partial Y_{\Sigma}} {\partial w} Y_{\Sigma}^{-1} {\rm d}w
 +\sum_{j=1}^N\sum_{j'=1}^N \frac{\partial Y_{\Sigma}}{\partial t_j}Y_{\Sigma}^{-1}
 \frac{\partial Y_{\Sigma}}{\partial t_{j'}} Y_{\Sigma}^{-1} {\rm d}t_j\wedge {\rm d}t_{j'}
=0.
\end{gather*}
So $\big(\tilde{\nabla}'\big)^{\mathrm{flat}}$
is an integrable connection which is an extension of
$\tilde{\nabla}'|_{\Delta_w\times{\mathcal L}}$.

The following proposition
is in fact included in a more general framework by T.~Mochizuki in \cite[Section~20.3]{Mochizuki},
which provides the existence of flat solution with asymptotic property
in a~general setting.

\begin{Proposition} \label{proposition: extendability to integrable connection in w variable}
Let ${\mathcal L}\subset M^{\circ}$ be a submanifold and
let $\big(\tilde{E}',\nabla'\big)|_{\Delta_w\times{\mathcal L}}$ be the restriction of
the family of connections constructed in \eqref{equation: local connection on ramified cover}.
Assume that for each point $t'\in{\mathcal L}$, there is an open neighborhood
${\mathcal L'}$ of $t'$ in ${\mathcal L}$ and a meromorphic integrable connection
\[
\big(\tilde{\nabla}'\big)^{\rm flat} \colon \
\tilde{E}'|_{\Delta_w\times{\mathcal L'}} \longrightarrow
\tilde{E}'|_{\Delta_w\times{\mathcal L'}} \otimes
\Omega^1_{\Delta_w\times{\mathcal L'}}((mr-r)\tilde{x}'),
\]
whose associated relative connection coincides with
$\tilde{\nabla}'|_{\Delta_w\times{\mathcal L'}}$.
Then $\big(\tilde{E}',\nabla'\big)|_{\Delta_w\times{\mathcal L}}$
is a local generalized isomonodromic deformation.
\end{Proposition}

\begin{proof}
We have
$\tilde{E}'|_{\Delta_w\times{\mathcal L'}}\cong{\mathcal O}_{\Delta_w\times{\mathcal L'}}^{\oplus r}$
and we can write
\[
 \big(\tilde{\nabla}'\big)^{\rm flat}
 =
 {\rm d}+ A'(w,t)\frac{{\rm d}w}{w^{mr-r}} +\sum_{j=1}^N B_j(w,t) {\rm d}t_j .
\]
After shrinking ${\mathcal L'}$ if necessary,
we can take a covering
$\{\Sigma=\Gamma_u\times{\mathcal L'}\}$ of $(\Delta_w\setminus\{0\})\times{\mathcal L'}$
with $\Gamma_u$ a sector in $\Delta_w\setminus\{0\}$ and
we can take a fundamental solution
$Y_{\Sigma}(w,t)$ of $\tilde{\nabla'}|_{\Sigma}$
with the uniform asymptotic relation
\begin{equation} \label{equation: family of asymptotic relation}
 Y_{\Sigma}(w,t)
 \exp\left(\int\Lambda(w)\right)
 \sim P(w,t),
 \qquad
 w\to 0, \ w\in\Sigma.
\end{equation}
Since $(\nabla')^{\mathrm{flat}}$ is an integrable connection
extending $\tilde{\nabla}'|_{\Delta_w\times{\mathcal L'}}$,
we can take a fundamental solution
$Y^{\mathrm{flat}}_{\Sigma}(w,t)$ of $(\nabla')^{\mathrm{flat}}$
on $\Sigma$
satisfying
$Y^{\mathrm{flat}}_{\Sigma}(w,t')=Y_{\Sigma}(w,t')$.
We can write
\begin{equation}
\label{equation: comparison between isomonodromic solution and relative asymptotic solution}
 Y^{\mathrm{flat}}_{\Sigma}(w,t)
 =
 Y_{\Sigma}(w,t)
 C(t),
 \qquad
 (w,t)\in \Sigma,
\end{equation}
for a matrix $C(t)=(c_{ij}(t))$ of holomorphic functions in $t\in {\mathcal L'}$ such that
$C(t')=I_r$ is the identity matrix.
Differentiating
(\ref{equation: comparison between isomonodromic solution and relative asymptotic solution})
in $t_j$, we have
\[
 \frac{\partial Y^{\rm flat}_{\Sigma}} {\partial t_j}
 =\frac{\partial Y_{\Sigma}} {\partial t_j} C(t)
 + Y_{\Sigma} \frac{\partial C(t)} {\partial t_j},
 \]
from which we have
\begin{gather} \label{equation: difference of dt-coefficient with asymptotic}
 Y_{\Sigma}(w,t) \frac{\partial C(t)} {\partial t_j}
 C(t)^{-1} Y_{\Sigma}(w,t)^{-1}
 =
 \frac{\partial Y^{\rm flat}_{\Sigma}(w,t)} {\partial t_j} Y^{\mathrm{flat}}(w,t)^{-1}
 -\frac{\partial Y_{\Sigma}(w,t)} {\partial t_j} Y_{\Sigma}(w,t)^{-1}.
\end{gather}
Since $Y^{\mathrm{flat}}(w,t)$ is a fundamental solution matrix of
$(\nabla')^{\mathrm{flat}}$, we have
\begin{equation} \label{equation: integrability induce meromorphic}
 \frac{\partial Y^{\rm flat}_{\Sigma}(w,t)} {\partial t_j}
 \, Y^{\rm flat}_{\Sigma}(w,t)^{-1}
 =-B_j(w,t).
\end{equation}
On the other hand,
since the asymptotic relation (\ref{equation: family of asymptotic relation})
is uniform in $t\in{\mathcal L}$,
we have the asymptotic relation
\[
 \frac {\partial Y_{\Sigma}} {\partial t_j} \,
 \exp \left( \hbox{$\int\Lambda(w)$} \right)
 +Y_{\Sigma} \, \exp \left( \hbox{$\int\Lambda(w)$} \right) \,
 \frac { \partial } {\partial t_j} \left( \hbox{$\int\Lambda(w)$} \right)
 \ \sim \ \frac {\partial P} {\partial t_j}
\]
on $\Sigma$.
Multiplying
$\left( Y_{\Sigma} \, \exp \left( \hbox{$\int\Lambda(w)$} \right) \right)^{-1}\sim P^{-1}$
from the right to the above, we have
\begin{equation}
\label{equation: differentiation of asymptotic relation}
 \frac{\partial Y_{\Sigma}} {\partial t_j} Y_{\Sigma}^{-1}
 \ \sim \
 \frac{\partial P}{\partial t_j} P^{-1}
 -P \: \frac{\partial}{\partial t_j} \left( \hbox{$\int\Lambda(w)$} \right) P^{-1}
 \qquad
 (w\to 0)
\end{equation}
on $\Sigma$.
Using the equality (\ref{equation: difference of dt-coefficient with asymptotic})
and substituting (\ref{equation: integrability induce meromorphic})
and (\ref{equation: differentiation of asymptotic relation}),
we have the asymptotic relation
\begin{gather*}
 \exp \left( \int\Lambda(w) \right) ^{-1} \frac{\partial C(t)} {\partial t_j} C(t)^{-1}
 \exp \left( \int\Lambda(w) \right)
\sim
 P^{-1} Y_{\Sigma} \frac{\partial C(t)} {\partial t_j}
 C(t)^{-1} Y_{\Sigma}^{-1} P
 \\
 \qquad{}=
 P^{-1} \left(
 \frac{\partial Y^{\rm flat}_{\Sigma}} {\partial t_j} (Y^{\mathrm{flat}}_{\Sigma})^{-1}
 -\frac{\partial Y_{\Sigma}} {\partial t_j} Y_{\Sigma}^{-1} \right) P
 \sim
 -P^{-1} B_j P
 -P^{-1} \frac{\partial P}{\partial t_j}
 +\frac{\partial }{\partial t_j} \left( \int\Lambda(w) \right)
\end{gather*}
on $\Sigma$.
So
$
w^N \exp \big( \int\Lambda(w) \big)^{-1}
\frac{\partial C(t)} {\partial t_j} C(t)^{-1}
\exp \big( \int\Lambda(w)\big)$
is bounded on $\Sigma$
for a large $N$,
because~$B_j$ is a matrix of meromorphic functions in $w$,.

Choose a point $(w_0,t')\in \Sigma$.
After replacing a frame of $E'_0$,
we can write
\begin{equation*}
 \int\Lambda(w)
 =
 \begin{pmatrix}
 A_1(w) & 0 & 0 \\
 0 & \ddots & 0 \\
 0 & 0 & A_s(w)
 \end{pmatrix}\frac{1}{w^{mr-r-1}},
 \qquad
 A_k(w)
 =
 \begin{pmatrix}
 a^{(k)}_1(w) & 0 & 0 \\
 0 & \ddots & 0 \\
 0 & 0 & a^{(k)}_{m_k}(w)
 \end{pmatrix},
\end{equation*}
such that
$a^{(k)}_p(w)=a^{(k)}_p(0)+b^{(k)}_{p,1}w+\cdots
 +b^{(k)}_{p,mr-r-2}w^{mr-r-2}+b^{(k)}_{p,mr-r-1} w^{mr-r-1}\log w$
satisfies $a^{(k)}_p(0)\neq a^{(l)}_q(0)$ for $(k,p)\neq (l,q)$
and that
$
 \rho_k=\mathrm{Re}\big( w_0^{-mr+r+1} a^{(k)}_p(0) \big)
$
holds for $1\leq p\leq m_k$ at $t'$
with $\rho_1>\rho_2>\cdots>\rho_s$.
Write
\begin{equation} \label{equation: differential of transition matrix}
 \frac{\partial C(t)} {\partial t_j} C(t)^{-1}
 =:\tilde{C}(t)=
 \begin{pmatrix}
 \tilde{C}_{11}(t) & \cdots & \tilde{C}_{1s}(t) \\
 \vdots & \ddots & \vdots \\
 \tilde{C}_{s1}(t) & \cdots & \tilde{C}_{ss}(t)
 \end{pmatrix},
\end{equation}
where $\tilde{C}_{kl}(t)$ is a matrix of size $(m_k,m_l)$.
Then we have
\begin{gather} \label{equation: adjoint of differential of transition matrix}
w^N\exp \left( \int\Lambda(w) \right)^{-1} \frac{\partial C(t)} {\partial t_j}
 C(t)^{-1} \exp \left( \int\Lambda(w) \right)
 \\
=
 w^N
 \begin{pmatrix}
 \exp\big(\frac{-A_1(w)}{w^{mr-r-1}}\big)\tilde{C}_{11}(t)\exp\big(\frac{A_1(w)}{w^{mr-r-1}}\big)
 & \cdots &
 \exp\big(\frac{-A_1(w)}{w^{mr-r-1}}\big)\tilde{C}_{1s}(t)\exp\big(\frac{A_s(w)}{w^{mr-r-1}}\big)
 \\
 \vdots & \ddots & \vdots \\
 \exp\big(\frac{-A_s(w)}{w^{mr-r-1}}\big)\tilde{C}_{1s}(t)\exp\big(\frac{A_1(w)}{w^{mr-r-1}}\big)
 & \cdots &
 \exp\big(\frac{-A_s(w)}{w^{mr-r-1}}\big)\tilde{C}_{ss}(t)\exp\big(\frac{A_s(w)}{w^{mr-r-1}}\big)
 \end{pmatrix},
\nonumber
\end{gather}
which is bounded on $\Sigma$.

Suppose that
$\tilde{C}_{kl}(t)\neq 0$ for $k>l$.
Then the growth order of the $(k,l)$ minor of (\ref{equation: adjoint of differential of transition matrix})
along the ray $\{ \theta w_0 \mid 0<\theta\leq 1\}$
is the same as
\[
 (\theta w_0)^N\exp \left( \mathrm{Re}
 \big( (\theta w_0)^{-mr+r+1}(A_k(0)-A_l(0)) \big) \right) \tilde{C}_{k,l}(t)
 =\theta^N w_0^N {\rm e}^{\frac{\rho_k-\rho_l}{\theta^{mr-r-1}}}
 \tilde{C}_{kl}(t).
\]
Since $\rho_k-\rho_l>0$, it is divergent as $\theta\to 0$, which is a contradiction.
If we write
\[
 \tilde{C}_{kk}(t)=
 \begin{pmatrix}
 \tilde{c}^{(k)}_{11}(t) & \cdots & \tilde{c}^{(k)}_{1m_k}(t) \\
 \vdots & \ddots & \vdots \\
 \tilde{c}^{(k)}_{m_k1} & \cdots & \tilde{c}^{(k)}_{m_km_k}(t)
 \end{pmatrix},
\]
then we have
\begin{gather}
 w^N\exp\big({-}w^{-mr+r+1}A_k(w)\big)\tilde{C}_{kk}(t)\exp\big(w^{-mr+r+1}A_k(w)\big)
\nonumber \\
 =
 w^N
 \begin{pmatrix}
 \tilde{c}^{(k)}_{11}(t) & \cdots &
 {\rm e}^{w^{-mr+r+1} (a^{(k)}_1(w)-a^{(k)}_{m_k}(w)) } \tilde{c}^{(k)}_{1m_k}(t)
 \\
 \vdots & \ddots & \vdots \\
 {\rm e}^{w^{-mr+r+1} (a^{(k)}_{m_k}(w)-a^{(k)}_1(w)) } \tilde{c}^{(k)}_{m_k1}(t)
 & \cdots & \tilde{c}_{m_km_k}(t)
 \end{pmatrix}.\label{equation: diagonal block of the differential of transition matrix}
\end{gather}
Suppose that $\tilde{c}^{(k)}_{pq}(t)\neq 0$ for $p\neq q$.
Since $a^{(k)}_p(0)\neq a^{(k)}_q(0)$,
we can find $\delta\neq 0$ with
$|\delta|$ small such that
$\big\{ \theta {\rm e}^{\sqrt{-1}\delta} w_0 \mid 0<\theta \leq 1 \big\}$
is contained in $\Gamma_u$
and that either
$\mathrm{Re}\big( \frac {a^{(k)}_p(0)-a^{(k)}_q(0)}
{(w_0{\rm e}^{\sqrt{-1}\delta})^{mr-r-1}} \big)>0$
or
$\mathrm{Re}\big( \frac {a^{(k)}_p(0)-a^{(k)}_q(0)}
{(w_0{\rm e}^{\sqrt{-1}\delta})^{mr-r-1}} \big)<0$ holds.
After replacing $\delta$ with $\pm\delta$, we may assume the inequality
$\mathrm{Re}\big( \frac {a^{(k)}_p(0)-a^{(k)}_q(0)}
{(w_0{\rm e}^{\sqrt{-1}\delta})^{mr-r-1}} \big)>0$.
Then the growth order of the $(p,q)$-entry of~(\ref{equation: diagonal block of the differential of transition matrix})
is the same as
$(w_0\theta)^N \exp\big( \frac {a^{(k)}_p(0)-a^{(k)}_q(0)}
{(w_0{\rm e}^{\sqrt{-1}\delta})^{mr-r-1}\theta^{mr-r-1}} \big)$,
which is divergent along
$\big\{ \theta {\rm e}^{\sqrt{-1}\delta} w_0 \mid 0<\theta \leq 1 \big\}$
as $\theta\to 0$.
Since~(\ref{equation: diagonal block of the differential of transition matrix})
is bounded on $\Gamma_u\times{\mathcal L}$, it is a contradiction.
So $\tilde{C}_{kk}(t)$ is a diagonal matrix for any~$k$.

Thus we have proved that
the matrix $\tilde{C}(t)$ given in (\ref{equation: differential of transition matrix})
is a block upper triangular matrix in the sense that
$\tilde{C}_{kl}(t)=0$ for $k>l$
and that $\tilde{C}_{kk}(t)$ is are diagonal matrices for $1\leq k\leq s$.
We will show that $C(t)$ is also a block upper triangular matrix.
Consider the Taylor expansion
\begin{equation} \label{equation: Taylor expansion of transition matrix}
 C(t)=
 \sum_{i_1,\dots,i_N} C_{i_1,\dots,i_n}t_1^{i_1}\cdots t_n^{i_n}
\end{equation}
around $t=t'$.
Suppose that one of $C_{i_1,\dots,i_n}$ is not block upper triangular
and put
\[
 l=\min \big\{i_1+\cdots+i_n \mid
 \text{$C_{i_1,\dots,i_n}$ is not a block upper triangular matrix}\big\}.
\]
By the minimality of $l$,
$C(t) \pmod{(t_1,\dots,t_n)^{l-1}}$ is a block upper triangular matrix
and so is $C(t)^{-1} \pmod{(t_1,\dots,t_n)^{l-1}}$.
Differentiating (\ref{equation: Taylor expansion of transition matrix}),
$\frac{\partial C(t)} {\partial t_j} \pmod{(t_1,\dots,t_n)^{l-1}}$
is not a block upper triangular matrix for some $j$.
So we can see that
$\frac{\partial C(t)} {\partial t_j}
C(t)^{-1} \pmod{(t_1,\dots,t_n)^{l-1}}$
is not a block upper triangular matrix of the above form, which is a contradiction.

Thus $C(t)$ is also a block upper triangular matrix of the above form.
Let $C_{\mathrm{diag}}(t)$ be the diagonal part of $C(t)$.
Then we have
\begin{gather}
 Y^{\mathrm{flat}}_{\Sigma}(w,t) \exp\left( \int\Lambda \right)
 =Y_{\Sigma}(w,t) \exp \left( \int\Lambda \right)
 \exp \left(- \int\Lambda \right) C(t) \exp\left( \int\Lambda \right)\nonumber\\
 \hphantom{Y^{\mathrm{flat}}_{\Sigma}(w,t) \exp\left( \int\Lambda \right)}{}
 \sim P(w,t) C_{\mathrm{diag}}(t)\label{equation: asymptotic relation for flat solution}
\end{gather}
on $\Sigma$.
If we take another sector
$\Sigma'=\Gamma_{u'}\times{\mathcal L}$
and a fundamental solution $Y^{\mathrm{flat}}_{\Sigma'}$ of $\big(\tilde{\nabla}'\big)^{\mathrm{flat}}$
satisfying $Y^{\mathrm{flat}}_{\Sigma'}=Y_{\Sigma'}C'(t)$ with $C'(t')=I_r$,
we have
\begin{equation} \label{equation: asymptotic relation for another flat solution}
 Y^{\mathrm{flat}}_{\Sigma'}(w,t) \exp\left( \int\Lambda \right)
 \sim P(w,t) C'_{\mathrm{diag}}(t)
\end{equation}
on $\Sigma'$.
Since both of $Y^{\mathrm{flat}}_{\Sigma}$ and $Y^{\mathrm{flat}}_{\Sigma'}$
are fundamental solutions of the integrable connection
$\big(\tilde{\nabla}'\big)^{\mathrm{flat}}$,
we can write
$Y^{\mathrm{flat}}_{\Sigma'}=Y^{\mathrm{flat}}_{\Sigma}K$ for a constant matrix~$K$.
Combining~(\ref{equation: asymptotic relation for flat solution}) and
(\ref{equation: asymptotic relation for another flat solution}),
we have
\begin{gather*}
 C_{\mathrm{diag}}(t)^{-1}C'_{\mathrm{diag}}(t)
 \sim
 \exp\left({-} \int\Lambda \right) (Y^{\mathrm{flat}}_{\Sigma})^{-1}
 Y^{\mathrm{flat}}_{\Sigma'} \exp\left( \int\Lambda \right)\\
 \hphantom{C_{\mathrm{diag}}(t)^{-1}C'_{\mathrm{diag}}(t)}{}
 =
 \exp\left({-} \int\Lambda \right) K \exp\left( \int\Lambda \right)
\end{gather*}
on $\Sigma\cap\Sigma'$.
Since the diagonal entries of the right-hand side of the above are those of
$K$, which are constant in $t$,
we can see that the left-hand side of the above is a constant matrix.
Since $C'_{\mathrm{diag}}(t')=C'(t')=I_r=C(t')=C_{\mathrm{diag}}(t')$,
we have $C'_{\mathrm{diag}}(t)=C_{\mathrm{diag}}(t)$.

Thus, the replacement of the formal transform $P(w,t)$ with
$P(w,t)C_{\mathrm{diag}}(t)$ is independent of~$\Sigma$.
So the replacement of $Y_{\Sigma}$ with $Y^{\mathrm{flat}}_{\Sigma}$ on each $\Sigma$
satisfies the condition of
Definition~\ref{definition: local generalized isomonodromic deformation}.
\end{proof}

Summarizing the above arguments, we get the following theorem,
which is the local version of a main consequence of the Jimbo--Miwa--Ueno theory.
It is the significance of
the formulation of generalized isomonodromic deformation
introduced in Section \ref{section: generalized isomonodromy equation} later.

\begin{Theorem}[Jimbo--Miwa--Ueno {\cite[Theorems 3.1 and 3.3]{Jimbo-Miwa-Ueno}}]
\label{theorem: Jimbo-Miwa-Ueno equation}
For a submanifold ${\mathcal L}$ of~$M^{\circ}$,
the restriction
$\big(\tilde{E}',\tilde{\nabla}'\big)|_{\Delta_w\times{\mathcal L}}$
of the family of connections to $\Delta_w\times{\mathcal L}$
is a local generalized isomonodromic deformation
if and only if
for each point $t'$ of ${\mathcal L}$, there is a neighborhood ${\mathcal L'}$
of $t'$ in ${\mathcal L}$ and a meromorphic integrable connection
$\big(\tilde{\nabla}'\big)^{\rm flat}\colon \tilde{E}'|_{\Delta_w\times{\mathcal L'}} \longrightarrow
\tilde{E}'|_{\Delta\times{\mathcal L'}}\otimes\Omega_{\Delta_w\times{\mathcal L'}}((mr-r)\tilde{x}')$
whose associated relative connection coincides with
$\tilde{\nabla}'|_{\Delta_w\times{\mathcal L'}}$.
\end{Theorem}

\begin{Remark}\quad
\begin{itemize}\itemsep=0pt
\item[(i)]
In the precise setting of \cite{Jimbo-Miwa-Ueno}, each sector is taken sufficiently large
so that the asymptotic solution $Y_{\Sigma}$ is determined uniquely.
Furthermore, the choice of formal transforms is also included in the system of
differential equation in \cite[Theorems 3.1 and~3.3]{Jimbo-Miwa-Ueno}.
\item[(ii)]
In our setting of Theorem~\ref{theorem: Jimbo-Miwa-Ueno equation},
there are ambiguities in the choice of asymptotic solutions~$Y_{\Sigma}(w,t)$
and in the choice of the formal transforms~$P(w,t)$.
Our statement of Theorem~\ref{theorem: Jimbo-Miwa-Ueno equation}
is a consequence of Proposition~\ref{proposition: extendability to integrable connection in w variable},
which is essentially the result by T.~Mochizuki in~\cite{Mochizuki}.
\item[(iii)] We introduce Definition \ref{definition: local generalized isomonodromic deformation}
based on the naive meaning of Stokes data,
but it will be better to explain the Stokes data by
using the notion of local system with Stokes filtration
as in \cite[Section 4.6]{Babbitt-Varadarajan}
or \cite[Chapter 3]{Mochizuki}.
\item[(iv)]
Theorem \ref{theorem: Jimbo-Miwa-Ueno equation}
is also mentioned in the appendix of \cite{Boalch-1}.
\item[(v)]
We can see from (\ref{equation: asymptotic property of B})
that the $dt_j$-coefficient of $\big(\tilde{\nabla}'\big)^{\mathrm{flat}}$
has a pole of order $mr-r-1$.
\end{itemize}
\end{Remark}

\begin{Proposition}\label{proposition: ramified local generalized isomonodromic deformation}
For the family of connections $\big(\tilde{E}',\tilde{\nabla}'\big)$ on $\Delta_w\times M^{\circ}$
which is constructed from $\big(\tilde{E},\tilde{\nabla}\big)|_{\Delta_z\times M^{\circ}}$
in (\ref{equation: local connection on ramified cover})
and for a submanifold ${\mathcal L}$ of $M^{\circ}$,
$\big(\tilde{E}',\tilde{\nabla}'\big)|_{\Delta_w\times{\mathcal L}}$
can be extended to an integrable connection if and only if
$\big(\tilde{E},\tilde{\nabla}\big)|_{\Delta_z\times{\mathcal L}}$
can be extended to an integrable meromorphic connection on $\Delta_z\times{\mathcal L}$.
\end{Proposition}

\begin{proof}
Assume that $\tilde{\nabla}'|_{\Delta_w\times{\mathcal L}}$
can be extended to an integrable connection
$\big(\tilde{\nabla}'\big)^{\rm flat}$.
Note that there is a canonical inclusion
$S(w)\colon p_{\mathcal L}^*\big(\tilde{E}|_{\Delta_z\times{\mathcal L}}\big)\hookrightarrow
\tilde{E}'|_{\Delta_w\times{\mathcal L}}$
which is Galois equivariant and compatible with the connections.
Consider the pullback
$S(w)^*\big(\tilde{\nabla}'\big)^{\mathrm{flat}}$.
If we write
\[
 \big(\tilde{\nabla}'\big)^{\mathrm{flat}}
 =
 {\rm d}+A'(w)\frac{{\rm d}w}{w^{mr-r}}
 +\sum_{j=1}^N B'_j(w){\rm d}t_j,
\]
then the connection $S(w)^*\big(\tilde{\nabla}'\big)^{\mathrm{flat}}$
on $p_{\mathcal L}^*\big(\tilde{E}|_{\Delta_z\times{\mathcal L}}\big)$ is given by
\begin{equation}
\label{equation: matrix of integrable connection to be descend}
 {\rm d}+S(w)^{-1}\left( \frac{\partial S(w)}{\partial w}+\frac{A'(w)S(w)}{w^{mr-r}}\right){\rm d}w
 +\sum_{j=1}^N S(w)^{-1} \left(\frac {\partial S(w)} {\partial t_j}+B'_j(w)S(w) \right){\rm d}t_j.
\end{equation}
Note that there is a canonical action of
$\mathrm{Gal}(\Delta_w/\Delta_z)$ on
$p_{\mathcal L}^*\big(\tilde{E}|_{\Delta_z\times{\mathcal L}}\big)\cong
p_{\mathcal L}^*\big({\mathcal O}^{\mathrm{hol}}_{\Delta_z\times{\mathcal L}}\big)^{\oplus r}$,
which induces a canonical Galois action on
$\End\big(p_{\mathcal L}^*\big(\tilde{E}|_{\Delta_z\times{\mathcal L}}\big)\big)\otimes
p_{\mathcal L}^*\Omega^1_{\Delta_z\times{\mathcal L}}(m\tilde{x})$.
If we denote the matrix of $\tilde{\nabla}|_{\Delta_z\times{\mathcal L}}$
by $\frac{A(z){\rm d}z}{z^m}$, then we have
\[
 \frac{A(z){\rm d}z}{z^m}-\tilde{\nu}_0(z)I_r=
 S(w)^{-1}\left( \frac{\partial S(w)}{\partial w}+\frac{A'(w)S(w)}{w^{mr-r}}\right){\rm d}w,
\]
which is Galois invariant.
On the other hand, the ${\rm d}t_j$-coefficient of
(\ref{equation: matrix of integrable connection to be descend})
may not be Galois invariant.
So we put
\[
 B_j: =
 -\int\frac {\partial \tilde{\nu}_0(z)} {\partial t_j} I_r
 \ + \ \frac{1}{r}\sum_{\sigma\in\mathrm{Gal}(\Delta_w/\Delta_z)}
 \left[ S(w)^{-1} \left(\frac {\partial S(w)} {\partial t_j}+B'_j(w)S(w) \right) \right]^{\sigma}.
\]
Then $B_j$ is $\mathrm{Gal}(\Delta_w/\Delta_z)$-invariant and
becomes a matrix of meromorphic functions on $\Delta_z\times{\mathcal L}$.
If we put
\[
 \tilde{\nabla}^{\mathrm{flat}}:=
 {\rm d}+\frac{A(z){\rm d}z}{z^m}+\sum_{j=1}^N B_j {\rm d}t_j,
\]
then $\tilde{\nabla}^{\mathrm{flat}}$ defines a meromorphic integrable connection on
$\tilde{E}|_{\Delta_z\times{\mathcal L}}$.
The converse is im\-me\-diate.
\end{proof}

We can see by a calculation that
\begin{equation} \label{equation: formal fundamental solution of ramified connection}
 \Psi(z,t):=
 \begin{pmatrix}
 1 & z^{\frac{1}{r}} & \cdots & z^{\frac{r-1}{r}} \\
 1 & \zeta_r z^{\frac{1}{r}} & \cdots & \zeta_r^{r-1}z^{\frac{r-1}{r}} \\
 \vdots & \vdots & \ddots & \vdots \\
 1 & \zeta_r^{r-1}z^{\frac{1}{r}} & \cdots & \zeta_r^{(r-1)^2}z^{\frac{r-1}{r}}
 \end{pmatrix}^{-1}
 {\rm e}^{\int\tilde{\nu}_0(z,t)}
 \exp\left({-}\int\Lambda\big(z^{\frac{1}{r}},t\big)\right)
\end{equation}
becomes a fundamental solution of
\begin{gather*}
{\rm d}+
 \begin{pmatrix}
 \tilde{\nu}_0(z) & z\tilde{\nu}_{r-1}(z) & \cdots & z\tilde{\nu}_1(z) \\
 \tilde{\nu}_1(z) & \tilde{\nu}_0(z)+\frac{dz}{rz} & \cdots & z\tilde{\nu}_2(z) \\
 \vdots & \vdots & \ddots & \vdots \\
 \tilde{\nu}_{r-1}(z) & \tilde{\nu}_{r-2}(z) & \cdots & \tilde{\nu}_0(z)+\frac{(r-1){\rm d}z}{rz}
 \end{pmatrix}
 \\
{} +
 \sum_{j=1}^N
 \begin{pmatrix}
 \frac{\partial}{\partial t_j} \int\tilde{\nu}_0(z) &
\! zw^{-r+1}\frac{\partial}{\partial t_j} \int w^{r-1}\tilde{\nu}_{r-1}(z)\!
 & \cdots & zw^{-1}\frac{\partial}{\partial t_j} \! \int w\tilde{\nu}_1(z) \vspace{1mm}\\
 w^{-1}\frac{\partial}{\partial t_j} \int w\tilde{\nu}_1(z) &
 \frac{\partial}{\partial t_j} \int\tilde{\nu}_0(z)
 & \cdots & \!zw^{-2}\frac{\partial}{\partial t_j} \int w^2\tilde{\nu}_2(z) \\
 \vdots & \vdots & \ddots & \vdots \\
 w^{-r+1}\frac{\partial}{\partial t_j} \int w^{r-1}\tilde{\nu}_{r-1}(z) &
 w^{-r+2}\frac{\partial}{\partial t_j} \int w^{r-2}\tilde{\nu}_{r-2}(z) & \cdots
 & \frac{\partial}{\partial t_j} \int\tilde{\nu}_0(z)
 \end{pmatrix} {\rm d}t_j
\end{gather*}
which is a matrix form of the integrable formal connection
\begin{gather*}
 \nabla_{\tilde{\nu}(w)+\sum\frac{\partial} {\partial t_j}(\int \tilde{\nu}) {\rm d}t_j}
 \colon \
 {\mathcal O}_{\mathcal L}[[w]] \longrightarrow
 {\mathcal O}_{\mathcal L}[[w]]\otimes \Omega_{\Delta_z\times{\mathcal L}}(m \tilde{x})
 \\
 f(w) \mapsto {\rm d}f(w)+
 \left( \tilde{\nu}(w)+\sum_{j=1}^N\frac{\partial} {\partial t_j}
 \left(\int \tilde{\nu}(w)\right) {\rm d}t_j \right)f(w)
\end{gather*}
with respect to the basis
$1,w,\dots,w^{r-1}$
of the free module
$ {\mathcal O}_{\mathcal L}[[w]]$
over $ {\mathcal O}_{\mathcal L}[[z]]$.
On the other hand,
recall that the elementary transform
$p_{\mathcal L}^*\big(\tilde{E},\tilde{\nabla}\big)|_{\Delta_w\times{\mathcal L}}
\mapsto \big(\tilde{E}',\tilde{\nabla}'\big)$
is given by the rational gauge transform
$S(w)\colon p_{\mathcal L}^*\big(\tilde{E}|_{\Delta_z\times {\mathcal L}}\big)
\hookrightarrow \tilde{E}'$.
If we put
$\overline{\Sigma}:=p_{M^{\circ}}(\Sigma)$, then
$p_{M^{\circ}}|_{\Sigma}\colon \Sigma \stackrel{\sim}\longrightarrow
\overline{\Sigma}$
is an isomorphism if $\Sigma$ is sufficiently small.
Substituting $z=w^{\frac{1}{r}}$ in the solution
$Y_{\Sigma}(w)$,
we can get a~fundamental solution
\[
 Z_{\overline{\Sigma}}(z,t):=S\big(z^{\frac{1}{r}},t\big)^{-1}Y_{\Sigma}\big(z^{\frac{1}{r}},t\big) {\rm e}^{\int\tilde{\nu}_0(z,t)}
\]
of $\tilde{\nabla}|_{\Delta_z\times{\mathcal L}}$.
Using the asymptotic property
$Y_{\Sigma}\exp\big( \int\Lambda(w) \big)\sim P(w)$
and the equality~(\ref{equation: formal transform using z parameter}),
we get the asymptotic relation
\begin{gather}
 Z_{\overline{\Sigma}}(z) \Psi(z)^{-1}
 =
 S\big(z^{\frac{1}{r}}\big)^{-1}Y_{\Sigma}\big(z^{\frac{1}{r}}\big)\nonumber\\
 \hphantom{ Z_{\overline{\Sigma}}(z) \Psi(z)^{-1}=}{}
 \times
 \exp\left( \int\Lambda(z^{\frac{1}{r}}) \right)
 \begin{pmatrix}
 1 & z^{\frac{1}{r}} & \cdots & z^{\frac{r-1}{r}} \\
 1 & \zeta_r z^{\frac{1}{r}} & \cdots & \zeta_r^{r-1} z^{\frac{r-1}{r}} \\
 \vdots & \vdots & \ddots & \vdots \\
 1 & \zeta_r^{r-1} z^{\frac{1}{r}} & \cdots & \zeta_r^{(r-1)^2} z^{\frac{r-1}{r}}
 \end{pmatrix}
 \sim Q(z)\label{equation: asymptotic solution on z-space}
\end{gather}
on $(z,t)\in \overline{\Sigma}$.
For another $\overline{\Sigma}'$, we have
\[
 Z_{\overline{\Sigma}'}(z,t)=Z_{\overline{\Sigma}}(z,t)C_{\overline{\Sigma},\overline{\Sigma}'}(t),
\]
where $C_{\overline{\Sigma},\overline{\Sigma}'}(t)=C_{\Sigma,\Sigma'}(t)$.
So we can in fact describe the Stokes data on
$\Delta_z$, without using a~ramified cover,
in the sense of patching data in \cite[Theorem 4.5.1]{Babbitt-Varadarajan}.

\begin{Definition}\label{definition: ramified local generalized isomonodromic deformation}
We say that a family of connections
$\big(\tilde{E},\tilde{\nabla}\big)|_{\Delta_z\times {\mathcal L}}$
over a submanifold ${\mathcal L}\subset M^{\circ}$
is a local generalized isomonodromic deformation,
if for each $t'\in{\mathcal L}$,
we can take an open neighborhood ${\mathcal L}_{t'}$ of $t'$ in ${\mathcal L}$,
a replacement of the formal transform $Q(z,t)$
given in (\ref{equation: setting of Q(z)})
and a~replacement of the covering $\{\overline{\Sigma}\}$
of $(\Delta_z\setminus\{0\})\times {\mathcal L}_{t'}$
such that
\begin{itemize}\itemsep=0pt
\item[(i)]
there is a fundamental solution
$Z_{\overline{\Sigma}}(z,t)$ of $\tilde{\nabla}$
on each $\overline{\Sigma}$
with the asymptotic property~(\ref{equation: asymptotic solution on z-space})
and
\item[(ii)]
all the Stokes matrices $C_{\overline{\Sigma},\overline{\Sigma}'}(t)$ are constant in
$t\in {\mathcal L}_{t'}$.
\end{itemize}
\end{Definition}

\begin{Corollary}\label{corollary: rigorous ramified local generalized isomonodromic deformation}
 For a submanifold ${\mathcal L}$ of $M^{\circ}$,
 the family $\big(\tilde{E},\tilde{\nabla}\big)|_{\Delta_z\times {\mathcal L}}$
 is a local isomonodromic deformation in the sense of
 Definition~{\rm \ref{definition: ramified local generalized isomonodromic deformation}}
 if and only if
 for each point $t'$ of ${\mathcal L}$, there is a~neighborhood~${\mathcal L'}$ of $t'$
 in ${\mathcal L}$ and an integrable meromorphic connection
 $\tilde{\nabla}\colon \tilde{E}|_{\Delta_z \times{\mathcal L'}}
 \longrightarrow \tilde{E}|_{\Delta_z\times{\mathcal L'}}\otimes
 \Omega^1_{\Delta_z\times{\mathcal L'}}(m\tilde{x})$
 whose associated relative connection coincides with
 $\big(\tilde{E},\tilde{\nabla}\big)|_{\Delta_z\times {\mathcal L'}}$.
\end{Corollary}

\begin{proof}Assume that there is an integrable connection
$\tilde{\nabla}^{\mathrm{flat}}$ on $\tilde{E}|_{\Delta_z\times{\mathcal L'}}$
which is an extension of $\tilde{\nabla}|_{\Delta_z\times{\mathcal L'}}$
as in Proposition~\ref{proposition: ramified local generalized isomonodromic deformation}.
Then there is a canonically induced integrable connection
$\big(\tilde{\nabla}'\big)^{\mathrm{flat}}$ on $\tilde{E}'|_{\Delta_w\times{\mathcal L'}}$.
If we take a fundamental solution
$Y^{\mathrm{flat}}(w,t)$ of $\big(\tilde{\nabla}'\big)^{\mathrm{flat}}$ as in the proof of
Proposition~\ref{proposition: extendability to integrable connection in w variable},
then
\[
 Z^{\mathrm{flat}}_{\overline{\Sigma}}(z):=
 S\big(z^{\frac{1}{r}}\big)^{-1}Y^{\mathrm{flat}}\big(z^{\frac{1}{r}}\big){\rm e}^{\int\tilde{\nu}_0(z)}
\]
is a fundamental solution of $\tilde{\nabla}^{\mathrm{flat}}$.
Since
$Y_{\Sigma} \exp\big( \int\Lambda(w) \big)
\sim Y^{\mathrm{flat}} \exp\big( \int\Lambda(w) \big)
C_{\mathrm{diag}}(t)^{-1}$
as in the proof of Proposition~\ref{proposition: extendability to integrable connection in w variable},
we can see from~(\ref{equation: asymptotic solution on z-space})
that the asymptotic relation
\[
 Z^{\mathrm{flat}}_{\overline{\Sigma}}(z,t)
 C_{\mathrm{diag}}(t)^{-1}
 \Psi(z,t)^{-1}
 \sim Q(z,t)
\]
holds on $(z,t)\in \overline{\Sigma}$.
Differentiating the above in~$t_j$, we have
\begin{gather} \label{equation: differentiation of asymptotic on z-space}
 \frac{\partial Z^{\mathrm{flat}}_{\overline{\Sigma}}} {\partial t_j}
 \big(Z^{\mathrm{flat}}_{\overline{\Sigma}}\big)^{-1} Q(z)
 -Q(z)\Psi(z)
 \frac {\partial C_{\mathrm{diag}}} {\partial t_j}
 C_{\mathrm{diag}}^{-1} \Psi(z)^{-1}
 +Q(z) \frac{\partial \Psi(z)} {\partial t_j} \Psi(z)^{-1}
 \sim
 \frac{\partial Q(z)}{\partial t_j}
\end{gather}
on $(z,t)\in \overline{\Sigma}$.
Note that
$-\frac{\partial Z^{\mathrm{flat}_{\overline{\Sigma}}}} {\partial t_j }
\big(Z^{\mathrm{flat}}_{\overline{\Sigma}}\big)^{-1}$
is $\mathrm{Gal}(\Delta_w/\Delta_z)$-invariant
because it is the ${\rm d}t_j$-co\-ef\-fi\-cient of $\tilde{\nabla}^{\mathrm{flat}}$.
We can see that
$-\frac {\partial \Psi(z)} {\partial t_j} \Psi(z)^{-1}$ is also
$\mathrm{Gal}(\Delta_w/\Delta_z)$-invariant
because it is the ${\rm d}t_j$-coefficient of the formal connection
$\nabla_{\nu(w)+\sum\int\frac{\partial \nu_0}{\partial t_j}{\rm d}t_j}$.
The transform $Q(z)$ is also $\mathrm{Gal}(\Delta_w/\Delta_z)$-invariant
as a matrix of formal power series.
So, from the asymptotic relation~(\ref{equation: differentiation of asymptotic on z-space}),
we can see that
$\Psi(z) \frac {\partial C_{\mathrm{diag}}} {\partial t_j}
C_{\mathrm{diag}}^{-1} \Psi(z)^{-1}$
is $\mathrm{Gal}(\Delta_w/\Delta_z)$-invariant.
If the Galois transform
$\sigma\in \mathrm{Gal}(\Delta_w/\Delta_z)$
is given by
$\sigma(w)=\zeta_r^kw$,
then the Galois transform by $\sigma$ on
$\Psi(z,t)$ in
(\ref{equation: formal fundamental solution of ramified connection})
is given by
\begin{gather*}
 \Psi(z,t)^{\sigma}
 =
 \begin{pmatrix}
 1 & \zeta^k z^{\frac{1}{r}} & \cdots & \zeta^{(r-1)k}z^{\frac{r-1}{r}} \\
 1 & \zeta_r^{k+1} z^{\frac{1}{r}} & \cdots & \zeta_r^{(r-1)(k+1)}z^{\frac{r-1}{r}} \\
 \vdots & \vdots & \ddots & \vdots \\
 1 & \zeta_r^{r-1+k}z^{\frac{1}{r}} & \cdots & \zeta_r^{(r-1)(r-1+k)}z^{\frac{r-1}{r}}
 \end{pmatrix}^{-1}
 {\rm e}^{\int\nu_0(z,t)}
 \exp\left({-} \int\Lambda(\zeta_r^k z^{\frac{1}{r}},t) \right)
 \\
\hphantom{\Psi(z,t)^{\sigma}=}{} =
 \begin{pmatrix}
 1 & z^{\frac{1}{r}} & \cdots & z^{\frac{r-1}{r}} \\
 1 & \zeta_r z^{\frac{1}{r}} & \cdots & \zeta_r^{r-1}z^{\frac{r-1}{r}} \\
 \vdots & \vdots & \ddots & \vdots \\
 1 & \zeta_r^{r-1}z^{\frac{1}{r}} & \cdots & \zeta_r^{(r-1)^2}z^{\frac{r-1}{r}}
 \end{pmatrix}^{-1}
 P_{\sigma}
 {\rm e}^{\int\nu_0(z,t)}
 P_{\sigma}^{-1}
 \exp\left({-} \int\Lambda(z^{\frac{1}{r}},t) \right)
 P_{\sigma},
\end{gather*}
where $P_{\sigma}$ is the permutation matrix
defined by
$P_{\sigma}=(e_{k+1},e_{k+2},\dots,e_r,e_1,\dots,e_k)$
for the canonical basis
$e_1,\dots,e_r$ of $\mathbb{C}^r$.
So the equation of Galois invariance
\[
 \Psi(z)^{\sigma} \frac{\partial C_{\mathrm{diag}}}{\partial t_j} C_{\mathrm{diag}}^{-1}
 \big(\Psi(z)^{\sigma}\big)^{-1}
 =
 \Psi(z) \frac {\partial C_{\mathrm{diag}}} {\partial t_j}
 C_{\mathrm{diag}}^{-1} \Psi(z)^{-1}
\]
deduces the equalities
\[
 P_{\sigma} \, \frac {\partial C_{\mathrm{diag}}} {\partial t_j}
 C_{\mathrm{diag}}^{-1} \, P_{\sigma}^{-1}
 =
\frac {\partial C_{\mathrm{diag}}} {\partial t_j}
 C_{\mathrm{diag}}^{-1}
\]
for cyclic permutation matrices
$P_{\sigma}$ corresponding to $\sigma\in \mathrm{Gal}(\Delta_w/\Delta_z)$.
Thus all the diagonal entries of
$\frac {\partial C_{\mathrm{diag}}} {\partial t_j}
C_{\mathrm{diag}}^{-1}$
are the same,
which implies that all the diagonal entries of $C_{\mathrm{diag}}(t)$ are the same.
After replacing $Q(z)$ with $Q(z) C_{\mathrm{diag}}(t)$,
we have the asymptotic relation
 \[
 Z^{\mathrm{flat}}_{\Sigma}(z)
 \Psi(z)^{-1}
 \sim \ Q(z)
 \qquad
 \text{as $z\to 0$ on $\overline{\Sigma}$}
\]
for all $\overline{\Sigma}$.
After replacing
$Z_{\overline{\Sigma}}(z,t)$ with
$Z^{\mathrm{flat}}_{\Sigma}(z,t)$
and shrinking ${\mathcal L}$ if necessary,
all the Stokes matrices
$\big\{ C_{\overline{\Sigma},\overline{\Sigma}'} \big\}$
become constant.
So $\big(\tilde{E},\tilde{\nabla}\big)|_{\Delta_z\times {\mathcal L}}$
becomes a local generalized isomonodromic deformation.

Conversely, assume that $\big(\tilde{E},\tilde{\nabla}\big)|_{\Delta_z\times {\mathcal L}}$
is a~local generalized isomonodromic deformation.
For the fundamental solution
$Z_{\overline{\Sigma}}(z,t)$ of $\tilde{\nabla}|_{\overline{\Sigma}}$
given in Definition~\ref{definition: ramified local generalized isomonodromic deformation},
\[
 Y_{\Sigma}(z,t)=S(w,t)Z_{\overline{\Sigma}}(z,t){\rm e}^{-\int \nu_0(z,t)}
\]
becomes a fundamental solution of
$\tilde{\nabla}'|_{\Sigma}$.
So we have
$C_{\Sigma,\Sigma'}(t)=C_{\overline{\Sigma},\overline{\Sigma}'}$
which is constant in~$t$.
Thus
$\tilde{\nabla}'|_{\Delta_w\times{\mathcal L}}$ is a local generalized isomonodromic deformation.
By Theorem~\ref{theorem: Jimbo-Miwa-Ueno equation},
we can extend $\tilde{\nabla}'|_{\Delta_w\times{\mathcal L}}$
to an integrable connection after shrinking ${\mathcal L}$ at each point.
So $\big(\tilde{E},\tilde{\nabla}\big)|_{\Delta_z\times {\mathcal L}}$
can be extended to an integrable connection by
Proposition~\ref{proposition: ramified local generalized isomonodromic deformation}.
\end{proof}

\begin{Remark}The achievement of the construction of the generalized isomonodromic deformation
by Bremer and Sage in \cite{Bremer-Sage-2} is based on
the Jimbo--Miwa--Ueno theory, which becomes
Corollary~\ref{corollary: rigorous ramified local generalized isomonodromic deformation}
in our setting.
\end{Remark}

\section {Horizontal lift of a universal family of connections}\label{section: horizontal lift}

We will extend the notion of local generalized isomonodromic deformation
in Section~\ref{section: local analytic theory}
to a global setting on the moduli space of connections.
Its differential equation is given as a subbundle of the tangent bundle of the moduli space,
which satisfies the integrability condition.
For its construction, we introduce the notion of horizontal lift of
a universal family of connections.

Let the notations
${\mathcal T}$, ${\mathcal C}$, $\lambda$, $\tilde{\mu}$, $\tilde{\nu}$,
$M^{\balpha}_{{\mathcal C},{\mathcal D}}(\lambda,\tilde{\mu},\tilde{\nu})$
be as in Section~\ref{section: construction of the moduli space}.
There is an \'etale surjective morphism
$\tilde{M}\longrightarrow M^{\balpha}_{{\mathcal C},{\mathcal D}}(\lambda,\tilde{\mu},\tilde{\nu})$
such that there is a universal family of
connections
$\big(\tilde{E},\tilde{\nabla},\tilde{l},\tilde{\ell},\tilde{\mathcal V}\big)$
on ${\mathcal C}_{\tilde{M}}$.
We may assume that
the generic $\tilde{\nu}$-ramified structure
$\tilde{\mathcal V}$ is given by a factorized $\tilde{\nu}$-ramified structure
$\big(\tilde{V}_k,\tilde{\vartheta}_k,\tilde{\varkappa}_k\big)_{0\leq k\leq r-1}$.

For a Zariski open subset
${\mathcal T'}\subset {\mathcal T}$,
we put $\tilde{M}':=\tilde{M}\times_{\mathcal T}{\mathcal T'}$.
Take a vector field
$v\in H^0({\mathcal T}',T_{\mathcal T}|_{\mathcal T'})$.
If we put
${\mathcal T}'[v]:={\mathcal T}'\times\Spec\mathbb{C}[\epsilon]$
with $\epsilon^2=0$,
then $v$ is characterized by a morphism
$I_v\colon {\mathcal T}'[v]\longrightarrow{\mathcal T'}$
whose restriction to ${\mathcal T'}$ is the identity.
Put
$\tilde{M}'[v]:=\tilde{M}'\times\Spec\mathbb{C}[\epsilon]$
and consider the fiber product
${\mathcal C}_{\tilde{M}'[v]}:=
{\mathcal C}\times_{\mathcal T} (\tilde{M}'\times\Spec\mathbb{C}[\epsilon])$
with respect to the projection
${\mathcal C}\longrightarrow {\mathcal T}$
and the composition
$\tilde{M}'\times\Spec\mathbb{C}[\epsilon]
\longrightarrow {\mathcal T'}\times\Spec\mathbb{C}[\epsilon]
\xrightarrow{I_v} {\mathcal T'}\hookrightarrow {\mathcal T}$.

We define a divisor ${\mathcal D}'$ on ${\mathcal C}$
by setting
\[
 {\mathcal D}':=\sum_{i=1}^{n_{\mathrm{un}}} \big(m^{\mathrm{un}}_i-1\big)\tilde{x}_i^{\mathrm{un}}
 +\sum_{i=1}^{n_{\mathrm{ram}}} \big(m^{\mathrm{ram}}_i-1\big)\tilde{x}_i^{\mathrm{ram}}.
\]
Consider the sheaf of differential forms
$\Omega^1_{{\mathcal C}_{\tilde{M}'[v]/\tilde{M}'}}$
with respect to the composition of trivial projections
\[
 {\mathcal C}_{\tilde{M}'[v]}
 ={\mathcal C}\times_{\mathcal T}\big(\tilde{M}'\times\Spec\mathbb{C}[\epsilon]\big)
 \longrightarrow
 \tilde{M}'\times\Spec\mathbb{C}[\epsilon]\longrightarrow M'.
\]
Take a local section $z_i^{\mathrm{log}}$
(resp.\ $z_i^{\mathrm{un}}$, $z_i^{\mathrm{ram}}$)
of ${\mathcal O}_{{\mathcal C}_{\mathcal T'}}$
which is a local defining equation of
$\tilde{x}_i^{\mathrm{log}}$
(resp.\ $\tilde{x}_i^{\mathrm{un}}$, $\tilde{x}_i^{\mathrm{ram}}$).
We write the induced local section of
${\mathcal O}_{{\mathcal C}_{\tilde{M}'[v]}}$
by the same symbol $z_i^{\mathrm{log}}$
(resp.\ $z_i^{\mathrm{un}}$, $z_i^{\mathrm{ram}}$).
Let
$\tilde{\Omega}_v$
be the coherent subsheaf of
$\Omega^1_{{\mathcal C}_{\tilde{M'}[v]}/\tilde{M'}}({\mathcal D}_{\tilde{M'}[v]})$
which is locally defined by
\begin{gather}
 \tilde{\Omega}_v
 =
 {\mathcal O}_{{\mathcal C}_{\tilde{M}'[v]}} \frac{dz_i^{\mathrm{reg}}} {z_i^{\mathrm{log}}}
 +{\mathcal O}_{{\mathcal C}_{\tilde{M}'}} d\epsilon
 \qquad \text{around $(\tilde{x}_i^{\mathrm{log}})_{\tilde{M}'[v]}$},
 \nonumber\\
 \tilde{\Omega}_v
=
 {\mathcal O}_{{\mathcal C}_{\tilde{M}'[v]}}
 \frac{dz_i^{\mathrm{un}}} {(z_i^{\mathrm{un}})^{m^{\mathrm{un}}_i}}
 +{\mathcal O}_{{\mathcal C}_{\tilde{M}'}}
 \frac{d\epsilon} {(z_i^{\mathrm{un}})^{m^{\mathrm{un}}_i-1}}
 \qquad
\text{around $(\tilde{x}_i^{\mathrm{un}})_{\tilde{M}'[v]}$},
 \nonumber\\
 \tilde{\Omega}_v
 =
 {\mathcal O}_{{\mathcal C}_{\tilde{M}'[v]}}
 \frac{dz_i^{\mathrm{ram}}} {(z_i^{\mathrm{ram}})^{m^{\mathrm{ram}}_i}}
 +{\mathcal O}_{{\mathcal C}_{\tilde{M}'}}
 \frac{d\epsilon} {(z_i^{\mathrm{ram}})^{m^{\mathrm{ram}}_i-1}}
 \qquad
\text{around $(\tilde{x}_i^{\mathrm{ram}})_{\tilde{M}'[v]}$}.\label{equation: definition of total differential}
\end{gather}

\begin{Definition}\label{definition: horizontal lift in one vector case}
We say that
$\big({\mathcal E}^v,\nabla^v, l^v,\ell^v,{\mathcal V}^v\big)$
is a global horizontal lift of
$\big(\tilde{E},\tilde{\nabla},\tilde{l},\tilde{\ell},\tilde{\mathcal V}\big)_{\tilde{M}'}$
with respect to $v$, if
\begin{itemize}\itemsep=0pt
\item[(i)]
${\mathcal E}^v$ is a vector bundle on
${\mathcal C}_{\tilde{M}'[v]}$ of rank $r$,
\item[(ii)]
$\nabla^v \colon {\mathcal E}^v \longrightarrow
{\mathcal E}^v \otimes \tilde{\Omega}_v$
is a morphism such that
$\nabla^v(fa)=a\otimes {\rm d}f+f\nabla^v(a)$ for
$f\in{\mathcal O}_{{\mathcal C}_{\tilde{M'}[v]}}$, $a\in {\mathcal E}^v$
and that the matrix
$\Gamma^v=\tilde{A}(z){\rm d}z+B(z) \, {\rm d}\epsilon$
corresponding to $\nabla^v$ with respect to
a local frame $e_0,\dots,e_{r-1}$ of ${\mathcal E}^v|_{U[v]}$
defined by
$
 (\nabla^v(e_0),\dots,\nabla^v(e_{r-1}))
 =
 (e_0,\dots,e_{r-1})\Gamma^v
$
satisfies
$\tilde{A}(z)\in M_r({\mathcal O}_{U[v]}({\mathcal D}_{\tilde{M}'[v]}\cap U))$
and
$B(z)\in M_r ( {\mathcal O}_U ({\mathcal D}'_{\tilde{M}'}\cap U))$
\item[(iii)]
$\nabla^v$ satisfies the integrability condition
${\rm d}\Gamma^v+\Gamma^v\wedge\Gamma^v=0$,
which means that
the equality
$\frac{\partial \tilde{A}}{\partial \epsilon}{\rm d}z\wedge {\rm d}\epsilon
={\rm d}B(z)\wedge d\epsilon+[\tilde{A}(z),B(z)]{\rm d}z\wedge {\rm d}\epsilon$
holds,
\item[(iv)]
for the relative connection
$\overline{\nabla^v}\colon
{\mathcal E}^v\longrightarrow{\mathcal E}^v\otimes
\Omega^1_{{\mathcal C}_{\tilde{M}'[v]/\tilde{M}'[v]}}\big({\mathcal D}_{\tilde{M}'[v]}\big)$
induced by $\nabla^v$,
\begin{enumerate}\itemsep=0pt
\item[(a)]
$l^v=(l^v_k)_{0\leq k\leq r-1}$ is a logarithmic $\lambda$-parabolic structure
on $\big({\mathcal E}^v,\overline{\nabla^v}\big)$
such that the subsheaf $\ker\big( {\mathcal E}^v\to
{\mathcal E}^v|_{({\mathcal D}_{\mathrm{log}})_{\tilde{M'}[v]}}/l^v_k\big)$
of ${\mathcal E}^v$ is preserved by $\nabla^v$ for $0\leq k\leq r-1$,
\item[(b)]
$\ell^v=(\ell^v_k)_{0\leq k\leq r-1}$ is a generic unramified
$I_v^*\tilde{\mu}$-parabolic structure on
$\big({\mathcal E}^v,\overline{\nabla^v}\big)$
such that the subsheaf
$\ker\big( {\mathcal E}^v\to
{\mathcal E}^v|_{({\mathcal D}_{\mathrm{un}})_{\tilde{M'}[v]}}/\ell^v_k\big)$
of ${\mathcal E}^v$
is preserved by $\nabla^v$ for $0\leq k\leq r-1$,
\item[(c)]
${\mathcal V}^v=\big(V^v_k,\vartheta^v_k,\varkappa^v_k\big)_{0\leq k\leq r-1}$
is a factorized $I_v^*\tilde{\nu}$-ramified structure on
$\big({\mathcal E}^v,\overline{\nabla^v}\big)$
such that the subsheaf
$\ker \big( {\mathcal E}^v \to
{\mathcal E}^v|_{({\mathcal D}_{\mathrm{ram}})_{\tilde{M}'[v]}}/V^v_k \big)$
of ${\mathcal E}^v$
is preserved by $\nabla^v$,
\end{enumerate}
\item[(v)]
$\big({\mathcal E}^v,\overline{\nabla^v}, l^v,\ell^v,{\mathcal V}^v\big)
\otimes{\mathcal O}_{\tilde{M'}[v]}/(\epsilon)
\cong\big(\tilde{E},\tilde{\nabla},\tilde{l},\tilde{\ell},\tilde{\mathcal V}\big)_{\tilde{M}'}$
holds.
\end{itemize}
\end{Definition}

We will prove the existence and uniqueness of the horizontal lift
in the above definition.
For its proof, we will show the local existence and the uniqueness of
the horizontal lift.

\begin{Definition}
Let $U$ be an open subset of ${\mathcal C}_{\tilde{M'}}$
such that $\tilde{E}|_U\cong{\mathcal O}_U^{\oplus r}$
and let $U[v]$ be the open subscheme of ${\mathcal C}_{\tilde{M}'[v]}$
whose underlying set is the same as $U$.
We say that
$\big({\mathcal E}^v_U,\nabla^v_U, l^v_U,\ell^v_U,{\mathcal V}^v_U\big)$
is a local horizontal lift of
$\big(\tilde{E}|_U,\tilde{\nabla}|_U,\tilde{l}|_U,\tilde{\ell}|_U,\tilde{\mathcal V}|_U\big)$
with respect to $v$, if
\begin{enumerate}\itemsep=0pt
\item[(i)]
${\mathcal E}^v_U$ is a vector bundle on $U[v]$ of rank $r$,
\item[(ii)]
$\nabla^v_U\colon {\mathcal E}^v_U\longrightarrow
{\mathcal E}^v_U\otimes\tilde{\Omega}_v|_{U[v]}$
is an integrable connection in the sense of
Definition \ref{definition: horizontal lift in one vector case} (ii) and (iii),
\item[(iii)] $\big(l^v_U,\ell^v_U,{\mathcal V}^v_U\big)$ satisfies the same condition
as (a), (b), (c) of Definition \ref{definition: horizontal lift in one vector case} and
for the induced relative connection $\overline{\nabla^v}$ on ${\mathcal E}^v$,
\[
\big({\mathcal E}^v_U,\overline{\nabla^v}_U, l^v_U,\ell^v_U,{\mathcal V}^v_U\big)
\otimes{\mathcal O}_{\tilde{M'}[v]}/(\epsilon)
\cong\big(\tilde{E}|_U,\tilde{\nabla}|_U,\tilde{l}|_U,\tilde{\ell}|_U,\tilde{\mathcal V}|_U\big)
\]
holds.
\end{enumerate}
\end{Definition}

\begin{Lemma}[logarithmic local horizontal lift]
\label{lemma: local horizontal lift of regular singular connection}
Let $U$ be an affine open subset of ${\mathcal C}_{\tilde{M}'}$
such that
$\tilde{E}|_U\cong{\mathcal O}_U^{\oplus r}$
and that ${\mathcal D}_{\tilde{M'}}\cap U=\big(\tilde{x}^{\mathrm{log}}_i\big)_{\tilde{M}'}\cap U$
for some $i$, which is defined by the equation $z_U=0$
for a section $z$ of ${\mathcal O}_{{\mathcal C}_{\mathcal T'}}$
on a Zariski open subset of ${\mathcal C}_{\mathcal T'}$.
Then there exists a local horizontal lift
$\big({\mathcal E}^v_U,\nabla^v_U, l^v_U\big)$
of
$\big(\tilde{E}|_U,\tilde{\nabla}|_U,\tilde{l}|_U\big)$
with respect to~$v$,
which is unique up to an isomorphism.
\end{Lemma}

\begin{proof}
Note that $\big(\tilde{\ell}|_U,\tilde{\mathcal V}|_U\big)$ is nothing in this case.
Put $\tilde{x}:=\big(\tilde{x}^{\mathrm{log}}_i\big)_{\tilde{M}'}\cap U$.
For a suitable choice of a frame
$e_0,\dots,e_{r-1}$ of $\tilde{E}|_U\cong{\mathcal O}_U^{\oplus r}$,
we may assume that $\tilde{l}_k\cap U$ is given by
$\langle e_k|_{\tilde{x}},\dots,e_{r-1}|_{\tilde{x}} \rangle$.
With respect to the frame $e_0,\dots,e_{r-1}$ of $\tilde{E}|_U$, we can write
$\tilde{\nabla}|_U={\rm d}+A(z){\rm d}z/z$,
where $A(z)$ is a matrix with values in ${\mathcal O}_U$
such that $A(0)$ is a lower triangular matrix with the diagonal entries
$\lambda^{(i)}_0,\dots,\lambda^{(i)}_{r-1}$.
Take a lift $\tilde{A}(z)$ of $A(z)$ as a matrix with values in
${\mathcal O}_{U[v]}$ such that
$\tilde{A}(0)$ is a lower triangular matrix with the diagonal entries
$\lambda^{(i)}_0,\dots,\lambda^{(i)}_{r-1}$.
After replacing $\tilde{A}(z)$, we may assume that the $d\epsilon$-coefficient
of each entry of $d\tilde{A}(z)$ in
$\Omega^1_{U[v]/\tilde{M}'}={\mathcal O}_{U[v]}{\rm d}z\oplus{\mathcal O}_U {\rm d}\epsilon$
vanishes.
Then $\nabla^v_U:={\rm d}+\tilde{A}(z){\rm d}z/z$ defines an integrable connection
on ${\mathcal E}^v_U:={\mathcal O}_{U[v]}^{\oplus r}$,
which preserves the parabolic structure $l^v_U$
on ${\mathcal E}^v_U$ defined by
$l^v_{U,k}=\langle e_k|_{\tilde{x}_{U[v]}},\dots,e_{r-1}|_{\tilde{x}_{U[v]}}\rangle$.

Assume that
$\big({\mathcal E'}_U,\nabla'_U,l'_U\big)$ is another local horizontal lift of
$\big(\tilde{E}|_U,\tilde{\nabla}|_U,\tilde{l}|_U\big)$.
Then we have ${\mathcal E'}_U\cong{\mathcal O}_{U[v]}^{\oplus r}$
and we can write $\nabla'_U={\rm d}+\tilde{A}'(z){\rm d}z/z+B'(z){\rm d}\epsilon$.
After replacing the frame $e_0,\dots,e_{r-1}$
of ${\mathcal E'}_U\cong{\mathcal O}_{U[v]}^{\oplus r}$,
we may assume that $l'_U$ is given by
$l'_{U,k}=\langle e_k|_{\tilde{x}_{U[v]}},\dots,e_{r-1}|_{\tilde{x}_{U[v]}}\rangle$.
Then $\tilde{A}'(0)$ is a lower triangular matrix and
$B'(0)$ is also lower triangular by the condition~(a) of
Definition~\ref{definition: horizontal lift in one vector case}.
Since $\nabla'_U$ is a lift of $\tilde{\nabla}|_U$,
we can write $\tilde{A}'(z)=\tilde{A}(z)+\epsilon C'(z)$,
with~$C'(0)$ a~lower triangular matrix whose diagonal entries are zero.
The integrability condition of~$\nabla'$ yields
$C'(z){\rm d}z/z={\rm d}B'(z)+[\tilde{A}(z),B'(z)]{\rm d}z/z$.
Applying the transform $I_r-\epsilon B'(z)$ to the connection~$\nabla'_U$,
the matrix of connection becomes
\begin{gather*}
 (I_r+\epsilon B'(z)){\rm d}(I_r-\epsilon B'(z))
 +(I_r+\epsilon B'(z))\big( (\tilde{A}(z)+\epsilon C'(z)){\rm d}z/z+B'(z){\rm d}\epsilon \big)(I_r-\epsilon B'(z))
 \\
 =\tilde{A}{\rm d}z/z
 +\epsilon \big( C'(z){\rm d}z/z-{\rm d}B'(z)-[\tilde{A}(z),B'(z)]{\rm d}z/z \big)
 -B'(z){\rm d}\epsilon+B'(z){\rm d}\epsilon
 =\tilde{A}(z){\rm d}z/z.
\end{gather*}
So $I_r-\epsilon B'(z)$ transforms
$\big({\mathcal E'}_U,\nabla'_U,l'_U\big)$
to
$\big({\mathcal E}^v_U,\nabla^v_U,l^v_U\big)$.
The transform $I_r-\epsilon B'(z)$ also preserves the parabolic structures
on the both sides.
Since the transform is uniquely determined by the ${\rm d}\epsilon$-coefficient,
we can see the uniqueness of the transform.
\end{proof}

The following lemma is essentially given in
\cite[Theorem~6.2]{Inaba-Saito}.

\begin{Lemma}[unramified irregular singular local horizontal lift]
\label{lemma: local horizontal lift of unramified connection}
Let $U$ be an affine open subset of ${\mathcal C}_{\tilde{M}'}$
such that
$\tilde{E}|_U\cong{\mathcal O}_U^{\oplus r}$,
${\mathcal D}_{\tilde{M'}}\cap U=m^{\mathrm{un}}_i\big(\tilde{x}^{\mathrm{un}}_i\big)_{\tilde{M}'}\cap U$
for some $i$
and that $\big(\tilde{x}^{\mathrm{un}}_i\big)_{\tilde{M}'}\cap U$ is defined by the equation $z_U=0$
for a section $z$ of ${\mathcal O}_{{\mathcal C}_{\mathcal T'}}$
on a Zariski open subset of ${\mathcal C}_{\mathcal T'}$.
Then there exists a local horizontal lift
$\big({\mathcal E}^v_U,\nabla^v_U, \ell^v_U\big)$
of
$\big(\tilde{E}|_U,\tilde{\nabla}|_U,\tilde{\ell}|_U\big)$
with respect to $v$,
which is unique up to an isomorphism.
\end{Lemma}

\begin{proof}
We put $\tilde{x}:=\big(\tilde{x}^{\mathrm{un}}_i\big)_{\tilde{M}'}\cap U$
and $m:=m^{\mathrm{un}}_i$.
Write
\[
 \mu^{(i)}_k(z)=\sum_{j=0}^{m-2} a_{k,j}(z) \frac {{\rm d}z} {z^m} +c_k \frac {{\rm d}z}{z}.
\]
We can write
\[
 I_v^*(a_{k,j})=
 a_{k,j}+\epsilon b_{k,j}
 \in {\mathcal O}_{{\mathcal T'}[v]}
 ={\mathcal O}_{{\mathcal T'}\times\Spec\mathbb{C}[\epsilon]}
 ={\mathcal O}_{\mathcal T'}\oplus \epsilon{\mathcal O}_{\mathcal T'}.
\]
We express the above equality by
\[
 I_v^*\mu^{(i)}_k(z)=
 \mu^{(i)}_k(z)+\epsilon \mu^{(i)}_{k,v}(z),
 \hspace{30pt}
 \mu_{k,v}(z)=
 \sum_{j=0}^{m-2}b_{k,j}z^j \frac{{\rm d}z}{z^m}.
\]
Take a local frame $e_0,\dots,e_{r-1}$ of $\tilde{E}|_U$
such that $\tilde{\ell}_k\cap U$ is given by
$\langle e_k|_{\tilde{x}},\dots,e_{r-1}|_{\tilde{x}}\rangle$.
After a~suitable replacement of the frame $e_0,\dots,e_{r-1}$,
we can write
$\tilde{\nabla}|_U={\rm d}+A(z){\rm d}z/z^m$
such that
$A(z){\rm d}z/z^m \pmod{z^{2m-1}{\rm d}z/z^m}$ is the diagonal matrix
with the diagonal entries
$\mu^{(i)}_0,\dots,\mu^{(i)}_{r-1}$.
We can take a matrix $\tilde{A}(z)$ with entries in ${\mathcal O}_{U[v]}$
which is a lift of $A(z)$ such that
$\partial \tilde{A}/\partial \epsilon=0$
and that
$\tilde{A}(z){\rm d}z/z^m\big|_{z^{2m-1}=0}$
is a diagonal matrix with the diagonal entries
$\mu^{(i)}_0,\dots,\mu^{(i)}_{r-1}$.
Set
\[
 B(z):=\int
 \begin{pmatrix}
 \mu^{(i)}_{0,v}(z) & 0 & 0 \\
 0 & \ddots & 0 \\
 0 & 0 & \mu^{(i)}_{r-1,v}(z)
 \end{pmatrix},
 \qquad
 C(z){\rm d}z/z^m
 :=
 {\rm d}B+[A,B]\frac{{\rm d}z}{z^m}.
\]
Then $\nabla^v_U:={\rm d}+(\tilde{A}(z)+\epsilon C(z)){\rm d}z/z^m+B{\rm d}\epsilon$
defines an integrable connection on
${\mathcal E}^v_U={\mathcal O}_{U[v]}^{\oplus r}$.
By construction, the connection~$\nabla^v_U$
preserves the parabolic structure $\ell^v_U$ on
${\mathcal E}^v_U$ induced by $e_0,\dots,e_{r-1}$.
So we can see the existence of the local horizontal lift
$\big({\mathcal E}^v,\nabla^v,\ell^v\big)$.

Assume that $\big({\mathcal E}'_U,\nabla'_U,\ell'_U\big)$
is another local horizontal lift of
$\big(\tilde{E}|_U,\tilde{\nabla}|_U,\tilde{\ell}|_U\big)$.
Note that ${{\mathcal E}'_U\cong{\mathcal O}_{U[v]}^{\oplus r}}$.
So we may write
$\nabla'_U={\rm d}+(\tilde{A}(z)+\epsilon C'(z)){\rm d}z/z^m+ B'(z){\rm d}\epsilon$
with $C'(z)\equiv C(z) \pmod{z^m}$.
The integrability condition
\begin{equation} \label{equation: integrability condition in local unramified case}
 C'(z){\rm d}z/z^m
 :=
 {\rm d}B'+[A,B']\frac{{\rm d}z}{z^m}
\end{equation}
yields $\big[A,z^{m-1}B'\big]\equiv 0 \pmod{z^{m-1}}$.
Since $A(z)|_{z^{2m-1}=0}$ is a diagonal matrix whose constant term $A(0)$ has distinct eigenvalues,
we can see that $z^{m-1}B'|_{z^{m-1}=0}$ is also a diagonal matrix.
Looking at~(\ref{equation: integrability condition in local unramified case}) again
and using $C(z)\equiv C'(z) \pmod{z^m}$,
we can see that
$B'(z)|_{z^{2m-1}=0}$ is also a diagonal matrix
with the diagonal entries $\mu^{(i)}_{0,v},\dots,\mu^{(i)}_{r-1,v}$.
So $B(z)-B'(z)$ is a matrix of regular functions on $U$,
whose constant term is diagonal.
We can see by the same calculation as in the proof of
Lemma \ref{lemma: local horizontal lift of regular singular connection}
that the automorphism
$I_r+\epsilon(B-B')$ transforms
$\nabla'_U$ to $\nabla^v_U$
and it also preserves the parabolic structures on the both sides.
We can see that such an automorphism is unique,
because it is determined by the $d\epsilon$-coefficient
of $\nabla'_U$.
\end{proof}

\begin{Lemma}[existence of ramified irregular singular local horizontal lift]\label{lemma: existence of local horizontal lift of ramified connection}
Let $U$ be an affine open subset of ${\mathcal C}_{\tilde{M}'}$
such that
$\tilde{E}|_U\cong{\mathcal O}_U^{\oplus r}$,
${\mathcal D}_{\tilde{M'}}\cap U=
m^{\mathrm{ram}}_i\big(\tilde{x}^{\mathrm{ram}}_i\big)_{\tilde{M}'}\cap U$
for some $i$ and that
$\big(\tilde{x}^{\mathrm{ram}}_i\big)_{\tilde{M}'}\cap U$
is defined by the equation $z_U=0$
for a section $z$ of ${\mathcal O}_{{\mathcal C}_{\mathcal T'}}$
on a Zariski open subset of ${\mathcal C}_{\mathcal T'}$.
Then there exists a local horizontal lift
$\big({\mathcal E}^v_U,\nabla^v_U, {\mathcal V}^v_U\big)$
of
$\big(\tilde{E}|_U,\tilde{\nabla}|_U,{\mathcal V}|_U\big)$
with respect to~$v$.
\end{Lemma}

\begin{proof}
Write $\tilde{x}=\big(\tilde{x}^{\mathrm{ram}}_i\big)_{\tilde{M}'}\cap U$
and $m=m^{\mathrm{ram}}_i$.
We denote the pullback of $\nu$ via the trivial first projection
${\mathcal T}'[v]\longrightarrow {\mathcal T'}\hookrightarrow {\mathcal T}$
by the same symbol $\nu$.
As in the proof of Lemma~\ref{lemma: local horizontal lift of unramified connection},
we express
\begin{gather*}
 I_v^*\nu(w)=
 \nu(w)+\epsilon \nu_v(w),
 \qquad\!\!
 \nu(w)
 =
 \sum_{k=0}^{r-1}\sum_{j=0}^{m-1}\!a_{k,j}z^j w^k \frac{{\rm d}z}{z^m},
 \qquad\!\!
 \nu_v(w)
 =
 \sum_{k=0}^{r-1}\sum_{j=0}^{m-2}\!b_{k,j}z^j w^k \frac{{\rm d}z}{z^m},
\end{gather*}
where $a_{1,0}\in {\mathcal O}_{\tilde{M}'}^{\times}$ and
$a_{k,m-1}=0$ for $1\leq k\leq r-1$.

We choose a local frame
$e_0,\dots,e_{r-1}$ of $\tilde{E}|_U$
whose restriction to $(2m-1)\tilde{x}$ corresponds to
$1,w,\dots,w^{r-1}$
via the isomorphism
$\tilde{E}|_{(2m-1)\tilde{x}}\cong {\mathcal O}_{\tilde{M}'}[w]/\big(w^{(2m-1)r}\big)$
given by Proposition \ref{proposition: normalization of formal type over finite scheme}
in the case $q=2m-1$.
Let
\[
 N\colon \ \tilde{E}|_U \longrightarrow \tilde{E}|_U
\]
be the homomorphism
defined by the representation matrix
\begin{equation} \label{equation: representation matrix of N}
 \begin{pmatrix}
 0 & 0 & \cdots & 0 & z \\
 1 & 0 & \cdots & 0 & 0 \\
 \vdots & \ddots & \ddots & \vdots & \vdots \\
 0 & \cdots & 1 & 0 & 0 \\
 0 & \cdots & 0 & 1 & 0
 \end{pmatrix}
\end{equation}
with respect to the basis
$e_0,\dots,e_{r-1}$ of $\tilde{E}|_U$.
As in the proof of Theorem~\ref{thm: smoothness of the moduli space},
we can construct homomorphisms
$\theta\colon \tilde{E}_{\alpha}|_{m\tilde{x}}^{\vee}
\longrightarrow \tilde{E}_{\alpha}|_{m\tilde{x}}$
and
$\kappa\colon E_{\alpha}|_{m\tilde{x}}
\longrightarrow E_{\alpha}|_{m\tilde{x}}^{\vee}$
which satisfy
$^t\theta=\theta$, $^t\kappa=\kappa$ and
$N|_{m\tilde{x}}=\theta\circ\kappa$.
We may assume that $(\tilde{\theta}_k)$ and $(\tilde{\kappa}_k)$ are induced by
$\theta$ and $\kappa$, respectively.

Write
$\tilde{\nabla}|_{\bar{U}_{\alpha}}={\rm d}+A(z)\frac{{\rm d}z}{z^m}$
with respect to the frame $e_0,\dots,e_{r-1}$
of $\tilde{E}|_U \cong{\mathcal O}_U^{\oplus r}$.
Since
$\big(\tilde{E}|_{(2m-1)\tilde{x}},\tilde{\nabla}|_{(2m-1)\tilde{x}}\big)\cong
\big({\mathcal O}_{\tilde{M}'}[w]/\big(w^{(2m-1)r}\big),\nabla_{\nu}\big)$
as in Proposition~\ref{proposition: normalization of formal type over finite scheme},
we can write
\begin{equation} \label{equation: normalization of a connection to (2m-1)x}
 A(z)=
 \sum_{k=0}^{r-1}\sum_{l=0}^{m-1} a_{k,l} z^l N^k
 +
 z^{m-1}R_r+z^{2m-1}A'(z)
\end{equation}
for some matrix $A'(z)$ of regular functions,
where we are putting
\begin{equation} \label{equation: definition of R}
 R_r:=
 \begin{pmatrix}
 0 & 0 & \cdots & 0 \\
 0 & \frac{1}{r} & \cdots & 0 \\
 \vdots & \vdots & \ddots & \vdots \\
 0 & 0 & \cdots & \frac{r-1}{r}
 \end{pmatrix}.
\end{equation}

Set ${\mathcal E}^v_U:={\mathcal O}_{U[v]}^{\oplus r}$ with
the identification
${\mathcal E}^v_U\otimes{\mathcal O}_{U[v]}/(\epsilon)=\tilde{E}|_U$.
Define the ${\mathcal O}_{U[v]}$-homomorphism
\[
 \tilde{N}\colon \ {\mathcal E}^v_U \longrightarrow {\mathcal E}^v_U
\]
by the same matrix (\ref{equation: representation matrix of N}) as $N$.
Then $\big({\mathcal E}^v_U,\tilde{N}\big)$ becomes a lift of $\big(\tilde{E}|_U, N\big)$.
Define matrices $\tilde{A}(z)$, $B(z)$, $C(z)$ by setting
\begin{gather*}
 \tilde{A}(z)
 :=
 \sum_{k=0}^{r-1}\sum_{l=0}^{m-1} a_{k,l} z^l \tilde{N}^k +
 z^{m-1}\tilde{R}_r+z^{2m-1}\tilde{A}'(z),
 \\
 B(z)
 :=
 \sum_{k=0}^{r-1}\sum_{l=0}^{m-2}
 \frac { r b_{k,l} } { (-mr+lr+r+k)z^{m-l-1} } \tilde{N}^k, \\
 C(z) \frac{{\rm d}z} {z^m}
 :=
 {\rm d}B(z)+ [A(z),B(z)]\frac{{\rm d}z} {z^m},
\end{gather*}
where $\tilde{R}_r$ is the endomorphisms of
${\mathcal E}^v_U$ whose representation matrix with respect to the basis
$e_0,\dots,e_{r-1}$ is the same as that of $R_r$ in (\ref{equation: definition of R})
and $\tilde{A}'(z)$
is a lift of $A'(z)$ such that
$\frac{\partial \tilde{A}'(z)}{\partial\epsilon}=0$.
Using the calculations
\begin{equation*}
 \tilde{N}^k
 =
 \begin{pmatrix}
 0 & \cdots & 0 & z & \cdots & 0 \\
 \vdots & \ddots & \vdots & \vdots & \ddots & \vdots \\
 0 & \cdots & 0 & 0 & \cdots & z \\
 1 & \cdots & 0 & 0 & \cdots & 0 \\
 \vdots & \ddots & \vdots & \vdots & \ddots & \vdots \\
 0 & \cdots & 1 & 0 & \cdots & 0
 \end{pmatrix},
 \qquad
 \big[ \tilde{R}_r,\tilde{N}^k \big]
 =
 \begin{pmatrix}
 0 & \cdots & 0 & -\frac {r-k} {r} z & \cdots & 0 \\
 \vdots & \ddots & \vdots & \vdots & \ddots & \vdots \\
 0 & \cdots & 0 & 0 & \cdots & -\frac {r-k} {r}z \\
 \frac {k} {r} & \cdots & 0 & 0 & \cdots & 0 \\
 \vdots & \ddots & \vdots & \vdots & \ddots & \vdots \\
 0 & \cdots & \frac {k} {r} & 0 & \cdots & 0
 \end{pmatrix},
\end{equation*}
we can check the equality
\begin{equation} \label{equation: differential of k-th power of N}
 {\rm d}\tilde{N}^k+\left[\tilde{R}_r\frac{{\rm d}z}{z},\tilde{N}^k\right]
 =
 \frac{k}{r}\tilde{N}^k\frac{{\rm d}z}{z}.
\end{equation}
Then we can see
\begin{gather*}
 \left( {\rm d}B(z)+ [A(z),B(z)]\frac{{\rm d}z} {z^m} \right) \Big|_{(2m-1)\tilde{x}} \\
 =\sum_{k=0}^{r-1}\sum_{l=0}^{m-2}
 \frac{rb_{k,l}} {-mr+lr+r+k} \left( {\rm d} \left( \frac{1}{z^{m-l-1}} \tilde{N}^k \right)
 +\left[ \tilde{R}_r\frac{{\rm d}z}{z} , \frac{1}{z^{m-l-1}}\tilde{N}^k \right] \right)
 \Big|_{(2m-1)\tilde{x}}
 \\
\qquad{} =
 \sum_{k=0}^{r-1}\sum_{l=0}^{m_i-2}
 \frac{rb_{k,l}} {-mr+lr+r+k}
 \left( \frac{-m+l+1}{z^{m-l}}\tilde{N}^k {\rm d}z +\frac{1}{z^{m-l-1}}\frac{k}{r}\tilde{N}^k\frac{{\rm d}z}{z} \right)
 \Big|_{(2m-1)\tilde{x}}
 \\
\qquad{} =
 \sum_{k=0}^{r-1}\sum_{l=0}^{m-2}
 \frac { r(-m+l+1)+k } {-mr+lr+r+k} \frac { b_{k,l} } {z^{m-l}} \tilde{N}^k {\rm d}z
 \bigg|_{(2m-1)\tilde{x}}
 =
 \nu_v(\tilde{N}) \big|_{(2m-1)\tilde{x}}.
\end{gather*}
So the matrix
\[
 \big(\tilde{A}(z)+\epsilon C(z)\big) \frac{{\rm d}z} {z^m}
 +B(z) {\rm d}\epsilon
\]
determines an integrable connection
$\nabla^{\mathrm{flat}}_{U[v]} \colon
{\mathcal O}_{U[v]}^{\oplus r} \longrightarrow
{\mathcal O}_{U[v]}^{\oplus r}\otimes_{{\mathcal O}_{U[v]}} \tilde{\Omega}_v$
such that the induced relative connection
$\overline{\nabla^{\mathrm{flat}}_{U[v]}}
\colon {\mathcal O}_{U[v]}^{\oplus r} \longrightarrow
{\mathcal O}_{U[v]}^{\oplus r}\otimes
 \Omega^1_{U[v]/{\mathcal T'}[v]}({\mathcal D}_{\mathcal T'}\cap U)$
satisfies
\[
\overline{\nabla^{\mathrm{flat}}_{U[v]}}\big|_{(2m-1)\tilde{x}}
=I_v^*\nu\big(\tilde{N}\big)\big|_{(2m-1)\tilde{x}}.
\]
We can give a filtration
${\mathcal O}_{U[v]}^{\oplus r}\big|_{{\mathcal D}_{\tilde{M}'[v]}}
=V^v_{U,0}\supset V^v_{U,1}\supset\cdots\supset
V^v_{U,r-1}\supset V^v_{U,r}=z V^v_{U,0}$
by setting
$V^v_{U,k}:=\im \big(\tilde{N}^k\big|_{{\mathcal D}_{\tilde{M}'[v]}}\big)$
for $k=0,1,\dots,r$.
So we can see that
$\{ V^v_{U,k} \}$ induces
$\overline{V}^v_{U,k}$, $\overline{W}^v_{U,k}$
and that the homomorphism
$\overline{\tilde{N}|_{V^v_{U,k}}}\colon \overline{V}^v_{U,k}
\longrightarrow \overline{V}^v_{U,k}$
induced by the restriction $\tilde{N}|_{V^v_{U,k}}$
has a factorization
\[
 \overline { \tilde{N}|_{V^v_{U,k}} }
 \colon \
 \overline{V}^v_{U,k} \xrightarrow{\kappa^v_{U,k}}
 \overline{W}^v_{U,k} \xrightarrow[\sim]{\theta^v_{U,k}}
 \overline{V}^v_{U,k}.
\]
Then $\big(V^v_{U,k},\theta^v_{U,k},\kappa^v_{U,k}\big)$
induces a factorized ramified structure
$\big(V^v_{U,k},\vartheta^v_{U,k},\varkappa^v_{U,k}\big)$
on $\big({\mathcal E}^v_U,\overline{\nabla^{\mathrm{flat}}_{U[v]}}\big)$,
where $\overline{\nabla^{\mathrm{flat}}_{U[v]}}$
is the relative connection induced by
$\nabla^{\mathrm{flat}}_{U[v]}$.
Thus
$\big(
{\mathcal E}^v_U ,,
\nabla^{\mathrm{flat}}_{U[v]} ,
\big\{
V^v_{U,k},\vartheta^v_{U,k},\varkappa^v_{U,k} \big\} \big)$
becomes a local horizontal lift of
$\big(\tilde{E},\tilde{\nabla},\{\tilde{V}_{k},
\tilde{\vartheta}_{k},\tilde{\varkappa}_{k}\}\big)
\big|_U$.
\end{proof}

\begin{Lemma}[uniqueness of ramified irregular singular local horizontal lift]\label{lemma: uniqueness of ramified local horizontal lift}
Under the same assumption as
Lemma~{\rm \ref{lemma: existence of local horizontal lift of ramified connection}},
a local horizontal lift
$\big({\mathcal E}^v_U,\nabla^v_U, {\mathcal V}^v_U\big)$
of
$\big(\tilde{E}|_U,\tilde{\nabla}|_U,{\mathcal V}|_U\big)$
with respect to~$v$
is unique up to an isomorphism.
\end{Lemma}

\begin{proof}
Let $\big( {\mathcal E}^v_U , \nabla^{\mathrm{flat}}_{U[v]} ,
\big\{ V^v_{U,k} , \vartheta^v_{U,k} , \varkappa^v_{U,k} \big\} \big)$
be the local horizontal lift constructed in
Lemma~\ref{lemma: existence of local horizontal lift of ramified connection}.
Take another local horizontal lift
$\big({\mathcal O}_{U_{\alpha}[v]}^{\oplus r}, \nabla',
\big\{ V'_k,\vartheta'_k,\varkappa'_k \big\} \big)$ of
$\big(\tilde{E},\tilde{\nabla},\big\{\tilde{V}_{k},
\tilde{\vartheta}_{k},\tilde{\varkappa}_{k}\big\}\big)
\big|_{U}$.
The connection
$\nabla' \colon {\mathcal O}_{U[v]}^{\oplus r} \longrightarrow
{\mathcal O}_{U[v]}^{\oplus r}
 \otimes \Omega^1_{U[v]/M'}(m(\tilde{x})_{M'})$
can be given by
\[
 \nabla'
 \begin{pmatrix} f_1 \\ \vdots \\ f_r \end{pmatrix}
 =
 \begin{pmatrix} {\rm d}f_1 \\ \vdots \\ {\rm d}f_r \end{pmatrix}
 +
 \left( \big( \tilde{A}(z)+\epsilon C'(z) \big) \frac {{\rm d}z} {z^m}
 + B'(z) {\rm d}\epsilon
 \right)
 \begin{pmatrix} f_1 \\ \vdots \\ f_r \end{pmatrix}
\]
with $B'(z)$ a rational function on $U$ admitting a pole at
$z=0$ of order at most $m-1$.
Note that $\nabla'$ satisfies the integrability condition
\begin{equation} \label{equation: integrability condition for another local connection}
 C'(z)\frac{{\rm d}z} {z^m} =
 {\rm d} B'(z) + [A(z),B'(z)] \frac{{\rm d}z} {z^m}.
\end{equation}

Now we apply Proposition~\ref{proposition: normalization of formal type over finite scheme}
in the case $q=2m-1$ to the relative connection
$\overline{\nabla'}$
on ${\mathcal O}_{U[v]}^{\oplus r}$
induced by $\nabla'$.
Then, after applying an automorphism of
${\mathcal O}_{U[v]}^{\oplus r}$
of the form $I_r+\epsilon h$,
we may assume that
\begin{equation} \label{equation: condition of exponent for another local horizontal lift}
 C'(z)\frac{{\rm d}z}{z^m}\Big|_{(2m-1)\tilde{x}}
 =
 \nu_v\big(\tilde{N}\big)\big|_{(2m-1)\tilde{x}},
\end{equation}
$V'_k=\im\big(\tilde{N}^k|_{m\tilde{x}}\big)$
and that $\theta'_k\circ\kappa'_k$ is induced by
the restriction $\tilde{N}|_{V'_k}$
for $0\leq k\leq r-1$.

By the equality~(\ref{equation: integrability condition for another local connection}),
 we have $\big[A(z),z^{m-1} B'(z)\big]\equiv 0 \pmod{z^{m-1}}$.
Note that $A(z)$ satisfies the equality
(\ref{equation: normalization of a connection to (2m-1)x})
with $a_{1,0}\in{\mathcal O}_{\mathcal T}^{\times}$ and
$B'(z) \in M_r({\mathcal O}_U ({\mathcal D}'_{\tilde{M}'}\cap U))$
by the condition~(ii) of Definition~\ref{definition: horizontal lift in one vector case}.
So we can find $c_0(z),\dots,c_{r-1}(z)\in {\mathcal O}_{U[v]}$
satisfying
\[
 z^{m-1}B'(z)\equiv \sum_{k=0}^{r-1} c_k(z) \tilde{N}^k \quad
 \big({\rm mod} \ z^{m-1}\End\big({\mathcal O}_{U_{\alpha}[v]}^{\oplus r}\big) \big)
\]
since the equality
$\ker \big(\mathrm{ad}\big( \tilde{N}|_{z=0}\big) \big)
={\mathcal O}_{U[v]}\big[\tilde{N}|_{z=0}\big]$
holds.
Then we can write
\[
 B'(z)=\sum_{k=0}^{r-1} \frac {c_k(z)} {z^{m-1}} \tilde{N}^k +B_m (z).
\]
with $B_m (z)$ a matrix of regular functions.
Furthermore,
we can see that
$B_m (0)$ is a lower triangular matrix, since
$\nabla^v={\rm d}+\big(\tilde{A}(z)+\epsilon C(z)\big)\frac{{\rm d}z}{z^m}+B'(z){\rm d}\epsilon$
preserves the filtration~$(V^v_k)$.
Looking at the equality (\ref{equation: integrability condition for another local connection})
again, we can see that
\begin{gather*}
 C'(z){\rm d}z-z^m\left( {\rm d}B'(z)+[R_r,B'(z)]\frac{{\rm d}z}{z} \right)
 \\
\qquad{}=
 \big[\tilde{A}(z)-z^{m-1}R_r,B'(z)\big]{\rm d}z
 \\
\qquad{} =\left[ \sum_{k=0}^{r-1}\sum_{l=0}^{m-1}a_{k,l}z^l\tilde{N}^k+z^{2m-1}\tilde{A}' ,
 \sum_{k=0}^{r-1} \frac{c_k(z)}{z^{m-1}} \tilde{N}^k+B_m(z) \right]{\rm d}z
 \\
\qquad{} \equiv
 \sum_{k=0}^{r-1}
 \left[ z^m c_k(z)\tilde{A}'(z)
 -\sum_{l=0}^{m-1}a_{k,l}z^l B_m(z) , \tilde{N}^k \right]{\rm d}z
 \in\mathrm{ad}\big(\tilde{N}\big){\rm d}z\quad \big({\rm mod} \ z^{2m-1}{\rm d}z\big).
\end{gather*}
In particular, we have
\begin{equation}\label{equation: trace condition for extending generalized isomonodromy}
\Tr \left( \tilde{N}^l \left( C'(z){\rm d}z-z^m\left( {\rm d}B'(z)+[R_r,B'(z)]\frac{{\rm d}z}{z} \right)
\right)\right)\equiv 0 \quad \big({\rm mod} \ z^{2m-1}{\rm d}z\big)
\end{equation}
for $0\leq l\leq r-1$.
Since
\begin{gather*}
 z^m \left( {\rm d}B'(z) +[R_r,B'(z)]\frac{{\rm d}z}{z} \right)
\\
\qquad{} \equiv
 \sum_{k=0}^{r-1} z^m \left( {\rm d}\left( \frac{c_k(z)} {z^{m-1}}\right)
 +\frac {kc_k(z) } {rz^m}{\rm d}z \right) \tilde{N}^k
 +z^{m-1}[R_r,B_m(z)] {\rm d}z \quad \big({\rm mod} \ z^m {\rm d}z\big)
\end{gather*}
and since
$[R_r,B_m]|_{z=0}$ is lower triangular nilpotent matrix,
we can see that the condition
(\ref{equation: trace condition for extending generalized isomonodromy})
implies
\begin{gather*}
 z^m \left( {\rm d}\left( \frac{c_0(z)} {z^{m-1}}\right) -\nu_{0,v}(z)\right)
\equiv 0
 \pmod{z^m},
 \\
 z^{m+1} \left( {\rm d}\left( \frac{c_{r-l}(z)} {z^{m-1}}\right)
 +\frac {(r-l)c_{r-l}(z) } {rz^m}{\rm d}z -\nu_{r-l,v}(z)\right)
 \equiv 0
 \pmod{z^m} \quad
 (1\leq l\leq r-1).
\end{gather*}
In other words, we have
\[
{\rm d}\left( \frac{c_0(z)} {z^{m-1}} \right)\Big|_{m\tilde{x}}
 = \nu_{0,v}(z)\big|_{m\tilde{x}},
\qquad
 \left( {\rm d}\left( \frac{c_k(z)} {z^{m-1}}\right)
 +\frac {kc_k(z) } {rz^m}{\rm d}z\right)\Big|_{(m-1)\tilde{x}}
 = \nu_{k,v}(z) \big|_{(m-1)\tilde{x}}
\]
for $1\leq k\leq r-1$,
which implies that
\[
 c_k(z)\equiv \sum_{l=0}^{m-2} \frac{rb_{k,l}}{-mr+lr+r+k} z^l
 \quad \big({\rm mod} \ z^{m-1}\big).
\]
Thus
\[
 Q(z):=
 B(z)-B'(z)
\]
becomes
a matrix of regular functions.
Furthermore, (\ref{equation: integrability condition for another local connection})
and (\ref{equation: condition of exponent for another local horizontal lift})
implies the equality
\[
{\rm d}Q(z)+[A(z),Q(z)]\frac{{\rm d}z}{z^m}
 \equiv
 0 \quad \big({\rm mod} \ z^{2m-1}{\rm d}z/z^m \big),
\]
from which we can see
$Q(z)|_{m\tilde{x}} \in {\mathcal O}_{m\tilde{x}}\big[\tilde{N}\big]$.
If we apply the transform
$I_r+\epsilon Q(z)$ to the connection~$\nabla'$,
then the consequent connection has the matrix form
\begin{gather*}
 (I_r\!+\epsilon Q(z))^{-1}
 {\rm d} (I_r+\epsilon Q(z))
 +
 (I_r\!+\epsilon Q(z))^{-1}\!
 \left( \!\big( \tilde{A}(z) + \epsilon C'(z) \big) \frac {{\rm d}z} {z^m}\!
 + B'(z) {\rm d}\epsilon \!\right)\!
 (I_r\!+\epsilon Q(z))
 \\
\qquad{} =
 (I_r-\epsilon Q(z))(\epsilon {\rm d}Q(z)+Q(z) {\rm d}\epsilon)
 +\tilde{A}(z)\frac{{\rm d}z}{z^m} +\epsilon ([A(z),Q(z)]+C'(z) )\frac{{\rm d}z}{z^m}
 +B'(z){\rm d}\epsilon
 \\
\qquad{} =
 \tilde{A}(z)\frac{{\rm d}z} {z^m}
 +\epsilon \left( {\rm d}B(z)-{\rm d}B'(z)
 +\big([A(z),B(z)-B'(z)]+C'(z)\big) \frac{{\rm d}z} {z^m} \right)\\
 \qquad\quad{}
 +(Q(z)+B'(z)){\rm d}\epsilon
 \\
\qquad{} =
 \tilde{A}(z)\frac{{\rm d}z} {z^m}
 +\epsilon \left( {\rm d}B(z) + [A(z),B(z)] \frac {{\rm d}z} {z^m}
 -{\rm d}B'(z) - [A(z),B'(z)]\frac{{\rm d}z} {z^m} + C'(z)\frac{{\rm d}z} {z^m} \right)\\
 \qquad\quad{}
 +B(z)){\rm d}\epsilon
 \\
\qquad{} =
 \big( \tilde{A}(z)+\epsilon C(z) \big) \frac {{\rm d}z} {z^m} +B(z) {\rm d}\epsilon,
\end{gather*}
which means that
$\big({\mathcal O}_{U_{\alpha}[v]}^{\oplus r}, \nabla'\big)$
is isomorphic to
$\big(
{\mathcal O}_{U_{\alpha}[v]}^{\oplus r} ,
\nabla^{\mathrm{flat}}_{U_{\alpha}[v]} \big)$
via $I_r+\epsilon Q(z)$.
Since $Q(z)|_{m\tilde{x}}$ belongs to
${\mathcal O}_{m\tilde{x}}[\tilde{N}|_{m\tilde{x}}]$,
we can see that
$I_r+\epsilon Q(z)$ induces an isomorphism
which transforms
$\big({\mathcal O}_{U[v]}^{\oplus r}, \nabla',
\{ V'_k,\vartheta'_k,\varkappa'_k \} \big)$
to
$\big(
{\mathcal O}_{U[v]}^{\oplus r} ,
\nabla^{\mathrm{flat}}_{U[v]} ,
\big\{
V^v_k,\vartheta^v_k,\varkappa^v_k \big\} \big)$.
We can see that such an isomorphism is unique,
because it is determined by
the coefficient of $d\epsilon$.
\end{proof}

\begin{Proposition} \label{proposition: horizontal lift in one variable}
For any vector field
$v\in H^0({\mathcal T'},T_{\mathcal T'})$,
there is a unique global horizontal lift
$\big({\mathcal E}^v,\nabla^v, l^v,\ell^v,{\mathcal V}^v\big)$
of
$\big(\tilde{E},\tilde{\nabla},\tilde{l},\tilde{\ell},\tilde{\mathcal V}\big)_{\tilde{M}'}$.
\end{Proposition}

\begin{proof}
We take an affine open covering
${\mathcal C}_{\tilde{M}'} =\bigcup _{\alpha} U_{\alpha}$
such that
$\tilde{E}|_{U_{\alpha}}
\cong {\mathcal O}_{U_{\alpha}}^{\oplus r}$ for each $\alpha$.
We may assume that
$\# \big\{ \alpha \mid U_{\alpha} \supset \tilde{x} \big\}
=1$
for each
$\tilde{x}=\big(\tilde{x}^{\mathrm{log}}_i\big)_{\tilde{M}'}$, $\tilde{x}=\big(\tilde{x}^{\mathrm{un}}_i\big)_{\tilde{M}'}$
and $\tilde{x}=\big(\tilde{x}^{\mathrm{ram}}_i\big)_{\tilde{M}'}$.
We may further assume that, for each $\alpha$,
$U_{\alpha} \cap {\mathcal D}_{\tilde{M}'}=\varnothing$ holds or
$U_{\alpha} \cap {\mathcal D}_{\tilde{M}'}=\tilde{x}$
holds for some $\tilde{x}=\big(\tilde{x}^{\mathrm{log}}_i\big)_{\tilde{M}'}$,
$\tilde{x}=\big(\tilde{x}^{\mathrm{un}}_i\big)_{\tilde{M}'}$
or $\tilde{x}=\big(\tilde{x}^{\mathrm{ram}}_i\big)_{\tilde{M}'}$.

Let $U_{\alpha}[v]$ be the open subscheme of ${\mathcal C}_{\tilde{M}'[v]}$
whose underlying set is $U_{\alpha}$.
If $U_{\alpha}\cap {\mathcal D}_{\tilde{M}'}=\varnothing$,
then we can write
$\tilde{\nabla}|_{U_{\alpha}}={\rm d}+A_{\alpha}(z){\rm d}z$
for a~matrix~$A_{\alpha}$ with values in ${\mathcal O}_{U_{\alpha}}$.
We can take a~matrix~$\tilde{A}_{\alpha}$ with values in
${\mathcal O}_{U_{\alpha}}[v]$ which is a lift of $A_{\alpha}$.
After adding an element of $\epsilon M_r({\mathcal O}_{U_{\alpha}})$,
we can assume that $\partial \tilde{A}_{\alpha}/\partial \epsilon=0$.
Then $\nabla_{\alpha}={\rm d}+\tilde{A}_{\alpha}{\rm d}z$
is an integrable connection
and $\big({\mathcal O}_{U_{\alpha}[v]}^{\oplus r},\nabla_{\alpha}\big)$
is a local horizontal lift of
$\big(\tilde{E}|_{U_{\alpha}},\tilde{\nabla}|_{U_{\alpha}}\big)$.
Furthermore, we can prove the uniqueness of the local horizontal lift
by the same proof as Lemma~\ref{lemma: local horizontal lift of regular singular connection}.

If $\alpha$ satisfies
$U_{\alpha} \cap {\mathcal D}_{\tilde{M}'}=\tilde{x}$
for some $\tilde{x}=\big(\tilde{x}^{\mathrm{log}}_i\big)_{\tilde{M}'}$,
$\tilde{x}=\big(\tilde{x}^{\mathrm{un}}_i\big)_{\tilde{M}'}$
or $\tilde{x}=\big(\tilde{x}^{\mathrm{ram}}_i\big)_{\tilde{M}'}$,
we can take a local horizontal lift
$\big({\mathcal E}^v_{U_{\alpha}},\nabla^v_{U_{\alpha}},l^v_{U_{\alpha}},\ell^v_{U_{\alpha}},
{\mathcal V}^v_{U_{\alpha}}\big)$
of
$\big(\tilde{E}|_{U_{\alpha}},\tilde{\nabla}|_{U_{\alpha}},
\tilde{l}|_{U_{\alpha}},\tilde{\ell}|_{U_{\alpha}},\tilde{\mathcal V}|_{U_{\alpha}}\big)$
by Lemmas \ref{lemma: local horizontal lift of regular singular connection},~\ref{lemma: local horizontal lift of unramified connection} and
\ref{lemma: existence of local horizontal lift of ramified connection}.
Since the local horizontal lifts are unique up to unique isomorphisms,
we can patch them and get a global horizontal lift
$\big({\mathcal E}^v,\nabla^v,l^v,\ell^v,{\mathcal V}^v \big)$
of
$\big(\tilde{E},\tilde{\nabla}, \tilde{l}, \tilde{\ell},
\tilde{\mathcal V}\big)_{\tilde{M}'}$,
which is unique up to an isomorphism.
\end{proof}

For a Zariski open subset
${\mathcal T'}\subset{\mathcal T}$,
consider a morphism
\begin{equation*}
 u\colon \ \Spec{\mathcal O}_{\mathcal T'}[\epsilon_1,\epsilon_2]/\big(\epsilon_1^2,\epsilon_2^2\big)
 \longrightarrow{\mathcal T'}
\end{equation*}
such that
$u|_{\mathcal T'}=\mathrm{id}_{\mathcal T'}$.
Let
\begin{equation*}
 \bar{u} \colon \ \Spec {\mathcal O}_{\mathcal T'}[\epsilon_1,\epsilon_2]
 /\big(\epsilon_1^2,\epsilon_1\epsilon_2,\epsilon_2^2\big)
 \longrightarrow {\mathcal T'}
\end{equation*}
be the induced morphism which corresponds to a pair
$(u_1,u_2)$ of vector fields.
We write
\[
 {\mathcal T'}[\bar{u}]:={\mathcal T}'\times
 \Spec\mathbb{C}[\epsilon_1,\epsilon_2]/\big(\epsilon_1^2,\epsilon_1\epsilon_2,\epsilon_2^2\big),
 \qquad
 {\mathcal T'}[u]:={\mathcal T}'\times
 \Spec\mathbb{C}[\epsilon_1,\epsilon_2]/\big(\epsilon_1^2,\epsilon_2^2\big)
\]
with the structure morphisms
${\mathcal T'}[\bar{u}]\xrightarrow{\bar{u}}{\mathcal T'}$
and
${\mathcal T'}[u]\xrightarrow{u}{\mathcal T'}$,
respectively.
We further set
\begin{alignat*}{3}
& \tilde{M'}[\bar{u}]
:=\tilde{M}'\times_{\mathcal T'}{\mathcal T'}[\bar{u}] ,\qquad &&
 {\mathcal C}_{ {\mathcal M'}[\bar{u}] }
 :=
 {\mathcal C}\times_{\mathcal T'} {\mathcal M'}[\bar{u}],&
 \\
& \tilde{M'}[u]
 :=
 \tilde{M}'\times_{\mathcal T'}{\mathcal T'}[u],
 \qquad&&
 {\mathcal C}_{ {\mathcal M'}[u] }
 :=
 {\mathcal C}\times_{\mathcal T'} {\mathcal M'}[u].&
\end{alignat*}
We define a coherent subsheaf
$\tilde{\Omega}_u$ of
$\Omega^1_{{\mathcal C}_{\tilde{M}'[u]}/\tilde{M}'}\big({\mathcal D}_{\tilde{M'}[u]}\big)$
in the same way as in (\ref{equation: definition of total differential})
and define a coherent subsheaf
$\tilde{\Omega}_{\bar{u}}$ of
$\Omega^1_{{\mathcal C}_{\tilde{M}'[\bar{u}]}/\tilde{M}'}\big({\mathcal D}_{\tilde{M'}[\bar{u}]}\big)$
similarly.

\begin{Definition} \label{definition: horizontal lifts in general}
We say that
$\big({\mathcal E}^u,\nabla^u,l^u,\ell^u,{\mathcal V}^u\big)$
\big(resp.\ $\big({\mathcal E}^{\bar{u}},\nabla^{\bar{u}},l^{\bar{u}},\ell^{\bar{u}},{\mathcal V}^{\bar{u}}\big)$\big)
is a horizontal lift of
$\big(\tilde{E},\tilde{\nabla},\tilde{l},\tilde{\ell},\tilde{\mathcal V}\big)_{\tilde{M}'}$
with respect to~$u$ (resp.\ $\bar{u}$)
if the conditions (i), (ii), (iii), (iv) and (v) of Definition~\ref{definition: horizontal lift in one vector case}
are satisfied after replacing
$\tilde{M}'[v]$ with $\tilde{M'}[u]$ (resp.\ $\tilde{M'}[\bar{u}]$),
replacing $\tilde{\Omega}_v$ with $\tilde{\Omega}_u$
(resp. $\tilde{\Omega}_{\bar{u}}$),
replacing $(\lambda,I_v^*\tilde{\mu},I_v^*\tilde{\nu})$-structure in~(vi)
with $(\lambda,I_u^*\tilde{\mu},I_u^*\tilde{\nu})$-structure
(resp.\ $(\lambda,I_{\bar{u}}^*\tilde{\mu},I_{\bar{u}}^*\tilde{\nu})$-structure)
and replacing the equality of integrability condition in (iii) with
\begin{gather*}
\frac{\partial A}{\partial \epsilon_1}{\rm d}z\wedge {\rm d}\epsilon_1
 + \frac{\partial A}{\partial \epsilon_2}{\rm d}z\wedge {\rm d}\epsilon_2
 +\frac{\partial B_1}{\partial \epsilon_2}{\rm d}\epsilon_2\wedge {\rm d}\epsilon_1
 +\frac{\partial B_2}{\partial \epsilon_1}{\rm d}\epsilon_1\wedge {\rm d}\epsilon_2
 \\
 \qquad{} =
 {\rm d}B_1\wedge {\rm d}\epsilon_1+[A,B_1]{\rm d}z\wedge {\rm d}\epsilon_1
 +
{\rm d}B_2\wedge {\rm d}\epsilon_1+[A,B_2]{\rm d}z\wedge d\epsilon_2
 +[B_1,B_2]{\rm d}\epsilon_1\wedge {\rm d}\epsilon_2
\end{gather*}
for
$\Gamma^u=A{\rm d}z+B_1{\rm d}\epsilon_1+B_2{\rm d}\epsilon_2$
(resp.\ replacing with
\begin{align*}
 \frac{\partial A}{\partial \epsilon_1}{\rm d}z\wedge {\rm d}\epsilon_1
 + \frac{\partial A}{\partial \epsilon_2}{\rm d}z\wedge {\rm d}\epsilon_2
 =
{\rm d}B_1\wedge d\epsilon_1+[A,B_1]{\rm d}z\wedge {\rm d}\epsilon_1
 +
{\rm d}B_2\wedge {\rm d}\epsilon_1+[A,B_2]{\rm d}z\wedge {\rm d}\epsilon_2
\end{align*}
for
$\Gamma^{\bar{u}}=Adz+B_1{\rm d}\epsilon_1+B_2{\rm d}\epsilon_2$).
\end{Definition}

The following proposition can be proved in the same way as
Proposition~\ref{proposition: horizontal lift in one variable}.
So we omit its proof.

\begin{Proposition} \label{proposition: horizontal lift in two variable}
There exists a unique horizontal lift
$\big({\mathcal E}^{\bar{u}},\nabla^{\bar{u}},l^{\bar{u}},\ell^{\bar{u}},{\mathcal V}^{\bar{u}}\big)$
of
$\big(\tilde{E},\tilde{\nabla},\tilde{l},\tilde{\ell},\tilde{\mathcal V}\big)_{\tilde{M}'}$
with respect to
$\bar{u}\colon
{\mathcal T'}[\bar{u}]=\Spec {\mathcal O}_{\mathcal T'}[\epsilon_1,\epsilon_2]
 /\big(\epsilon_1^2,\epsilon_1\epsilon_2,\epsilon_2^2\big)
 \longrightarrow {\mathcal T'}$.
\end{Proposition}

If a horizontal lift
$\big({\mathcal E}^u,\nabla^u,l^u,\ell^u,{\mathcal V}^u\big)$
of
$\big(\tilde{E},\tilde{\nabla},\tilde{l},\tilde{\ell},\tilde{\mathcal V}\big)_{\tilde{M}'}$
with respect to $u$ exists,
it can be obtained as a lift of
$\big({\mathcal E}^{\bar{u}},\nabla^{\bar{u}},l^{\bar{u}},\ell^{\bar{u}},{\mathcal V}^{\bar{u}}\big)$
whose existence is ensured by Proposition~\ref{proposition: horizontal lift in two variable}.

\begin{Proposition} \label{proposition: horizontal lift in two variable of deep order}
There exists a unique horizontal lift
$\big({\mathcal E}^u,\nabla^u,l^u,\ell^u,{\mathcal V}^u\big)$
of
$\big(\tilde{E},\tilde{\nabla},\tilde{l},\tilde{\ell},\tilde{\mathcal V}\big)_{\tilde{M}'}$
with respect to
$u\colon {\mathcal T'}[u]=\Spec {\mathcal O}_{\mathcal T'}[\epsilon_1,\epsilon_2]
 /\big(\epsilon_1^2,\epsilon_2^2\big)
 \longrightarrow {\mathcal T'}$.
\end{Proposition}

\begin{proof}By Proposition~\ref{proposition: horizontal lift in two variable},
there is a unique horizontal lift
$\big({\mathcal E}^{\bar{u}}\! ,\nabla^{\bar{u}} \! ,l^{\bar{u}} \! ,\ell^{\bar{u}} \! ,{\mathcal V}^{\bar{u}}\big)$
of
$\big(\tilde{E},\tilde{\nabla},\tilde{l},\tilde{\ell},\tilde{\mathcal V}\big)_{\tilde{M}'}$
with respect to
$\bar{u}\colon \Spec {\mathcal O}_{\mathcal T'}[\epsilon_1,\epsilon_2]
 /\big(\epsilon_1^2,\epsilon_1\epsilon_2,\epsilon_2^2\big)
 \rightarrow {\mathcal T'}$.
So we only have to show the existence and the uniqueness of a lift of
$\big({\mathcal E}^{\bar{u}},\nabla^{\bar{u}},l^{\bar{u}},\ell^{\bar{u}},{\mathcal V}^{\bar{u}}\big)$,
which is a horizontal lift of
$\big(\tilde{E},\tilde{\nabla},\tilde{l},\tilde{\ell},\tilde{\mathcal V}\big)_{\tilde{M}'}$
with respect to the morphism
$u\colon {\mathcal T'}[u]=\Spec {\mathcal O}_{\mathcal T'}[\epsilon_1,\epsilon_2]
 /\big(\epsilon_1^2,\epsilon_2^2\big)
 \longrightarrow {\mathcal T'}$.
The method of the proof is similar to that of
Proposition \ref{proposition: horizontal lift in one variable}.

We take an affine open covering
${\mathcal C}\times_{\mathcal T}\tilde{M}'
=\bigcup U_{\alpha}$
as in the proof of Proposition~\ref{proposition: horizontal lift in one variable}.
If $U_{\alpha}$ is an open neighborhood of $\big(\tilde{x}_i^{\mathrm{un}}\big)_{\tilde{M}'}$,
then the existence and the uniqueness of the local horizontal lift
with respect to~$u$ is given in the proof of \cite[Lemma~5.5]{Inaba-3}.
If $U_{\alpha}$ is an open neighborhood of
$\big(\tilde{x}_i^{\mathrm{log}}\big)_{\tilde{M}'}$,
then it is much easier to prove the existence and the uniqueness of
a logarithmic local horizontal lift.

So assume that $\tilde{x}:=\big(\tilde{x}_i^{\mathrm{ram}}\big)_{\tilde{M}'}$
is contained in $U_{\alpha}$.
If $u$ is given by
\begin{align*}
 u^*(\nu(w))
 &=
 \nu(w) + \epsilon_1 \nu_{u_1}(w)
 +\epsilon_2 \nu_{u_2}(w) +\epsilon_1\epsilon_2 \nu_{u_{12}}(w)
 \\
 &=
 \sum_{k=0}^{r-1}\left( a_{k,m-1} z^{m-1}
 +\sum_{l=0}^{m-2} ( a_{k,l} + \epsilon_1 b_{1,k,l}
 +\epsilon_2 b_{2,k,l} +\epsilon_1\epsilon_2 b_{1,2,k,l} ) z^l \right)w^k,
\end{align*}
then, by the proof of Proposition \ref{proposition: horizontal lift in one variable},
the restriction of $\nabla^{\bar{u}}$ to
$U_{\alpha}[\bar{u}]=U_{\alpha}[u]\otimes{\mathcal O}_{{\mathcal T'}[u]/(\epsilon_1\epsilon_2)}$
can be given by
\[
 A(z)\frac{{\rm d}z} {z^m}
 +\epsilon_1C_1(z)\frac{{\rm d}z} {z^m}
 +\epsilon_2C_2(z)\frac{{\rm d}z} {z^m}
 +B_1(z) {\rm d}\epsilon_1+B_2(z){\rm d}\epsilon_2
\]
where $\frac{\partial A(z)} {\partial\epsilon_1}=\frac {\partial A(z)} {\partial \epsilon_2}=0$
and
\begin{gather}
 A(z)
 =
 \sum_{k=0}^{r-1} \sum_{l=0}^{m-1} a_{k,l} z^l \tilde{N}^k
 + z^{m-1}R_r+z^{3m-1}A'(z),
 \label{equation: assumption of A(z)} \\
 B_1(z) =
 \sum_{k=0}^{r-1} \sum_{l=0}^{m-2}
 \frac {rb_{1,k,l}} {(-mr+lr+r+k)z^{m-l-1}} \tilde{N}^k,
 \notag \\
 B_2(z)
 =
 \sum_{k=0}^{r-1} \sum_{l=0}^{m-2}
 \frac {rb^{(i)}_{2,k,l}} {(-mr+lr+r+k)z^{m-l-1}} \tilde{N}^k,
 \notag \\
 C_1(z)\frac{{\rm d}z}{z^m}
 =
 {\rm d}B_1(z)+[A(z),B_1(z)]\frac{{\rm d}z}{z^m},
 \qquad
 C_2(z)\frac{{\rm d}z} {z^m}
 = {\rm d}B_2(z)+[A(z),B_2(z)]\frac{{\rm d}z} {z^m}.
 \notag
\end{gather}
Then we can see by the above equality that
\[
 C_j(z)|_{2m\tilde{x}}
 =
 \sum_{k=0}^{r-1} \sum_{l=0}^{m-2}
 b_{j,k,l} z^l \tilde{N}^k|_{2m\tilde{x}}
\]
for $j=1,2$.
So we have
$[C_1(z),B_2(z)],
[C_2(z),B_1(z)]
\in z^{m+1}\End(\tilde{E}|_{U_{\alpha}})$.

\begin{Claim}
$[C_1(z),B_2(z)]=[C_2(z),B_1(z)]$.
\end{Claim}

\begin{proof}
First notice that we can check the equality
\[
 {\rm d}B_1(z)
 =
 \sum_{k=0}^{r-1}\sum_{l=0}^{m-2}
 \frac { r(-m+l+1)+k} {-mr+lr+r+k}
 \frac {b_{1,k,l}} {z^{m-l}} \tilde{N}^k {\rm d}z
 +\left[B_1(z),R_r \right]\frac{{\rm d}z} {z}
\]
using (\ref{equation: differential of k-th power of N}).
So we have
\begin{align*}
 [{\rm d}B_1(z),B_2(z)]
 &=
 \left[\left[B_1(z),R_r\right],B_2(z)\right]\frac{{\rm d}z} {z}
 \\
 &=
 \left[
 \left[B_2(z),R_r \right], B_1(z)\right] \frac{{\rm d}z} {z}
 +\left[ [B_1(z),B_2(z) ] , R_r \right]\frac{{\rm d}z}{z}
 \\
 &=[{\rm d}B_2(z),B_1(z)],
\end{align*}
because
$[B_1(z),B_2(z)]=0$.
Thus we have
\begin{align*}
 \left[C_1(z)\frac{{\rm d}z}{z^m} , B_2(z) \right]
 &=
 \left[ {\rm d}B_1(z)+[A(z),B_1(z)]\frac{{\rm d}z}{z^m} , B_2(z)\right]
 \\
 &=
 [{\rm d}B_1(z),B_2(z)]+[[A(z),B_1(z)],B_2(z)]\frac{{\rm d}z}{z^m}
 \\
 &=
 [{\rm d}B_2(z),B_1(z)]+[[A(z),B_2(z)],B_1(z)]\frac{{\rm d}z}{z^m}
 =\left[C_2(z)\frac{{\rm d}z}{z^m},B_1(z)\right].\!\!\!\tag*{\qed}
\end{align*}\renewcommand{\qed}{}
\end{proof}

We put
\begin{gather*}
 \tilde{A}(z)
 :=
 \sum_{k=0}^{r-1} \sum_{l=0}^{m-1} a_{k,l} z^l \tilde{N}^k
 + z^{m-1}R_r+z^{3m-1}\tilde{A}'(z),
 \\
 B_{1,2}(z)
 :=
 \sum_{k=0}^{r-1} \sum_{l=0}^{m-2}
 \frac {rb_{1,2,k,l}} {(-mr+lr+r+k)z^{m-l-1}} \tilde{N}^k,
 \\
 C_{1,2}(z)\frac{{\rm d}z} {z^m}
:=
 {\rm d}B_{1,2}(z)+\left([A(z),B_{1,2}(z)]+[C_1(z),B_2(z)]\right)\frac{{\rm d}z} {z^m},
\end{gather*}
where
$\tilde{A}'(z)$ is a lift of $A'(z)$
as a matrix with coefficients in
${\mathcal O}_{U_{\alpha}[u]}$
such that
$\frac {\partial \tilde{A}(z)} {\partial \epsilon_1}=
\frac {\partial \tilde{A}(z)} {\partial \epsilon_2}=0$.
Define a connection
$\nabla^u_{\alpha} \colon
{\mathcal O}_{U_{\alpha}[u]}^{\oplus r}
\longrightarrow
{\mathcal O}_{U_{\alpha}[u]}^{\oplus r}\otimes
\Omega^1_{{\mathcal C}_{\tilde{M}'[u]/\tilde{M}'}}({\mathcal D}_{\tilde{M}'[u]})$
by setting
\[
 \nabla^u_{\alpha}
 =
{\rm d}+
 \big(\tilde{A}+\epsilon_1C_1+\epsilon_2C_2+\epsilon_1\epsilon_2 C_{1,2}\big)\frac{{\rm d}z} {z^m}
 +B_1{\rm d}\epsilon_1+B_2{\rm d}\epsilon_2+B_{1,2}(\epsilon_1{\rm d}\epsilon_2+\epsilon_2{\rm d}\epsilon_1).
\]
Then $\nabla^u_{\alpha}$ is an integrable connection
because its curvature form becomes
\begin{gather*}
( C_1+\epsilon_2 C_{1,2} ) {\rm d}\epsilon_1\wedge \frac {{\rm d}z} {z^m}
 +
 ( C_2+\epsilon_1C_{1,2} ) {\rm d}\epsilon_2\wedge \frac {{\rm d}z} {z^m}
 +( {\rm d}B_1+\epsilon_2 {\rm d}B_{1,2} ) \wedge {\rm d}\epsilon_1
 \\
 \qquad\quad{}+
 ( {\rm d}B_2+\epsilon_1{\rm d}B_{1,2} ) \wedge {\rm d}\epsilon_2
 +
 B_{1,2}{\rm d}\epsilon_1\wedge {\rm d}\epsilon_2+B_{1,2}{\rm d}\epsilon_2\wedge {\rm d}\epsilon_1\\
 \qquad\quad{}
 +
 [A,B_1+\epsilon_2 B_{1,2}]
 \frac{{\rm d}z} {z^{m_i}} \wedge {\rm d}\epsilon_1
+[A,B_2+\epsilon_1 B_{1,2}]
 \frac{{\rm d}z} {z^{m_i}} \wedge {\rm d}\epsilon_2
 +\epsilon_2 [C_2,B_1]\frac{{\rm d}z} {z^m} \wedge {\rm d}\epsilon_1\\
 \qquad\quad{}
 +\epsilon_1[C_1,B_2]\frac{{\rm d}z} {z^m} \wedge {\rm d}\epsilon_2
 +[B_1,B_2]{\rm d}\epsilon_1 \wedge {\rm d}\epsilon_2
 \\
 \qquad{} =
 \left( {\rm d}B_1+(-C_1+[A,B_1])\frac{{\rm d}z} {z^m} \right)\wedge {\rm d}\epsilon_1
 +
 \left( {\rm d}B_2+(-C_2+[A,B_2])\frac{{\rm d}z} {z^m} \right)\wedge {\rm d}\epsilon_2
 \\
 \qquad\quad{}
 +
 \epsilon_2
 \left( {\rm d}B_{1,2}+(-C_{1,2}+[A,B_{1,2}]+[C_2,B_1])\frac{{\rm d}z} {z^m}
 \right) \wedge {\rm d}\epsilon_1
 \\
 \qquad \quad{}
 +
 \epsilon_1
 \left( {\rm d}B_{1,2}+(-C_{1,2}+[A,B_{1,2}]+[C_1,B_2])\frac{{\rm d}z} {z^m}
 \right) \wedge {\rm d}\epsilon_2
 =0.
\end{gather*}
We can define
$V^u_{k,\alpha}$, $\vartheta^u_{k,\alpha}$, $\varkappa^u_{k,\alpha}$
on ${\mathcal O}_{{\mathcal D}[u]}^{\oplus r}$
in the same way as in the proof of
Lemma~\ref{lemma: existence of local horizontal lift of ramified connection}.
So we can get a local horizontal lift
$\big(
{\mathcal O}_{U_{\alpha}[u]}^{\oplus r},\nabla^u_{\alpha},
(V^u_{k,\alpha},\vartheta^u_{k,\alpha},\varkappa^u_{k,\alpha})_{0\leq k\leq r-1} \big)$,
which is a lift of the restriction
$\big({\mathcal E}^{\bar{u}},\nabla^{\bar{u}},l^{\bar{u}},\ell^{\bar{u}},{\mathcal V}^{\bar{u}}\big)
|_{U_{\alpha}[\bar{u}]}$.

Let
$\big(
{\mathcal O}_{U_{\alpha}[u]}^{\oplus r},\nabla'_{\alpha},
(V'_{k,\alpha},\vartheta'_{k,\alpha},\varkappa'_{k,\alpha})_{0\leq k\leq r-1}\big)$
be another local horizontal lift with respect to~$u$,
which is a lift of
$\big({\mathcal E}^{\bar{u}},\nabla^{\bar{u}},l^{\bar{u}},\ell^{\bar{u}},{\mathcal V}^{\bar{u}}\big)
|_{U_{\alpha}[\bar{u}]}$.
Then we can write
\[
 \nabla'_{\alpha}=
 {\rm d}+
\big( \tilde{A}+\epsilon_1C_1+\epsilon_2 C_2+\epsilon_1\epsilon_2 C'_{1,2} \big)
 \frac {{\rm d}z} {z^m}
 +B_1 {\rm d}\epsilon_1+B_2{\rm d}\epsilon_2+B'_{1,2}\epsilon_1{\rm d}\epsilon_2
 +B'_{2,1}\epsilon_2{\rm d}\epsilon_1.
\]
The integrability condition of $\nabla'$ implies the equalities
\begin{align}
 C'_{1,2}(z)\frac{{\rm d}z} {z^m}
 &=
 {\rm d}B'_{1,2}(z)+\big( [A(z), B'_{1,2}(z)]+[C_1(z),B_2(z)] \big)\frac{{\rm d}z} {z^m}
 \nonumber\\
 &={\rm d}B'_{2,1}(z)+\big( [A(z),B'_{2,1}(z)]+[C_1(z),B_2(z)] \big)\frac{{\rm d}z} {z^m}\label{equation: another integrability condition in two variables}
\end{align}
and $B'_{1,2}=B'_{2,1}$.
Since $\nabla'$ has the property of local horizontal lift,
we have
\[
 C'_{1,2}(z) \big|_{m\tilde{x}}
 =
 \sum_{k=0}^{r-1} \sum_{l=0}^{m-2}
 b_{1,2,k,l} z^l \big(\tilde{N}\big)^k\big|_{m\tilde{x}}.
\]
We can see that
$[A(z),B'_{1,2}(z)]$ is regular from the equality
(\ref{equation: another integrability condition in two variables}).
Since $A(z)$ satisfies~(\ref{equation: assumption of A(z)}),
we can first verify
$z^{m-1}B'_{1,2}(z)\in {\mathcal O}_{(m-1)\tilde{x}}\big[\tilde{N}|_{(m-1)\tilde{x}}\big]$.
Combining with the condition (c) of
Definition \ref{definition: horizontal lift in one vector case}\,(iv),
we can take $\beta_{k,l}$ such that
\[
 B'_{1,2}(z)-\sum_{k=0}^{r-1}\sum_{l=0}^{m-1}
 \beta_{k,l} z^{l-m-1}\tilde{N}^k
 \in \End\big(\tilde{E}|_{U_{\alpha}}\big).
\]
is a matrix of regular functions whose constant term
is a lower triangular matrix.
Using the same argument as in the proof of
Lemma \ref{lemma: uniqueness of ramified local horizontal lift},
we can see
\[
 \beta_{k,l}=\frac { rb_{1,2,k,l} } { -mr+l+r+k }.
\]
So
$B_{12}(z)-B'_{12}(z)$ becomes a matrix of regular functions
and
$I_r+\epsilon_1\epsilon_2(B_{12}(z)-B'_{12}(z))$
gives an automorphism of
${\mathcal O}_{U_{\alpha}[u]}^{\oplus r}$
which transform $\nabla'_{\alpha}$ to $\nabla^u_{\alpha}$
and which sends $V'_{k,\alpha}$ to $V^u_{k,\alpha}$.
Furthermore, we can see that such a transform is uniquely
determined by the coefficient of $\epsilon_2 \, d\epsilon_1$.
Thus the existence and the uniqueness of a ramified local horizontal lift
with respect to $u$ is proved.

Patching the local horizontal lifts together,
we get a unique horizontal lift
$\big({\mathcal E}^{\bar{u}},\nabla^{\bar{u}},l^{\bar{u}},\ell^{\bar{u}},{\mathcal V}^{\bar{u}}\big)$
of
$\big(\tilde{E},\tilde{\nabla},\tilde{l},\tilde{\ell},\tilde{\mathcal V}\big)_{\tilde{M}'}$
on
${\mathcal C}\times_{\mathcal T}\Spec{\mathcal O}_{\tilde{M}'}
[\epsilon_1,\epsilon_2]/\big(\epsilon_1^2,\epsilon_2^2\big)$
with respect to $u$.
\end{proof}

\section{Global generalized isomonodromic deformation}\label{section: generalized isomonodromy equation}

\begin{Definition}\label{definition: generalized isomonodromy splitting}
For each vector field $v\in T_{\mathcal T'}$,
the relative connection
$\big({\mathcal E}^v,\overline{\nabla^v},l^v,\ell^v, {\mathcal V}^v \big)$
induced by the global horizontal lift
$\big({\mathcal E}^v,\nabla^v, l^v,\ell^v,{\mathcal V}^v \big)$
(which exists by Proposition~\ref{proposition: horizontal lift in one variable})
defines a morphism
\begin{equation*}
 I_{\Phi(v)}\colon \ \tilde{M}' \times\Spec\mathbb{C}[\epsilon]/\big(\epsilon^2\big)
 \longrightarrow \tilde{M}'
\end{equation*}
which makes the diagram
\begin{equation} \label{equation: diagram defining isomonodromic vector}
 \begin{CD}
 \tilde{M}' \times\Spec\mathbb{C}[\epsilon]/\big(\epsilon^2\big)
 @> I_{\Phi(v)} >> \tilde{M}' \\
 @V\pi_{\mathcal T'}\times\mathrm{id} VV @V V \pi_{\mathcal T'} V \\
 {\mathcal T}'\times\Spec\mathbb{C}[\epsilon]
 @> I_v >> {\mathcal T}'
 \end{CD}
\end{equation}
commutative.
We can see by the uniqueness of the horizontal lift that
the morphism $I_{\Phi(v)}$ descends to a~morphism
$
 M^{\balpha}_{{\mathcal C},{\mathcal D}}(\lambda,\tilde{\mu},\tilde{\nu})_{\mathcal T'}
 \times\mathbb{C}[\epsilon]
 \longrightarrow M^{\balpha}_{{\mathcal C},{\mathcal D}}(\lambda,\tilde{\mu},\tilde{\nu})_{\mathcal T'}
$
which corresponds to a~vector field
\begin{equation*}
 \Phi(v)\in
 H^0\big(M^{\balpha}_{{\mathcal C},{\mathcal D}}(\lambda,\tilde{\mu},\tilde{\nu})_{\mathcal T'},
 T_{M^{\balpha}_{{\mathcal C},{\mathcal D}}(\lambda,\tilde{\mu},\tilde{\nu})_{\mathcal T'}}\big).
\end{equation*}
We call this vector field $\Phi(v)$
a generalized isomonodromic vector field.
\end{Definition}

\begin{Proposition}
\label{proposition: isomonodromic splitting is a homomorphism}
The map
\[
 \Phi \colon \ H^0( {\mathcal T'},T_{\mathcal T'})
 \ni v \mapsto \Phi(v) \in
 H^0\big(M^{\balpha}_{{\mathcal C},{\mathcal D}}(\lambda,\tilde{\mu},\tilde{\nu})_{\mathcal T'},
 T_{M^{\balpha}_{{\mathcal C},{\mathcal D}}(\lambda,\tilde{\mu},\tilde{\nu})_{\mathcal T'}}\big)
\]
is a homomorphism of
$H^0({\mathcal T'},{\mathcal O}_{\mathcal T'})$-modules.
\end{Proposition}

\begin{proof}
Take vector fields
$v_1,v_2\in H^0({\mathcal T'},T_{\mathcal T'})$.
Then $(v_1,v_2)$ corresponds to a morphism
\[
 \bar{u}\colon \ {\mathcal T'}\times
 \Spec\mathbb{C}[\epsilon_1,\epsilon_2]/\big(\epsilon_1^2,\epsilon_1\epsilon_2,\epsilon_2^2\big)
 \longrightarrow {\mathcal T'}
\]
such that the composition
${\mathcal T'}\times\Spec\mathbb{C}[\epsilon_i]/\big(\epsilon_i^2\big)
\hookrightarrow
{\mathcal T'}\times
 \Spec\mathbb{C}[\epsilon_1,\epsilon_2]/\big(\epsilon_1^2,\epsilon_1\epsilon_2,\epsilon_2^2\big)
\xrightarrow{\bar{u}}{\mathcal T'}$
coincides with the morphism $I_{v_i}$
for $i=1,2$.
Let
\[
 \Delta_{\mathcal T'}\colon \
 {\mathcal T'}\times\Spec\mathbb{C}[\epsilon]/\big(\epsilon^2\big)
 \longrightarrow
 {\mathcal T'}\times
 \Spec\mathbb{C}[\epsilon_1,\epsilon_2]/\big(\epsilon_1^2,\epsilon_1\epsilon_2,\epsilon_2^2\big)
\]
be the morphism corresponding to the ring homomorphism
\[
 {\mathcal O} _{\mathcal T'}[\epsilon_1,\epsilon_2]
 /\big(\epsilon_1^2,\epsilon_1\epsilon_2,\epsilon_2^2\big)
 \ni a+b_1\epsilon_1+b_2\epsilon_2
 \mapsto a+b_1\epsilon+b_2\epsilon
 \in
 {\mathcal O}_{\mathcal T'}[\epsilon]/\big(\epsilon^2\big).
\]
Then the composition
\[
 \bar{u}\circ\Delta_{\mathcal T'} \colon \
 {\mathcal T'}\times\Spec\mathbb{C}[\epsilon]/\big(\epsilon^2\big)
 \xrightarrow { \Delta_{\mathcal T'}}
 {\mathcal T'}\times
 \Spec\mathbb{C}[\epsilon_1,\epsilon_2]/\big(\epsilon_1^2,\epsilon_1\epsilon_2,\epsilon_2^2\big)
 \xrightarrow{\bar{u}}
 {\mathcal T'}
\]
coincides with the morphism
$I_{v_1+v_2}$ corresponding to the vector field $v_1+v_2$.
By virtue of Proposition~\ref{proposition: horizontal lift in two variable},
there exists a horizontal lift
$\big({\mathcal E}^{\bar{u}},\nabla^{\bar{u}},l^{\bar{u}},\ell^{\bar{u}},{\mathcal V}^{\bar{u}}\big)$
of
$\big(\tilde{E},\tilde{\nabla},\tilde{l},\tilde{\ell},\tilde{\mathcal V}\big)_{\tilde{M}'}$
with respect to
$\bar{u}$.
By the same procedure as Definition~\ref{definition: generalized isomonodromy splitting},
the flat family of connections induced by the horizontal lift
$\big({\mathcal E}^{\bar{u}},\nabla^{\bar{u}},l^{\bar{u}},\ell^{\bar{u}},{\mathcal V}^{\bar{u}}\big)$
provides a morphism
$I_{\Phi(\bar{u})}\colon
\tilde{M'}
 \times\mathbb{C}[\epsilon_1,\epsilon_2]/\big(\epsilon_1^2,\epsilon_1\epsilon_2,\epsilon_2^2\big)
 \longrightarrow\tilde{M'}$
such that the right square of the diagram
\[
 \begin{CD}
 \tilde{M}'\times\Spec\mathbb{C}[\epsilon]/\big(\epsilon^2\big)
 @>\Delta_{\tilde{M'}}>>
 \tilde{M}'\times\Spec\mathbb{C}[\epsilon_1,\epsilon_2]/\big(\epsilon_1^2,\epsilon_1\epsilon_2,\epsilon_2^2\big)
 @>I_{\Phi(\bar{u})} >> \tilde{M'}
 \\
 @VVV @VVV @VVV \\
 {\mathcal T'}\times\Spec\mathbb{C}[\epsilon]/\big(\epsilon^2\big)
 @> { \Delta_{\mathcal T'}} >>
 {\mathcal T'}\times
 \Spec\mathbb{C}[\epsilon_1,\epsilon_2]/\big(\epsilon_1^2,\epsilon_1\epsilon_2,\epsilon_2^2\big)
 @>{\bar{u}}>>
 {\mathcal T'}
 \end{CD}
\]
is commutative.
The left square of the above diagram is defined as a Cartesian diagram.
By the definition of horizontal lift, the pullback
$\Delta_{\tilde{M}'}^*
\big({\mathcal E}^{\bar{u}},\nabla^{\bar{u}},l^{\bar{u}},\ell^{\bar{u}},{\mathcal V}^{\bar{u}}\big)$
is a horizontal lift of
$\big(\tilde{E},\tilde{\nabla},\tilde{l},\tilde{\ell},\tilde{\mathcal V}\big)_{\tilde{M}'}$
with respect to $I_{v_1+v_2}$.
So the composition
$I_{\Phi(\bar{u})}\circ\Delta_{\tilde{M'}}$
coincides with the morphism
$I_{\Phi(v_1+v_2)}$ determined by the vector field
$\Phi(v_1+v_2)$.
On the other hand, the morphism~$I_{\Phi(\bar{u})}$ corresponds to the pair
$(\Phi(v_1),\Phi(v_2))$ of vector fields and
the composition
$I_{\Phi(\bar{u})}\circ\Delta_{\tilde{M'}}$ corresponds to
the vector field $\Phi(v_1)+\Phi(v_2)$.
So we have the equality
\[
 I_{\Phi(v_1+v_2)}=I_{\Phi(v_1)+\Phi(v_2)}
\]
which means the equality $\Phi(v_1+v_2)=\Phi(v_1)+\Phi(v_2)$.

Take a vector field $v\in H^0({\mathcal T'},T_{\mathcal T'})$
and a regular function
$f\in H^0({\mathcal T'},{\mathcal O}_{\mathcal T'})$.
Consider the morphism
\[
 \alpha_f\colon \ {\mathcal T'}\times\Spec\mathbb{C}[\epsilon]/\big(\epsilon^2\big)
 \longrightarrow
 {\mathcal T'}\times\Spec\mathbb{C}[\epsilon]/\big(\epsilon^2\big)
\]
corresponding to the ring homomorphism
\[
 {\mathcal O}_{\mathcal T'}[\epsilon]/\big(\epsilon^2\big)
 \ni a+\epsilon b \mapsto a+\epsilon fb \in
 {\mathcal O}_{\mathcal T'}[\epsilon]/\big(\epsilon^2\big).
\]
Then the composition
\[
 {\mathcal T'}\times\Spec\mathbb{C}[\epsilon]/\big(\epsilon^2\big)
 \xrightarrow { \alpha_f}
 {\mathcal T'}\times\Spec\mathbb{C}[\epsilon]/\big(\epsilon^2\big)
 \xrightarrow {I_v} {\mathcal T'}
\]
coincides with the morphism~$I_{fv}$
corresponding to the vector field~$fv$.
As in Definition~\ref{definition: generalized isomonodromy splitting},
the horizontal lift
$\big({\mathcal E}^v,\nabla^v, l^v,\ell^v,{\mathcal V}^v \big)$
induces a morphism
$I_{\Phi(v)} \colon
\tilde{M}'\times\Spec\mathbb{C}[\epsilon]/\big(\epsilon^2\big)
\longrightarrow\tilde{M'}$
which makes the diagram
\[
 \begin{CD}
 \tilde{M'}\times\Spec\mathbb{C}[\epsilon]/\big(\epsilon^2\big)
 @> (\alpha_f)_{\tilde{M'}} >>
 \tilde{M'}\times\Spec\mathbb{C}[\epsilon]/\big(\epsilon^2\big)
 @>I_{\Phi(v)} >> \tilde{M'}
 \\
 @VVV @VVV @VVV
 \\
 {\mathcal T'}\times\Spec\mathbb{C}[\epsilon]/\big(\epsilon^2\big)
 @> \alpha_f >>
 {\mathcal T'}\times\Spec\mathbb{C}[\epsilon]/\big(\epsilon^2\big)
 @>I_v >> {\mathcal T'}
 \end{CD}
\]
commutative, where the right square is Cartesian.
By the definition of horizontal lift,
the pullback
$(\alpha_f)_{\tilde{M'}}^*\big({\mathcal E}^v,\nabla^v, l^v,\ell^v,{\mathcal V}^v \big)$
is a horizontal lift of
$\big(\tilde{E},\tilde{\nabla},\tilde{l},\tilde{\ell},\tilde{\mathcal V}\big)_{\tilde{M}'}$
with respect to $fv$.
So the composition
$I_{\Phi(v)}\circ (\alpha_f)_{\tilde{M'}}$
coincides with the morphism
$I_{\Phi(fv)}$ corresponding to $\Phi(fv)$.
On the other hand,
the composition
$I_{\Phi(v)}\circ (\alpha_f)_{\tilde{M'}}$
coincides with the morphism
$I_{f\Phi(v)}$ corresponding to the vector field
$f\Phi(v)$.
So we have
$\Phi(fv)=f\Phi(v)$.
\end{proof}

By Proposition \ref{proposition: isomonodromic splitting is a homomorphism},
$\Phi$ defines a homomorphism
\[
 \Phi\colon \
 T_{\mathcal T} \longrightarrow
 (\pi_{\mathcal T})_*
 T_{M^{\balpha}_{{\mathcal C},{\mathcal D}}(\lambda,\tilde{\mu},\tilde{\nu})}
\]
of sheaves of ${\mathcal O}_{\mathcal T}$-modules.
By the adjoint property, $\Phi$ corresponds to a
homomorphism
\begin{equation} \label{equation: generalized isomonodromic splitting homomorphism}
 \Psi \colon \
 (\pi_{\mathcal T})^* T_{\mathcal T}
 \longrightarrow
 T_{M^{\balpha}_{{\mathcal C},{\mathcal D}}(\lambda,\tilde{\mu},\tilde{\nu})}
\end{equation}
Since the diagram (\ref{equation: diagram defining isomonodromic vector})
in Definition \ref{definition: generalized isomonodromy splitting} is commutative,
we can see
${\rm d}\pi_{\mathcal T}\circ\Psi=\mathrm{id}_{T_{\mathcal T}}$
for the canonical surjection
${\rm d}\pi_{\mathcal T} \colon
T_{M^{\balpha}_{{\mathcal C},{\mathcal D}}(\lambda,\tilde{\mu},\tilde{\nu})}
\longrightarrow
(\pi_{\mathcal T})^* T_{\mathcal T}$.
In particular,
the image
$\im \Psi$ is a~subbundle of
$T_{M^{\balpha}_{{\mathcal C},{\mathcal D}}(\lambda,\tilde{\mu},\tilde{\nu})}$.

\begin{Definition}\label{definition: generalized isomonodromic subbundle}
We call $\im\Psi$ the
generalized isomonodromic subbundle of
$T_{M^{\balpha}_{{\mathcal C},{\mathcal D}}(\lambda,\tilde{\mu},\tilde{\nu})}$.
\end{Definition}

By using the generalized isomonodromic subbundle $\im\Psi$,
we can extend the relative symplectic form
$\omega_{M^{\balpha}_{{\mathcal C},{\mathcal D}}(\lambda,\tilde{\mu},\tilde{\nu})}$
constructed in Theorem \ref{theorem: existence of symplectic form and d-closedness}
to a total $2$-form on the moduli space
$M^{\balpha}_{{\mathcal C},{\mathcal D}}(\lambda,\tilde{\mu},\tilde{\nu})$
in the following.

\begin{Definition}\label{definition: generalized isomonodromic 2-form}
We define a $2$-form
$\omega_{M^{\balpha}_{{\mathcal C},{\mathcal D}}(\lambda,\tilde{\mu},\tilde{\nu})}^{\mathrm{GIM}}$
on $M^{\balpha}_{{\mathcal C},{\mathcal D}}(\lambda,\tilde{\mu},\tilde{\nu})$
by setting
\[
 \omega_{M^{\balpha}_{{\mathcal C},{\mathcal D}}(\lambda,\tilde{\mu},\tilde{\nu})}^{\mathrm{GIM}}
 (v_1,v_2)
 =
 \omega_{M^{\balpha}_{{\mathcal C},{\mathcal D}}(\lambda,\tilde{\mu},\tilde{\nu})}
 \big(v_1-\Psi({\rm d}\pi_{\mathcal T}(v_1)), v_2-\Psi({\rm d}\pi_{\mathcal T}(v_2))\big)
\]
for $v_1,v_2\in
T_{M^{\balpha}_{{\mathcal C},{\mathcal D}}(\lambda,\tilde{\mu},\tilde{\nu})}$
and call it the generalized isomonodromic $2$-form.
\end{Definition}

\begin{Remark}
In the logarithmic case,
the above formulation of isomonodromic $2$-form is given by
A.~Komyo in~\cite{Komyo-1}.
For a vector field $v\in T_{M^{\balpha}_{{\mathcal C},{\mathcal D}}(\lambda,\tilde{\mu},\tilde{\nu})}$
we can immediately see the equivalence
\[
 v \in \im \Psi \Leftrightarrow
 \omega_{M^{\balpha}_{{\mathcal C},{\mathcal D}}(\lambda,\tilde{\mu},\tilde{\nu})}^{\mathrm{GIM}}
 (v,w)=0 \qquad \text{for any $w\in T_{M^{\balpha}_{{\mathcal C},{\mathcal D}}(\lambda,\tilde{\mu},\tilde{\nu})}$}
\]
from the definition of the generalized isomonodromic $2$-form.
So the generalized isomonodromic $2$-form recovers the generalized isomonodromic
subbundle.
\end{Remark}

\begin{Theorem} \label{theorem: integrability of generalized isomonodromy}
For any vector fields $v_1,v_2\in T_{\mathcal T}$,
the equality
\begin{equation*}
 \Phi([v_1,v_2])=[\Phi(v_1),\Phi(v_2)]
\end{equation*}
holds, where $[v_1,v_2]=v_1v_2-v_2v_1$ is the commutator of the vector fields $v_1$, $v_2$.
In particular,
the generalized isomonodromic subbundle $\im\Psi$ of
$T_{M^{\balpha}_{{\mathcal C},{\mathcal D}}(\lambda,\tilde{\mu},\tilde{\nu})}$
satisfies the integrability condition
\[
 [\im\Psi,\im\Psi]\subset\im\Psi.
\]
\end{Theorem}

\begin{proof}Take vector fields
$v_1,v_2\in H^0({\mathcal T'},T_{\mathcal T}|_{\mathcal T'})$
over a Zariski open subset ${\mathcal T'}$ of ${\mathcal T}$.
Let
\[
 \tilde{I}_{v_1} \colon \
 {\mathcal T'}\times\Spec\mathbb{C}[\epsilon_1,\epsilon_2]/\big(\epsilon_1^2,\epsilon_2^2\big)
 \longrightarrow
 {\mathcal T'}\times\Spec\mathbb{C}[\epsilon_1,\epsilon_2]/\big(\epsilon_1^2,\epsilon_2^2\big)
\]
be the automorphism corresponding to the ring automorphism $\tilde{I}_{v_1}^*$ of
${\mathcal O}_{\mathcal T'}[\epsilon_1,\epsilon_2]/\big(\epsilon_1^2,\epsilon_2^2\big)$
defined by
\[
 \tilde{I}_{v_1}^* \left( a+b_1\epsilon_1+b_2\epsilon_2+c\, \epsilon_1\epsilon_2 \right)
 =
 a+(v_1(a)+b_1)\epsilon_1+b_2\epsilon_2+(v_1(b_2)+c)\epsilon_1\epsilon_2.
\]
Similarly, we can define an automorphism
$\tilde{I}_{v_2}$ of
${\mathcal T'}\times\Spec\mathbb{C}[\epsilon_1,\epsilon_2]/\big(\epsilon_1^2,\epsilon_2^2\big)$
corresponding to $v_2$.
By construction, we can see that
$\tilde{I}_{-v_1}=\tilde{I}_{v_1}^{-1}$ and $\tilde{I}_{-v_2}=\tilde{I}_{v_2}^{-1}$.
The composition
$\tilde{I}_{v_2}\circ\tilde{I}_{v_1}\circ\tilde{I}_{-v_2}\circ\tilde{I}_{-v_1}$
corresponds to the ring automorphism of
${\mathcal O}_{\mathcal T'}[\epsilon_1,\epsilon_2]/\big(\epsilon_1^2,\epsilon_2^2\big)$
determined by
\begin{gather*}
 \tilde{I}_{-v_1}^*\circ\tilde{I}_{-v_2}^*\circ\tilde{I}_{v_1}^*\circ\tilde{I}_{v_2}^*
 ( a+b_1\epsilon_1+b_2\epsilon_2+c\,\epsilon_1\epsilon_2 )
 \\
\quad{} =
 \tilde{I}_{-v_1}^*\circ\tilde{I}_{-v_2}^*\circ\tilde{I}_{v_1}^*
 ( a+b_1\epsilon_1+(v_2(a)+b_2)\epsilon_2+(c+v_2(b_1))\epsilon_1\epsilon_2 )
 \\
\quad{} =
 \tilde{I}_{-v_1}^*\circ\tilde{I}_{-v_2}^*
 ( a+(v_1(a)+b_1)\epsilon_1+(v_2(a)+b_2)\epsilon_2
 +(v_1v_2(a)+v_1(b_2)+c+v_2(b_1))\epsilon_1\epsilon_2 )
 \\
 \quad {} =
 \tilde{I}_{-v_1}^* (
 a+(v_1(a)+b_1)\epsilon_1+b_2\epsilon_2
 +(-v_2v_1(a)+v_1v_2(a)+v_1(b_2)+c)\epsilon_1\epsilon_2
 )
 \\
\quad{} =
 a+b_1\epsilon_1+b_2\epsilon_2+
 ((v_1v_2-v_2v_1)(a)+c)\epsilon_1\epsilon_2.
\end{gather*}
Let
\[
 \rho \colon \
 {\mathcal T'}\times\Spec\mathbb{C}[\epsilon_1,\epsilon_2]/\big(\epsilon_1^2,\epsilon_2^2\big)
 \longrightarrow
 {\mathcal T'}\times\Spec\mathbb{C}[\epsilon]/\big(\epsilon^2\big)
\]
be the morphism
corresponding to the ring homomorphism
$\rho^* \colon
{\mathcal O}_{\mathcal T'}[\epsilon]/\big(\epsilon^2\big)
\rightarrow {\mathcal O}_{\mathcal T'}[\epsilon_1,\epsilon_2]/\big(\epsilon_1^2,\epsilon_2^2\big)$
determined by
$\rho^* (a+c \epsilon) =a+c\epsilon_1\epsilon_2$.
Then the composition
\begin{equation} \label{composition on T-space (1)}
 {\mathcal T'}\times\Spec\mathbb{C}[\epsilon_1,\epsilon_2]/\big(\epsilon_1^2,\epsilon_2^2\big)
 \stackrel {\rho} \longrightarrow
 {\mathcal T'}\times\Spec\mathbb{C}[\epsilon]/\big(\epsilon^2\big)
 \xrightarrow{I_{v_1v_2-v_2v_1}}
 {\mathcal T'}
\end{equation}
coincides with the composition
\begin{gather}
 {\mathcal T'}\times\Spec\mathbb{C}[\epsilon_1,\epsilon_2]/\big(\epsilon_1^2,\epsilon_2^2\big)
 \xrightarrow { \tilde{I}_{v_2}\circ\tilde{I}_{v_1}\circ\tilde{I}_{v_2}^{-1}\circ\tilde{I}_{v_1}^{-1} }
 {\mathcal T'}\times\Spec\mathbb{C}[\epsilon_1,\epsilon_2]/\big(\epsilon_1^2,\epsilon_2^2\big)\nonumber\\
 \hphantom{{\mathcal T'}\times\Spec\mathbb{C}[\epsilon_1,\epsilon_2]/\big(\epsilon_1^2,\epsilon_2^2\big)
 \xrightarrow { \tilde{I}_{v_2}\circ\tilde{I}_{v_1}\circ\tilde{I}_{v_2}^{-1}\circ\tilde{I}_{v_1}^{-1} }}{}
 \xrightarrow {\text{trivial projection}}
 {\mathcal T'}. \label{composition on T-space (2)}
\end{gather}

By Proposition \ref{proposition: horizontal lift in two variable},
there exists a horizontal lift
$\big({\mathcal E}^{\tilde{v}_i},\nabla^{\tilde{v}_i},
l^{\tilde{v}_i},\ell^{\tilde{v}_i},{\mathcal V}^{\tilde{v}_i} \big)$
of $\big(\tilde{E},\tilde{\nabla},\tilde{l},\tilde{\ell},\tilde{\mathcal V}\big)_{\tilde{M}'}$
with respect to the morphism
\[
 {\mathcal T'}\times\Spec\mathbb{C}[\epsilon_1,\epsilon_2]/\big(\epsilon_1^2,\epsilon_2^2\big)
 \xrightarrow {\tilde{I}_{v_i}}
 {\mathcal T'}\times\Spec\mathbb{C}[\epsilon_1,\epsilon_2]/\big(\epsilon_1^2,\epsilon_2^2\big)
 \xrightarrow {\text{trivial projection}} {\mathcal T'}.
\]
For the relative connection
$\overline { \nabla^{\tilde{v}_i} }$ induced by $\nabla^{\tilde{v}_i}$,
the flat family
$\big({\mathcal E}^{\tilde{v}_i},\overline{\nabla^{\tilde{v}_i}} ,
l^{\tilde{v}_i},\ell^{\tilde{v}_i},{\mathcal V}^{\tilde{v}_i} \big)$
determines a morphism
$I_{\Phi(\tilde{v}_i)}\colon
\tilde{M}'\times\Spec\mathbb{C}[\epsilon_1,\epsilon_2]/\big(\epsilon_1^2,\epsilon_2^2\big)
\longrightarrow \tilde{M}'$
which is canonically extended to a morphism
\[
 \tilde{I}_{\Phi(\tilde{v}_i)}
 \colon \
 \tilde{M}'\times\Spec\mathbb{C}[\epsilon_1,\epsilon_2]/\big(\epsilon_1^2,\epsilon_2^2\big)
 \longrightarrow
 \tilde{M}'\times\Spec\mathbb{C}[\epsilon_1,\epsilon_2]/\big(\epsilon_1^2,\epsilon_2^2\big)
\]
over $\Spec\mathbb{C}[\epsilon_1,\epsilon_2]/\big(\epsilon_1^2,\epsilon_2^2\big)$.
Furthermore, the diagram
\[
 \begin{CD}
 \tilde{M}'\times\Spec\mathbb{C}[\epsilon_1,\epsilon_2]/\big(\epsilon_1^2,\epsilon_2^2\big)
 @> \tilde{I}_{\Phi(\tilde{v}_2)}\circ\tilde{I}_{\Phi(\tilde{v}_1)}
 \circ\tilde{I}_{\Phi(\tilde{v}_2)}^{-1}\circ\tilde{I}_{\Phi(\tilde{v}_1)}^{-1} >>
 \tilde{M}'\times\Spec\mathbb{C}[\epsilon_1,\epsilon_2]/\big(\epsilon_1^2,\epsilon_2^2\big)
 \\
 @VVV @VVV \\
 {\mathcal T'}\times\Spec\mathbb{C}[\epsilon_1,\epsilon_2]/\big(\epsilon_1^2,\epsilon_2^2\big)
 @> \tilde{I}_{v_2}\circ\tilde{I}_{v_1}
 \circ\tilde{I}_{v_2}^{-1}\circ\tilde{I}_{v_1}^{-1} >>
 {\mathcal T'}\times\Spec\mathbb{C}[\epsilon_1,\epsilon_2]/\big(\epsilon_1^2,\epsilon_2^2\big)
 \end{CD}
\]
is commutative.

By the definition of horizontal lift,
we can see that the pullback
\[
 \big(\tilde{I}_{\Phi(\tilde{v}_1)}^{-1}\big)^*\big(\tilde{I}_{\Phi(\tilde{v}_2)}^{-1}\big)^*
 \tilde{I}_{\Phi(\tilde{v}_1)}^*\big({\mathcal E}^{\tilde{v}_2},\nabla^{\tilde{v}_2} ,
 l^{\tilde{v}_2},\ell^{\tilde{v}_2},{\mathcal V}^{\tilde{v}_2} \big)
\]
becomes a horizontal lift of
$\big(\tilde{E},\tilde{\nabla},\tilde{l},\tilde{\ell},\tilde{\mathcal V}\big)_{\tilde{M}'}$
with respect to the morphism
(\ref{composition on T-space (2)}).
On the other hand, there is a canonical commutative diagram
\[
 \begin{CD}
 \tilde{M'}\times\Spec\mathbb{C}[\epsilon_1,\epsilon_2]/\big(\epsilon_1^2,\epsilon_2^2\big)
 @> \rho_{\tilde{M'}} >>
 \tilde{M'}\times\Spec\mathbb{C}[\epsilon]/\big(\epsilon^2\big)
 @> I_{\Phi(v_1v_2-v_2v_1)} >> \tilde{M'}
 \\
 @VVV @VVV @VVV
 \\
 {\mathcal T'}\times\Spec\mathbb{C}[\epsilon_1,\epsilon_2]/\big(\epsilon_1^2,\epsilon_2^2\big)
 @>\rho>>
 {\mathcal T'}\times\Spec\mathbb{C}[\epsilon]/\big(\epsilon^2\big)
 @> I_{v_1v_2-v_2v_1} >>
 {\mathcal T'},
 \end{CD}
\]
whose left square is Cartesian.
So we can see that the pullback
\[
 \rho_{\tilde{M'}}^*\big({\mathcal E}^{v_1v_2-v_2v_1},\nabla^{v_1v_2-v_2v_1} ,
 l^{v_1v_2-v_2v_1},\ell^{v_1v_2-v_2v_1},{\mathcal V}^{v_1v_2-v_2v_1} \big)
\]
becomes a horizontal lift of
$\big(\tilde{E},\tilde{\nabla},\tilde{l},\tilde{\ell},\tilde{\mathcal V}\big)_{\tilde{M}'}$
with respect to the morphism
(\ref{composition on T-space (1)}).
Since the morphism (\ref{composition on T-space (2)})
coincides with the morphism (\ref{composition on T-space (1)}),
we can deduce an isomorphism
\begin{gather*}
 \big(\tilde{I}_{\Phi(\tilde{v}_1)}^{-1}\big)^*\big(\tilde{I}_{\Phi(\tilde{v}_2)}^{-1}\big)^*
 \tilde{I}_{\Phi(\tilde{v}_1)}^*
 \big({\mathcal E}^{\tilde{v}_2},\nabla^{\tilde{v}_2} ,
 l^{\tilde{v}_2},\ell^{\tilde{v}_2},{\mathcal V}^{\tilde{v}_2}\big)
 \\
 \qquad{} \cong
 \rho_{\tilde{M'}}^*\big({\mathcal E}^{v_1v_2-v_2v_1},\nabla^{v_1v_2-v_2v_1} ,
 l^{v_1v_2-v_2v_1},\ell^{v_1v_2-v_2v_1},{\mathcal V}^{v_1v_2-v_2v_1} \big)
\end{gather*}
by the uniqueness of horizontal lift
proved in Proposition \ref{proposition: horizontal lift in two variable}.
Considering the induced morphism, we have
\[
 (\text{trivial projection})\circ
 \tilde{I}_{\Phi(\tilde{v}_2)\circ\tilde{I}_{\Phi(\tilde{v}_1)}
 \circ\tilde{I}_{\Phi(\tilde{v}_2)}^{-1}\circ\tilde{I}_{\Phi(\tilde{v}_1)}^{-1}}
 =
 I_{\Phi(v_1v_2-v_2v_1)}\circ\rho_{\tilde{M'}},
\]
from which we get
$\Phi(v_1v_2-v_2v_1)=
\Phi(v_1)\Phi(v_2)-\Phi(v_2)\Phi(v_1)$.
\end{proof}

\begin{Definition}\label{definition: generalized isomonodromic foliation}
Since the subbundle
$\im\Psi\subset T_{M^{\balpha}_{{\mathcal C},{\mathcal D}}(\lambda,\tilde{\mu},\tilde{\nu})}$
satisfies the integrability condition by Theorem~\ref{theorem: integrability of generalized isomonodromy},
it determines a foliation
${\mathcal F}^{\rm GIM}_{T_{M^{\balpha}_{{\mathcal C},{\mathcal D}}(\lambda,\tilde{\mu},\tilde{\nu})}}$
on the moduli space
$M^{\balpha}_{{\mathcal C},{\mathcal D}}(\lambda,\tilde{\mu},\tilde{\nu})$.
We call it the generalized isomonodromic foliation.
\end{Definition}

Take a point $t_0\in {\mathcal T}$
and a point $y$ of the fiber
$M^{\balpha}_{{\mathcal C},{\mathcal D}}(\lambda,\tilde{\mu},\tilde{\nu})_{t_0}$
over $t_0$.
Then we can take an analytic open neighborhood
${\mathcal M'}$ of $y$ in
$M^{\balpha}_{{\mathcal C},{\mathcal D}}(\lambda,\tilde{\mu},\tilde{\nu})$
and an analytic open neighborhood ${\mathcal T'}$ of $t_0$ in ${\mathcal T}$
together with an analytic isomorphism
\begin{equation}
\label{equation: analytic isomorphism from isomonodromic foliation}
 {\mathcal M'}\cong {\mathcal M}'_{t_0}\times{\mathcal T'}
\end{equation}
such that the restriction $\pi_{\mathcal T}|_{\mathcal M'}$ of
$\pi_{\mathcal T} \colon M^{\balpha}_{{\mathcal C},{\mathcal D}}(\lambda,\tilde{\mu},\tilde{\nu})
\longrightarrow {\mathcal T}$
coincides with the second projection
and that the fibers $\{ \{ y' \} \times {\mathcal T'} \}_{y'\in {\mathcal M'_{t_0}}}$
over ${\mathcal M'_{t_0}}$ are leaves in
${\mathcal F}^{\rm GIM}_{M^{\balpha}_{{\mathcal C},{\mathcal D}}(\lambda,\tilde{\mu},\tilde{\nu})}$.

Take a holomorphic system of coordinates
$\theta=(\theta_1,\dots,\theta_N)$ of ${\mathcal T'}$.
If we set
\[
 {\mathcal T'}[\partial\theta]:=
 {\mathcal T'}\times\Spec\mathbb{C}[\epsilon_1,\dots,\epsilon_N]
 \big/ \big( \epsilon_i\epsilon_j \mid 1\leq i,j \leq N \big),
\]
then the tuple
$\partial\theta=
(\partial/\partial\theta_1,\dots,\partial/\partial\theta_N)$
of vector fields on ${\mathcal T'}$
corresponds to a morphism
\[
 I_{\partial\theta}\colon \
 {\mathcal T'}[\partial\theta]
 \longrightarrow {\mathcal T'},
\]
whose restriction to ${\mathcal T'}\subset {\mathcal T'}[\partial \theta]$ is the identity morphism.

By the same proof as Proposition \ref{proposition: horizontal lift in two variable},
we can construct a horizontal lift
$\big({\mathcal E}^{\partial\theta} ,\nabla^{\partial\theta} ,l^{\partial\theta} ,
\ell^{\partial\theta} ,\allowbreak {\mathcal V}^{\partial \theta}\big)$
of the universal family
$(E,\nabla,l,\ell,{\mathcal V})$
on ${\mathcal C}\times_{\mathcal T}{\mathcal M'}$
with respect to the morphism $I_{\partial\theta}$.
On a~small open subset
$U\subset {\mathcal C}\times_{\mathcal T}{\mathcal M'}$,
we may assume
$E|_U\cong{\mathcal O}_U^{\oplus r}$.
Then we can write
$\nabla|_U={\rm d}+A{\rm d}z$
where $z$ is a holomorphic coordinate on ${\mathcal C}\times_{\mathcal T}{\mathcal M'}$
over ${\mathcal M'}$ and $A$ is a matrix of meromorphic functions in $z$.
Let $U[\partial\theta]\subset {\mathcal C}\times_{\mathcal T}{\mathcal M'}[\partial\theta]$
be the open subscheme whose underlying set is $U$.
Then we have
${\mathcal E}|_{U[\partial\theta]}\cong{\mathcal O}_{U[\partial\theta]}^{\oplus r}$
and we can write
$\nabla^{\partial\theta}={\rm d}+A(\epsilon){\rm d}z+\sum_{j=1}^N B_j {\rm d}\epsilon_j$,
where $A(\epsilon)$ is a lift of~$A$.
By the integrability condition of $\nabla^{\partial\theta}$, we have the equality
\[
 \sum_{j=1}^N \frac{\partial A(\epsilon)}{\partial \epsilon_j}{\rm d}\epsilon_j\wedge {\rm d}z
 +\sum_{j=1}^N \frac{\partial B_j}{\partial z}{\rm d}z\wedge {\rm d}\epsilon_j
 +\sum_{j=1}^N [A,B_j]{\rm d}z\wedge {\rm d}\epsilon_j
 =0.
\]
Take a holomorphic coordinate system
$x_1,\dots,x_{\delta}$ of ${\mathcal M}'_{t_0}$.
With respect to the coordinate system
$z,x_1,\dots,x_{\delta},\theta_1,\dots,\theta_N$,
the partial derivative
$\partial/\partial \theta_j$ coincides with the vector field
$\Phi(\partial /\partial \theta_j)$
and the partial derivative $\partial A/\partial \theta_j$
coincides with
$\partial A(\epsilon)/\partial \epsilon_j$.
So the above integrability condition of~$\nabla^{\partial\theta}$ is the same as the
integrability condition
\[
 \sum_{j=1}^N \frac{\partial A}{\partial \theta_j}{\rm d}\theta_j\wedge {\rm d}z
 +\sum_{j=1}^N \frac{\partial B_j}{\partial z}{\rm d}z\wedge {\rm d}\theta_j
 +\sum_{j=1}^N [A,B_j]{\rm d}z\wedge {\rm d}\theta_j=0
\]
of the connection
\[
 \nabla_U^{\rm flat}={\rm d}+A{\rm d}z+\sum_{j=1}^NB_j{\rm d}\theta_j
\]
on $E|_U$ relative to the composition
\[
U\hookrightarrow{\mathcal C}\times_{\mathcal T}{\mathcal M'}
 \longrightarrow
 {\mathcal M'}\cong{\mathcal M}'_{t_0}\times{\mathcal T'}
 \xrightarrow{\pi_{{\mathcal M}'_{t_0}}} {\mathcal M}'_{t_0},
 \]
where
$\pi_{{\mathcal M}'_{t_0}}\colon
{\mathcal M'}\cong{\mathcal M}'_{t_0}\times{\mathcal T'}
\longrightarrow {\mathcal M}'_{t_0}$
is the first projection.
So we can see from Theorem \ref{theorem: Jimbo-Miwa-Ueno equation}
and Corollary \ref{corollary: rigorous ramified local generalized isomonodromic deformation}
that $\nabla|_U$ is a local generalized isomonodromic deformation
in the sense of Definition \ref{definition: local generalized isomonodromic deformation}
or Definition \ref{definition: ramified local generalized isomonodromic deformation}.

We can patch $\nabla_U^{\rm flat}$ together to get a global connection on $E$.
Indeed, take another open subset $U'\subset{\mathcal C}\times_{\mathcal T}{\mathcal M'}$
and write $\nabla|_{U'}={\rm d}+A'{\rm d}z$.
Then we have
$P^{-1}{\rm d}P+P^{-1}A{\rm d}zP=A'{\rm d}z$
for a transition matrix~$P$.
There is a local horizontal lift ${\rm d}+A'(\epsilon){\rm d}z+\sum_{j=1}^NB'_j{\rm d}\epsilon_j$
of $\nabla|_{U'}$
and by the uniqueness of the local horizontal lift,
we have a uniquely lift $P(\epsilon)$ of~$P$
satisfying
\begin{gather*}
 P(\epsilon)^{-1}\left( \frac{\partial P(\epsilon)}{\partial z}{\rm d}z
 +\sum_{j=1}^N \frac{\partial P(\epsilon)}{\partial \epsilon_j} {\rm d}\epsilon_j \right)
 +P(\epsilon)^{-1} \left( A(\epsilon){\rm d}z+\sum_{j=1}^N B_j{\rm d}\epsilon_j \right) P(\epsilon)\\
 \qquad{}
 =A'(\epsilon){\rm d}z+\sum_{j=1}^NB'_j{\rm d}\epsilon_j.
\end{gather*}
Since
$\partial P(\epsilon)/\partial \epsilon_j
=\partial P/\partial \theta_j$,
the above equality yields the equality
\[
 P^{-1}\left( \frac{\partial P} {\partial z}{\rm d}z
 +\sum_{j=1}^N \frac{\partial P} {\partial \theta_j} {\rm d}\theta_j \right)
 +P^{-1} \left( A{\rm d}z+\sum_{j=1}^N B_j {\rm d}\theta_j \right) P
 =A'{\rm d}z+\sum_{j=1}^N B'_j {\rm d}\theta_j.
\]
So we can patch the local connections
$\nabla^{\rm flat}_U$ together
to get an integrable connection
\begin{equation} \label{global integrable connection analytic local over moduli}
 \nabla^{\rm flat} \colon \
 E \longrightarrow E\otimes
 \Omega_{{\mathcal C}\times_{\mathcal T}{\mathcal M'}/{\mathcal M}'_{t_0}}(D_{\mathcal M'})
\end{equation}
relative to the composition
${\mathcal C}\times_{\mathcal T}{\mathcal M'}
 \longrightarrow{\mathcal M'}\cong{\mathcal M}'_{t_0}\times{\mathcal T'}
 \longrightarrow {\mathcal M}'_{t_0}$.

\begin{Corollary} \label{corollary d-closedness of generalized isomonodromic 2-form}
The generalized isomonodromic $2$-form
$\omega_{M^{\balpha}_{{\mathcal C},{\mathcal D}}(\lambda,\tilde{\mu},\tilde{\nu})}^{\mathrm{GIM}}$
constructed in Definition~{\rm \ref{definition: generalized isomonodromic 2-form}} is ${\rm d}$-closed.
\end{Corollary}

\begin{proof}
Under the above notations, we will prove the equality
\begin{equation} \label{equation: isomonodromy 2-form is a pullback}
 \omega_{M^{\balpha}_{{\mathcal C},{\mathcal D}}
 (\lambda,\tilde{\mu},\tilde{\nu})}^{\mathrm{GIM}}\big|_{{\mathcal M}'}
 =
 \pi_{{\mathcal M}'_{t_0}}^* \big(
 \omega_{M^{\balpha}_{{\mathcal C},{\mathcal D}}
 (\lambda,\tilde{\mu},\tilde{\nu})}\big|_{{\mathcal M}'_{t_0}} \big)
\end{equation}
where
$\pi_{{\mathcal M}'_{t_0}} \colon {\mathcal M}'\cong {\mathcal M}'_{t_0}\times{\mathcal T'}
\longrightarrow {\mathcal M}'_{t_0}$
corresponds to the first projection with respect to the isomorphism
(\ref{equation: analytic isomorphism from isomonodromic foliation}).
The corollary follows from this equality, since
$\omega_{M^{\balpha}_{{\mathcal C},{\mathcal D}}
(\lambda,\tilde{\mu},\tilde{\nu})}|_{{\mathcal M}'_{t_0}}$
is $d$-closed by Theorem \ref{theorem: existence of symplectic form and d-closedness}.

Take two tangent vectors
$v,v'\in T_{{\mathcal M}'}(y,t_0)$
at $(y,t)\in{\mathcal M}'_{t_0}\times{\mathcal T'}$.
We have the equalities
\[
(\pi_{{\mathcal M}'_{t_0}})_*(v)=(\pi_{{\mathcal M}'_{t_0}})_*(v-\Phi(\pi_{\mathcal T'})_*(v)),
\qquad
(\pi_{{\mathcal M}'_{t_0}})_*(v')=(\pi_{{\mathcal M}'_{t_0}})_*(v-\Phi(\pi_{\mathcal T'})_*(v'))
\]
because $\{ y\times{\mathcal T'}\}_{y\in{\mathcal M}'_{t_0}}$
are leaves of the foliation
${\mathcal F}^{\rm GIM}_{M^{\balpha}_{{\mathcal C},{\mathcal D}}
(\lambda,\tilde{\mu},\tilde{\nu})}$,
which is determined by the subbundle $\im\Phi$
of $T_{M^{\balpha}_{{\mathcal C},{\mathcal D}}
(\lambda,\tilde{\mu},\tilde{\nu}) }$.
The tangent vector
$(\pi_{{\mathcal M}'_{t_0}})_*(v-\Phi(\pi_{\mathcal T'})_*(v))$
corresponds to
a morphism
$\Spec\mathbb{C}[\epsilon]/\big(\epsilon^2\big)\longrightarrow
{\mathcal M}'_{t_0}$.
Let{\samepage
\[
 \tilde{I}_v\colon \
 \Spec\mathbb{C}[\epsilon]/\big(\epsilon^2\big)\times{\mathcal T'}
 \longrightarrow {\mathcal M}'_{t_0}\times{\mathcal T'}
 \cong{\mathcal M}'
\]
be its base change.}

We can construct a complex
${\mathcal F}^{\bullet}$ of sheaves on ${\mathcal C}\times_{\mathcal T}{\mathcal M'}$
from the universal family
$(E,\nabla,l,\ell,{\mathcal V})$
in the same way as~(\ref{equation: definition of tangent complex})
in Section~\ref{section: tangent space}.
Since $(\mathrm{id}\times \tilde{I}_v)^*(E,\nabla,l,\ell,{\mathcal V})$
is a lift of $(E,\nabla,l,\ell,{\mathcal V})|_{{\mathcal C}_{y\times{\mathcal T'}}}$,
it induces a gluing data
$\{u_{\alpha\beta},v_{\alpha},\eta_{\alpha}\}$
with respect to an open covering $\{U_{\alpha}\}$ of
${\mathcal C}_{y\times{\mathcal T'}}:={\mathcal C}\times_{\mathcal T}(y\times{\mathcal T'})$.
as in Proposition~\ref{prop: tangent space of the moduli space}.
Set
\[
 \tilde{v}:=[\{ u_{\alpha\beta}, v_{\alpha},\eta_{\alpha} \}]
 \in \mathbb{R}^1(p_{y\times{\mathcal T'}})_*
 \big( {\mathcal F}^{\bullet}| _{{\mathcal C}_{y\times{\mathcal T'}}} \big).
\]
Then we can see from the construction of $\tilde{v}$ that the equalities
\begin{gather*}
 \tilde{v}|_{(y,t_0)}
 =
 (\pi_{{\mathcal M'}_{t_0}})_*(v)
\in \mathbb{H}^1\big({\mathcal F}^{\bullet}|_{{\mathcal C}_{(y,t_0)}}\big),
 \\
\tilde{v}|_{(y,t)}
 =
 v-\Phi(\pi_{\mathcal T'})_*(v)
 \in \mathbb{H}^1\big({\mathcal F}^{\bullet}|_{{\mathcal C}_{(y,t)}}\big).
\end{gather*}
hold.
We can similarly construct an element
$\tilde{v'}=[\{ u'_{\alpha\beta}, v'_{\alpha},\eta'_{\alpha} \}]$ of
$\mathbb{R}^1(p_{y\times{\mathcal T'}})_*
\big( {\mathcal F}^{\bullet}| _{{\mathcal C}_{y\times{\mathcal T'}}} \big)$
from the tangent vector $v'$.
Recall the construction of the complex ${\mathcal F}^{\bullet}$
in~(\ref{equation: definition of tangent complex}).
Since the map ${\mathcal G}^1\longrightarrow G^1$ is surjective
and the map ${\mathcal G}^0 \longrightarrow {\mathcal S}^1_{\rm ram}$
is a surjection to the kernel of the surjection
${\mathcal S}^1_{\rm ram}\longrightarrow A^1$,
we can replace $u_{\alpha\beta}$, $v_{\alpha}$ so that
$\eta_{\alpha}=0$ holds.
Similarly we may assume $\eta'_{\alpha}=0$.

Consider the pairing
\[
 \omega \big(\tilde{v},\tilde{v'} \big)
 :=
 \big[ \big\{ \Tr(u_{\alpha\beta} u'_{\beta\gamma}) ,
 -\Tr(u_{\alpha\beta}v_{\beta}-v_{\alpha}u'_{\alpha\beta}) \big\} \big]
 \in \mathbb{R}^2(p_{y\times{\mathcal T'}})_*
 \big( \Omega^{\bullet}_{{\mathcal C}_{y\times{\mathcal T'}}/{\mathcal T'}} \big)
 \cong{\mathcal O}_{\mathcal T'}
\]
of $\tilde{v}$ and $\tilde{v'}$.
Then we have the equalities
\begin{gather*}
 \pi_{{\mathcal M'}_{t_0}}^* \big(
 \omega_{M^{\balpha}_{{\mathcal C},{\mathcal D}}(\lambda,\tilde{\mu},\tilde{\nu})}
 |_{{\mathcal M'}_{t_0}} \big) (v,v')
 =
 \omega_{M^{\balpha}_{{\mathcal C},{\mathcal D}}(\lambda,\tilde{\mu},\tilde{\nu})_{t_0}}
 \big( (\pi_{{\mathcal M'}_{t_0}})_*(v) , (\pi_{{\mathcal M'}_{t_0}})_*(v') \big)
 =
 \omega \big(\tilde{v},\tilde{v'} \big)|_{(y,t_0)} ,
 \\
 \omega^{\rm GIM}_{M^{\balpha}_{{\mathcal C},{\mathcal D}}(\lambda,\tilde{\mu},\tilde{\nu})}
 (v,v')
 =
 \omega_{M^{\balpha}_{{\mathcal C},{\mathcal D}}(\lambda,\tilde{\mu},\tilde{\nu})_{t}}
 \big( v-\Phi(\pi_{\mathcal T'})_*(v) , v'-\Phi(\pi_{\mathcal T'})_*(v') \big)
 =
 \omega \big(\tilde{v},\tilde{v'} \big)|_{(y,t)}.
\end{gather*}
So, in order to prove (\ref{equation: isomonodromy 2-form is a pullback}),
we only have to prove that
$\omega \big(\tilde{v},\tilde{v'} \big)
 \in {\mathcal O}_{\mathcal T'}$
is constant on ${\mathcal T'}$.
We may assume that ${\mathcal T}'$ is isomorphic to a polydisk.
Then it is sufficient to show that
$\omega \big(\tilde{v},\tilde{v'} \big)$
belongs to the image of the canonical map
\begin{gather}
 \mathbb{C}\cong
 \mathbb{H}^2\big({\mathcal O}_{{\mathcal C}_{y\times{\mathcal T}'}}
\xrightarrow{{\rm d}} \Omega^1_{{\mathcal C}_{y\times{\mathcal T}'}}
\xrightarrow{{\rm d}}\cdots\xrightarrow{{\rm d}} \Omega^{N+1}_{{\mathcal C}_{y\times{\mathcal T}'}} \big)\nonumber\\
\hphantom{\mathbb{C}}{}
 \hookrightarrow
\mathbb{R}^2(p_{y\times{\mathcal T'}})_*\big( {\mathcal O}_{{\mathcal C}_{y\times{\mathcal T}'}}
\xrightarrow{{\rm d}} \Omega^1_{{\mathcal C}_{y\times{\mathcal T}'}/{\mathcal T'}} \big)
\cong{\mathcal O}_{\mathcal T'}. \label{equation: map from de Rham cohomology of total differentials}
\end{gather}

Recall that $(E,\nabla)$ can be extended to the family of integrable connections
$\big(E,\nabla^{\rm flat}\big)$ in~(\ref{global integrable connection analytic local over moduli}).
Then the pullback
$\big(\mathrm{id}\times\tilde{I}_v\big)^*\big(E,\nabla^{\rm flat}\big)$
is a family of integrable connections relative to $\Spec\mathbb{C}[\epsilon]$
whose induced relative connection is
$\big(\mathrm{id}\times\tilde{I}_v\big)^*(E,\nabla)$.
So we can extend the relative meromorphic differential
$v_{\alpha}=B_{\alpha}{\rm d}z$ to a total differential
$v^{\rm flat}_{\alpha}=B_{\alpha}{\rm d}z+\sum_{i=1}^N C^i_{\alpha}{\rm d}\theta_i$
which satisfies the patching condition
\[
 (\mathrm{id}+\epsilon u_{\alpha\beta})\circ\big(\nabla^{\rm flat}+\epsilon v^{\rm flat}_{\beta}\big)
 =
 \big(\nabla^{\rm flat}+\epsilon v^{\rm flat}_{\alpha}\big)\circ(\mathrm{id}+\epsilon u_{\alpha\beta})
\]
on $E|_{U_{\alpha\beta}}\otimes\mathbb{C}[\epsilon]$
and the integrability condition
\[
 \big(\nabla^{\rm flat}+\epsilon v^{\rm flat}_{\alpha}\big)
 \circ\big(\nabla^{\rm flat}+\epsilon v^{\rm flat}_{\alpha}\big)
 =0.
\]
Let
\begin{gather*}
 \nabla^{\rm flat}_{\dag}\colon \
 {\mathcal E}{\rm nd}(E_{y\times{\mathcal T'}})
 \ni u \mapsto \nabla^{\rm flat}\circ u-(u\otimes\mathrm{id})\circ \nabla^{\rm flat}
 \in {\mathcal E}{\rm nd}(E_{y\times{\mathcal T'}})\otimes
 \Omega^1_{{\mathcal C}\times_{\mathcal T}(y\times {\mathcal T}')}(D_{y\times{\mathcal T}'})
\end{gather*}
be the induced connection on ${\mathcal E}{\rm nd}(E|_{{\mathcal C}_{y\times{\mathcal T'}}})$.
Then the above two equalities become
\begin{gather*}
 \nabla^{\rm flat}_{\dag}(u_{\alpha\beta})
 =
 v_{\beta}^{\rm flat}-v_{\alpha}^{\rm flat}, \qquad
 \nabla^{\rm flat}_{\dag}\big(v^{\rm flat}_{\alpha}\big) =0.
\end{gather*}
We can check the equalities
\begin{gather*}
 {\rm d} \Tr\big(u_{\alpha\beta}u'_{\beta\gamma}\big)
 =
 \Tr\big(\nabla^{\rm flat}_{\dag}(u_{\alpha\beta}u'_{\beta\gamma})\big)
 =
 \Tr\big( \nabla^{\rm flat}_{\dag}(u_{\alpha\beta})u'_{\beta\gamma}
 +u_{\alpha\beta}\nabla^{\rm flat}_{\dag}(u'_{\beta\gamma})\big)
 \\
 \hphantom{{\rm d} \Tr\big(u_{\alpha\beta}u'_{\beta\gamma}\big)}{}
 =
 \Tr \big( ( v_{\beta}^{\rm flat}- v_{\alpha}^{\rm flat} ) u'_{\beta\gamma}
 +u_{\alpha\beta} ( v'^{\rm flat}_{\gamma} - v'^{\rm flat}_{\beta}) \big)
 \\
 \hphantom{{\rm d} \Tr\big(u_{\alpha\beta}u'_{\beta\gamma}\big)}{}=
 \Tr \big( v_{\beta}^{\rm flat}u'_{\beta\gamma}
 -v_{\alpha}^{\rm flat} (u'_{\alpha\gamma}-u'_{\alpha\beta})
 + (u_{\alpha\gamma}-u_{\beta\gamma}) v'^{\rm flat}_{\gamma}
 - u_{\alpha\beta} v'^{\rm flat}_{\beta} \big)
 \\
 \hphantom{{\rm d} \Tr\big(u_{\alpha\beta}u'_{\beta\gamma}\big)}{}=
 \Tr \big( {-}u_{\beta\gamma} v'^{\rm flat}_{\gamma}
 +v^{\rm flat}_{\beta} u'_{\beta\gamma} \big)
 -
 \Tr \big( {-}u_{\alpha\gamma} v'^{\rm flat}_{\gamma}
 +v^{\rm flat}_{\alpha} u'_{\alpha\gamma} \big)\\
\hphantom{{\rm d} \Tr\big(u_{\alpha\beta}u'_{\beta\gamma}\big)=}{}
 +\Tr \big( {-}u_{\alpha\beta} v'^{\rm flat}_{\beta}
 +v^{\rm flat}_{\alpha} u'_{\alpha\beta} \big),
\\
 {\rm d}\Tr \big( {-}u_{\alpha\beta} v'^{\rm flat}_{\beta} +v^{\rm flat}_{\alpha} u'_{\alpha\beta} \big) =
 \Tr \big( \nabla^{\rm flat}_{\dag} \big({-}u_{\alpha\beta}v'^{\rm flat}_{\beta}+v^{\rm flat}_{\alpha}u'_{\alpha\beta}\big)\big)
 \\
\hphantom{{\rm d}\Tr \big( {-}u_{\alpha\beta} v'^{\rm flat}_{\beta} +v^{\rm flat}_{\alpha} u'_{\alpha\beta} \big)}{}
 =
 \Tr \big( {-}\nabla^{\rm flat}_{\dag}(u_{\alpha\beta})\wedge v'^{\rm flat}_{\beta}
 -v^{\rm flat}_{\alpha} \wedge \nabla^{\rm flat}_{\dag}(u'_{\alpha\beta}) \big)
 \\
 \hphantom{{\rm d}\Tr \big( {-}u_{\alpha\beta} v'^{\rm flat}_{\beta} +v^{\rm flat}_{\alpha} u'_{\alpha\beta} \big)}{}=
\Tr \big( \big( v^{\rm flat}_{\alpha} - v^{\rm flat}_{\beta} \big) \wedge v'^{\rm flat}_{\beta} \big)
-\Tr \big( v^{\rm flat}_{\alpha} \wedge \big( v'^{\rm flat}_{\beta}-v'^{\rm flat}_{\alpha} \big) \big)
\\
\hphantom{{\rm d}\Tr \big( {-}u_{\alpha\beta} v'^{\rm flat}_{\beta} +v^{\rm flat}_{\alpha} u'_{\alpha\beta} \big)}{} =
-\Tr \big( v^{\rm flat}_{\beta} \wedge v'^{\rm flat}_{\beta} \big)
+\Tr \big( v^{\rm flat}_{\alpha} \wedge v'^{\rm flat}_{\alpha} \big)
\end{gather*}
and
\begin{gather*}
 {\rm d}\Tr\big(v_{\alpha}^{\rm flat}\wedge v'^{\rm flat}_{\alpha}\big)
 =
 \Tr \big( \nabla^{\rm flat}_{\dag}\big(v^{\rm flat}_{\alpha}\wedge v'^{\rm flat}_{\alpha}\big)\big)
 =0.
\end{gather*}
Therefore,
$\big[ \big\{ \Tr(u_{\alpha\beta}u'_{\beta\gamma}) ,
-\Tr\big(u_{\alpha\beta}v'^{\rm flat}_{\beta}+v^{\rm flat}_{\alpha}u'_{\alpha\beta}\big) ,
\Tr\big(v^{\rm flat}_{\alpha}\wedge v'^{\rm flat}_{\alpha}\big) \big\} \big]$
defines an element of
$\mathbb{H}^2\big(\Omega^{\bullet}_{{\mathcal C}_{y\times{\mathcal T}'}}\big)\cong\mathbb{C}$
and its image by the map~(\ref{equation: map from de Rham cohomology of total differentials})
coincides with
\[
\omega\big(\tilde{v},\tilde{v'} \big)
=\big[ \big\{ \Tr(u_{\alpha\beta}u'_{\beta\gamma}) ,
-\Tr(u_{\alpha\beta}v'_{\beta}+v_{\alpha}u'_{\alpha\beta}) \big\} \big].
\]
Thus we have proved the corollary.
\end{proof}

\subsection*{Acknowledgements}

The author would like to thank Professor Takuro Mochizuki for
having a discussion and giving a useful advice.
The author also would like to thank Professor Arata Komyo for
having useful discussions frequently.
The author would like to thank the referee for valuable comments to
improve of the paper.
This work is partially supported by
JSPS Grant-in-Aid for Scientific Research (C) 19K03422.

\LastPageEnding

\end{document}